\documentclass[10pt,a4paper]{amsart}
\usepackage[english]{babel}
\usepackage{amsmath,amsfonts, amsthm, amssymb}
\usepackage{fancyhdr}
\usepackage{setspace}
\usepackage{tikz-cd}
\usepackage{mathrsfs}
\usepackage{stmaryrd}
\usepackage{enumitem}
\usepackage{verbatim}

\title{}
\author{}

\theoremstyle{definition}
\newtheorem{Theo}{Theorem}[section]

\newtheorem{Def}[Theo]{Definition}
\newtheorem{Cor}[Theo]{Corollary}

\newtheorem{Cons}[Theo]{Construction}

\theoremstyle{plain}
\newtheorem{Lemma}[Theo]{Lemma}
\newtheorem{Prop}[Theo]{Proposition}

\theoremstyle{remark}
\newtheorem{Rmk}[Theo]{Remark}

\DeclareMathOperator{\id}{id}

\DeclareMathOperator{\Hom}{Hom}
\DeclareMathOperator{\rank}{rank}
\DeclareMathOperator{\Aut}{Aut}

\newcommand{\inj}{\hookrightarrow}
\newcommand{\surj}{\twoheadrightarrow}

\newcommand{\overbar}[1]{\mkern 1.5mu\overline{\mkern-1.5mu#1\mkern-1.5mu}\mkern 1.5mu}

\renewcommand{\bar}{\overbar}
\renewcommand{\tilde}{\widetilde}

\begin{document}
	\author{Shanxiao Huang}
    \title{Families of Paraboline $(\varphi,\Gamma_K)$-modules}
    \maketitle
    \begin{abstract}
    	Let $p$ be a prime and $K$ be a $p$-adic local field. We study the stack of quasi-deRham $(\varphi,\Gamma_K)$-modules, i.e. $(\varphi,\Gamma_K)$-modules that are deRham up to twist by characters. These objects are used to construct and then study the so called the paraboline varieties, which parametrize successive extensions of quasi-deRham $(\varphi,\Gamma_K)$-modules of a certain type, generalizing the trianguline varieties of \cite{Breuil2017a},\cite{Breuil2017}. On the automorphic side, We construct relative eigenvarieties, and prove the existence of some local-global compatible morphism between them via showing the density of "classical points".
    \end{abstract}
	\tableofcontents
	\pagebreak
	{ \ \ }\\
	\bibliographystyle{alpha}
	\section{Introduction}
	Let $p$ be a prime and let $K$ be a finite field extension of $\mathbb{Q}_p$. The main purpose of this article is to define the \textit{refined paraboline varieties} and \textit{paraboline varieties}, which are generalizations of the \textit{trianguline varieties} and prove some parallel results in \cite{Breuil2017a} and \cite{Breuil2017}. More precisely, we fix an integer $n$ with a partition $n=n_1+\cdots+n_l$, and type a $\tau_i$ of the inertia group $I_K$ of dimension $n_i$ for each $i\in\{1,\dots,l\}$. The refined paraboline varieties and the paraboline varieties are two different compactifications of some rigid spaces, which parametrize $(\varphi,\Gamma_K)$-modules with filtrations such that, for $i=1,\dots,l$, the $i$-th graded pieces are deRham of type $\tau_i$ up to twist by characters (with some extra conditions). We compare the refined paraboline varieties with the paraboline varieties, and study their geometry. Moreover, on the automorphic side, we modified Loeffler's recipe in \cite{Loeffler2011} to construct a Hecke eigenvariety corresponding to a given paraboline variety. The points of such eigenvariety parametrize the $p$\textit{-adic overconvergent eigenforms} (see the definition in \textit{loc. cit.}) of some Hecke algebra, whose $v$-component is the Bernstein center of $\tau_1\otimes\cdots\otimes \tau_l$ for some place $v|p$. Then we give a comparison between the Hecke eigenvarieties and  the paraboline varieties. To further elucidate the motivation, let us briefly recall some results about trianguline varieties in \cite{Kedlaya2012}, \cite{BJG2009}.  
	
    Let $C$ be a finite field extension of $\mathbb{Q}_p$ such that $|\Hom(K,C)|=[K:\mathbb{Q}_p]$. Write $\mathcal{R}_{K,C}$ for the \textit{relative Robba ring over} $C$ \textit{for} $K$, see e.g.  \cite[Def~2.2.2]{Kedlaya2012}. Given a continuous character $\delta:K^{\times}\rightarrow C^{\times}$, one can attach a rank one $(\varphi,\Gamma_K)$-module $\mathcal{R}(\delta)$ over $\mathcal{R}_{K,C}$. Actually, the inverse is true (\cite[Lem~6.2.13]{Kedlaya2012}). Namely, let $D$ be a $(\varphi,\Gamma_K)$-module of rank one over $\mathcal{R}_{K,C}$. Then there exists a unique continuous character $\delta: K^{\times}\rightarrow C^{\times}$ such that
    	$$D\cong \mathcal{R}(\delta).$$
    Let $\delta_i:K^{\times}\rightarrow C^{\times}$ be continuous characters for $i=1,\dots,n$. A $(\varphi,\Gamma_K)$-module $D$ over $\mathcal{R}_{K,C}$ (of rank $n$) is called \textit{trianguline with parameter} $(\delta_1,\dots,\delta_n)$ ({[Def~6.3.1] in \textit{loc. cit.}}) if $D$ admits a filtration 
    \begin{equation} \label{eq:Triangulation}
    	0=D_0\subset D_1\subset\cdots\subset D_n=D
    \end{equation}
	given by $(\varphi,\Gamma_K)$-submodules of $D$ such that $D_i/D_{i-1}$ is a rank one $(\varphi,\Gamma_K)$-module, which is isomorphic to $\mathcal{R}(\delta_i)$ for $i=1,\dots,n$. In other word, a trianguline $(\varphi,\Gamma_K)$-module $D$ can be regarded as a successive extension of $(\delta_1\dots,\delta_n)$.  
	
	Now write $G_K:= \mathrm{Gal}(\bar{K}/K)$ for the absolute Galois group of $K$, and $k_C$ for the residue field of $C$. Fix a continuous representation $\bar{r}: G_K\rightarrow \mathrm{GL}_n(k_C)$ and let $R_{\bar{r}}^{\Box}$ be the framed local deformation ring of $\bar{r}$ (a local complete noetherian $\mathcal{O}_C$-algebra with residue field $k_C$). We write $\mathfrak{X}_{\bar{r}}:=(\mathrm{Spf}R_{\bar{r}}^{\Box})^{\mathrm{rig}}$ for the rigid $C$-space associated to the formal scheme $\mathrm{Spf}R_{\bar{r}}^{\Box}$, then the $C$-points of $\mathfrak{X}_{\bar{r}}$ parametrize the lifts of $\bar{r}$ to $\mathcal{O}_C$. Let $\mathcal{T}_K:= \widehat{K^{\times}}$ be the rigid $C$-space of continuous characters of $K^{\times}$. The \textit{trianguline variety} $X_{\mathrm{tri}}(\bar{r})$ is defined to be the reduced rigid $C$-space, which is the Zariski closure in $\mathfrak{X}_{\bar{r}}\times \mathcal{T}_K^{n}$ of:
	\begin{equation*}
		U_{\mathrm{tri}}(\bar{r}) :=\{(r,\delta)\in\mathfrak{X}_{\bar{r}}\times (\mathcal{T}^{n}_K)_{\mathrm{reg}}\ |\ \mathrm{D}_{\mathrm{rig}}^{\dagger}(r)\mathrm{\ is\ trianguline\ with\ parameter\ }\delta\},
	\end{equation*}
	 where $(\mathcal{T}^{n}_K)_{\mathrm{reg}}$ is a Zariski-open subspace of $\mathcal{T}_K^{n}$ defined for some technical reason, and  $\mathrm{D}_{\mathrm{rig}}^{\dagger}(r)$ denotes the $(\varphi,\Gamma_K)$-module attached to $r$, constructed by Berger in \cite{Berger2002}.
	 
	  An important fact is that, if $r$ is a crystalline representation, then the attached $(\varphi,\Gamma_K)$-module $\mathrm{D}_{\mathrm{rig}}^{\dagger}(r)$ is trianguline (\cite[Prop~2.4.1~and~Rmk~2.4.3]{BJG2009}). Indeed, Bella\"{\i}che and Chenevier show that for a crystalline representation $r$, the set of filtrations of the associated $(\varphi,\Gamma_K)$-module $\mathrm{D}_{\mathrm{rig}}^{\dagger}(r)$ by $(\varphi,\Gamma_K)$-submodules is naturally bijective to the set of filtrations of the associated filtered $\varphi$-module $\mathbf{D}_{\mathrm{cris}}(r)$ by sub filtered $\varphi$-modules, where $\mathbf{D}_{\mathrm{cris}}(-)$ is Fontaine's functor of crystalline periods in $p$-adic Hodge theory. Then one can see that, after perhaps enlarging $C$, one can choose a basis of $\mathbf{D}_{\mathrm{cris}}(-)$ such that the matrix associated to the $\varphi$-action under this basis in is an upper triangular matrix, which implies that $\mathrm{D}_{\mathrm{rig}}^{\dagger}(r)$ is trianguline (this is the reason why such $(\varphi,\Gamma_K)$-module is called trianguline). Therefore for a very general crystalline representation $r$ (with some extra data), one can regard it as a point in the trianguline variety $X_{\mathrm{tri}}(\bar{r})$.
	
	In this article, we want to generalize this approach to the case of deRham representations, and define the paraboline varieties. Let $r:G_K\rightarrow \mathrm{GL}_n(C)$ be a deRham representation. Suppose that $r$ is semi-stable restricted on $G_{L}$ for some Galois extension $L/K$ (recall that a deRham representation $r$ is deRham if and only if it is potentially semi-stable \cite[Cor~5.22]{Berger2002}), and let $\mathbf{D}_{\mathrm{st},L}(r)$ be the associated filtered $(\varphi,N,G_{L/K})$-module, which is defined by Colmez and Fontaine (see e.g. \cite{PierreColmez2000ConstructionDR},\cite{Berger2002}). Thanks to \cite[Cor~III.2.5]{Berger2008}, we have similar results as in \cite[Rmk~2.4.3]{BJG2009} for deRham representations, i.e. the set of filtrations of $\mathbf{D}_{\mathrm{st},L}(r)$ by sub filtered $(\varphi,N,G_{L/K})$-modules is naturally bijective to the set of filtrations of $\mathrm{D}_{\mathrm{rig}}^{\dagger}(r)$ by $(\varphi,\Gamma_K)$-submodules, induced by the fully faithful functor $\mathbf{D}_K$ (see the definition in \ref{FamilyofPhiGamma}, also in \cite[Def~II.2.4]{Berger2008}). It follows that the $(\varphi,\Gamma_K)$-module $D:=\mathrm{D}_{\mathrm{rig}}^{\dagger}(r)$ admits a filtration
	$$0= D_0 \subset D_1 \subset \cdots \subset D_l = D$$
	by $(\varphi,\Gamma_K)$-submodules, such that the filtered $(\varphi,N,G_{L/K})$-modules $\mathbf{D}_K^{-1}(D_i/D_{i-1})$ are irreducible for $i=1,\dots,l$. 
	
	The discussion above motivate us to give the following definition. We call a $(\varphi,\Gamma_K)$-module $D$ over $\mathcal{R}_{K,C}$ \textit{quasi-deRham irreducible} if there exist an irreducible filtered $(\varphi,N,G_{L/K})$-module $M$ and a continuous character $\delta:K^{\times}\rightarrow C^{\times}$ such that
	$$D\cong \mathbf{D}_{K}(M_i)\otimes_{\mathcal{R}_{K,C}} \mathcal{R}(\delta_i).$$
	 We call a $(\varphi,\Gamma_K)$-module $D$ over $\mathcal{R}_{K,C}$ \textit{paraboline} if $D$ admits a filtration 
	$$0=D_0\subset D_1\subset\cdots\subset D_l=D$$
	given by $(\varphi,\Gamma_K)$-submodules such that each graded piece $D_i/D_{i-1}$ is quasi-deRham irreducible (and say $D$ is \textit{with parameter} $(D_1,D_2/D_1,\dots, D_l/D_{l-1})$). Therefore if $r$ is a deRham representation, then $\mathrm{D}_{\mathrm{rig}}^{\dagger}(r)$ is paraboline. We then want to proceed as follows: given a parabolic subgroup $\mathcal{P}\subset \mathrm{GL}_n$, we want to parametrize rank $n$ $(\varphi,\Gamma_K)$-modules with a filtration of type $\mathcal{P}$ whose graded pieces are quasi-deRham irreducible.  
	
	But to define such ``paraboline varieties", we still have a problem to solve. Unlike the trianguline varieties, it is not clear that what is the definition of the parameter spaces, i.e. the space parametrizing the graded pieces. One approach is to compute the moduli space $\mathcal{Z}$ (see section \ref{ParemeterSpaces}) of irreducible quasi-deRham $(\varphi,\Gamma_K)$-module (with some extra trivialization data) with the property that 
	$$\mathcal{Z}(E):=\{\mathrm{quasi}\text{-}\mathrm{deRham\ irreducible\ }(\varphi,\Gamma_K)\text{-module\ } D\mathrm{\ over\ }\mathcal{R}_{K,E}  \}$$	
	for every $E/C$ finite field extension (and for a point $x\in\mathcal{Z}(E)$, write $\mathcal{R}(x)$ for the associated $(\varphi,\Gamma_K)$-module) and use it to define the "paraboline varieties". More precisely, fix a partition $n:=n_1+\cdots +n_l$ and fix a connected component $\mathcal{S}_i$ of $\mathcal{Z}$ with the rank of objects in $\mathcal{S}_i$ equal to $n_i$ and fix a continuous representation $\bar{r}: G_K\rightarrow \mathrm{GL}_n(k_C)$.  We define the \textit{refined paraboline variety} $X_{\mathrm{par}}(\bar{r})$ to be the reduced rigid $C$-space, which is the Zariski closure in $\mathfrak{X}_{\bar{r}}\times \mathcal{S}$ of:
	\begin{equation*}
		U_{\mathrm{par}}(\bar{r}) :=\{(r,x)\in\mathfrak{X}_{\bar{r}}\times \mathcal{S}_{\mathrm{reg}}\ |\ \mathrm{D}_{\mathrm{rig}}^{\dagger}(r)\mathrm{\ is\ paraboline\ with\ parameter\ }x\},
	\end{equation*}   
	where $\mathcal{S}:= \mathcal{S}_1\times\cdots\times\mathcal{S}_l$ and $\mathcal{S}_{\mathrm{reg}}\subset \mathcal{S}$ is some Zariski open subspace of $\mathcal{S}$ for the technical reason similar to the trianguline case. The following theorem (see theorem \ref{Thm:Geoofpar}) describes the geometry of the refined paraboline variety $X_{\mathrm{par}}(\bar{r})$, which is the parallel result of \cite[Thm~2.6]{Breuil2017}.
	
	\begin{Theo}\label{GeoofRefPar}{\ }
		\begin{enumerate}
			\item the rigid space $X_\mathrm{par}(\bar{r})$ is equidimensional of dimension 
			$$[K:\mathbb{Q}_p](\frac{n(n-1)}{2}+l)+n^2;$$
			\item the set $U_{\mathrm{par}}(\bar{r})$ is Zariski open in $X_{\mathrm{par}}(\bar{r})$, hence it is also Zariski dense in $X_{\mathrm{par}}(\bar{r})$;
			\item the rigid space $U_{\mathrm{par}}(\bar{r})$ is smooth. 
		\end{enumerate}
	\end{Theo}

    However, we are motivated to consider another reasonable definition of the "paraboline varieties" for two reasons. Indeed, if $z_1 = (r,x_1,\dots,x_l)$ and $z_2=(r,y_1,\dots,y_l)$ are two points in $U_{\mathrm{par}}(\bar{r})$ with the same underlying Galois representation $r$ such that $\mathcal{R}(x_i)[1/t]\cong \mathcal{R}(y_i)[1/t]$ for each $i$ then $z_1=z_2$ (i.e. $x_i=y_i$ for any $i$). By the construction of the functor $\mathbf{D}_K$, inverting $t$ on the level of deRham $(\varphi,\Gamma_K)$-modules corresponding to forget the filtration for the associated filtered $(\varphi,N,G_{L/K})$-modules. Let $\mathcal{T}_i$ be the Stein space characterized by $\Gamma(\mathcal{T}_i,\mathcal{O}_{\mathcal{T}_i})\cong \Gamma(\mathcal{S}_i,\mathcal{O}_{\mathcal{S}_i})$, and write
    $$\mathrm{pr}_i:\mathcal{S}_i\rightarrow \mathcal{T}_i$$
    be the canonical projection for $i=1,\dots,l$. Then by definition $$\mathcal{S}_i \cong \mathcal{T}_i\times \mathrm{Flag}_i,$$ here $\mathrm{Flag}_i$ is some flag variety encoding the information about filtration, and $\mathrm{pr}_i$ is the projection to the first factor. It follows from the discussion above that (See subsection \ref{ParemeterSpaces}) if $z_1 = (r,(x_i)_i)$ and $z_2=(r,(y_i)_i)$ are two points in $U_{\mathrm{par}}(\bar{r})$ with the same $r$ such that $\mathrm{pr}_i(x_i) = \mathrm{pr}_i(y_i)$ for each $i$, then $z_1=z_2$ (then we can say $z_1$ is \textit{paraboline with parameter} $(\mathrm{pr}_1(x_1),\dots,\mathrm{pr}_l(x_l))$). Hence the induced projection ($\mathcal{T} := \mathcal{T}_1\times\cdots\times\mathcal{T}_l$)
    $$\Xi:\mathfrak{X}_{\bar{r}}\times \mathcal{S} \rightarrow \mathfrak{X}_{\bar{r}}\times \mathcal{T}$$
    restricted on $U_{\mathrm{par}}$ is injective, which indicates that $\mathcal{T}$ may be another considerable candidate for the parameter space. We write $V_{\mathrm{par}}$ for the image of $U_{\mathrm{par}}$ under $\Xi$, and define the \textit{paraboline variety} $Y_{\mathrm{par}}(\bar{r})$ to be the reduced rigid space, which is the Zariski closure of $V_{\mathrm{par}}(\bar{r})$ in $\mathfrak{X}_{\bar{r}}\times \mathcal{T}$. We prove the following crucial comparison theorem (see remark \ref{RefinedParandPar}), which in particular implies that the geometry of the set $U_{\mathrm{par}}(\bar{r}) \cong V_{\mathrm{par}}(\bar{r})$ is intrinsic and does not depend on the ambient spaces.
    
    \begin{Theo}
    		The map
    			 $$\Xi:X_{\mathrm{par}}(\bar{r}) \rightarrow Y_{\mathrm{par}}(\bar{r})$$ is proper birational and surjective with the property that
    			 $\Xi^{-1}(V_{\mathrm{par}}(\bar{r})) = U_{\mathrm{par}}(\bar{r}) $ and $\Xi|_{U_{\mathrm{par}}(\bar{r})}$ is an isomorphism onto $V_{\mathrm{par}}(\bar{r})$.
    \end{Theo}		
     We can derive similar geometric properties of $Y_{\mathrm{par}}(\bar{r})$ as those in theorem \ref{GeoofRefPar}:
    		\begin{enumerate}
    			\item the rigid space $Y_\mathrm{par}(\bar{r})$ is equidimensional of dimension 
    			$$[K:\mathbb{Q}_p](\frac{n(n-1)}{2}+l)+n^2;$$
    			\item the set $V_{\mathrm{par}}(\bar{r})$ is Zariski open in $Y_{\mathrm{par}}(\bar{r})$, hence it is also Zariski dense in $Y_{\mathrm{par}}(\bar{r})$;
    			\item the rigid space $V_{\mathrm{par}}(\bar{r})$ is smooth.
    		\end{enumerate}
    	 
    Another reason to define the paraboline variety $Y_{\mathrm{par}}(\bar{r})$ is for the seek to compare with a corresponding Hecke eigenvariety $D(V)$ on the automorphic side (see the definition of $D(V)$ in construction \ref{ConsforEigenVar}), which generalizes the results in \cite{Breuil2017a}, \cite{Breuil2017}. Roughly speaking, $D(V)$ is the eigenvariety constructed by Loeffler's work in \cite{Loeffler2011} (with some modification as the Hecke algebras we consider about is larger than that in \textit{loc. cit.}) for the following setting. We fix a unitary group $\mathcal{G} $ for some CM pair $(F,F^{+})$, a parabolic subgroup $\mathcal{P}$ of $\mathcal{G}\times_{F^{+}}F^{+}_p$, a weight $W$ (i.e. an irreducible algebraic representation of $\mathcal{G}(F^{+}\otimes_{\mathbb{Q}}\mathbb{R})$) and a supercuspidal Bernstein component $[\sigma]$ of $\mathcal{M}$, the Levi of $\mathcal{P}$ (and $V$ is a representation of a maximal compact subgroup of $\mathcal{M}(F^{+}_p)$ induced by $W$ and $\sigma$, see subsection \ref{Subsect:Notations}). Then $D(V)$ interpolate Hecke eigenvalues on automorphic representation $$\pi:=\otimes'_v\pi_v$$ for the unitary $\mathcal{G}$ of weight $W$ whose local component at $p$ corresponds (via the local Langlands correspondence) to a direct sum of irreducible Weil-Deligne representations of the types corresponding to the fixed Bernstein component $[\sigma]$. Clearly, $\mathcal{T}$ is the better parameter space of ``paraboline varieties" for this purpose: in this setup, the ``parameter space" of the eigenvariety parametrize Hecke eigenvalues, i.e. functions, hence it should be a Stein space (that is, it must be controlled by its global section). Moreover, the local Langlands correspondence attaches to $\pi_v$ a Weil-Deligne representation rather than a filtered Weil-Deligne representation (or equivalently, a filtered $(\varphi,N,G_{L/K})$-module), which gives another motivation why in the comparison of paraboline varieties with eigenvarieties, the parameter space should forget the information about filtration. By the construction, $D(V)$ is a reduced closed rigid subspace of
    $$\mathfrak{X}_{\bar{\rho}}\times\prod_{v\in S_p} \mathcal{T}_v$$
    where $\bar{\rho}:\mathrm{Gal}(\bar{F}/F)\rightarrow \mathrm{GL}_n(k_C)$ is some absolutely irreducible representation, $\mathfrak{X}_{\bar{\rho}}:=(\mathrm{Spf}(R_{\bar{\rho},S}))^{\mathrm{rig}}$ for $R_{\bar{\rho},S}$, the complete local $\mathcal{O}_C$-algebra with residue field $k_C$ pro-representing the functor of deformations $\rho$ of $\bar{\rho}$ satisfies some extra conditions and  $S_p$ is the set of finite places of $F$ dividing $p$ and $\mathcal{T}_v$ the parameter space of the paraboline variety $Y_{\mathrm{par}}(\bar{\rho}_v)$ for $\bar{\rho}_v := \bar{\rho}|_{G_{F_{\tilde{v}}}}$ ($\tilde{v}$ is a place of $F$ dividing $v$). Then one may view a classical automorphic form $f$ (with some extra technical conditions) as a point $z=(\rho,(x_v)_{v\in S_p})$ of such an eigenvariety, such that the Galois representation $\rho: \mathrm{Gal}(\bar{F}/F) \rightarrow \mathrm{GL}_n(\bar{\mathbb{Q}}_p)$ restricted on $G_{F_{\tilde{v}}}$ is paraboline with parameter $x_v$ at every place $v$ of $F$ dividing $p$ (see the first part of Theorem \ref{ComparisonTheorem}). Then we prove the density theorem for classical points (see theorem \ref{VCNareZarDense}, and see definition \ref{Def:ClPt} for the definition of classical points)
    
    \begin{Theo}
    		The set of dominant, very regular, classical, non-critical points in $D(V)$  is Zariski dense.
    \end{Theo}  
    
    A very important application of the theorem above is that we can construct a morphism from the eigenvariety $D(V)$ to the paraboline variety $Y_{\mathrm{par}}(\bar{r})$, i.e. (see theorem \ref{ComparisonTheorem}):
    
    \begin{Cor}
    	The natural map
    	\begin{align*}
    		\mathfrak{X}_{\bar{\rho}}\times\prod_{v\in S_p} \mathcal{T}_v & \rightarrow \prod (\mathfrak{X}_{\bar{\rho}_v}\times \mathcal{T}_v)\\
    		(\rho,(x_v)_v) & \mapsto (\rho|_{G_{F_{\tilde{v}}}},j_v'(x_v))_v
    	\end{align*}
    induces a map (via restriction)
    	$$D(V) \rightarrow \prod\limits_{v\in S_p} Y_{\mathrm{par}}(\bar{\rho}_v),$$
    	here $j_v'$ is an isomorphism of $\mathcal{T}_v$ induced by twisting by a character (see the definition in subsection \ref{Subsect:Density}).
    \end{Cor}
     	
    Finally, we remark that in \cite{ChristopheBreuil2021BernsteinE}, Christophe Breuil and Yiwen Ding prove results that parallel the second part of remark \ref{RefinedParandPar} (i.e. the description of the geometry of $Y_{\mathrm{par}}(\bar{r})$), proposition \ref{CharOfDomClPt}, theorem \ref{ComparisonTheorem}. 
    
    \textbf{Acknowledgment.} 
    
    I would like to express my gratitude to my supervisor, Eugen Hellmann, who guided me throughout this project, and also be grateful for the support by Germany's Excellence Strategy EXC 2044-390685587 ``Mathematics Münster: Dynamics-Geometry-Structure" and by the CRC 1442 ``Geometry: Deformations and rigidity" of the DFG (German Research Foundation).

	\textbf{Notation and convention.} 
	
	Through out this article, we fix $p$ a prime number. 
	
	For any local field $E$ over $\mathbb{Q}_p$, let $G_E$ denote the absolute Galois group of $E$, let $W_E$ denote the Weil group of $E$, let $I_E$ denote the inertial group of $G_E$ and let $k_E  $ denote the residue field of $E$ with cardinality $q_E$. Let $G_{E'/E} := \mathrm{Gal}(E'/E)$ for any Galois extension $E'$ over $E$, and let $E_0:= W(k_E)[1/p]$ (i.e. the maximal unramified sub field of $E$). We normalized the local Artin map $\mathrm{Art}: W_E^{\mathrm{ab}} \rightarrow E^{\times}$ such that $\mathrm{Art}(\mathrm{Fr}^{-1}) = \varpi_E$, where $\mathrm{Fr}^{-1}$ is the geometric Frobenius (i.e., the induced action on the residue field is $x\mapsto x^{1/q_E}$) and $\varpi_E$ is some uniformizer in $E$.
	
	We denote $\mathbb{Z}_{+}^{d} :=\{(k_1\leq \cdots \leq k_d)\in\mathbb{Z}^{d}\}$ and $\mathbb{Z}_{++}^{d} :=\{(k_1< \cdots < k_d)\in\mathbb{Z}^{d}\}$.
	
	
	 In this article, we will define many objects consisting of a module (or a vector space) with some extra structures, for convention, we use the notation of the module (or the vector space) to represent the object if there is no ambiguity. For a group acting on an object, we also use the notation of an element of this group to represent the corresponding action for short.
	
	\section{Stack of Quasi-deRham $(\varphi,\Gamma_K)$-modules}
	
	Through out this section, we fix a $p$-adic local field $K$. Let letter $L$ will always denote some finite Galois extension over $K$, and the letter $C$ will always denote a field over $\mathbb{Q}_p$.
	
	\subsection{Decomposition of Weil-Deligne Representations}\label{Subsect:DecompofWDRep}
	In this and the subsection, we assume $C$ is algebraically closed, and let $A$ denote some $C$-algebra. We write $q:= q_K$ for the cardinality of the residue field of $K$, and write $I:= I_K$ for the inertial group of $K$.

	\begin{Def}
		A \textit{Weil-Deligne representation $(V,\varrho,N)$ of $K$ over $A$} (also called a \textit{WD representation} for short) consists of a finitely generated projective $A$-module $V$, and a group homomorphism $\varrho$ from $W_K$ to $\Aut_A(V)$, which is trivial on some open subgroup $I_{L}\subseteq W_L\subseteq W_K$ for some finite field extension $L$ over $K$ (in this case, we say that $\varrho$ is unramified over $L$, and moreover we say $V$ is $\varrho$-\textit{unramified} if $\varrho$ is unramified over $K$), and an $A$-linear endomorphism $N$ of $V$ such that $q^{\parallel g\parallel_K}gN = Ng$ for any $g$ in $W_K$. Here the group homomorphism $\parallel \cdot \parallel_K : W_K \rightarrow \mathbb{Z}$ is characterized by $\bar{g}: \bar{\mathbb{F}}_p \rightarrow \bar{\mathbb{F}}_p, x \mapsto x^{q^{-\parallel g\parallel_K}}$. 
		
		Moreover, a \textit{ morphism $\alpha:(V,\varrho,N) \rightarrow (V',\varrho',N')$ of Weil-Deligne representations} is an $A$-homomorphism $\alpha:V\rightarrow V'$ with $\alpha\circ N = N'\circ \alpha$, and $\alpha\circ \varrho(g) = \varrho'(g)\circ \alpha$ for any $g\in W_K$, we denote the set of morphisms by $\Hom_{\mathrm{WD}}(V_1,V_2)$. 
	\end{Def}

    \begin{Rmk}
    	Recall that, for a Weil-Deligne representation $(V,\varrho,N)$, people usually say $V$ is unramified if $N=0$ and $\varrho|_{I_K}$ is trivial. Hence we use the name $\varrho$-unramified to discriminate these different definition (actually, we will never say $V$ is unramified in this article).
    \end{Rmk}
	
	\begin{Def}
		Let $\mathfrak{W}^{L/K,d}$ denote the category fibered in groupoids on $C$-schemes that assign to $A$ the groupoid of rank $d$ Weil-Deligne representations $(V,\varrho,N)$ of $K$ with $\varrho$ unramified over $L$.
	\end{Def}

	\begin{Lemma}\label{Lem:Idecomp}
		Let $V$ be a WD representation of $K$ over $A$, one has the following decomposition as $I$-representation over $A$:
		$$V = \bigoplus_{\tau} V_\tau,$$
		here $\tau$ runs through all the isomorphism classes of irreducible smooth $I$-representation over $C$ and $V_\tau \cong \tau\otimes_C U_{\tau}$ for some finitely generated projective $A$-module $U_\tau$. Moreover, each $V_\tau$ is stable under the $N$-action (but not stable under $W_K$-action in general).
	\end{Lemma}
	\begin{proof}
		As the restriction of $\varrho$ to $I$ factors through some finite quotient $I/H$ of $I$, hence $\varrho$ can be regarded as a representation of $I/H$. In addition, $V$ can be regarded as an $C$-vector space with linear $(I/H)$-action. Hence we get such decomposition. 
		
		To see that $V_\tau$ is an $A$-module, one can rewrite $V_\tau$ as the image of the natural injection:
		\begin{align*}
		i_{\tau}:\tau \otimes_C \Hom_I(\tau,V)  &\hookrightarrow V\\
			v\otimes f &\mapsto f(v)
		\end{align*}
		and note that the natural $A$-module structure of $\Hom_I(\tau,V)$, which is defined by $$r(f)(v):= rf(v)$$ for any $f\in \Hom_I(\tau,V)$, is compatible with the $A$-module structure of $V$. Hence $V_\tau$ is an $A$-submodule of $V$ ($V_\tau$ is projective as it is a direct summand of $V$). 
		
		Moreover, if we let $N(f)(v) := N(f(v))$, then for any $g\in I$, one has $$N(f)(g(v)) = N\circ f \circ g (v) = N\circ g\circ f(v) = g\circ N\circ f (v).$$
		Hence $N(f)\in\Hom_I(\tau,V)$. This defines an $N$-structure of $\Hom_I(\tau,V)$ such that the natural injection $\tau \otimes \Hom_I(\tau,V)  \hookrightarrow V$ is $N$-equivariant.
		
		 It follows that for any $v\in V_\tau$, which can be expressed as
		$\sum\limits_{t\in T} f_t(v_t) $ (for a finite index set $T$ and $f_t \in \Hom_I(\tau,V)$ and $v_t\in \tau$), one has $$N(v) =\sum\limits_{t\in T} N(f_t)(v_t).$$ Hence $N(v)$ is in $V_\tau$, which implies $V_\tau$ is $N$-stable.
	\end{proof}
	
	\begin{Lemma}\label{Lem:extprop}
		Let $(V,\tau)$ be an irreducible smooth representation of $I$ over $C$. And set 
		$$W_\tau = \{g\in W_K | g^{-1}\tau g \sim \tau \}$$
		Then $W_\tau/I\cong \mathbb{Z}$. Up to scalar, $\tau$ uniquely extends to a smooth representation $\tilde{\tau}$ of $W_\tau$ (i.e. after choosing some $g\in W_\tau$ which is a generator of $W_\tau/I$, all the possible choice of $\tilde{\tau}(g)$ is differed by scalars).
	\end{Lemma}

	\begin{proof}
		To show $W_\tau/I\cong \mathbb{Z}$, we only need to show that $W_\tau\neq I$ as $W_K/I \cong \mathbb{Z}$.
		 
		Let $n := [I:\ker \tau]$. Then for any element $g\in W_K\setminus I$, the conjugate action of $g^{n!}$ is trivial on $I/\ker\tau$, hence $g^{n!} \in W_\tau$.
		
		One can choose $g\in W_\tau$ which is a generator of $W_\tau/I$. Then there exists some $A\in GL(V)$ such that for any $h\in I$, $\tau (g^{-1}hg) = A^{-1}\tau(g)A$. As $\tau$ is irreducible, such $A$ is unique up to scalar. Hence, up to scalar, one can uniquely extend $\tau$ to $W_\tau$ by requiring $\tilde{\tau} (g) = A$.
	\end{proof}
	
	\begin{Rmk}\label{taupart}
	From now on, for every irreducible smooth representation $\tau$ of $I$ over $C$, fix an extension $\tilde\tau$ to $W_\tau$. We denote $e_\tau := [W_K:W_\tau]$ and denote by $K_\tau$ the unique unramified extension of $K$ of degree. It is easy to see that $W_\tau = W_{K_\tau}$ (i.e. $W_\tau$ is the Weil group of $K_\tau$).
	
	For any Weil-Deligne representation $(V,\varrho,N)$ of $K$ over $A$, from the above discussion, one can regard $\Hom_I(\tilde{\tau},V)$ as a Weil-Deligne representation of $K_\tau$ over $R$, via defining $ g(f) := g\cdot f\cdot g^{-1}$, and $N(f) := N\cdot f$ for any $g\in W_\tau$, and $f\in \Hom_I(\tilde{\tau},V)$. It is indeed a Weil-Deligne representation as $$N\circ g(f) = N\circ g \circ f \circ g^{-1} = q^{\parallel g\parallel_K} g\circ N \circ f\circ g^{-1} =q_{K_\tau}^{\parallel g\parallel_{K_\tau}}g\circ N(f)$$
	(here $q_{K_\tau}$ is the cardinality of $K_\tau$ and note that $q^{\parallel g\parallel_K}= q_{K_\tau}^{\parallel g\parallel_{K_\tau}}$). 
	Moreover, $\Hom_I(\tilde{\tau},V)$ is $\varrho$-unramified as it restricts on $I_{K_{\tau}} = I$ is trivial.

	 If we regard $\tilde\tau$ as a WD representation of $K_\tau$ over $C$, the injection (which is the same map in the proof of lemma \ref{Lem:Idecomp} as the left hand side is the same as a vector space)
	$$i_\tau:\tilde\tau \otimes_C \Hom_I(\tilde\tau,V)  \hookrightarrow V$$
	is a morphism of WD representations of $K_\tau$ over $A$.
	\end{Rmk}

    \begin{Lemma}\label{inversefordecomp}
    	Assume $\tau$ is a smooth irreducible representation of $I$ over $C$. Let $(V,\varrho,N)$ be a Weil-Deligne representation of $K_\tau$ over $A$ such that $V$ is $\varrho$-unramified. The canonical map
    	\begin{align*}
    		h:V\rightarrow & \Hom_I(\tilde{\tau},\tilde{\tau}\otimes_C V)\\
    		v & \mapsto (w\mapsto w\otimes v)
    	\end{align*}
    is an isomorphism of Weil-Deligne representations of $K_\tau$.
    \end{Lemma}

    \begin{proof}
    	As $\tilde{\tau}$ is a finite vector space, then 
    	$$\Hom_I(\tilde{\tau},\tilde{\tau}\otimes_C V) \cong (\Hom_C(\tilde{\tau},\tilde{\tau}\otimes_C V))^I \cong (\Hom_C(\tilde{\tau},\tilde{\tau})\otimes_C V)^I.$$
    	Note that $I$ acts trivially on $V$, hence
    	$$(\Hom_C(\tilde{\tau},\tilde{\tau})\otimes_C V)^I\cong (\Hom_C(\tilde{\tau},\tilde{\tau}))^I\otimes_C V\cong (\Hom_I(\tilde{\tau},\tilde{\tau})\otimes_C V)\cong V.$$
    	One can check that the isomorphism above coincides with the canonical map $h$, and finish the proof.
    \end{proof}
	
	For now on, for every $W_K$-conjugacy class $[\tau]$ of irreducible smooth representation of $I$ over $C$, we pick and fix a representative $\tau$ of this class, and denote by $\Omega$ the set of the collection of $\tau$ for all the  $W_K$-conjugacy classes $[\tau]$. And let 
	$$\Omega_L:= \{\tau\in \Omega\ |\ \tau|_{I_L} \mathrm{\ is\ trivial} \}.$$ 
	
	\begin{Prop}\label{Prop:WD-decomp}
		Let $V$ be a Weil-Deligne representation over $A$. Then one has a canonical isomorphism	$$V \cong \bigoplus_{\tau\in\Omega} \mathrm{Ind}_{W_\tau}^{W_K}(\tilde{\tau}\otimes_C \Hom_I(\tilde{\tau}, V))$$  
		of Weil-Deligne representations of $K$ over $R$. 
		
		Moreover, if we denote the WD representation $\Hom_I(\tilde{\tau}, V)$ of $K_\tau$  by $V(\tilde{\tau})$, then for any two WD representations  $V_1,V_2$ of $K$, one has
		$$\Hom_{\mathrm{WD}}(V_1,V_2) = \bigoplus_{\tau\in\Omega}\Hom_{\mathrm{WD}}(V_1(\tilde{\tau}),V_2(\tilde{\tau}))$$
	\end{Prop}
	
	\begin{proof}
		By the discussion in remark \ref{taupart}, one has $i_\tau \in \Hom_{\mathrm{WD}}(\tilde{\tau}\otimes_C \Hom_I(\tilde{\tau}, V),V|_{W_\tau})$. Applying the left adjoint functor $\mathrm{Ind}^{W_K}_{W_\tau}$ of $(-)|_{W_\tau}$, one has the following morphism of WD representations:
		\begin{align*}
			\iota_{\tau}: \mathrm{Ind}_{W_\tau}^{W_K}(\tilde{\tau}\otimes_C \Hom_I(\tilde{\tau}, V)) & \rightarrow V\\
			 g\otimes v \otimes f & \mapsto g\circ f(v).
		\end{align*}
	
	  For any $\tau$ and $g\in G$, denote by $\tau^g$ the same space but with $I$-action twisted by $g$, namely, for any $v\in \tau$ and $h\in W_\tau$, set $h(v^g) := g^{-1}hg(v)$ (here $v^g$ denotes the same element as $v$ but with twisted $W_\tau$-action). We define 
	  \begin{align*}
	  	\Hom_I(\tau,V) & \rightarrow \Hom_I(\tau^{g},V)\\
	  	           f   & \mapsto f^{g} := (v^{g}\mapsto g\circ f(v))
	  \end{align*}
      $f^g$ is indeed an element in $\Hom_I(\tau^{g},V)$ as
      $$h^{-1}f^{g}h(v^{g}) = h^{-1}f^{g}(g^{-1}hg(v)) = g(g^{-1}h^{-1}g)f(g^{-1}hg)(v)=gf(v)=f^{g}(v^{g})$$
      for any $v\in \tau$ and $h\in I$. This map is bijective as one can easily construct the inverse map directly. Then consider the map
	  \begin{align*}
	  	\psi: {\mathrm{Ind}_{W_\tau}^{W_K}(\tilde{\tau}\otimes_C \Hom_I(\tilde{\tau}, V))} &\rightarrow \bigoplus\limits_{\tau'\in[\tau]} \tau'\otimes_C \Hom_I(\tau',V)\\
	  	g\otimes v\otimes f& \mapsto v^{g}\otimes f^{g}.
	  \end{align*}
      Note that if we choose a set $\{g_1,\dots,g_{e_\tau}\}$ of representatives of $W_K/W_{\tau}$, then the right side of the map above is $\bigoplus\limits_{1\leq i\leq e_{\tau}} \tau^{g_i}\otimes_C \Hom_I(\tau^{g_i},V)$. It follows that $\psi$ is a bijection map of sets. 
      	  
	  Let $\tau\in \Omega$, note that the following diagram:
	  \begin{equation*}
	  	\begin{tikzcd}
	  		 {\mathrm{Ind}_{W_\tau}^{W_K}(\tilde{\tau}\otimes_C \Hom_I(\tilde{\tau}, V))} \arrow[dd, "\psi"'] \arrow[r, "\iota_{\tau}"] &  V\\
	  		&                                                                                                                    \\
	  		{\bigoplus\limits_{\tau'\in[\tau]} \tau'\otimes_C \Hom_I(\tau',V)} \arrow[ruu,"(i_{\tau'})_{\tau'\in[\tau]}"', hook] &                                                       
	  	\end{tikzcd}
	  \end{equation*}
	  commutes as $$\left((i_{\tau'})_{\tau'}\right)\circ\psi (g\otimes v\otimes f) = i_{\tau^{g}}(v^{g}\otimes f^{g})=g\circ f(v) = \iota_{\tau}(g\otimes v\otimes f).$$ Hence $\iota_{\tau}$ is injective with image $\bigoplus\limits_{\tau'\in[\tau]} V_{\tau'}$. It follows from lemma \ref{Lem:Idecomp} that the map 
	  $$\iota = (\iota_\tau)_{\tau\in\Omega} :  \bigoplus_{\tau\in\Omega}\mathrm{Ind}_{W_\tau}^{W_K}(\tilde{\tau}\otimes_C \Hom_I(\tilde{\tau}, V))  \rightarrow V$$ 
	  is an isomorphism.
	  
		For the moreover part, consider the functor 
		\begin{align*}
			\mathcal{F}_\tau: \{\mathrm{WD\ representaions\ of\ }W_K\} &\rightarrow \{\varrho \mathrm{-unramified\ } \mathrm{WD\ representaions\ of\ }W_\tau \}\\
			V&\rightarrow \mathcal{F}_\tau(V):= V(\tilde{\tau})
		\end{align*}
	   and let $\mathcal{F}:= \bigoplus\limits_{\tau\in\Omega} \mathcal{F}_\tau$. We claim that the functor 
	   $$\mathcal{G}: (U_\tau)_{\tau\in\Omega}\mapsto \bigoplus_{\tau\in\Omega}\mathrm{Ind}_{W_\tau}^{W_K}(\tilde{\tau}\otimes_C U_\tau)$$
	   is the inverse of $\mathcal{F}$, where $U_\tau$ is a $\varrho$-unramified WD representation of $W_\tau$. Actually, what we have proved shows that $\mathcal{G}\circ\mathcal{F} \cong \id$. Note that for any $\tau$,
	   $$\mathrm{Ind}_{W\tau}^{W_K}(\tilde{\tau}\otimes_C U_\tau) \cong \bigoplus_{g\in W_K/W_\tau}\tilde{\tau}^{g}\otimes_C U_\tau$$
	   as WD representations over $K_{\tau}$.
	   And for any $\tau'\ncong \tau\in\Omega$, or $g\in W_K\setminus W_\tau$,
	   $$\Hom_I (\tilde{\tau},(\tilde{\tau}')^{g}\otimes_C U_\tau) = 0,$$
	   as $\tau\ncong (\tau')^{g}$ by the definition of $W_\tau$. Combined with the conclusion in lemma \ref{inversefordecomp}, for any smooth irreducible $I$-representation $\tau_0$, the natural map 
	   \begin{align*}
	   U_{\tau_0}&\rightarrow \Hom_{I}(\tilde{\tau}_0, \bigoplus_{\tau\in\Omega}\mathrm{Ind}_{W_\tau}^{W_K}(\tilde{\tau}\otimes_C U_\tau)) \\
	   v &\mapsto (w\mapsto w\otimes v)
	   \end{align*}
       is a bijection. This means we have a natural isomorphism $\id \xrightarrow{\sim} \mathcal{F}\circ\mathcal{G}$. We conclude that $\mathcal{F}$ is an equivalence of categories, and hence
       $$\Hom_{\mathrm{WD}}(V_1,V_2) = \bigoplus_{\tau\in\Omega}\Hom_{\mathrm{WD}}(V_1(\tilde{\tau}),V_2(\tilde{\tau}))$$
	   
 	\end{proof}
    \begin{Theo}\label{Thm:WDstack}
For any irreducible smooth representation $\tau$ of $I$ over $C$, denote by $d_\tau$ the dimension of the underlying vector space. Then there is an isomorphism 
$$\Theta: \mathfrak{W}^{L/K,d} \xrightarrow{\sim}\coprod_{\{(n_\tau)_{\tau\in\Omega_L}|\star\}} \prod_{\tau\in\Omega_L} \mathfrak{W}^{K_\tau/K_\tau,n_\tau}  $$
as $C$-groupoid, here the $\star$-condition in the index set is
$$\sum_{\tau\in\Omega_L} n_\tau d_\tau e_\tau = d.$$
    \end{Theo}
	\begin{proof}
		Let $A$ be a connected $C$-algebra, and let $(V,\varrho,N)$ be a Weil-Deligne representation in $\mathfrak{W}^{L/K,d}(A)$. Then the map $\Theta (V) := (\Hom_I(\tilde{\tau},V))_{\tau\in\Omega_L}$ gives the desired isomorphism by the proposition above. 
	\end{proof}
	\begin{Rmk}
		\label{WArtin}
		Follow the notation in \cite[Theorem~3.2]{Hartl2013}, we can see that $\mathfrak{W}^{K/K,d}$ is isomorphic to $[P_{K_0,d}/\mathrm{Res}_{K_0/\mathbb{Q}_p}(\mathrm{GL}_{d,K_0})]\otimes_{\mathrm{Spec}{(\mathbb{Q}_p)}} \mathrm{Spec}(C)$. Hence it is an Artin stack  of dimension $0$ and consequently so is $\mathfrak{W}^{L/K,d}$.
	\end{Rmk}

	\subsection{Filtration on Weil-Deligne Representations}
	
	It is a well known result \cite{Breuil2007} that the category of $(\varphi,N,G_K)$-modules is equivalent to the category of Weil-Deligne representations of $K$. Now we want to give a good description of the stack of filtered $(\varphi,N,G_K)$-module (see the definitions below). We do this by generalizing this result to the category of filtered $(\varphi,N,G_K)$-modules via equipping some filtration structures on Weil-Deligne representations, as we have decomposed the stack of Weil-Deligne representations into the stacks of Weil-Deligne representations of trivial inertial type, which are equivalent to the stacks of $(\varphi,N)$-modules, and the latter one are well studied in \cite{Hartl2013}.
	
	\begin{Def} Let $E$ be a field extension over $\mathbb{Q}_p$.
		\begin{enumerate}[leftmargin=2em,label=(\alph*)]
			\item A \textit{locally free $E$-module $M$ over $A$} is a finitely generated $(E\otimes_{\mathbb{Q}_p}A)$-modules, which is Zariski locally on $\mathrm{Spec}A$ free over $E\otimes_{\mathbb{Q}_p}A$.
			\item A \textit{filtered locally free $E$-module $(M,\mathcal{F}^\bullet)$ over $A$} consists of an $E$-free module $M$ over $A$ with a decreasing separated and exhaustive $\mathbb{Z}$-filtration $\mathcal{F}^\bullet$ on $M$ by $(E\otimes_\mathrm{\mathbb{Q}_p}A)$-submodules such that $\mathrm{gr}_\mathcal{F}^i M := \mathcal{F}^{i}M/\mathcal{F}^{i+1}M$ is locally free as an $A$-module. 
			
			A \textit{morphism $\alpha: (M,\mathcal{F}^\bullet) \rightarrow (M',\tilde{\mathcal{F}}^\bullet)$ of filtered $E$-free modules} is a $(E\otimes_{\mathbb{Q}_p} R)$-homomorphism $\alpha: M \rightarrow M'$ such that $\alpha(\mathcal{F}^iM)\subseteq \tilde{\mathcal{F}}^i M'$ for any $i\in\mathbb{Z}$.
			\item A \textit{$(\varphi,N,G_{L/K})$-module $(M,\varphi_M,N,\gamma)$ over $A$} consists of an $L_0$-free module $M$ over $R$, and a Frobenius-linear map $\varphi_M: M \rightarrow M$, and an $(L_0\otimes_{\mathbb{Q}_p}A)$-linear map $N:M\rightarrow M$ and a group homomorphism $\gamma: G_{L/K} \rightarrow \Aut_{A}(M)$, such that
			    \begin{enumerate}[leftmargin=3em,label=(\arabic*)]
			    \item $\gamma(g)((l\otimes 1)\cdot m) = (g(l)\otimes 1)\cdot g(m)$,
				\item $\varphi_M$ is a bijection,
				\item $N\circ\varphi_M$ = $p\cdot\varphi_M\circ N$,
				\item $N\circ \gamma(g)$ = $\gamma(g)\circ N$ for any $g \in G_{L/K}$.
			    \end{enumerate}
		    
		    	If $L=K$, we call $M$ a $(\varphi,N)$-\textit{module over} $A$ for short.
		    
			A \textit{morphism $\alpha$: $(M,\varphi_M,N,\gamma) \rightarrow (M',\varphi_{M'},N',\gamma')$ of $(\varphi,N,G_{L/K})$-modules}  is an $(L_0\otimes_{\mathbb{Q}_p}A)$-homomorphism $\alpha: M\rightarrow M'$ with
			\begin{enumerate}[leftmargin=3em,label=(\arabic*)]
				\item $\alpha\circ\varphi_{M} =  \varphi_{M'}\circ\alpha$,
				\item $\alpha\circ N = N'\circ\alpha$,
				\item $\alpha\circ \gamma(g)= \gamma'(g)\circ\alpha$ for any $g\in G_{L/K}$.
			\end{enumerate}	
		\item A \textit{filtered $(\varphi,N,G_{L/K})$-module $(M,\mathcal{F}^\bullet)$ over $A$} consists of a $(\varphi,N,G_{L/K})$-module $M$ over $A$ together with a filtered locally free $L$-module $(M_L,\mathcal{F}^\bullet)$ over $A$ (here $M_L:= M\otimes_{L_0}L$) such that
	    $g(\mathcal{F}^iM_L)\subseteq (\mathcal{F}^iM_L)$ for any $g \in G_{L/K}$ and $i\in \mathbb{Z}$. 
	    
	    	If $L=K$, we call $M$ a \textit{filtered} $(\varphi,N)$-\textit{module over} $A$ for short.
	    
	    A \textit{morphism $\alpha$: $(M,\mathcal{F}^\bullet) \rightarrow (M',\tilde{\mathcal{F}}^\bullet)$ of filtered $(\varphi,N,G_{L/K})$-modules}  is a morphism of $(\varphi,N,G_{L/K})$-modules such that $(\alpha\otimes \id_{L})(\mathcal{F}^iM_L)\subseteq \tilde{\mathcal{F}}^iM'_L$. 
		\end{enumerate}
 	    \end{Def}
   
   \begin{Def} 
   	We define the following categories fibered in groupoids on $C$-schemes:
   	\begin{enumerate}[leftmargin=1.5em]
   		\item denote by $\mathrm{M}^{K}$ the category that assigns to $A$ to the groupoid of locally free $K$-modules;
   		\item denote by $\mathrm{FM}^{K}$ the category that assigns to $A$ to the groupoid of locally free $K$-modules;
   		\item denote by $\mathfrak{M}_0^{L/K}$ the category that assigns to $A$ to the groupoid of $(\varphi,N,G_{L/K})$-modules;
   		\item denote by $\mathfrak{M}^{L/K}$ the category that assigns to $A$ to the groupoid of filtered $(\varphi,N,G_{L/K})$-modules.
   	\end{enumerate}
   	
   \end{Def}

   \begin{Rmk}
   	If we add a superscript $d$ on any above notation (for example $\mathrm{M}^{K,d}$), it means the fully faithful sub groupoid of objects of rank $d$. The other groupoids in this article also follow this convention in the understandable way.
   \end{Rmk}

   \begin{Lemma}\label{fieldecomp}
   	Let $E$ be a field and let $E'$, $E''$ be finite field extensions of $E$ such that $ |\Hom_E(E',E'')|= [E':E]$. Let $A$ be an $E''$-algebra, and let $M$ be an $(E'\otimes_E A)$-module.
   	\begin{enumerate}
   		\item Then the ring homomorphism: $$E'\otimes_E A \rightarrow \prod_{\eta\in\Hom_E(E',E'')} A,\ a\otimes b \mapsto (\eta (a)b)_\eta$$
   		is an isomorphism.
   		\item Let $e_\eta$ denote the unique element in $\prod\limits_{\eta\in\Hom_E(E',E'')} E''$ which is equal to 1 in $\eta$-th coordinate and vanishes in other coordinates. One has the canonical decomposition:
   		$$M \cong \prod_{\eta\in\Hom_E(E',E'')} M_\eta(:= e_\eta M).$$
   		In particular, let $\eta\in\Hom_E(E'E'')$. The composition of the maps
   		$$e_\eta M \inj M \surj M\otimes_{(E'\otimes_E A),\eta\otimes \id}A$$
   		is an isomorphism of $(E'\otimes_E A)$-modules.
   	\end{enumerate}

   \end{Lemma}
   \begin{proof}
   		\text{ }
   	\begin{enumerate}[leftmargin=1.5em]
   	
   		\item Note that the condition $[E':E] = \# \Hom_E(E',E'')$ implies $E'/E$ is a separable field extension. Hence by primitive element theorem, there exists an element $x\in E'$ such that $E'=E(x)$.\\
   		If we choose $\{1,x,\dots,x^{n-1}\}$ as the $E$-basis  of $E'$ (note that also an $A$-basis of $E'\otimes_E A$), and choose $\{e_\eta\}_{\eta\in\Hom_E(E',E'')}$ as the $A$-basis of $\prod\limits_{\eta\in\Hom_E(E',E'')} A$, then the ring homomorphism above under these two bases corresponds the matrix
   		$B:=(\eta(x^i))_{i,\eta}$. Then, by direct computation, $$\det(B^TB) = \prod_{\eta\neq \eta'}(\eta(x)-\eta'(x)),$$ which is not equal to zero as $x$ is a separable element. Hence the ring homomorphism above is an isomorphism.
        \item The second part of the lemma follows from the following isomorphism:
        $$M\cong M\otimes_{(E'\otimes_EA)}(\prod_{\eta\in\Hom_E(E',E'')}A)\cong \prod_{\eta\in\Hom_E(E',E'')} (M\otimes_{(E'\otimes_EE''),\eta\otimes\id}A). $$
        Note that $M\otimes_{(E'\otimes_E A),\eta\otimes\id}E''$ is exactly $e_\eta M$ by definition.
    \end{enumerate}
   	
   \end{proof}
    Let $E,E',E'',A$ be the setting as in the lemma above. For a $(E'\otimes_E A)$-module $M$, we always write $M_\eta$ for the submodule $e_{\eta}M$ of $M$.
    
    Fix an embedding $\eta_0: L_0 \hookrightarrow C$, by \cite[Proposition~4.1]{Breuil2007}, there is an isomorphism of groupoids $$F: \mathfrak{M}^{L/K,d}_0 \rightarrow \mathfrak{W}^{L/K,d}, (M,\varphi_M,N_M,\gamma) \mapsto (V,\varrho,N_V),$$
    	Where $V = M_{\eta_0}$, $N_V = N_M|_V$ and $\varrho$ is determined by $\varphi_M$ and $\gamma$.

      \begin{Theo}\label{FWDstack}
    	There is an isomorphism of groupoids:
    	$$ \mathrm{FWD}: \mathrm{FM}^{K,d} \times_{\mathrm{M}^{K,d},P} \mathfrak{W}^{L/K,d}\rightarrow \mathfrak{M}^{L/K,d},$$
    	here $\mathrm{FM}^{K,d}\rightarrow \mathrm{M}^{K,d}$ is the natural forgetful functor, and the morphism $$P:\mathfrak{W}^{L/K,d}\rightarrow \mathrm{M}^{K,d}$$ is defined by $V \rightarrow V\otimes_{\mathbb{Q}_p}K$.
    \end{Theo}
    
    \begin{proof}
    	Consider the morphism
    	\begin{align*}
    		P': \mathfrak{M}^{L/K,d}_0 &\rightarrow \mathrm{M^{K,d}}\\
    		(M,\varphi_M,N,\gamma) &\mapsto M_K:=(M\otimes_{L_0}L)^{G_{L/K}}.
    	\end{align*}
    	Note that, by Galois descent, we have the following isomorphism:
    	\begin{align*}
    		\mathrm{FM}^{K,d} \times_{\mathrm{M}^{K,d}} \mathfrak{M}^{L/K,d}_0 & \xrightarrow{\sim} \mathfrak{M}^{L/K,d}\\
    		((D,\mathcal{F}^\bullet),(M,\varphi_M,N,\gamma),\alpha:D\xrightarrow{\sim} M_K)& \mapsto (M,\varphi_M,N,\gamma,\tilde{\mathcal{F}}^\bullet),
    	\end{align*}	
    	here $\tilde{\mathcal{F}}^i M_L := \alpha (\mathcal{F}^iD)\otimes_KL$. Hence it is enough to construct a canonical isomorphism $P \circ F \xrightarrow{\sim} P'$, i.e. $M_{\eta_0}\otimes_{{\mathbb{Q}}_p} K$ is canonically isomorphic to $(M_L)^{G_{L/K}}$.
    	
    	Choose a lifting $\tilde{\eta}_0:L_0\rightarrow C$ of $\eta_0$. Note that one has the canonical isomorphism
    	$$M_L\otimes_{(L\otimes_{\mathbb{Q}_p}A),\tilde{\eta}_0\otimes 1}A = (M\otimes_{(L_0\otimes_{\mathbb{Q}_p}A)}(L\otimes_{\mathbb{Q}_p}A))\otimes_{(L\otimes_{\mathbb{Q}_p}A),\tilde{\eta}_0\otimes 1}A\cong M\otimes_{(L_0\otimes_{\mathbb{Q}_p}A),\eta_0\otimes 1}A$$
    	of $A$-modules.	Consider the following diagram
    	\begin{equation*}
    		\begin{tikzcd}
    			M_{\eta_0} \arrow[r, hook] \arrow[d, dashed, "f_{\eta}"] & M \arrow[d, hook] \arrow[r, two heads] &  {M\otimes_{(L_0\otimes_{\mathbb{Q}_p}A),\eta_0\otimes 1}A} \arrow[d, "\cong"] \\
    			{M_{L,\tilde{\eta}_0}} \arrow[r, hook]       & M_L \arrow[r, two heads]               &  {M_L\otimes_{(L\otimes_{\mathbb{Q}_p}A),\tilde{\eta}_0\otimes 1}A}  .         
    		\end{tikzcd}
    	\end{equation*}
    	As the compositions of the horizontal arrows of both rows and the rightmost vertical map are isomorphisms, then there exists a unique and canonical map $$f_{\eta}: M_{\eta_{0}} \rightarrow M_{L,\tilde{\eta}_{0}}$$ such that the whole diagram commutes and $f_{\eta}$ is an isomorphism of $A$-modules. Then it suffices to show that $M_{L,\tilde{\eta}_0} \otimes_{\mathbb{Q}_p} K$ is canonically isomorphic to $(M_L)^{G_{L/K}}$ as $(K\otimes_{{\mathbb{Q}}_p}A)$-modules.
    	
    	By the Galois decent, one has a canonical isomorphism $$(M_L)^{G_{L/K}}\otimes_{K}L\xrightarrow{\sim} M_L.$$
    	Note that the map $A\otimes_{\mathbb{Q}_p} L\rightarrow A,\ a\otimes b\mapsto \tilde{\eta}_0(b)\cdot a$ restricted on $e_{\tilde{\eta}_0}(A\otimes_{\mathbb{Q}_p} L)$ is a bijection (by lemma \ref{fieldecomp}(2)). Hence we have a canonical isomorphism
    	$$e_{\tilde{\eta}_0}(A\otimes_{\mathbb{Q}_p} L)\otimes_{\mathbb{Q}_p}K\xrightarrow{\sim} A\otimes_{\mathbb{Q}_p}K$$
    	After tensoring with $(M_L)^{G_{L/K}}$ over $A\otimes_{\mathbb{Q}_p}K$, the left hand side is
    	\begin{align*}
    		(M_L)^{G_{L/K}}\otimes_{A\otimes_{\mathbb{Q}_p}K} e_{\tilde{\eta}_0}(A\otimes_{\mathbb{Q}_p} L)\otimes_{\mathbb{Q}_p}K & \cong e_{\tilde{\eta}_0}((M_L)^{G_{L/K}}\otimes_{K}L)\otimes_{{\mathbb{Q}}_p} K\\
    		& \cong e_{\tilde{\eta}_0} M_L\otimes_{{\mathbb{Q}}_p}K \\
    		& \cong M_{L,\tilde{\eta}_0} \otimes_{\mathbb{Q}_p} K,
    	\end{align*}
    	and the right hand side is $(M_L)^{G_{L/K}}$. Then the conclusion follows.
    \end{proof} 
    
   \begin{Rmk}\label{flag}
   
   \begin{enumerate}[leftmargin=1.5em]
   	\text{ }
   \item Assume $A$ is connected. For any locally free $K$-module $M$ over $A$ of rank $d$, by lemma \ref{fieldecomp}, $M \cong \prod\limits_{\eta:K\inj C}M_\eta$ and the set of the $(K\otimes_{\mathbb{Q}_p} A)$-submodules $M'$ of $M$ is in bijection with the set of the systems of $A$-submodules $M'_\eta$ of $M_\eta$. Moreover $M_\eta$ is a locally free module over $A$ of finite rank $d$. Hence giving a filtration on $M$ is equal to give a system of decreasing separated and exhaustive $\mathbb{Z}$-filtrations on $\{M_\eta\}_\eta$ by locally free $A$-submodules. Hence for any $(M,\mathcal{F}^\bullet)\in\mathrm{FM}^{K,d}$ and for each $\eta\in\Hom(K,C)$, there exists a unique non-decreasing sequence of $d$ integers 
   $$(k_{\eta,1}\leq \cdots \leq k_{\eta,d})\in\mathbb{Z}_{+}^{d}$$ 
   such that
   \begin{equation*}
   \rank(\mathrm{gr}^i_\mathcal{F}M_\eta) = |\{m|k_{\eta_i}=m\}|
   \end{equation*}
   We call $\mathrm{wt}_{\eta}(M) := (k_{\eta,1}\leq \cdots \leq k_{\eta,d})$ the \textit{$\eta$-filtration weight of} $M$. 
   \item In the language of stacks, by definition, $\mathrm{M}^{K,d} = \mathrm{Res}_{K/\mathbb{Q}_p}[\mathrm{Spec}{(C)}/\mathrm{GL}_{d,C}]$, where $\mathrm{Res}_{K/\mathbb{Q}_p}$ is the Weil restriction. The observations implies $$\mathrm{M}^{K,d}\cong \prod_{\eta:K\hookrightarrow C }[\mathrm{Spec}{(C)}/\mathrm{GL}_{d,C}]$$
   as a $C$-stack. Similarly, for a filtration weight data $(\mathrm{wt}_\eta)_{\eta}$, if we denote by $P_{\mathrm{wt}_\eta,C}$ the $C$-variety of block upper triangular matrices of type $\mathrm{wt}_\eta$ (i.e. $P_{\mathrm{wt}_\eta,C}$ is of size $(n_1,\dots,n_l)$ for some partition $(n_1,\dots,n_l)$ of $n$ such that $k_{\eta,s}<k_{\eta,s+1}$ if and only if $s=n_1+\cdots+n_i$ for some $i$) and denote by  $\mathrm{Flag}^{\mathrm{wt}_\eta} := [\mathrm{GL}_{d,C}/P_{\mathrm{wt}_\eta,C}]$ the flag variety of type $\mathrm{wt}_\eta$, the observation also implies
   $$\mathrm{FM}^{K,d} \cong \coprod_{(\mathrm{wt}_\eta)_\eta}\prod_{\eta:K\hookrightarrow C} [\mathrm{Flag}^{\mathrm{wt}_\eta}/\mathrm{GL}_{d,C}],$$
   here $(\mathrm{wt}_\eta)_{\eta}$ runs through all filtration weights. Moreover, the forgetful functor
   \begin{align*}
   \mathrm{FM}^{K,d} & \rightarrow \mathrm{M}^{K,d}\\
   (M,\mathcal{F}^\bullet)&\mapsto M
   \end{align*}
   is induced by $\mathrm{Flag}^{\mathrm{wt}_\eta} \rightarrow \mathrm{Spec}(C)$. Hence the morphism is smooth whose fibers are union of products of flag varieties.
   
   \end{enumerate}
   
   \end{Rmk}

   \begin{Cor}
   	The groupoid $\mathfrak{M}^{L/K,d}$ is an Artin stack.
   \end{Cor}
   \begin{proof}
   	It follows from remark \ref{WArtin}, remark \ref{flag} and theorem \ref{FWDstack}.
   \end{proof}

   \subsection{The Construction of Family of ($\varphi,\Gamma_K$)-modules }\label{FamilyofPhiGamma}
   In this subsection, we fix $C$ some field finite over $\mathbb{Q}_p$ such that $|\Hom(L,C)| = [L:\mathbb{Q}_p]$.
   
   In \cite{Berger2008}, Berger constructed a fully faithful functor from the category of filtered $(\varphi,N,G_{L/K})$-modules over $\mathbb{Q}_p$ to the category of $(\varphi,\Gamma_K)$-modules over $\mathcal{R}_{\mathbb{Q}_p}$ (see the definition below). In this subsection we generalize this result to families of those objects parametrized by rigid spaces. First, let us briefly recall some basic notations about Robba rings (see \cite{Berger2008}):
   
   Let $E$ be a finite field extension over $\mathbb{Q}_p$ (in this subsection, $E=K$ or $L$), and let $(\varepsilon^{(n)})_{n\in\mathbb{N}}$ be a compatible system of $p^n$-th root of unity in $\bar{\mathbb{Q}}_p$ (i.e. $(\varepsilon^{(n)})^p=\varepsilon^{(n-1)}$). We write $E_n:= E(\varepsilon^{(n)})$ for $n\geq 1$ and write $E_\infty:= \cup_{n\geq 1}E_n$. We write $H_E:=G_{E_\infty}$ and write $\Gamma_E:= G_E/H_E$.
   
   Let $e_E$ be the ramification degree of $E_\infty/(E_{0})_{\infty}$ (here $E_0$ is the maximal unramified subfield of $E$). For any $s\geq r$, we define  $$\mathcal{R}^{[r,s]}_E := \{f(X_E) \in (E_\infty)_0\llbracket X_E^{\pm 1} \rrbracket|\ f(x) \mathrm{\ converges\ }\forall x \mathrm{\ with\ }\frac{1}{e_Es}\leq \mathrm{val}(x)\leq \frac{1}{e_Er}\}$$ and $\mathcal{R}_E^{r}:= \cap_{s\geq r}\mathcal{R}^{[r,s]}_E$ ($(E_{\infty})_0$ is the maximal unramified subfield of $E_\infty$). By convention, we write $\mathcal{R}_{E}^{[r,\infty]}:= \mathcal{R}_E^{r}$. The ring $\mathcal{R}_E^{[r,s]}$ is noetherian if $s\neq\infty$, otherwise not (note that the definition of $\mathcal{R}_E^{[r,s]}$ here follows from \cite{Berger2008}, and the similar notation in \cite{Kedlaya2012} defines a similar ring but the interval is scaled by the factor $e_E$).
   
   In this subsection, we will define several rings and functors whose symbols contain an interval $[r,s]$. Their definitions depend on $\mathcal{R}_E^{[r,s]}$. Unless explicitly stated, for convention, we allow $s=\infty$ (but not $r$) when an interval $[r,s]$ occurs in these notations, and define them by replacing $\mathcal{R}_E^{[r,s]}$ by $\mathcal{R}_E^{r}$ in them definitions respectively.
   
   There is a compatible $\Gamma_E$-action on $\mathcal{R}^{[r,s]}_E$, and a compatible injective Frobenius homomorphism $\varphi: \mathcal{R}^{[r,s]}_E\rightarrow \mathcal{R}^{[pr,ps]}_E$. We define the Robba ring $$\mathcal{R}_E := \bigcup\limits_{r > 0}\mathcal{R}_E^r$$ which is then equipped with a continuous $(\varphi,\Gamma_E)$-action. In each $\mathcal{R}^{[r,s]}_E$, we have a special element $t$ such that $\varphi(t) = pt$ and $\gamma(t)=\chi(\gamma)t$ (here $\chi:\Gamma_E\rightarrow \mathbb{Z}_p^{\times}$ is the cyclotomic character). For any $ s\geq p^{n-1}(p-1)\geq r \geq r(E)$, we have an injection $\iota_n: \mathcal{R}^{[r,s]}_E\rightarrow E_n\llbracket t_n \rrbracket$, such that $t\mapsto t_n/p^n$, here $r(E)$ is some positive real number depending on $E$. 
   
   Moreover, let $\log X$ be a formal variable, we extend the ($\varphi,\Gamma_E$)-action on $\mathcal{R}_E[\log X]$ by $$\varphi(\log X) := p\log X + \log (\varphi(X)/X^p),$$ and $$\gamma(\log X) := \log X + \log (\varphi(X)/X).$$ 
   (We also extend the $\Gamma_E$-action on $\mathcal{R}^{[r,s]}_E[\log X]$ by the same formula.)
   We extend the map $\iota_n$ on $\mathcal{R}^{[r,s]}_E[\log X]$ by
   $$\iota_n(\log X) = \log(\varepsilon^{(n)}\cdot\exp(t_n/p^n)-1).$$
   For $? \in\{[s,r],r,\emptyset\}$, we define $N$ to be the unique $\mathcal{R}_E^{?}$-derivation on $\mathcal{R}^?_E[\log X]$ such that $$N(\log X) = -p/(p-1).$$ . 
   
   We also define the relative Robba rings for any affoind $C$-algebra $A$ in similar settings (see \cite[Sect~2]{Kedlaya2012} for more details). For example, we set $$\mathcal{R}_{K,A}^{[s,r]}:= \mathcal{R}_E^{[s,r]} \hat{\otimes}_{\mathbb{Q}_p} A  \text{ and } \mathcal{R}^r_{E,A}:= \bigcap_{0 < s\leq r}\mathcal{R}^r_{E,A} \text{ and } \mathcal{R}_{E,A}:= \bigcap_{r > 0}\mathcal{R}^r_{K,A}.$$ Let $E_{n,A} := A \otimes_{\mathbb{Q}_p} E_n$, we also have an injection $\iota_n:\mathcal{R}^r_{E,A}\rightarrow E_{n,A}\llbracket t_n \rrbracket$ induced by the base change from $\mathbb{Q}_p$ to $A$. If $X$ is a rigid $C$-space, the relative Robba rings $\mathcal{R}_{K,A}^{?}$, where $? \in\{[s,r],r,\emptyset\}$, glues to a sheaf of Robba rings on $X$. 
   
   \begin{Rmk}\label{Rmk:FactAboutRobba}
   Here we list some facts about Robba rings that we need in this subsections:
   \begin{enumerate}
   	\item The following diagram commutes:
   	\begin{equation*}
   		\begin{tikzcd}
   			{\mathcal{R}_{E}^{[r,s]}} \arrow[r, "\varphi"] \arrow[d, "\iota_n"] & {\mathcal{R}_{E}^{[pr,ps]}} \arrow[d, "\iota_{n+1}"] \\
   			{E_{n}\llbracket t_n \rrbracket} \arrow[r, "t_n\mapsto t_{n+1}", hook]                  & {E_{n+1}\llbracket t_{n+1} \rrbracket}                                    
   		\end{tikzcd}
   	\end{equation*}
   	\item Let $L/K$ be a finite Galois extension of $p$-adic local fields. There is a natural $G_K$-action on $\mathcal{R}^{[r,s]}_L$ and a natural injective ring homomorphism: $\mathcal{R}^{[r,s]}_K \inj \mathcal{R}^{[r,s]}_L$ such that
   	\begin{enumerate}
   		\item $H_L\subset H_K$ acts trivially on $\mathcal{R}^{[r,s]}_L$ and the $\Gamma_L$-action is induced by the isomorphism $\Gamma_L\cong (G_L/H_L) \subseteq (G_K/H_L)$.
   		\item $\mathcal{R}^{[r,s]}_L$ is naturally a finite \'{e}tale $\mathcal{R}^{[r,s]}_K$-algebra, free of rank $[H_K:H_L]$, and $(\mathcal{R}^{[r,s]}_L)^{H_K/H_L}\cong \mathcal{R}^{[r,s]}_K$, such that the $\Gamma_K$-action on $\mathcal{R}^{[r,s]}_K$ is induced by the isomorphism $\Gamma_K \cong G_K/H_K$.
   	\end{enumerate}
   	\item Let $h$ be a positive integer and $n$ be an integer with $p^{n-1}(p-1)\in[r,s]$. There is an element $t_{n,h}\in \mathcal{R}^{[r,s]}_{E}$ such that
   	\begin{equation*}
   		\iota_m(t_{n,h})=\left\{
   		\begin{array}{cc}
   			1\ (\mathrm{mod\ } t_n^h) & m= n\\
   			0\ (\mathrm{mod\ } t_m^h )& m\neq n
   		\end{array}
   		\right.
   	\end{equation*}
   \item The natural map
   \begin{align*}
   	\mathcal{R}^{[r,s]}_{E} &\rightarrow \prod_{\{n|p^{n-1}(p-1)\in[r,s]\}} E_{n}\llbracket t_n \rrbracket\\
   	x &\mapsto (\iota_n(x))_n
   \end{align*}
   induces an isomorphism $\varprojlim\limits_h(\mathcal{R}^{[r,s]}_{E}/t^h\mathcal{R}^{[r,s]}_{E}) \cong \prod\limits_{\{n|p^{n-1}(p-1)\in[r,s]\}} E_{n}\llbracket t_n \rrbracket$.
   \end{enumerate}
   \end{Rmk}

   \begin{proof} \text{ }
   	\begin{enumerate}
   		\item It follows from the construction of $\iota_n$, see \cite{Cherbonnier1999}.
   		\item It follows from \cite[Cor~5.46]{FontaineOuyang}.
   		\item It follows from \cite[Lem~I.2.1]{Berger2008}.
   		\item It follows from \cite[Prop~I.2.2]{Berger2008}.
   	\end{enumerate}
   \end{proof}
   
   Through out the rest part of this subsection, we always write $A$ for an affinoid $C$-algebra. 
   
   Our main goal in this subsection is to associate a filtered $(\varphi,N,G_{L/K})$-module over $A$ (of rank $d$) to a $(\varphi,\Gamma_K)$-module over $\mathcal{R}_{K,A}$ (of rank $d$), which is defined as follows:
   
   \begin{Def}\cite[Def~2.2.12]{Kedlaya2012}
   	\begin{enumerate}
   		\item
   			A $(\varphi,\Gamma_K)$\textit{-module} $(D^{r},\varphi_{D^r},\Gamma_K)$ \textit{over} $\mathcal{R}^{r}_{K,A}$ (\textit{of rank} $d$) consists of a finitely generated and projective $\mathcal{R}_{K,A}^{r}$-module $D^r$ (of rank $d$) with a semi-linear $\Gamma_K$-action and a $\varphi$-linear map
   			$$\varphi_{D^r}: D^{r}\rightarrow D^{pr}:= D^{r}\otimes_{\mathcal{R}^{r}_{K,A}}\mathcal{R}^{pr}_{K,A}$$
   			such that the linearization map $1\otimes \varphi_{D^r}: \mathcal{R}^{pr}_{K,A}\otimes_{\varphi,\mathcal{R}^{r}_{K,A}}D^{r}\rightarrow D^{pr}$ is a bijection, and $\varphi_{D^r}$ commutes with the $\Gamma_K$-action.
   		\item A $(\varphi,\Gamma_K)$\textit{-module} $(D,\varphi_D,\Gamma_K)$ \textit{over} $\mathcal{R}_{K,A}$ (\textit{of rank} $d$) is the base change to $\mathcal{R}_{K,A}$ of a $(\varphi,\Gamma_K)$-module $D^r$ over $\mathcal{R}^r_{K,A}$ (of rank $d$) for some $r\geq r(L)$.
   	\end{enumerate}
   	
   \end{Def}
   
   

   	Let $M$ be an $(L_0\otimes_{\mathbb{Q}_p} A)$-module equipped with a Frobenius-linear bijective map $$\varphi_M: M\rightarrow M,$$
   	and a decreasing filtration on $M_L := M\otimes_{L_0} L$
   	$$\mathcal{F}^{\bullet}M_L =(\cdots \supseteq \mathcal{F}^{i-1}M_L \supseteq \mathcal{F}^{i}M_L \supseteq \mathcal{F}^{i+1}M_L \supseteq \cdots)$$
   	given by $(L\otimes_{\mathbb{Q}_p} A)$-submodules. 
   	
   	Then the isomorphism
   	\begin{align*}
   		 M\otimes_{L_0} L &\rightarrow M\otimes_{L_0,\varphi^{-n}} L\\
   		x\otimes a &\mapsto \varphi_M^{n}(x) \otimes a
   	\end{align*}
   of $(L\otimes_{\mathbb{Q}_p} A)$-modules induces a decreasing filtration on $M\otimes_{L_0,\varphi^{-n}} L$ from $M\otimes_{L_0} L$. We denote by $\varphi^{n}(M_L)$ the module $M\otimes_{L_0,\varphi^{-n}} L$ with this decreasing filtration structure.
   
   We also define the $\mathbb{Z}$-filtration on $L_{n,A}(\!( t_n )\!)$ by
    $$\mathcal{F}^{i}L_{n,A}(\!( t_n )\!):= t_n^{i}L_{n,A}[\![ t_n ]\!].$$
    
   \begin{Def}\label{BergersFunc1}
   	Let $r,s$ be positive real numbers (or $s=\infty$) and $n$ be a positive integer such that $$r(L) \leq r\leq p^{n-1}(p-1) \leq s.$$
   	Let $M$ be a filtered $(\varphi,N,G_{L/K})$-module $M$ over $A$ of rank $d$, we define
   		
   		$$\mathbf{D}^{[r,s]}_0(M) := (\mathcal{R}^{[r,s]}_{L,A}[\log X]\otimes_{(L_0\otimes_{\mathbb{Q}_p} A)} M)^{N=0},$$
   and
   		$$\Lambda^{n}_{L}(M):= \mathcal{F}^0(L_{n,A}(\!( t_n )\!)\otimes_{(L\otimes_{\mathbb{Q}_p} A)} \varphi^n(M_L))$$
  (for two filtered modules $M_1,M_2$, we define $\mathcal{F}^{i}(M_1\otimes M_2):= \sum\limits_{p+q=i} \mathcal{F}^{p}M_1\otimes \mathcal{F}^{q}M_2$).
  
  Hence $\mathbf{D}^{[r,s]}_0(-)$ (resp. $\Lambda^{n}_{L}(-)$) is a functor from the category of filtered $(\varphi,N,G_{L/K})$-modules over $A$ to the category of $\mathcal{R}^{[r,s]}_{L,A}$-modules (resp. $L_{n,A} \llbracket t_n \rrbracket$-modules) as the maps on the Hom-sets are obvious.
   \end{Def}

   \begin{Rmk}\label{Rmk:ForBergersFunc1}
   Let $r,s,n,M$ be the settings as above.
   \begin{enumerate}
   	\item $\mathbf{D}_{0}^{[r,s]}(M)$ is a flat $\mathcal{R}^{[r,s]}_{L,A}$-module of rank $d$.
   	\item The natural morphism
   	$$\mathcal{R}^{[r,s]}_{L,A}[\log X]\otimes_{\mathcal{R}^{[r,s]}_{L,A}} \mathbf{D}^{[r,s]}_0(M) \rightarrow \mathcal{R}^{[r,s]}_{L,A}[\log X]\otimes_{(L_0\otimes_{\mathbb{Q}_p} A)} M$$
   	of $\mathcal{R}^{[r,s]}_{L,A}[\log X]$-modules is an isomorphism.
   	
   	In particular, the natural morphism 
   	$$L_{n,A} \llbracket t_n \rrbracket\otimes_{\iota_n,\mathcal{R}^{[r,s]}_{L,A}} \mathbf{D}^{[r,s]}_0(M) \rightarrow L_{n,A}\llbracket t_n \rrbracket \otimes_{(L\otimes_{\mathbb{Q}_p}A)} M_L$$
   	is an isomorphism.
   	\item Zariski locally on $\mathrm{Spec}(A)$, the $L_{n,A}\llbracket t_n \rrbracket$-module $\Lambda_{n,L}(M)$ is a free of rank $d$. In particular, $\Lambda_{n,L}(M)$ is a finite projective $L_{n,A}\llbracket t_n \rrbracket$-module of rank $d$.
   \end{enumerate}
   \end{Rmk}

   \begin{proof}
   \text{ }
   \begin{enumerate}
   \item 	We may assume $M$ is a free $(L_0\otimes_{\mathbb{Q}_p}A)$-module of rank $d$.
   
   Choose a $(L_0\otimes_{\mathbb{Q}_p}A)$-basis $(m_1,\dots,m_d)$ of $M$. For any $1 \leq i\leq d$, let $$e_i := \sum_{k=0}^{\infty}((p-1)/p)^k(\log X)^kN^k(m_i).$$
   Each $e_i$ is well defined as $N$ is nilpotent.
   
   Note that any $x\in \mathcal{R}_{L,A}^{[r,s]}[\log X]\otimes_{(L_0\otimes_{{\mathbb{Q}}_p}A)} M$ has the following unique expansion:
   $$x = \sum_{0\leq k\ll\infty}(\log X)^k x_k$$
   for some $x_k\in\mathcal{R}_{L,A}^{[r,s]}\otimes_{(L_0\otimes_{\mathbb{Q}_p}A)} M$. As $\mathcal{R}_{L,A}^{[r,s]}\otimes_{(L_0\otimes_{\mathbb{Q}_p}A)} M$ is a free $\mathcal{R}^{[r,s]}_{L,A}$-module with basis $(m_1,\dots,m_d)$, one can write $x_0 = a_1m_1+\dots+a_dm_d$ for unique $a_i\in \mathcal{R}_{L,A}^{[r,s]}$. Then $N(x) = 0$ if and only if $x = a_1e_1+\dots+a_de_d$ by direct computation, which implies that $(e_1,\dots,e_d)$ is a  $\mathcal{R}_{L,A}^{[r,s]}$-basis of $\mathbf{D}^{[r,s]}_0(M)$.
   \item Let $(m_1,\dots,m_d)$ and $(e_1,\dots,e_d)$ be the notations as in the proof above. Let $B_N$ be the nilpotent matrix in $\mathrm{Max}_{d\times d}(L_0\otimes_{\mathbb{Q}_p}A)$ corresponding to $N$ w.r.t. the basis $(m_1,\dots,m_d)$. Then one has
   $$(e_1,\dots,e_d) = (m_1,\dots,m_d)\left(\sum_{k=0}^{\infty}((p-1)/p)^k(\log X)^kB_N^k\right).$$
   Then our assertion follows from the fact that $$\sum_{k=0}^{\infty}((p-1)/p)^k(\log X)^kB_N^k$$ is invertible in $\mathcal{R}^{[r,s]}_{L,A}[\log X]$ (write $B:=\sum_{k=1}^{\infty}((p-1)/p)^k(\log X)^kB_N^k$ which is nilpotent, then $(1+B)(1+B+B^{2}+\cdots) = 1$), and in particular
   \begin{align*}
   	&L_{n,A}\llbracket t_n \rrbracket \otimes_{\iota_n,\mathcal{R}^{[r,s]}_{L,A}} \mathbf{D}^{[r,s]}_0(M)\\ \cong &L_{n,A} \llbracket t_n \rrbracket \otimes_{\iota_n,\mathcal{R}^{[r,s]}_{L,A}[\log X]} \left(\mathcal{R}^{[r,s]}_{L,A}[\log X]\otimes_{\mathcal{R}^{[r,s]}_{L,A}} \mathbf{D}^{[r,s]}_0(M)\right)\\
   	\cong &L_{n,A}\llbracket t_n \rrbracket \otimes_{\iota_n,\mathcal{R}^{[r,s]}_{L,A}[\log X]}\left(\mathcal{R}^{[r,s]}_{L,A}[\log X] \otimes_{(L_0\otimes_{\mathbb{Q}_p} A)}M\right)\\
   	\cong &L_{n,A}\llbracket t_n \rrbracket \otimes_{(L\otimes_{\mathbb{Q}_p} A)}M_L
   \end{align*}
   \item We may assume that each $\mathcal{F}^{i}M_L$ is a free $A$-module for $i\in\mathbb{Z}$. More precisely, consider the decomposition by lemma \ref{fieldecomp}(2):
   $$\mathcal{F}^{i}(M_L) = \bigoplus_{\eta\in\Hom(L,C)} \left(\mathcal{F}^{i}(M_L)\right)_{\eta}$$
   for all $i\in\mathbb{Z}$. Then each $\left(\mathcal{F}^{i}(M_L)\right)_{\eta}$ is a free $A$-module and one has the decreasing filtration of $(M_L)_{\eta}$:
   $$\cdots \supseteq \left(\mathcal{F}^{i-1}(M_L)\right)_{\eta}\supseteq \left(\mathcal{F}^{i}(M_L)\right)_{\eta}\supseteq \left(\mathcal{F}^{i+1}(M_L)\right)_{\eta}\supseteq \cdots$$
   
    Let $(m_{\eta,1},\dots,m_{\eta,d})$ be a basis of $(M_L)_{\eta}$ which is compatible with the filtration $\left(\mathcal{F}^{\bullet}M_L\right)_{\eta}$ (i.e. if $\left(\mathcal{F}^{i}M_L\right)_{\eta}$ is of rank $d_{\eta,i}$, then $(m_{\eta,1},\dots,m_{\eta,d_{\eta,i}})$ is a basis of $\mathcal{F}^{i}(M_L)$ for all $i$). For each $1\leq i\leq d$, we write $k_{\eta,i}$ for the unique integer such that 
    $$m_{\eta,i}\in \left(\mathcal{F}^{i}M_L\right)_{\eta} \mathrm{\ and \ } m_{\eta,i}\notin \left(\mathcal{F}^{i+1}M_L\right)_{\eta}.$$
    Then one can find that $$\left((t_n^{-k_{\eta,1}}\otimes
     \varphi_M^{n}(m_{\eta,1}))_{\eta},\dots,(t_n^{-k_{\eta,d}}\otimes
     \varphi_M^{n}(m_{\eta,d}))_{\eta}\right)$$ is an $ L_{n,A}\llbracket t_n \rrbracket$-basis of $\Lambda_L^n(M)$. 
   \end{enumerate}
   \end{proof}

    	Let  $s\geq r$ be positive real numbers (or $s=\infty$) with $r\geq r(L)$. Let $M$ be a filtered $(\varphi,N,G_{L/K})$-module $M$ over $A$ of rank $d$.  Let $h$ be a positive integer such that 
    	$$\mathcal{F}^{h}(M_L) = 0 \mathrm{\ and\ } \mathcal{F}^{-h}(M_L) = M_L.$$
    	By definition, one has
    	$$ t_n^{h} \Lambda^{n}_L(M)\subseteq L_{n,A}\llbracket t_n \rrbracket \otimes_{(L\otimes_{\mathbb{Q}_p}A)} M_L\subseteq t_n^{-h} \Lambda^{n}_L(M).$$ 
    	It follows that $\Lambda^{n}_L(M)[1/t_n] \cong L_{n,A}(\!( t_n )\!). \otimes_{(L\otimes_{\mathbb{Q}_p}A)} M_L$
    	According to remark \ref{Rmk:ForBergersFunc1}(2), one has the canonical isomorphism (also denote by $\iota_n$)  $$\iota_n:\mathbf{D}^{[r,s]}_0(M)[1/t]\otimes_{\mathcal{R}_{L,A}^{[r,s]}[1/t],\iota_n} L_{n,A}(\!(t_n)\!) \xrightarrow{\sim} \Lambda_L^n(M)[1/t_n]$$
    	for all $n$ such that $r\leq p^{(n-1)}(p-1)\leq s$.
    \begin{Def} 	
    	We define
    	$$\mathbf{D}^{[r,s]}_L(M):= \{x\in\mathbf{D}^{[r,s]}_0(M)[1/t]\ |\ \iota_n(x)\in \Lambda_L^n(M),\ \forall n\mathrm{\ with\ } p^{n-1}(p-1)\in[r,s] \},$$	
    Moreover, note that $G_K/H_L$ acts semi-linearly on $\mathcal{R}^{[s,r]}_{L,R}[\log X,1/t]\otimes_{(L_0\otimes_{\mathbb{Q}_p} A)} M$ which is induced by the $G_{L/K}$-action on $M$ via the formula (note that $G_{L/K}$ is a quotient of $G_K/H_L$)
    $$g(a\otimes m) \rightarrow g(a)\otimes g(x).$$
    One can easily see that the subspace $\mathbf{D}^{[r,s]}_L(M)\subset \mathcal{R}^{[r,s]}_{L,R}[\log X,1/t]\otimes_{(L_0\otimes_{\mathbb{Q}_p} A)} M$ is stable under the $(G_K/H_L)$-action as the $G_{L/K}$-action on $M$ commutes with the $N$-action and is compatible with the filtration.
    
    Then we define $$\mathbf{D}^{[r,s]}_K(M):= \left(\mathbf{D}^{[r,s]}_L(M)\right)^{H_K/H_L},$$
    which is a $\mathcal{R}_{K,A}^{[r,s]}$-module with a $\Gamma_K$-action.
    
    Hence $\mathbf{D}^{[r,s]}_L(-)$ (resp. $\mathbf{D}^{[r,s]}_{K}(-)$) is a functor from the category of filtered $(\varphi,N,G_{L/K})$-modules over $A$ to the category of $\mathcal{R}^{[r,s]}_{L,A}$-modules with semi-linear $G_K/H_L$-actions (resp. $\mathcal{R}^{[r,s]}_{K,A}$-modules with semi-linear $\Gamma_K$-actions) as the maps on the Hom-sets are obvious.
    \end{Def}
     For simplicity, we write
    $$\mathbf{D}^{r}_0(-):=\mathbf{D}^{[r,\infty]}_0(-)\mathrm{\ and\ } \mathbf{D}^{r}_L(-):= \mathbf{D}^{[r,\infty]}_L(-) \mathrm{\ and\ } \mathbf{D}^{r}_K(-):= \mathbf{D}^{[r,\infty]}_K(-).$$
    In the rest of this subsection, we always write $M$ for a filtered $(\varphi,N,G_{L/K})$-module over $A$ of rank $d$.
   
   Let $R$ be a ring and $M$ be a $R$-module. The module $M$ is called $f$\textit{-regular} for an element $f\in R$ if $M$ has no $f$-torsion. Before we prove the following proposition, let us briefly recall the Beauville-Laszlo theorem (see \cite{Beauville-Laszlo} for the proof).
   
   \begin{Theo}\label{Thm:Beauville-Laszlo}
   	Let $R$ be a ring, let $f$ be a nonzero divisor in $A$, let $\hat{R}$ be the completion of $R$ for the $f$-adic topology. We denote by $R_f$ (resp. $\hat{R}_f$) the ring of fraction $R[1/f]$ (resp. $\hat{R}[1/f]$). Suppose given:
   	\begin{enumerate}
   		\item an $R_f$-module $M_1$;
   		\item an $f$-regular $\hat{R}$-module $M_2$;
   		\item an $\hat{R}_f$-linear isomorphism $\gamma: \hat{R}\otimes_R M_1\rightarrow  M_2 \otimes_R R_f$.
   	\end{enumerate} 
   Then there exists a $f$-regular $R$-module $M$ and isomorphisms $\alpha:  M\otimes_R R_f \rightarrow M_1$ and $\beta:\hat{R}\otimes_R M \rightarrow M_2$ such that $\gamma$ is the composition of
   $$\hat{R}\otimes_R M_1\xrightarrow{1\otimes\alpha^{-1}} \hat{R}\otimes_R M \otimes_R R_f \xrightarrow{\beta\otimes 1} M_2\otimes_R R_f.$$
   The triple $(M,\alpha,\beta)$ is uniquely determined up to a unique isomorphism.
   
   If $M_1$ and $M_2$ are finitely generate (resp. flat, resp. projective and finitely generated), then $M$ has the same property. 
   \end{Theo}

   \begin{Rmk}
   	All the functors (i.e. $\mathbf{D}^{[r,s]}_0(-)$, $\Lambda_L^n(-)$, $\mathbf{D}^{[r,s]}_L(-)$ and $\mathbf{D}^{[r,s]}_K(-)$) we defined above commutes with the base change along any morphism $A\rightarrow A'$ of affinoid $C$-algebras. It is obvious for $\mathbf{D}^{[r,s]}_0(-)$ and $\Lambda_L^n(-)$. For $\mathbf{D}^{[r,s]}_L(-)$, it follows from the Beauville-Laszlo theorem and part (1) of the following proposition. Hence the base change property also holds for $\mathbf{D}^{[r,s]}_K(-)$.
   \end{Rmk}

   \begin{Prop}\label{BergersFunc2}
   Let $s\geq r$ be positive real numbers (or $s=\infty$) with $r\geq r(L)$. 
   \begin{enumerate}
   \item  $n$ be a positive integer such that $$r\leq p^{n-1}(p-1) \leq s.$$
   The natural map
   $$\mathbf{D}_L^{[r,s]}(M)\otimes_{\mathcal{R}^{[r,s]}_{L,A},\iota_n} L_{n,A}\llbracket t_n \rrbracket\rightarrow \Lambda_L^{n} (M)$$ is an isomorphism for $p^{n-1}(p-1)\in[r,s]$.
   \item  Let $r',s'$ be positive real numbers (or $s'=\infty$) such that $r\leq r'\leq s'\leq s$,  The natural map
   \begin{equation}\label{BCforBerger}
   \mathbf{D}_L^{[r,s]}\otimes_{\mathcal{R}^{[r,s]}_{L,A}}\mathcal{R}^{[r',s']}_{L,A} \rightarrow \mathbf{D}_L^{[r',s']}(M)
   \end{equation}
   is an isomorphism of $\mathcal{R}^{[r',s']}_{L,A}$-modules.
   \item
   $\mathbf{D}_L^{[r,s]}(M)$ is a flat $\mathcal{R}_{L,A}^{[r,s]}$-module of rank $d$.
   \item The natural map $$\mathbf{D}_K^{r}(M)\otimes_{\mathcal{R}^{r}_{K,A}}{\mathcal{R}^{r}_{L,A}}\rightarrow \mathbf{D}_L^{r}(M)$$
   is an isomorphism of ${\mathcal{R}^{r}_{L,A}}$-modules.
   
   In particular, $\mathbf{D}_K^{r}(M)$ is a projective $\mathcal{R}_{K,A}^{r}$-module of rank $d$.
   \end{enumerate}
   \end{Prop}

   \begin{proof}\text{ }
   	\begin{itemize}
   		\item[(1)]
   		    Let $h$ be an positive integer such that $ \mathcal{F}^{h}(M_L)=0$, which means $$\mathbf{D}_L^{[r,s]}(M)\subseteq 1/t^{-h}\mathbf{D}_0^{[r,s]}(M) .$$ By remark \ref{Rmk:ForBergersFunc1}(2), one has
   		    $$\mathbf{D}_L^{[r,s]}(M)\otimes_{\mathcal{R}^{[r,s]}_{L,A},\iota_n} L_{n,A}\llbracket t_n \rrbracket \subseteq  L_{n,A}\llbracket t_n \rrbracket \otimes_{(L\otimes_{\mathbb{Q}_p}A)} M_L.$$
   		    Hence $\mathbf{D}_L^{[r,s]}(M)\otimes_{\mathcal{R}^{[r,s]}_{L,A},\iota_n} L_{n,A}\llbracket t_n \rrbracket$ is $t_n$-adic completed.
   		    
   			By definition, the map 
   			$$\mathbf{D}_L^{[r,s]}(M)\otimes_{\mathcal{R}^{[r,s]}_{L,A},\iota_n} L_{n,A}\llbracket t_n \rrbracket\rightarrow \Lambda_L^{n} (M)$$
   			is injective. As both sides are $t_n$-adic completed. It suffices to show that for all $h$ large enough, the natural map
   			$$\mathbf{D}_L^{[r,s]}(M)\otimes_{\mathcal{R}^{[r,s]}_{L,A},\iota_n} L_{n,A}\llbracket t_n \rrbracket\rightarrow \Lambda_L^{n} (M)/(t_n^h \Lambda_L^{n}(M))$$ is surjective.
   			
            We may assume $h$ is large enough such that 
   			$$\mathcal{F}^{h}(M_L) = 0 \mathrm{\ and\ } \mathcal{F}^{-h}(M_L) = M_L.$$
   			By definition, one has
   			$$ t_n^{h} \Lambda^{n}_L(M)\subseteq L_{n,A}\llbracket t_n \rrbracket\otimes_{\iota_n,\mathcal{R}^{[r,s]}_{L,A}}\mathbf{D}^{[r,s]}_0(M)\subseteq t_n^{-h} \Lambda^{n}_L(M)  $$
   			
   			It follows that for any $x\in \Lambda_L^{n}(M)$, there exists some $y\in t^{-h}\mathbf{D}^{[r,s]}_0(M)$ such that $\iota_n(y)\in t_n^{h}\Lambda_L^{n}(M)$. Let $z= t_{n,3h} y$. Then one has
   			$$\iota_n(z)-\iota_n(y) \in t^{2h}L_{n,A}\llbracket t_n \rrbracket \otimes_{\iota_n,\mathcal{R}^{[r,s]}_{L,A}}\mathbf{D}^{[r,s]}_0(M) \subseteq t_n^{h}\Lambda_L^{n}(M)$$
   			and
   			$$ \iota_m(z) \in  t^{2h}L_{m,A}\llbracket t_m \rrbracket \otimes_{\iota_n,\mathcal{R}^{[r,s]}_{L,A}}\mathbf{D}^{[r,s]}_0(M) \subseteq t^{h}_n\Lambda_L^{m}(M)\subseteq \Lambda_L^{m}(M)$$
   			for $m\neq n$, which implies that $z\in\mathbf{D}_L^{[r,s]}(M)$.
   			
   			Hence the natural map
   			$$\mathbf{D}_L^{[r,s]}(M)\otimes_{\mathcal{R}^{[r,s]}_{L,A},\iota_n} L_{n,A}\llbracket t_n \rrbracket\rightarrow \Lambda_L^{n} (M)/(t_n^h \Lambda_L^{n}(M))$$ is surjective.
   		
   	   \item[(2)+(3)] The $\mathcal{R}^{[r,s]}_{L,A}$-module $\mathbf{D}_L^{[r,s]}(M)[1/t]$ is finitely generated and projective of rank $d$ as $\mathbf{D}_L^{[r,s]}(M)[1/t]\cong \mathbf{D}_0^{[r,s]}(M)[1/t]$. Hence 
   	   map (\ref{BCforBerger}) is an isomorphism after inverting $t$ as $$\mathbf{D}_0^{[r,s]}\otimes_{\mathcal{R}^{[r,s]}_{L,A}}\mathcal{R}^{[r',s']}_{L,A}  \cong \mathbf{D}_0^{[r',s']}(M)$$
   	   (note that the $N$-action on $\mathcal{R}^{[r,s]}_{L,A}[\log X]$) (resp. on $\mathcal{R}^{[r',s']}_{L,A}[\log X]$) is $\mathcal{R}^{[r,s]}_{L,A}$-linear (resp. $\mathcal{R}^{[r',s']}_{L,A}$-linear), and $\mathbf{D}_L^{[r,s]}(M)[1/t]$ is finitely generated and projective of rank $d$.
   	   
   	   On the other hand, by remark \ref{Rmk:FactAboutRobba}(4) and assertion (1) above, one has
   	   \begin{align*}
   	   	& \mathbf{D}_L^{[r,s]}(M)\otimes_{\mathcal{R}_L{[r,s]}}\varprojlim_{h}(\mathcal{R}_L{[r,s]}/t^h\mathcal{R}_L{[r,s]})\\
   	   	 \cong & \prod_{\{n|p^{n-1}(p-1)\in[r,s]\}}\mathbf{D}_L^{[r,s]}(M)\otimes_{\mathcal{R}_{L,A}{[r,s]},\iota_n} L_{n,A}[\![t_n]\!]\\
   	   	 \cong & \prod_{\{n|p^{n-1}(p-1)\in[r,s]\}} \Lambda_{n,L}(M).
   	   \end{align*}
       Then by the proof of remark \ref{Rmk:ForBergersFunc1}(3), locally on $\mathrm{Spec}(A)$ (such that $\mathcal{F}^{i}M_L$ is free for each $i\in\mathbb{Z}$), $\Lambda_{n,L}(M)$ is free of rank $d$. Therefore then map (\ref{BCforBerger}) is an isomorphism after taking $t$-completion and $$\mathbf{D}_L^{[r,s]}(M)\otimes_{\mathcal{R}_L{[r,s]}}\varprojlim_{h}(\mathcal{R}_L{[r,s]}/t^h\mathcal{R}_L{[r,s]})$$
   	   is finitely generated and projective of rank $d$. Then assertions (2) and (3) follows from the Beauville-Laszlo theorem (note that this theorem does not need any noetherian condition, hence our assertions in particular holds for $s=\infty$).
       \item[(4)] By remark \ref{Rmk:FactAboutRobba}(2), $\mathcal{R}_{L,A}^{[r,s]}$ is \'{e}tale Galois over $\mathcal{R}_{K,A}^{[r,s]}$ with Galois group $H_K/H_L$. One has $\mathbf{D}_K^{[r,s]}(M) = (\mathbf{D}_L^{[r,s]}(M))^{H_K/H_L}$ is finitely generated and projective by Galois decent.
            
	\end{itemize}
    \end{proof}
    Let $M$ be a filtered $(\varphi,N,G_{L/K})$-module over $A$, we define the $\varphi$-linear map of $\mathcal{R}^{r}_{L,A}$-modules by
    \begin{align*}
    	\varphi: \mathcal{R}^r_{L,A}[\log X]\otimes_{\mathcal{R}^r_{L,A}}M & \rightarrow \mathcal{R}^{pr}_{L,A}[\log X]\otimes_{\mathcal{R}^r_{L,A}}M\\
    	a\otimes m& \mapsto \varphi(a)\otimes \varphi_M(m).
    \end{align*}
    One can easily check the equation $N\circ\varphi=p\varphi\circ N$ also holds. Hence this $\varphi$-action can be restricted on their $N$-invariant. Namely, the induced $\varphi$-linear map $\varphi:\mathbf{D}_0^{r}(M)\rightarrow \mathbf{D}_0^{pr}(M)$ is well-defined.
    
    Moreover the proof of remark \ref{Rmk:ForBergersFunc1}(3) implies that the natural map
    \begin{align*}
    	\varphi_n:L_{n,A}(\!(t_n)\!)\otimes_{(L\otimes_{\mathbb{Q}_p} A)} \varphi^{n}(M_L) &\rightarrow L_{n+1,A}(\!(t_{n+1})\!) \otimes_{(L\otimes_{\mathbb{Q}_p} A)} \varphi^{n+1}(M_L)\\
    	a\otimes m&\mapsto a\otimes \varphi(m)
    \end{align*}
    induces a well-define map $\varphi_n:\Lambda_{L}^{n}(M)\rightarrow \Lambda_{L}^{n+1}(M)$. Actually, the proof of \ref{BergersFunc1}(3) also implies that the map
    \begin{align*}
    	\varphi_n:L_{n+1,A}\llbracket t_{n+1} \rrbracket\otimes_{L_{n,A}\llbracket t_n \rrbracket}\Lambda_{L}^{n}(M) & \xrightarrow{\sim} \Lambda_{L}^{n+1}(M)\\
    	a\otimes m&\mapsto a \varphi(m)
    \end{align*}
    is an isomorphism of $L_{n+1,A}\llbracket t_{n+1} \rrbracket$-lattices. Note that \ref{BergersFunc1}(2), combined with remark \ref{Rmk:FactAboutRobba}(1), means that the following diagram
    \begin{equation*}
    	\begin{tikzcd}
    		{L_{n,A} (\!(t_n)\!)\otimes_{\iota_n,\mathcal{R}^{r}_{L,A}[1/t]} \mathbf{D}^{r}_0(M)[1/t]} \arrow[d, "\varphi"'] \arrow[rr, "\sim","\iota_n"'] &  & {L_{n,A}(\!(t_n)\!) \otimes_{(L\otimes_{\mathbb{Q}_p}A)} \varphi^{n}(M_L)} \arrow[d, "\varphi_n"] \\
    		{L_{n,A} (\!(t_{n+1})\!)\otimes_{\iota_{n+1},\mathcal{R}^{pr}_{L,A}[1/t]} \mathbf{D}^{pr}_0(M)[1/t]} \arrow[rr, "\sim","\iota_{n+1}"']                       &  & {L_{n,A}(\!(t_{n+1})\!) \otimes_{(L\otimes_{\mathbb{Q}_p}A)} \varphi^{n+1}(M_L)}                       
    	\end{tikzcd}
    \end{equation*}
    commutes. This means for an element $x\in \mathbf{D}^{r}_0(M)[1/t]$ such that $\iota_{n}(x)\in\Lambda_L^n(M)$, then $\iota_{n+1}(\varphi(x))$ is in $\iota_{n}(x)\in\Lambda_L^{n+1 }(M)$. Namely, one has a well-defined $\varphi$-linear map of $\mathcal{R}^{r}_{L,A}$-modules:
    $$\varphi:\mathbf{D}_L^{r}(M)\rightarrow \mathbf{D}_L^{pr}(M).$$
    Moreover, as the $\varphi_M$-action on $M$ commutes with the $G_{L/K}$-action and the $\varphi$-action on $\mathcal{R}_{L,A}^{r}$ also commutes with the $G_K/H_L$-action, by definition, one can easily check that the map $\varphi:\mathbf{D}_L^{r}(M)\rightarrow \mathbf{D}_L^{pr}(M)$ commutes with the $(G_K/H_L)$-action. 
    Hence after taking the $H_K/H_L$-invariant on both side, one has a well-defined $\varphi$-linear map
    $$\varphi:\mathbf{D}_K^{r}(M)\rightarrow \mathbf{D}_K^{pr}(M),$$
    which commutes with the $\Gamma_K$-action. 
 
    \begin{Def}
    	Let $M$ be a filtered $(\varphi,N,G_{L/K})$-module. For any $r\geq r(L)$, we define the $\Gamma_K$-compatible $\varphi$-linear map
    	$$\varphi:\mathbf{D}_K^{r}(M)\rightarrow \mathbf{D}_K^{pr}(M)$$
    	as the discussion above.
    \end{Def}

   \begin{Prop}
   	Let $M$ be a filtered $(\varphi,N,G_{L/K})$-module over $A$ of rank $d$. For any $r\geq r(L)$, The map
   	\begin{align*}
   		\mathcal{R}^{pr}_{K,A}\otimes_{\varphi,\mathcal{R}^{r}_{K,A}}\mathbf{D}_K^{r}(M)& \rightarrow \mathbf{D}_K^{pr}(M) \\
   		a\otimes x\rightarrow a\cdot\varphi(x)
   	\end{align*}
   is an isomorphism. In particular, $\mathbf{D}_K^{r}(M)$ is a $(\varphi,\Gamma_K)$-module over $\mathcal{R}^r_{K,A}$ of rank $d$.
   \end{Prop}

   \begin{proof}
   	Note that $\mathcal{R}^{pr}_{L,A}$ is faithfully flat over $\mathcal{R}^{pr}_{K,A}$. By proposition \ref{BergersFunc2}(4), we can apply the tensor functor $\mathcal{R}^{pr}_{L,A}\otimes_{\mathcal{R}^{pr}_{K,A}}(-)$ on both sides and reduce to show that the map

   		\begin{align}\label{phimap}
   			\mathcal{R}^{pr}_{L,A}\otimes_{\varphi,\mathcal{R}^{r}_{L,A}}\mathbf{D}_L^{r}(M)& \rightarrow \mathbf{D}_L^{pr}(M) \\
   			a\otimes x\rightarrow a\varphi(x) \nonumber
   		\end{align}
   is an isomorphism. By beauville laszlo theorem, it suffices to show that, after inverting $t$ and taking the $t$-adic completion respectively, the map (\ref{phimap}) is an isomorphism.
   
   After inverting $t$, one has $\mathbf{D}_L^{r}(M)[1/t]\cong \mathbf{D}_0^{r}(M)[1/t]$. We may assume $M$ is a free $(L_0\otimes_{{\mathbb{Q}}_p} A)$-module. Then by the proof of remark \ref{Rmk:ForBergersFunc1}(1), one can explicitly write down a basis, and see that the map $\mathcal{R}^{pr}_{L,A}\otimes_{\varphi,\mathcal{R}^{r}_{L,A}}\mathbf{D}_0^{r}(M)[1/t] \rightarrow \mathbf{D}_0^{pr}(M)[1/t]$ is an isomorphism.
   
   As for taking the $t$-completion, the proof is similar as the one in proposition \ref{BergersFunc2}(2), and use the fact that the map
   \begin{align*}
   	\varphi_n:L_{n+1,A}\llbracket t_{n+1} \rrbracket\otimes_{L_{n,A}\llbracket t_n \rrbracket}\Lambda_{L}^{n}(M) & \xrightarrow{\sim} \Lambda_{L}^{n+1}(M)\\
   	a\otimes m&\mapsto a\cdot \varphi(m)
   \end{align*}
   is an isomorphism.
   \end{proof}
   
   \begin{Def}
   Let $M$ be a filtered $(\varphi,N,G_{L/K})$-module of rank $d$. Choosing some $r\geq r(L)$, we define $$\mathbf{D}_K(M):= \mathbf{D}^{r}_K(M)\otimes_{\mathcal{R}^r_{K,A}}\mathcal{R}_{K,A}.$$ The $(\varphi,\Gamma_K)$-module $\mathbf{D}_K(M)$ over $\mathcal{R}_{K,A}$ of rank $d$ is called \textit{the} $(\varphi,\Gamma_K)$\textit{-module associated to} $M$ (by proposition \ref{BergersFunc2}(2), the definition is well-defined and does not depends on the choice of $r$). 
   \end{Def}

   \begin{Rmk}
   Now $\mathbf{D}_K(-)$ is a well-defined functor from the category of filtered $(\varphi,N,G_{L/K})$-module over $A$ to the category of $(\varphi,\Gamma_K)$-module over $\mathcal{R}_{K,A}$. When $A =\mathbb{Q}_p$, the functor $\mathbf{D}_K(-)$ is exactly the functor defined in \cite[Def~II.2.4]{Berger2008}. The following theorem is the parallel result for [Thm~II.2.6] of \textit{loc. cit.}.
   \end{Rmk}
   
   \begin{Theo} \label{Thm: PropofBergerFun}
   	   The functor $M\mapsto \mathbf{D}_K(M)$ is a faithful exact functor from the category of filtered $(\varphi,N,G_{L/K})$-modules over $A$ (of rank $d$) to the category of $(\varphi,\Gamma_K)$-modules over $\mathcal{R}_{K,A}$ (of rank $d$), which commutes with tensor product.
   	   
   	   Moreover, for any two filtered $(\varphi,N,G_{L/K})$-modules $M_1$ and $M_2$ over $A$, if $f$ is an element in $\Hom(M_1,M_2)$ such that $\mathbf{D}_K(f)$ is an isomorphism, then so is $f$.
   \end{Theo}

   \begin{proof}
   To prove the first part of this theorem, one only need check that the functors $\mathbf{D}_0^{r}(-)$ and $\Lambda_L^n(-)$ are faithful exact and commutes with tensor product. Then apply Beauville-Laszlo theorem and Galois descent to show $\mathbf{D}^r_K(-)$ is faithful exact and commutes with tensor product. Hence so is $\mathbf{D}^r_K(-)$ as $\mathcal{R}_{K,A} \rightarrow \mathcal{R}_{K,A}$ is flat. We omit the details as these tricks have been used several times in the proofs before.
   
   For the moreover part, one has
   $ \mathbf{D}^{r}_K(f)\in\Hom_{\varphi,\Gamma_K}(\mathbf{D}^r_K(M_1),\mathbf{D}^r_K(M_2))$
   is an isomorphism for some $r$, by \cite[Lem~2.2.9]{Kedlaya2012}. Note that $$\mathbf{D}^r_K(M_i)\otimes_{\mathcal{R}^{[r,s]}_{K,A}} \mathcal{R}^{[r,s]}_{L,A} \cong \mathbf{D}^r_L(M_i)$$ for $i=1,2$, we only need to show $\mathbf{D}^{r}_L(f)$ is an isomorphism implies $f$ is an isomorphism.
   
   As $\mathbf{D}^{r}_L(M_i)[1/t] \cong \mathbf{D}^{r}_0(M_i)[1/t]$, it follows that $f:M_1\rightarrow M_2$ is an isomorphism of $(\varphi,N,G_{L/K})$-modules if we forget the filtration structure of them (note that the functor $\mathbf{D}^{r}_0(-)$ has nothing to do with filtration structure and $\mathbf{D}^{r}_0(-)[1/t]$ is faithful and exact). On other hand, suppose that $f|_{\mathcal{F}^{i}M_{1,L}}$ is not an isomorphism for some $i\in\mathbb{Z}$. It follows that the induced map $\Lambda_{L}^{n}(f): \Lambda_{L}^{n}(M_1)\inj \Lambda_{L}^{n}(M_2)$ is a strict injection. This contradicts to the assumption that $\mathbf{D}_L^{r}(f)$ is an isomorphism by proposition \ref{BergersFunc2}(1).

\end{proof}
   
   \begin{Rmk}
   If replacing the affinoid $C$-algebra $A$ by some $C$-rigid space $X$, then correspondingly $\mathbf{D}_K(-)$ can be automatically understood as the functor from the category of sheaves of filtered $(\varphi,N,G_{L/K})$-module over $X$ (of rank $d$) to the category of sheaves of $(\varphi,\Gamma_K)$-module over $\mathcal{R}_{K,A}$ (of rank $d$), which satisfies similar properties as in theorem \ref{Thm: PropofBergerFun}.
   \end{Rmk}
   
   \subsection{Quasi-deRham $(\varphi,\Gamma_K)$-Modules}\label{QdeRhamPhiGammaMod}
   In this subsection, we want to study the groupoid of quasi-deRham $(\varphi,\Gamma_K)$-modules (see the definition below). In this article, all characters are supposed to be continuous.
   
    We fix $\varpi$ a uniformizer of $K$. We fix $C$ some field finite over $\mathbb{Q}_p$ such that $|\Hom(L,C)| = [L:\mathbb{Q}_p]$, and always write $A$ for some affinoid $C$-algebra. For a $(\varphi,\Gamma_K)$-module $D$ over $\mathcal{R}_{K,A}$, we denote by $$D^{\vee}:= \Hom_{\mathcal{R}_{K,A}}(D,\mathcal{R}_{K,A})$$
   the Poincare duality, which is a $(\varphi,\Gamma_K)$-module of the same rank as $D$.
   
   From now on, let $\mathbf{D}$ denote the functor $\mathbf{D}_K$ for short. The stacks we constructed in section 2.2 (e.g. $\mathfrak{M}^{L/K,d}$) are understood as rigid $C$-stacks, which fibered in groupoids on rigid $C$-spaces. One can check that all definitions and statements automatically hold in this setting.
   \begin{Def}
   A $(\varphi,\Gamma_K)$-modules $D$ over $\mathcal{R}_{K,A}$ is called \textit{quasi-deRham} if there exists a filtered $(\varphi,N,G_{L/K})$-module $M$ over $A$ and a character $\delta: K^\times \rightarrow A^\times$, such that $D\cong\mathbf{D}(M)(\delta) := \mathbf{D}(M)\otimes_{\mathcal{R}_{K,A}} \mathcal{R}_{K,A}(\delta)$, here $\mathcal{R}_{K,A}(\delta)$ is the free rank one ($\varphi,\Gamma_K$)-module associated to $\delta$ and for the precise definition, we refer to \cite[Construction~6.2.4]{Kedlaya2012}. 
   \end{Def}

   \begin{Rmk}
   	For a character $\delta: K^\times \rightarrow A^{\times}$, one has $\left(\mathcal{R}_{K,A}(\delta)\right)^{\vee} \cong \mathcal{R}_{K,A}(\delta^{-1})$.
   \end{Rmk}
   
   \begin{Def}
   Let $\mathcal{Z}^{L/K}$ denote the groupoid of quasi-deRham $(\varphi,\Gamma_K)$-modules on the category of rigid $C$-spaces.
   \end{Def}
   
   \begin{Def}
   let $\mathcal{T}_K$ denote the moduli rigid $C$-space of characters of $K^\times$, and let $\mathcal{W}_K$ denote the moduli rigid $C$-space of character of $\mathcal{O}_K^\times$.
   \end{Def}

   \begin{Rmk}\label{Rmk:Smchar}
   Recall that a character $\delta: K^\times \rightarrow R^\times$ is called 
   \begin{enumerate}
   	\item \textit{algebraic} if 
   	$$\delta(a) = \prod_{\eta: K\hookrightarrow C}\eta(a)^{k_\eta}$$
   	for some $k_\eta\in\mathbb{Z}$ (in this case,we also denote such character by $\delta_{\underline{k}}$ for $\underline{k}:= (k_\eta)_{\eta:K\inj C}$),
   	\item \textit{smooth} if $\ker \delta$ is open in $K^\times$,
   	\item \textit{deRham} if $\delta = \delta_{\mathrm{sm}}\cdot\delta_{\mathrm{alg}}$ for some smooth character $\delta_{\mathrm{sm}}$ and some algebraic character $\delta_{\mathrm{alg}}$. Obviously, if $\delta$ is deRham, such decomposition is unique. We say $\delta_{\mathrm{sm}}$ is the \textit{smooth part} of $\delta$ and $\delta_{\mathrm{alg}}$ is the \textit{algebraic part} of $\delta$.
   \end{enumerate}
   \end{Rmk}

   From now on, for a deRham character $\delta: K^{\times}\rightarrow A^{\times}$, we always write $\delta_{\mathrm{sm}}$ (resp. $\delta_{\mathrm{alg}}$) for the smooth part (resp. the algebraic part) of $\delta$.

   \begin{Rmk}\label{deRhamchar}
   	\text{ }
   \begin{enumerate}[leftmargin=1.5em]
   \item By the local class field theory, we have the canonical isomorphism: $W_K^{\mathrm{ab}}\cong K^\times$. Hence the restriction of a character $\delta$ on some subgroup $H$ of $W_K$ makes sense.
   
   \item Let $\delta:K^{\times}\rightarrow A^{\times}$ be a character. Suppose that $\delta$ is deRham and the smooth part $\delta_{\mathrm{sm}}$ restricted on $I_L$ (we regard $I_L$ as an open subgroup of $W_K$ via the injection $I_L\inj W_L\inj W_K$) is trivial, and suppose that $\delta_{\mathrm{alg}} = \delta_{\underline{k}}$ for some $\underline{k}\in \mathbb{Z}^{\Hom(K,C)}$. Then we define $\mathbf{M}(\delta)$ to be the unique free rank one filtered $(\varphi,N,G_{L/K})$-module over $A$ such that the $\eta$-filtration weight of $\mathbf{M}(\delta)$ is $-\underline{k}$, and the $(\varphi,N,G_{L/K})$-module structure is the one associated to the rank one Weil-Deligne representation $\delta_{\mathrm{sm}}: W_K^{\mathrm{ab}}\cong K^{\times}\rightarrow A^{\times}$, constructed by \cite[Prop~4.1]{Breuil2007}. 
   
   Then, comparing with the construction in \cite[Const~6.2.4]{Kedlaya2012}, one can directly check that $$\mathbf{D}(\mathbf{M}(\delta)) \cong \mathcal{R}_{K,R}(\delta).$$ Moreover, for any such two character $\delta_1,\delta_2$, one has $\mathbf{M}(\delta_1\cdot \delta_2) = \mathbf{M}(\delta_2)\otimes \mathbf{M}(\delta_1)$.
   
   For a filtered$(\varphi,N,G_{L/K})$-module $M$ over $A$ and a character $\delta:K^{\times}\rightarrow A^{\times}$, we write
   $$M(\delta) := M\otimes_{L_0\otimes_{\mathbb{Q}_p} A} \mathbf{M}(\delta). $$
   In particular, we have $\mathbf{D}(M(\delta)) \cong \mathbf{D}(M) \otimes_{\mathcal{R}_{K,A}} \mathcal{R}_{K,A}(\delta)$.
   \item Conversely, for any filtered $(\varphi,N,G_{L/K})$-module $M$ over $A$ which is free of rank one, their exists some some deRham character $\delta: K^\times\rightarrow A^\times$, whose smooth part $\delta_{\mathrm{sm}}$ is trivial on $I_L$ such that $\mathbf{D}(M) \cong \mathcal{R}_{K,R}(\delta)$ by \cite[Prop~4.1]{Breuil2007} (even though, in the statement of that proposition, $A$ is a finite field over $\mathbb{Q}_p$, but the proof automatically holds if we allow $A$ to be any $C$-algebra).
   \end{enumerate}
   \end{Rmk}
   
   \begin{Lemma}
   Let $R$ be a commutative ring and let $M_i$ be $R$-modules for $i=1,2,3$, such that there is an injection $f: M_1\otimes_R M_2\inj M_3$. Suppose that $M_2$ is finite projective, then the induced map
   $$\tilde{f}: M_1\rightarrow \Hom_R(M_2,M_3),a\mapsto(b\mapsto f
   (a\otimes b))$$ is injective.
   \end{Lemma}

   \begin{proof}
   For any $a\in\ker\tilde{f}$, one has $f(a\otimes b) = 0$ (hence $a\otimes b =0$) for any $b\in M_2$. As $M_2$ is faithfully flat, it follows that $a =0$.
   \end{proof}
   
   \begin{Cor}
   	Let $\delta_i: K^\times \rightarrow A^\times$ be characters and $M_i$ be filtered $(\varphi,N,G_{L/K})$-modules over $A$ for $i=1,2$. Suppose that $\mathbf{D}(M_1)(\delta_1) = \mathbf{D}(M_2)(\delta_2)$. Let $\delta := \delta_1\delta_2^\vee$ and $D:= \mathbf{D}(M_2)\otimes \mathbf{D}(M_1)^\vee$. Then for any maximal ideal $\mathfrak{m}\subseteq A$, and $n\in\mathbb{N}$, the natural map: $$\mathcal{R}_{K,A}(\delta)\otimes_A A/\mathfrak{m}^n \rightarrow D\otimes_A A/\mathfrak{m}^n$$ is an injection.
   \end{Cor}
    
   \begin{proof}
   Note that the Poincare duality functor on ($\varphi,\Gamma_K$)-modules commutes with base change. In particular $$D\otimes_A A/\mathfrak{m}^n \cong (\mathrm{D}(M_2)\otimes_A A/\mathfrak{m}^n)\otimes_{A/\mathfrak{m}^n}(\mathrm{D}(M_2)\otimes_A A/\mathfrak{m}^n)^\vee$$. Hence we can apply the lemma above to the isomorphism
   $$\mathrm{D}(M_1)(\delta)\otimes_A A/\mathfrak{m}^n \cong \mathrm{D}(M_2)\otimes_A A/\mathfrak{m}^n$$
   and get the conclusion.
   \end{proof}
   \begin{Lemma}\label{Inj}
   	Let $A$ be a local Artin ring over $C$ with $\dim_C A\leq \infty$, let $\delta: K^\times \rightarrow A^\times$ be a character and let $M$ be a $(\varphi,N,G_{L/K})$-module over $\mathcal{R}_{K,A}$. Suppose that there exists an injective morphism $$\mathcal{R}_{K,A}(\delta) \inj \mathbf{D}(M)$$ of $(\varphi,\Gamma_K)$-modules, then $\delta$ is deRham such that the smooth part $\delta_\mathrm{sm}$ is trivial restricted on $I_L\subset W_K$.
   \end{Lemma}

   \begin{proof}
   	Firstly, note that the statement is true when $A$ is a field. Indeed, by \cite[Corollary~III.2.5]{Berger2008}, one has $\mathcal{R}_{K} (\delta) \cong \mathbf{D}(M')$ for some rank one filtered $(\varphi,N,G_{L/K})$-module $M$ over $A$. Then $\delta$ is deRham by remark \ref{deRhamchar}(3)).
   	
   	Back to general $A$. Then we forget the $A$-structure of $\mathcal{R}(\delta)$ (resp. $\mathbf{D}(M)$) but only remember the $C$-structure, and denote it by $D_1$ (resp. $D_2$). Then $D_1\cong \mathbf{D}(M_1)$ for some filtered $(\varphi,N,G_{L/K})$-module $M_1$ over $\mathbb{Q}_p$ by the corollary in loc. cit..
   	
   	Note that having an $A$-structure on $D_1$ is equivalent to have a homomorphism of $C$-algebras: $A \rightarrow \Hom_{\varphi,\Gamma_K}(D_1,D_1)$. As $\Hom_{\varphi,\Gamma_K}(D_1,D_1) \cong \Hom_{\mathfrak{M}^{L/K}}(M_1,M_1)$, then the $A$-structure on $\mathcal{R}(\delta)$ induces an $A$-structure on $M_1$ and make it in into a filtered $(\varphi,N,G_{L/K})$-module over $A$ and $\mathbf{D}(M_1) = \mathcal{R}(\delta)$. It remains to show that $M_1$ is a free rank one module over $L_0\otimes_{\mathbb{Q}_p} A$.
   	
   	The claim is trivial when $A$ is a field. For the general case,   let $\mathfrak{m}$ be the maximal ideal of $A$. Then the base change of $M_1$ along $A/\mathfrak{m}$ is a free $(A/\mathfrak{m})$-module of rank one (note that the functor $\mathbf{D}(-)$ commutes with base change). By Nakayama lemma, $M_1$ is generated by one element.
   	Note that $$\dim_C(M_1)/[L_0:\mathbb{Q}_p] = \rank_{\mathcal{R}_{K,C}}D_1=\dim_CA,$$ hence it is free of rank one via considering its $C$-dimension.    
   \end{proof}
   
   \begin{Prop}\label{deRhamlemma}
   Let $\delta: K^\times \rightarrow A^\times$ be a character for some connected reduced affoind $C$-algebra $A$. Suppose that 
   $$\delta(x): K^\times \rightarrow R^\times\rightarrow (R/\mathfrak{m}_x)^\times$$
   is deRham and the smooth part $\delta(x)_{\mathrm{sm}}$ of  $\delta(x)$ is trivial on $I_L \subset W_K$ for any point $x\in\mathrm{Sp}R$. Then $\delta$ is deRham and the smooth part $\delta_{\mathrm{sm}}$ of $\delta$ is trivial on $I_L$.
   \end{Prop}
   \begin{proof}
   Suppose that 
   $$\delta(x)_{\mathrm{alg}} = \delta_{\underline{k}'(x)},$$
   for any $x\in\mathrm{Sp}R$ (here $\underline{k}'(x)\in\mathbb{Z}^{\Hom(K,C)}$).
   
    Choose $a\in\mathcal{O}_K$ such that $\mathbb{Q}_p(a) = K$. Then for any $1\leq i \leq [K:\mathbb{Q}_p]$, $m\in\mathbb{N}$, define:
   $$\delta_i^m := \frac{\delta(1+p^ma^i)-1}{p^m} \in A.$$
   Hence for any $x\in \mathrm{Sp}R$, and for $m\gg 0$, we have:
   \begin{align*}
   	\delta^m_i(x) & = \frac{\delta(x)_{\mathrm{sm}}(1+p^ma^i)\prod\limits_{\eta:K\inj C}\eta(1+p^ma^i)^{k'_\eta(x)}-1}{p^m}\\
   	&= \frac{\prod\limits_{\eta:K\inj C}\eta(1+p^ma^i)^{k'_\eta(x)}-1}{p^m} = \sum_{\eta:K\inj C} k'_\eta(x)\eta(a^i)+p^mb,
   \end{align*}
   here $b$ is some element in $\mathcal{O}_C$.
   
    Hence $|\delta^m_i(x)-\delta^{m'}_i(x)| \leq p^{-m}$ for any $m' > m$, which implies $|\delta^m_i-\delta^{m'}_i| \leq p^{-m}$ and $\{\delta_i^m\}_{m\in\mathbb{N}}$ is a Cauchy sequence in $R$ (note that $A$ is a Banach algebra and when $A$ is reduced, the norm of $a\in A$ is defined by $\sup\limits_{x\in\mathrm{Sp}A}\{|a(x)|\}$).
    
     Now choose an order of $\{\eta|\eta:K\inj C\} = \{\eta_1,\dots,\eta_n\}$, and define the matrix $\Delta:= (\eta_j(a^i))_{1\leq j,i\leq n}\in \mathrm{Mat}_{n\times n}(R)$. Then we have $$\lim\limits_{m\rightarrow +\infty} (\delta^m_1(x),\dots,\delta^m_n(x))=(k'_{\eta_1}(x),\dots,k'_{\eta_n}(x)\Delta.$$ Note that $\Delta$ is invertible as $|\det(\Delta)| = N_{K/\mathrm{Q}_p}(a)\prod\limits_{1\leq j <k\leq n}|\eta_k(a)-\eta_j(a)| \neq 0$ by direct computation. If set: $$(k_{\eta_1},\dots,k_{\eta_n}) := \lim\limits_{m\rightarrow +\infty} (\delta^m_1,\dots,\delta^m_n)\Delta^{-1},$$ then we have $k_\eta(x) = k_\eta'(x)$ for any $\eta: K\inj C$ and $x\in\mathrm{Sp}R$.
     
   Now note that for any $\eta\in\Hom(K,C)$, it defines a continuous map $$\mathrm{Sp}R \rightarrow \bar{C}, x\mapsto k_\eta(x),$$ whose image is in $\mathbb{N}$ (note that $\mathbb{N}$ carries the discrete topology as a topological subspace of $\bar{C}$ ). Hence $k_\eta(x)$ is constant for any  $x\in\mathrm{Sp}A$ as $A$ is connected.
   
   Hence if we set $$\delta_s:= \delta\cdot\delta_{\underline{k}}^{-1}$$ for $\underline{k}:= (k_\eta)_\eta$, then $\delta_s(x)$ is trivial on $I_{L}$ for any $x\in\mathrm{Sp}A$, which implies $\delta_s$ is trivial on $I_{L}$ (note that $R$ is reduced). Hence $\delta$ is deRham with smooth part $\delta_s$ and algebraic part $\delta_{\underline{k}}$.  
   \end{proof}
   
   \begin{Prop}\label{Prop:deRhamcondition}
   Let $A$ be a connected affoind $C$-algebra. Let $\delta_i: K^\times \rightarrow A^\times$ be characters and $M_i$ be filtered $(\varphi,N,G_{L/K})$-modules over $\mathcal{R}_{K,A}$ for $i=1,2$, such that
   $$\mathbf{D}(M_1)(\delta_1)\cong \mathbf{D}(M_2)(\delta_2),$$
   then $\delta_1\delta_2^{-1}$ is a deRham character and the smooth part is trivial on $I_L$.
   
    Moreover, $M_1(\delta_1\delta_2^{-1})\cong M_2$
   \end{Prop}

   \begin{proof}
   Let $\delta:= \delta_1\delta_2^{-1}$. For any $\mathfrak{m}\in\mathrm{Sp}A$ and $n\in \mathbb{N}$, by lemma \ref{Inj}, one has $$\delta_{\mathfrak{m}^n}:= K^\times \xrightarrow{\delta} A^\times \rightarrow (A/\mathfrak{m}^n)^\times$$
   is deRham and the smooth part is trivial restricted on $I_{L}$. It is obvious that the algebraic part of $\delta_{\mathfrak{m}^n}$ is the same as the algebraic part of $\delta_{\mathfrak{m}}$. Hence by proposition \ref{deRhamlemma}, there exists an algebraic character $\delta_{\mathrm{alg}}: K^\times \rightarrow A^\times$, such that $\delta_{\mathfrak{m}^n}\cdot \delta_{\mathrm{alg}}^{-1}$ is trivial on $I_{L}$ (here we apply $R/\sqrt{0}$ to proposition \ref{deRhamlemma}). Let $\delta_s := \delta\cdot \delta_{\mathrm{alg}}^{-1}$, then the discussion above means $\delta_s(I_{L}) \in 1+\cap \mathfrak{m}^n = 1$ as $A$ is noetherian. Hence $\delta_s$ is smooth and restricted on $I_{L}$ is trivial.
   
   Moreover, by theorem \ref{Thm: PropofBergerFun}, one has $$M_1(\delta_1\delta_2^{-1})\cong M_2$$ as $\mathbf{D}(M_1(\delta_1\delta_2^{-1}))\cong \mathbf{D}(M_2)$.
   \end{proof}

   Let $X(W_K/I_L)$ denote the rigid $C$-group that assigns to $A$ the group
   $$\{\delta:W_K\rightarrow A^{\times}|\delta\mathrm{\ is\ a\ character\ with\ }\delta|_{I_L} \mathrm{\ is\ trivial}\}$$
   Then $X(W_K/I_L)$ can be regarded as a closed subspace of $\mathcal{T}$ by the local Artin map $\mathrm{Art}: W_K^{\mathrm{ab}}\xrightarrow{\sim} K^{\times}$. Let $\mathfrak{M}^{L/K,d,\leq 0}$ denote the sub stack of  $\mathfrak{M}^{L/K,d}$ such that for any $M\in\mathfrak{M}^{L/K,d,\leq 0}$, $\mathcal{F}^0(M_\eta) \neq 0$ but $\mathcal{F}^1(M_\eta) = 0$ for any $\eta\in\Hom(K,C)$.
   
   \begin{Cor}\label{Parameterstack}
   	Define the $X(W_K/I_L)$-action on $\mathcal{T}_K\times\mathfrak{M}^{L/K,d,\leq 0}$ as following:
   	\begin{align*}
   		X(W_K/I_L)\times \mathcal{T}_K\times\mathfrak{M}^{L/K,d,\leq 0} & \rightarrow \mathcal{T}_K\times\mathfrak{M}^{L/K,d,\leq 0}\\
   		(\delta,\delta',M)&\mapsto (\delta'\delta^{-1},M(\delta)).
   	\end{align*}
   	Then we have the following isomorphism:
   	\begin{align*}
   		[(\mathcal{T}_K\times\mathfrak{M}^{L/K,d,\leq 0})/X(W_K/I_L)]&\xrightarrow{\sim} \mathcal{Z}^{L/K,d,\leq 0}\\
   		(\delta,M)&\mapsto \mathbf{D}(M)(\delta)
   	\end{align*}
   	
   \end{Cor}

   \begin{proof}
   	Let $A$ be a connected affinoid $C$-algebra. 
   	
   	Note that, for any filtered $(\varphi,N,G_{L/K})$-module $M$ over $A$ of rank $d$, there exists a unique algebraic character $\delta_{\underline{k}}$ for some $k\in\mathbb{Z}^{\Hom(K,C)}$ such that $M(\delta_{\underline{k}})$ is in $\mathcal{Z}^{L/K,d,\leq 0}$. Hence we have a surjective morphism 
   	$$\mathcal{T}_K\times\mathfrak{M}^{L/K,d,\leq 0} \surj \mathcal{Z}^{L/K,d,\leq 0}$$
   	Then this corollary is exactly the restatement of proposition $\ref{Prop:deRhamcondition}$.
   \end{proof}
   
   Consider the decomposition 
   \begin{align*}
   	\mathcal{W}\times \mathbb{G}_m^{\mathrm{rig}} & \xrightarrow{\sim} \mathcal{T} \\
   	(\delta,b)& \mapsto (a\mapsto b^{v_K(a)}\cdot \delta(\frac{a}{\varpi^{v_K(a)}})).
   \end{align*}
   Note that such decomposition is not canonical and depends on the choice of $\varpi$. 
   
   Let $I_K^{a}$ (resp. $I_L^{a}$) denote the image of $I_{K}$ (resp. $I_L$) in the abelianization $W_K^{\mathrm{ab}}$ (resp. $W_L^{\mathrm{ab}}$) (note that it is different from the abelianization of the inertia group). Let $X(I^{a}_K/I^{a}_L)$ denote the rigid $C$-group that assigns to $A$ the group
   $$\{\delta:I_K^{a}\rightarrow A^{\times}|\delta\mathrm{\ is\ a\ character\ with\ }\delta|_{I^{a}_L} \mathrm{\ is\ trivial}\}$$
   Then $X(I^{a}_K/I^{a}_L)$ can be regarded as a closed subspace of $\mathcal{W}$ as the image of $I_K$ is $\mathcal{O}_K$ by the local Artin map $\mathrm{Art}: W_K^{\mathrm{ab}}\xrightarrow{\sim} K^{\times}$. Then we have (depends on the choice of $\varpi$)
   $$X(I^{a}_K/I^{a}_L)\times \mathbb{G}^{\mathrm{rig}}_{m}\cong X(W_K/I_L),$$
   which is induced from the decomposition above.
   
    Even though the morphisms in corollary \ref{Parameterstack} are canonical and do not depend on the choice of $\varpi$, but the following statement strongly depends on the choice of $\varpi$ and the morphism is not canonical.   
   
   \begin{Theo}\label{Galoiscover}
   Via the decomposition $\mathcal{T} \cong \mathcal{W} \times\mathbb{G}_m^{\mathrm{rig}}$ (depends on a choice of the uniformizer $\varpi\in K$), one can regard $\mathcal{W}$ as a sub rigid group $C$-space of $\mathcal{T}$. Then the morphism:
   \begin{align*}
   	\mathcal{W}_K\times\mathfrak{M}^{L/K,d,\leq 0} & \rightarrow \mathcal{Z}^{L/K.d}\\
   	(\delta,M) &\mapsto \mathbf{D}(M)(\delta)
   \end{align*}
   is a Galois cover with
   $$\Aut_{\mathcal{Z}^{L/K,d}}(\mathcal{W}_K\times \mathfrak{M}^{L/K,d,\leq 0}) \cong X(I^{a}_K/I^{a}_L)$$
   \end{Theo}
   
   \begin{proof}
   	Note that for any group $C$-rigid stack action $G \times X \rightarrow X$ on a rigid $C$-stack $X$ and a subgroup rigid $C$-stack $G'\subseteq G$, we have the following isomorphism:
   	\begin{align*}
   	(G\times X) / G' & \xrightarrow{\sim} (G/G') \times X\\
   	(g,x) & \mapsto (\bar{g},gx),
   	\end{align*}
   	where the action $G'\times G \times X \rightarrow G \times X$ is define by $$(h,g,x)\mapsto (gh^{-1},hx).$$
   	Hence we have:
   	\begin{align*}
   	\mathcal{Z}^{L/K,d} \cong& \left[(\mathcal{T}_K\times\times \mathfrak{M}^{L/K,d,\leq 0})/X(W_K/I_L)\right]\\  = & \left[(\mathcal{T}_K\times\times \mathfrak{M}^{L/K,d,\leq 0})/(X(I^{a}_K/I^{a}_L)\times \mathbb{G}^{\mathrm{rig}}_m)\right]\\
   	\cong & \left[(\mathcal{W}_K\times \mathbb{G}^{\mathrm{rig}}_m\times \mathfrak{M}^{L/K,d,\leq 0})/(X(I^{a}_K/I^{a}_L)\times \mathbb{G}^{\mathrm{rig}}_m)\right]  \\
   \cong & \left[(\mathcal{W}_K\times \mathfrak{M}^{L/K,d,\leq 0})/X(I^{a}_K/I^{a}_L)\right]
   	\end{align*}
   	and get the conclusion.   
   	
   \end{proof}
   
   \section{Global Paraboline $(\varphi,\Gamma_K)$-modules}\label{Global}
   
   Through out this section, we also fix $K$ a $p$-adic local filed and use the letter $L$ to denote some Galois extension over $K$ as before. We write $C$ for the coefficient field, which is finite over $\mathbb{Q}_p$ and large enough. Moreover, for a group $G$, write $X(G)$ for the group $C$-rigid space of the character of $G$.
   
   In particular, we require $|\Hom(K,C)|=[K:\mathbb{Q}_p]$.
   
   \subsection{Parameter Spaces}\label{ParemeterSpaces}
   In this subsection, we are going to construct a rigid space $\mathcal{S}^{\underline{k}}_\tau$, which is a chart of some connected component $\mathcal{Z}_{\tau}^{\underline{k}}$ of $\mathcal{Z}^{L/K,d}$ (see the definition below). This will be used to define the \textit{parameter spaces} of the \textit{refined paraboline varieties}. By construction, $\mathcal{S}^{\underline{k}}_\tau$ is the form $\mathcal{T}_\tau\times \mathrm{Flag}$, where $\mathcal{T}_\tau$ will be used to define the \textit{parameter spaces} of the \textit{paraboline varieties} and be identified with the parameter space of the \textit{Hecke eigenvarieties} on the automorphic side.

   Let $(\rho,V)$ be an absolutely irreducible WD representation ($N=0$) of $K$ over $C$, then $\rho$ is of the form $\mathrm{Ind}_{W_\tau}^{W_K}(\tilde{\tau}\otimes\delta_0)$ by proposition \ref{Prop:WD-decomp}, where $\delta_0: W_\tau\rightarrow C^\times$ is an unramified character. In this case, $\rho$ is called \textit{of type} $\tau$. Through out this subsection, fix a smooth absolutely irreducible representation $\tau$ of rank $d_\tau$ over $C$. Let $d= d_\tau e_\tau$. Let $\mathfrak{W}^{L/K,d}_{n_\tau=1}$ denote the connected component  of $\mathfrak{W}^{L/K,d}$, consisting of irreducible WD representation of type $\tau$. Then $\mathfrak{W}^{L/K,d}_{n_\tau=1}$ can be defined over $C$, and is isomorphic to $\mathfrak{W}^{K_\tau/K_\tau,1}$ according to theorem \ref{Thm:WDstack} (recall that $K_{\tau}$ is the unique unramified extension over $K$ of degree $e_{\tau}$, and write $W_\tau\subseteq W_K$ for the Weil group of $K_\tau$). The following are some terminologies about quasi-deRham $(\varphi,\Gamma_K)$-modules used in this subsection and subsequent part of this article.
   
   \begin{Def}\label{Def:Phigammaoftypetau}
   	Let $A$ be a affinoid $C$-algebra. A quasi-deRham $(\varphi,\Gamma_K)$-module $D$ over $A$ with $D\cong\mathbf{D}(M)(\delta)$ for some filtered $(\varphi,N,G_{L/K})$-module $M$ and some continuous character $\delta: K^{\times} \rightarrow A^{\times}$ is called
   	\begin{enumerate}
   		\item \textit{quasi-free} if $M$ is a free $(L_0\otimes_{\mathbb{Q}_p} A)$-module;
   		\item \textit{of (filtration) weight type} $\underline{k}\in (\mathbb{Z}_{+}^{d'})^{\Hom(K,C)}$ if the filtration weight of $M$ is $\underline{k}$, here $d'$ is the rank of $D$;
   		\item \textit{irreducible} if there exists some type $\tau$' such that $\mathrm{WD}(M)$ corresponds to an element in $\mathfrak{W}^{L/K,d'}_{n_{\tau'}=1}(A)$, where $\mathrm{WD}(M)$ is the WD representation attached to $M$ and $d'=d_{\tau'}e_{\tau'}$. In this case, we also called $D$ is \textit{of type} $\tau'$.
   		
   		By theorem \ref{Thm:WDstack}, it is equivalent to say that there exists a line bundle $\mathcal{L}$ over $\mathrm{Spec}(A)$ with an unramified character $\delta: W_{\tau'} \rightarrow A^{\times} \cong \mathrm{End}_A(\mathcal{L})$ such that
   		$$\mathrm{WD}(M) \cong \mathrm{Ind}_{W_{\tau'}}^{W_K}(\tilde{\tau}'\otimes\mathcal{L}),$$
   		here $\mathcal{L}$ is regarded as a rank one WD representation of $W_{\tau'}$. 
   		
   	\end{enumerate} 
   \end{Def}
   	In particular, if $D$ is of type $\tau'$ and of filtration weight type $\underline{k}$, we say $D$ is \textit{of type} $(\tau',\underline{k})$ for short.
   	\begin{Def}
   		Let $\mathcal{Z}^{\underline{k}}_{\tau}$ denote the substack of $\mathcal{Z}^{L/K,d}$, such that on any affinoid $C$-algebra $A$,
   		$\mathcal{Z}^{\underline{k}}_{\tau}(A)${ is the groupoid of } $(\varphi,\Gamma_K)\text{-module\ } D \mathrm{\ over\ } A\mathrm{\ of\ type\ }(\tau,\underline{k})$.
   	\end{Def}

   \begin{Rmk}
   	By section \ref{QdeRhamPhiGammaMod}, for a quasi-deRham $(\varphi,\Gamma_K)$-module $D$, the expression $D\cong\mathbf{D}(M)(\delta)$ is not unique. But one can see that the quasi-freeness and the irreducibility of $D$ are independent of the choice of $M$ and $\delta$, while the type $\tau$ and the filtration weight type $\underline{k}$ are not. More precisely, suppose that $D$ is irreducible. If there exist two filtered $(\varphi,N,G_{L/K})$-modules $M_1$ (of type $\tau_1$ and filtration weight $\underline{k}_1$) and $M_2$ (of type $\tau_2$ and filtration weight $\underline{k}_2$), and two continuous character $\delta_1,\delta_2$ such that 
   	$$D\cong \mathbf{D}(M_1)(\delta_1)\cong \mathbf{D}(M_2)(\delta_2),$$
   	then by proposition \ref{Prop:deRhamcondition}, there exist an element $g\in W_K$ and smooth character $\delta':I_K\rightarrow C^{\times}$ such that $\tau_1\cong\tau_2^{g}(\delta')$, and a element $\underline{k}'\in \mathbb{Z}^{\Hom(K,C)}$ such that $\underline{k}_1=\underline{k}_2-\underline{k}'$. Hence when we say $D$ is of type $(\tau,\underline{k})$, the type $\tau$ can be regarded as an element in the set of $W_K$-conjugacy classes of of irreducible smooth representations of $I_K$ up to twist by smooth characters, and $\underline{k}$ can be regarded as an element in $ (\mathbb{Z}_{+}^{d}/\mathbb{Z})^{\Hom(K,C)}$.
   \end{Rmk}
   
   Let $X$ be a $C$-scheme, let $(\mathcal{L},\rho,N=0)$ be a rank one Weil-Deligne representation of $K_\tau$ over $X$ such that $\rho$ is unramified. Then $\mathcal{L}$ is a line bundle over $X$, and $\rho$ is a unramified character $W_\tau\rightarrow \mathcal{O}_X^\times(X)$. Let $s$ be a $q^{e_\tau}$-Frobenius element in $W_\tau$ ($q$ is the cardinality of the residue field $k_K$), then $\rho$ is determined by $\rho(s)$. In the language of stack, the analysis above means 
   $$\mathfrak{W}^{K_\tau/K_\tau,1}\cong \mathbb{G}_m\times [\mathrm{Spec}(C)/\mathbb{G}_m]\cong [\mathbb{G}_m/\mathbb{G}_m],$$
   where $\mathbb{G}_m$ acts trivially on $\mathbb{G}_m$.
   
   By direct computation, the universal sheaf $\mathcal{L}^{\mathrm{univ}}$ of WD representations of $K_\tau$ on $[\mathbb{G}_m/\mathbb{G}_m]$ can be describe as follows:
   
   Let $X$ be a $C$-scheme, with a morphism $f: X\rightarrow [\mathbb{G}_m/\mathbb{G}_m]$. By the universal property of fiber product, $f$ corresponds to a pair $(f_1,f_2)$, where $f_1$ is a morphism $X\rightarrow \mathbb{G}_m$ and $f_2$ is a morphism $X\rightarrow[\mathrm{Spec}(C)/\mathbb{G}_m]$. If we denote $\mathbb{G}_m$ by $\mathrm{Spec}(C[T^{\pm 1}])$, then $f_1$ induces a ring homomorphism $f_1^*:C[T^{\pm 1}]\rightarrow \mathcal{O}_X(X)$. We pick an element $f_1^*(T)\in \mathcal{O}^\times_X(X) = \mathrm{End}(\mathcal{L})$ and construct a line bundle $\mathcal{L}$ via $f_2$. Then We define 
   $$\mathcal{L}^{\mathrm{univ}}(f) :=(\mathcal{L},\rho),$$
   where $\rho(s) = f_1^*(T)$.
   
   

   Now for each $\eta\in \Hom(K,C)$, we fix an $\eta$-filtration weight $$\underline{k}_\eta:= (k_{\eta,1}<\cdots<k_{\eta,d}=0),$$
   and denote $\underline{k} := (\underline{k_\eta})_{\{\eta:K\hookrightarrow C\}}$. Let $\mathfrak{M}^{L/K,\underline{k}}_{n_\tau=1}$ denote the connected component of $\mathfrak{M}^{L/K,d,\leq 0}$, consisting of those filtered $(\varphi,N,G_{L/K})$-modules, whose filtration weights are $\underline{k}$ and corresponding WD representations are of type $\tau$ (hence $N=0$). Let $B_d$ denote the standard Borel subgroup of $\mathrm{GL}_d$ (over $C$) consisting of upper triangular matrices. By theorem \ref{FWDstack}, one can compute that $$\mathfrak{M}^{L/K,\underline{k}}_{n_\tau=1}\cong (\prod\limits_{\eta:K\inj C}[(\mathrm{GL}_d/B_d)/\mathrm{GL}_d])\times_{[\mathrm{Spec}(C)/\mathrm{GL}_d]} [\mathbb{G}_m/\mathbb{G}_m]$$
   where $\mathrm{GL_d}$ acts on the flag variety $(\mathrm{GL_d}/B_d)$ via conjugation, and the morphism $[\mathbb{G}_m/\mathbb{G}_m]\rightarrow [\mathrm{Sp}(C)/\mathrm{GL}_d]$ comes from the natural embedding $\mathbb{G}_m\cong \mathrm{GL}_1\inj \mathrm{GL}_d$. The following proposition shows that the Artin stack $\mathfrak{M}^{L/K,\underline{k}}_{n_\tau=1}$ becomes a scheme after the base change along the natural projection $\mathrm{pr}: \mathbb{G}_m \surj [\mathbb{G}_m/\mathbb{G}_m]$. 
   \begin{Prop}
   	Choose a basis of the irreducible WD representation  $\mathrm{Ind}_{W_{\tau}}^{W_K}(\tilde{\tau})$, i.e. $\mathrm{Ind}_{W_{\tau}}^{W_K}(\tilde{\tau}): W_K \rightarrow \mathrm{GL}_d(C)$. The $C$-scheme $$M_{\tau}:=(\prod\limits_{\eta:K\inj C}\mathrm{GL}_d/B_d)\times \mathbb{G}_m$$ represents the functor
   	$$X \rightarrow \{(\delta:W_{\tau}\rightarrow \mathcal{O}_X^\times,\mathcal{F}^\bullet)| \delta  \mathrm{\ is \ unramified}\}$$
   	where $\mathcal{F}^\bullet$ is a filtration of $\left(\mathrm{Ind}_{W_{\tau}}^{W_K}(\tilde{\tau}\otimes_C \delta)\right)\otimes_{\mathbb{Q}_p} K$, which equals to the functor
   	$$X\mapsto \{\mathrm{free\ of\ rank\ }d \mathrm{\ filtered\ }(\varphi,N,G_{L/K})\text{-}\mathrm{module\ over\ }X \mathrm{\ of\ type\ }\tau\}/\sim.$$ 
   	The  morphism
   	$$\phi_{\tau}:M_\tau\rightarrow \mathfrak{M}_{n_\tau= 1}^{L/K,\underline{k}},$$
   	induced by the natural projections $$(\mathrm{GL}_d/B_d)\surj [(\mathrm{GL}_d/B_d)/\mathrm{GL}_d] \text{ and } \mathrm{pr}:\mathbb{G}_m\surj [\mathbb{G}_m/\mathbb{G}_m],$$ makes the following diagram 
   	\begin{equation*}
   		\begin{tikzcd}
   			M_\tau \arrow{r}\arrow{d}{\phi_\tau}&\mathbb{G}_m\arrow{d}{\mathrm{pr}} \\
   			\mathfrak{M}_{n_\tau= 1}^{L/K,\underline{k}} \arrow{r} & \mathit{[}\mathbb{G}_m/\mathbb{G}_m\mathit{]}
   		\end{tikzcd}.
   	\end{equation*}
   	commute and become a Cartesian product. 
   	
   	Moreover, according to theorem \ref{FWDstack}, if we denote the universal character by $\delta^{\mathrm{univ}}$, and the universal filtration of $\left(\mathrm{Ind}_{W_{\tau}}^{W_K}(\tilde{\tau}\otimes_C \delta^{\mathrm{univ}})\right)\otimes_{{\mathbb{Q}}_p} K$ on $M_\tau$ by $\mathcal{F}^{\mathrm{univ},\bullet}$, then $\mathrm{FWD}(\delta^{\mathrm{univ}},\mathcal{F}^{\mathrm{univ},\bullet})$ is the induced filtered $(\varphi,N,G_{L/K})$-module via $\phi_\tau$ (where $\mathrm{FWD}$ is the isomorphism in theorem \ref{FWDstack}).
   \end{Prop}
   
   \begin{proof}
   	We only prove that the diagram is Cartesian as the other assertions are just direct corollary of theorem \ref{FWDstack}, combined with our discussions as above.
   	
   	Let $\mathrm{Flag}$ denote $\prod\limits_{\eta:K\inj C}\mathrm{GL}_d/B_d$ for short. Then consider the following commutative cubic diagram (where all the arrows are canonical in the understandable way.)
   	\begin{equation*}  		
   		\begin{tikzcd}
   			& M_\tau \arrow[ld, "q_\tau"] \arrow[dd] \arrow[rr] &                                            & \mathrm{Flag} \arrow[ld] \arrow[dd] \\
   			{\mathfrak{M}_{n_\tau=1}^{L/K,\underline{k}}} \arrow[dd] \arrow[rr] &                                                   & {[\mathrm{Flag}/\mathrm{GL}_d]} \arrow[dd] &                                     \\
   			& \mathbb{G}_m \arrow[ld, "\mathrm{pr}"] \arrow[rr] &                                            & \mathrm{Sp}(C) \arrow[ld]           \\
   			{[\mathbb{G}_m/\mathbb{G}_m]} \arrow[rr]                            &                                                   & {[\mathrm{Sp}(C)/\mathrm{GL}_d]}            &                                    
   		\end{tikzcd}
   	\end{equation*}
   By theorem \ref{FWDstack} again, the front square diagram is Cartesian, and by definition, the right and the back square diagrams are also  Cartesian (while the bottom one is not), which implies the left square diagram is Cartesian.
   \end{proof}

   \begin{Lemma}\label{Lem:Twistingbychar}
   Let $\delta: W_K\rightarrow C^{\times}$ be a smooth character. Denote $\delta_0 := \delta|_{W_\tau}$. Then there is an isomorphism
   $$\mathrm{Ind}_{W_\tau}^{W_K}(\tilde{\tau}\otimes\delta_0)\rightarrow \mathrm{Ind}_{W_\tau}^{W_K}(\tilde{\tau})\otimes\delta$$
   of $W_K$ representations.
   
   Moreover, for any smooth character $\delta: W_{\tau}\rightarrow C^{\times}$, then $$\mathrm{Ind}_{W_\tau}^{W_K}(\tilde{\tau}) \cong \mathrm{Ind}_{W_\tau}^{W_K}(\tilde{\tau}\otimes\delta)$$ if and only if $\tilde{\tau}\otimes\delta\cong\tilde{\tau}^{g}$ for some $g\in W_K/W_\tau$.
   \end{Lemma}

   \begin{proof}
   	Consider the $W_\tau$-equivariant map:
   	\begin{align*}
   		\tilde{\tau}\otimes\delta_0 & \rightarrow \mathrm{Ind}_{W_\tau}^{W_K}(\tilde{\tau})\otimes\delta\\
   		v\otimes 1 & \mapsto (1\otimes v)\otimes 1
   	\end{align*}
   Hence by Frobenius reciprocity, one has the induced $W_K$-equivariant map
   \begin{align*}
   	\mathrm{Ind}_{W_\tau}^{W_K}(\tilde{\tau}\otimes\delta_0)& \rightarrow \mathrm{Ind}_{W_\tau}^{W_K}(\tilde{\tau})\otimes\delta\\
   	g\otimes(v\otimes1) & \mapsto (g\otimes v)\otimes \delta(g)
   \end{align*}
   The map is obviously surjective, and is injective by considering the $C$-dimension on both sides.
   
   By our computation in section 1, one has
   $$\mathrm{Ind}_{W_{\tau}}^{W_K}(\tilde{\tau})\cong \bigoplus_{g\in W_K/W_{\tau}} \tilde{\tau}^{g}$$
   as $W_{\tau}$ representations. Then the second assertion follows. 
   \end{proof}

   Let $X_{\tau}$ be the subgroup scheme of $X(W_{K}/I_L)$, representing the characters $\delta$ of $W_{K}/I_L$ such that $\mathrm{Ind}_{W_\tau}^{W_K}(\tilde{\tau})\otimes \delta\cong \mathrm{Ind}_{W_\tau}^{W_K}(\tilde{\tau})$.
   
   \begin{Rmk}\label{Rmk:Xtau}
   	Let $Y_{\tau}$ denote $\{\psi \in X(I^{a}_K/I^{a}_L)| \tau\otimes\psi\cong \tau^{g} \text{ for some } g\in W_K/W_{\tau} \}$. The canonical map
   	\begin{equation*}
   		\begin{aligned}
   			X_{\tau} & \rightarrow  Y_{\tau}\\
   			\delta &\mapsto \delta|_{I_K}
   		\end{aligned}
   	\end{equation*}
   is an surjective morphism of group scheme with kernel $X(W_{K}/I_K)$ (recall that $I_K^{a}$ (resp. $I_L^{a}$) denotes the image of $I_{K}$ (resp. $I_L$) in the abelianization $W_K^{\mathrm{ab}}$).
   \end{Rmk}
   
   \begin{proof} It just rephrases lemma \ref{Lem:Twistingbychar}.
   \end{proof}
   Now we give a chart of $\mathcal{Z}_{\tau}^{\underline{k}}$ in a explicit way and describe the associated sheaf of $(\varphi,\Gamma_K)$-modules.
   
   Recall that $\Omega_L$ denotes the set of $W_K$-conjugacy class of irreducible $I_K/I_L$-representations (also fix a representative element in each class). If we denote
   $$\Omega_{\tau}:=\{\tau'\in \Omega_L\ |\ [\tau']= [\tau\otimes\delta]\text{ for some }\delta\in X(W_K/I_L)\},$$
   then $X(W_K/I_L)$ acts transitively one $\Omega_{\tau}$ with kernel $X_{\tau}$. By corollary \ref{Parameterstack}, one has
   $$\mathcal{Z}_{\tau}^{\underline{k}} \cong [\left(\coprod_{\tau'\in\Omega_{\tau}} \mathcal{T}\times \mathfrak{M}_{n_{\tau'}=1}^{L/K,\underline{k}}\right)/X(W_K/I_L)]\cong[ (\mathcal{T}\times \mathfrak{M}_{n_{\tau'}=1}^{L/K,\underline{k}})/X_{\tau}].$$
   
   Let $\mathcal{T}_{\tau}$ denote the $C$-rigid space representing the functor of continuous characters of the image of the norm map $N_{K_\tau/K}: K_\tau^\times\rightarrow K^\times$. As $K_\tau$ is unramified over $K$, then $N_{K_\tau/K}(K_\tau^\times) = \{a\in K^\times|v_K(a)\in e_\tau\mathbb{Z}\}$. Then $\mathcal{T}_{\tau}$ is a quotient of $\mathcal{T}$ with kernel $X(W_K/W_{\tau})$. The canonical map $\mathcal{T}\times X(W_{\tau}/I_K))\rightarrow \mathcal{T}_{\tau}, (\delta_1,\delta_2)\mapsto \bar{\delta}_1\delta_2$ induces an isomorphism
   $$[(\mathcal{T}\times X(W_{\tau}/I_K))/X(W_{K}/I_{K})]\cong \mathcal{T}_{\tau}.$$
   
    Let $Y_\tau$ denote the group $C$-rigid space representing the functor $$A\rightarrow \{\psi \in X(I_K/I_L)(A)| \tau\otimes\psi\cong \tau^{g} \text{ for some } g\in W_K/W_{\tau} \}$$.
   By remark \ref{Rmk:Xtau}, one has short exact sequence
   $$0\rightarrow  X(W_{K}/I_{K})\rightarrow X_{\tau} \rightarrow Y_{\tau}\rightarrow 0$$

   Hence
   \begin{equation*}
   	\begin{aligned}
   		[ (\mathcal{T}\times \mathfrak{M}_{n_{\tau'}=1}^{L/K,\underline{k}})/X_{\tau}]&\cong [((\mathcal{T}\times X(K_{\tau}/I_K)\times \mathrm{Flag})/X(W_K/I_K))/(Y_{\tau}\times \mathbb{G}_m)]\\
   		& \cong [(\mathcal{T_{\tau}}\times \mathrm{Flag})/(Y_{\tau}\times \mathbb{G}_m)]\\
   	\end{aligned}
   \end{equation*}

   The above discussion can be summarized as the following proposition:
   
   \begin{Prop}
   	The morphism 
   \begin{align*}
   	[\mathcal{T}_{\tau}\times \mathrm{Flag}/\mathbb{G}_m]&\rightarrow \mathcal{Z}^{\underline{k}}_{\tau}
   \end{align*}
   is \'{e}tale with Galois group $Y_{\tau}$.
   \end{Prop}

   \begin{Def}
   	We denote $\mathcal{S}^{\underline{k}}_{\tau}:= \mathcal{T}_{\tau}\times \mathrm{Flag}$, and $\theta$ the composition of the canonical morphisms:
   	$$\mathcal{S}_{\tau}^{\underline{k}} \rightarrow [\mathcal{T}_{\tau}\times \mathrm{Flag}/\mathbb{G}_m] \rightarrow \mathcal{Z}^{\underline{k}}_{\tau}.$$
   \end{Def}

   \begin{Rmk}
   	\label{RmkforParameterspace}
   	{\ }
   	\begin{enumerate}
   		\item Via pulling back along $\theta$, we have a canonical family $D^{u}:=\theta^{*}(D^{\mathrm{univ}})$ of $(\varphi,\Gamma_K)$-module  over $\mathcal{S}_\tau^{\underline{k}}$, here $D^{\mathrm{univ}}$ is the universal family of $(\varphi,\Gamma_K)$-module over $\mathcal{Z}_{\tau}^{\underline{k}}$. 
   		
   		\item
   		Let $\delta^{\mathrm{univ}}$ denote the universal character of $K_{\tau}^{\times}$ over $\mathcal{T}_{\tau}$. By the construction, one can compute:
   		$$\mathrm{pr}_*(D^{u,r})[1/t] \cong \mathrm{Ind}_{K_{\tau}}^{K}\left((\tilde{\tau}\otimes_{C}\mathcal{R}^{r}_{K_{\tau},C})^{H_{K_{\tau}}}\otimes_{C} \mathcal{R}_{K_{\tau},C}(\delta^{\mathrm{univ}})\right)[1/t],$$
   		here $r>r(K)$, $\mathrm{pr}$ is the natural projection $\mathcal{S}_{\tau}^{\underline{k}}\rightarrow \mathrm{Flag}$ and $\mathrm{Ind}_{K_\tau}^{K}$ is the induction functor defined in \cite[Sect. 2.2]{LiuR07}.
   		
   		Let $\mathcal{F}^{\mathrm{univ},\bullet}$ denote the universal filtration over $\mathrm{Flag}$ (hence a filtration on $A\otimes_C(\mathrm{Ind}_{W_{\tau}}^{W_K}\tilde{\tau})\otimes_{K_0}K$ for each affinoid subdomain $\mathrm{Sp}(A)\subseteq \mathrm{Flag}$). Then for $n$ large enough, the canonical isomorphism
   		\begin{equation*}
   			D^{u,r}(\mathrm{Sp}(A))[1/t]\otimes^{\iota_n}_{\mathcal{R}_{K,C}} K_n[\![t]\!]\cong A\otimes_C(\mathrm{Ind}_{W_{\tau}}^{W_K}\tilde{\tau}\otimes\delta^{\mathrm{univ}})\otimes_{K_0}^{\iota_n}K_n(\!(t)\!)
    		\end{equation*}
   	induces a filtration on $D^{u,r}(\mathrm{Sp}(A))[1/t]\otimes ^{\iota_n}_{\mathcal{R}_{K,C}} K_n[\![t]\!]$, where
   	\begin{equation*}
   		\mathcal{F}^{i}(A\otimes_C(\mathrm{Ind}_{W_{\tau}}^{W_K}\tilde{\tau})\otimes_{K_0}^{\iota_n}K_n(\!(t)\!)):= \sum_{j\in\mathbb{Z}} \mathcal{F}^{j}(A\otimes_C\mathrm{Ind}_{W_\tau}^{W_K}\tilde{\tau}\otimes_{K_0}^{\iota_n}K)\otimes_K t^{i-j}K_n[\![t]\!].
   	\end{equation*}
   
        Then $D^{u} = D^{u,r} \otimes_{\mathcal{R}_{K,C}^{r}}\mathcal{R}_{K,C}$ for 
        \begin{equation*}
        	D^{u,r}=\{x\in D^{u,r}[1/t]\ |\ \iota_{n}(x) \in \mathcal{F}^{0} (D^{u,r}[1/t]\otimes ^{\iota_n}_{\mathcal{R}_{K,C}} K_n[\![t]\!])\}.
        \end{equation*}
   		\item Even though the construction of $\mathcal{S}_\tau^{\underline{k}}$, $\mathcal{Z}_\tau^{\underline{k}}$ and the canonical map $\theta$ does not depend on the choose of uniformizer $\varpi$ in $K$. We can construct the $(\varphi,\Gamma_K)$-module $\mathcal{D}^{\mathrm{u}}$ on $\mathcal{S}_\tau^{\underline{k}}$ via choosing a uniformizer as follows.
   		
   		Choose a uniformizer $\varpi\in K$, then $\varpi^{e_\tau}$ is a uniformizer of $K_\tau$. Then the universal character $\delta^{\mathrm{u}}$ of $N_{L/K}(K^\times_\tau)$ on  $\mathcal{S}_{\tau}^{\underline{k}}$ can be decomposed by $\delta_1\delta_2$, where 
   		$\delta_1(a) := \delta^{\mathrm{u}}(\varpi^{e_\tau})^{v_{K_\tau}(a)}$ is unramified and $\delta_2$ is the unique character such that $\delta_2|_{\mathcal{O}_K^\times} =\delta^{\mathrm{u}}|_{\mathcal{O}_K^\times}$ and $\delta_2(\varpi^{e_\tau}) = 1$. We can extend $\delta_2$ to a character $\tilde{\delta}_2$ of $K^\times$ (for example, set $\tilde{\delta}_2(\varpi) = 1$). Then we can define
   		$$\mathcal{D}^{\mathrm{u}} := \mathbf{D}(\mathrm{Ind}_{W_\tau}^{W_K}(\tilde{\tau}\otimes_C \delta_1))(\tilde{\delta_2}).$$
   		Up to isomorphism, $\mathcal{D}^{\mathrm{u}}$ does not depend on the extension $\tilde{\delta}_2$ of $\delta_2$, and also not depend on the choice of the uniformizer $\varpi$ by definition.
   		
   		\item Let $X$ be a rigid $C$-space, and let $f$ be a morphism $X\rightarrow \mathcal{S}_\tau^{\underline{k}}$. From now on, we use the notation $\mathcal{R}(f)$ to denote the pullback $f^*(\mathcal{D}^{\mathrm{u}})$ of the universal $(\varphi,\Gamma_K)$-module $\mathcal{D}^{\mathrm{u}}$.
   		
   		\item For a point $x = (\delta_x,\mathcal{F}^{\bullet})\in \mathcal{S}_\tau^{\underline{k}}$, the specialization $\mathcal{R}(x) = \mathcal{D}^{\mathrm{u}}_x$ can be computed via the recipe in (2). One can also compute it in the following way. Let $V_0 := \mathrm{Ind}_{W_\tau}^{W_K}(\tilde{\tau})$. By theorem \ref{FWDstack}, one can define a filtered $(\varphi,N,G_{L/K})$-module $M_x:= \Omega(\mathcal{F}^{\bullet},V_0)$. After perhaps enlarging $C$, let $\delta: K^\times \rightarrow k(x)^\times$ be a continuous character extending $\delta_x$.
   		Then one can show that there exists an isomorphism of $(\varphi,\Gamma_K)$-modules: $$\mathbf{D}(M_x)(\delta) \cong \mathcal{D}^{\mathrm{u}}_x$$
   		
   		\item Let $\tau_1,\tau_2$ be absolutely irreducible smooth $I_K$-representations over $C$. Suppose that there exists a smooth character $\delta: I_K \rightarrow C^\times$ and a element $g\in W_K$ such that $\tau_1\otimes\delta \cong \tau_2^g$. Then $\delta$ induces a canonical isomorphism 
   		$$\kappa(\delta,g):\mathcal{S}_{\tau_1}^{\underline{k}}\rightarrow \mathcal{S}_{\tau_2}^{\underline{k}}$$
   		such that $\kappa(\delta,g)^*(\mathcal{D}^{\mathrm{u}}_2) \cong \mathcal{D}^{\mathrm{u}}_1$, where $\mathcal{D}^{\mathrm{u}}_i$ is the universal object of $\mathcal{S}_{\tau_i}^{\underline{k}}$.
   	\end{enumerate}
   \end{Rmk}

   \begin{Theo}
   	$\mathcal{S}_{\tau}^{\underline{k}}/Y_{\tau}$ is the coarse moduli space (in the sense of \cite[Def~5.9]{mumford1982geometric}) of the functor
   	\begin{align*}
   		\mathcal{Q}_{\tau,\underline{k}}:C\text{-rigid spaces} & \rightarrow \text{Sets} \\
   		A & \mapsto \{ (\varphi,\Gamma_K)\text{-module }D\text{ over }A\text{ of type }{(\tau,\underline{k})} \}/\sim.
   	\end{align*}

   \end{Theo}
   \begin{proof}
   	Note that $\mathcal{S}_{\tau}^{\underline{k}}/Y_\tau \cong (\mathcal{T}\times X(W_{\tau}/I_L))/X_{\tau}\times \mathrm{Flag}$, hence is indeed a rigid space.
   	
   	Let $Y$ be a rigid space over $C$ equipped with a $(\varphi,\Gamma_K)$-module $D$ of type $(\tau,\underline{k})$, then $D$ induces a morphism $f_D:Y\rightarrow \mathcal{Z}_{\tau}^{\underline{k}}\cong [(\mathcal{S}_{\tau}^{\underline{k}}/Y_{\tau})/\mathbb{G}_m]$, which corresponds to a diagram
   	\begin{equation*}
   		\begin{tikzcd}
   			\mathcal{E} \arrow[r,"\tilde{f}_D"] \arrow[d,"p"] & \mathcal{S}_{\tau}^{\underline{k}}/Y_\tau \\
   			Y                               &  
   		\end{tikzcd}
   	\end{equation*}
    here $p$ is a $\mathbb{G}_m$-torsor and $\tilde{f}_D$ is $\mathbb{G}_m$ equivariant.
    
   	As the $\mathbb{G}_m$-action on $\mathcal{S}_{\tau}^{\underline{k}}/Y_\tau$ is trivial, then \'{e}tale locally there exists a unique morphism $g_D:Y\rightarrow \mathcal{S}_{\tau}^{\underline{k}}/Y_\tau$ such that $\tilde{f}_D = g_D\circ p$. As such factorization is unique, one can apply \'{e}tale descent to get a unique global morphism $g_D:Y\rightarrow \mathcal{S}_{\tau}^{\underline{k}}/Y_\tau$. The map $\Psi:f_D\mapsto g_D$ is indeed a functor from $\mathrm{Hom}(-,\mathcal{Z}_{\tau}^{\underline{k}})$ to $\mathrm{Hom}(-,\mathcal{S}_{\tau}^{\underline{k}}/Y_{\tau})$. 
   	
   	Suppose $X$ is a $C$-rigid space such that there exists a functor $$\Psi_X: \mathrm{Hom}(-,\mathcal{Z}_{\tau}^{\underline{k}})\rightarrow \mathrm{Hom}(-,X).$$
   	Then $\Psi_X(\mathcal{S}_{\tau}^{\underline{k}}/Y_{\tau}\rightarrow \mathcal{Z}_{\tau}^{\underline{k}}): \mathcal{S}_{\tau}^{\underline{k}}/Y_{\tau} \rightarrow X$ is the unique morphism such that 
   	for any $f_D:Y\rightarrow \mathcal{Z}_{\tau}^{\underline{k}}$, the composition $\Psi_X(\mathcal{S}_{\tau}^{\underline{k}}/Y_{\tau}\rightarrow \mathcal{Z}_{\tau}^{\underline{k}})\circ \Psi(f_D)$ equals to $\Psi_X(f_D)$.
   	
   	Finally, when $Y=\mathrm{Sp}(C')$ for some finite field extension $C'/C$, then unique $\mathbb{G}_m$-torsor (up to isomorphism) is $\mathbb{G}_{m,C'}$. Hence any $\mathbb{G}_m$-equivariant morphism $\mathbb{G}_{m,C'}\rightarrow \mathcal{S}_{\tau}^{\underline{k}}/Y_\tau$ comes from a morphism $\mathrm{Sp}(C')\rightarrow \mathcal{S}_{\tau}^{\underline{k}}/Y_\tau$. It follows that $\Psi$ is a bijective on a point and hence $\mathcal{S}_{\tau}^{\underline{k}}/Y_\tau$ is the coarse moduli space of the functor $\mathcal{Q}_{\tau,\underline{k}}$.

\end{proof}

   \subsection{Regular locus of Parameter Spaces}\label{Subsect:RLoP}
   For convenience, from now on, the tensor product of two $(\varphi,\Gamma_K)$-modules $D_1,D_2$, which live over the same affinoid algebra or the same analytic rigid space, is just denoted by $D_1\otimes D_2$ and omit the basis. Similarly, the space of morphisms of $(\varphi,\Gamma_K)$-modules is denoted by $\Hom(D_1,D_2)$ and omit the "$\varphi,\Gamma_K$" subscript.
   
   From now on, we fix the following settings and notations once and for all:
   \begin{enumerate}
   	\item a positive integer $n$ with an ordered partition $n:= n_1+\cdots+n_l$;
   	\item for each $1\leq i\leq l$, and each $\eta\in\Hom(K,C)$, a $\eta$-filtration weight
   	$$\underline{k_\eta^{(i)}}:= (k_{\eta,1}^{(i)}<\cdots<k_{\eta,n_i}^{(i)}=0);$$
   	\item $\underline{k^{(i)}}:= (\underline{k_\eta^{(i)}})_{\{ \eta:K\inj C\} }$;
   	\item for each  $1\leq i\leq l$, an absolutely irreducible smooth $I_K$-representation $\tau_i$ over $C$ (actually we can choose some finite Galois extension $L/K$ such that every $\tau_i$ is trivial on $I_L$), such that $d_{\tau_i} e_{\tau_i} = n_i$, where $d_{\tau_i}$ is the dimension of $\tau_i$ and $e_{\tau_i} := [K_\tau:K]$.
   	\item $\mathcal{T}_i:=\mathcal{T}_{\tau_i}$ for $1\leq i\leq l$, and $\mathcal{T}:= \prod\limits_{1\leq i\leq l} \mathcal{T}_i$.
   	\item $\mathcal{S}_i:=\mathcal{S}_{\tau_i}^{\underline{k}^{(i)}}$ for $1\leq i\leq l$, and $\mathcal{S}:= \prod\limits_{1\leq i\leq l} \mathcal{S}_i$.
   	\item For each $i$, denote 
   	$$\Delta: \mathcal{S}_i \rightarrow \mathcal{T}_i$$
   	be the natural projection, and also use the same notation $\Delta$ denote the natural projection $\mathcal{S}\rightarrow \mathcal{T}$.
   \end{enumerate}

   \begin{Rmk}
   For convenience, we may assume that $\tau_i\otimes \delta \cong \tau_j^{g}$ only happens when $\tau_i\cong \tau_j$, via replacing $\tau_j$ by $\tau_j^{g}\otimes \delta^{-1}$. It would not change $\mathcal{S}_j$ according to remark \ref{RmkforParameterspace}(6). And for such pair $(i,j)$, if we also have $\underline{k^{(i)}}=\underline{k^{(j)}}$, then 
   the canonical isomorphism $$\kappa(\delta,g): \mathcal{S}_i\xrightarrow{\sim} \mathcal{S}_j$$ is induced by the identity map on $\mathcal{T}_{i}\times \prod\limits_{\eta:K\inj C} \mathrm{GL}_{n_i}^{\mathrm{rig}}/B_{n_i}^{\mathrm{rig}}$.
   \end{Rmk}

   \begin{Rmk}\label{Rmk:dimofS}
   	As $\mathcal{S}_i\cong \mathbb{G}_m^{\mathrm{rig}}\times\mathcal{W}\times (\prod\limits_{\eta: K\hookrightarrow C}\mathrm{GL}_{n_i}^{\mathrm{rig}}/B_{n_i}^{\mathrm{rig}})$ as $C$-rigid spaces, then one can compute that $\dim(\mathcal{S}_i) = 1+[K:\mathbb{Q}_p]+[K:\mathbb{Q}_p](\frac{n_i(n_i-1)}{2})$. Hence
   	$$\dim(\mathcal{S}) = l+[K:\mathbb{Q}_p]\sum_{i=1}^{l}\frac{n_i^2-n_i+2}{2}$$
   \end{Rmk}
   
   Now let $|z|$ denote the character $z\mapsto |N_{K/\mathbb{Q}_p}(z)|$ of $K^\times$. 
   \begin{Def}\label{Def:RegLocus}
   	For a pair $1\leq i< j\leq l$, such that $\tau_i\cong \tau_j$, we write $\mathcal{T}_{i,j} \subset \mathcal{T}_{\tau_i}$ for the Zariski-open complement of $C$-valued points $\delta_{-\underline{k}}$, $|z|\delta_{\underline{k}+1}$ with $\underline{k} =(k_\eta)_{\eta:K\inj C}$ such that $k_\eta\leq k^{(j)}_{\eta,t}-k^{(i)}_{\eta,t}$ for all $\eta$ and $1\leq t\leq n_i$ (even though $\delta_{-\underline{k}}$ and $|z|\delta_{\underline{k}+1}$ are character of $W_K$, here means their restrictions to $W_\tau$).
   	
   	We define the \textit{regular locus} $\mathcal{T}_{\mathrm{reg}}$ \textit{of} $\mathcal{T}$ as the Zariski open subset of points $(\delta_1,\cdots,\delta_l)$ satisfies the condition that, $\delta_i\delta_j^{-1}\in \mathcal{T}_{i,j}$ for every $i < j$ such that $\tau_i\cong \tau_j$.
   	And we define the \textit{regular locus} $\mathcal{S}_{\mathrm{reg}}$ \textit{of} $\mathcal{S}$ as $\Delta^{-1}(\mathcal{T}_{\mathrm{reg}})$.
   \end{Def}
   
   The motivation to give the definition of the regular locus as above is because we need a Zariski dense and open subspace of the parameter space $\mathcal{S}$ satisfies the following property, similar as \cite[Prop~6.2.8]{Kedlaya2012} for rank one $(\varphi,\Gamma_K)$-modules.
   
   \begin{Prop}\label{Prop:dimforcoh}
   For a point $x=(x_1,\cdots,x_l)\in \mathcal{S}_{\mathrm{reg}}$, we write $D_i:= \mathcal{D}^{\mathrm{u}}_{x_i}$. One has
   \begin{enumerate}
   	\item $H^0_{\varphi,\gamma_K}(D_i^{\vee}\otimes D_i) = k(x)$;
   	\item $H^m_{\varphi,\gamma_K}(D_i^{\vee}\otimes D_j) = 0$ for $i < j$, and $m=0,2$;
   	\item $\dim_{k(x)}H^1_{\varphi,\gamma_K}(D_i^{\vee}\otimes D_j) = [K:\mathbb{Q}_p]n_in_j$ for $i< j$.
   \end{enumerate}
   \end{Prop}
   \begin{proof}{\ }
   	\begin{enumerate}
   		\item By remark \ref{RmkforParameterspace}(3), we can write $D_i$ as the form $\mathbf{D}(M_{x_i})(\delta_i)$. As the functor $\mathbf{D}$ is fully faithful, one has 
   		$$H^0_{\varphi,\gamma_K}(D_i^{\vee}\otimes D_i) = \Hom(D_i,D_i) = \Hom(M_{x_i},M_{x_i}).$$
   		Note that $M_x$ is irreducible as a $(\varphi,N,G_{L/K})$-module, which implies $$\dim_{k(x)}H^0_{\varphi,\gamma_K}(D_i^{\vee}\otimes D_i)=1.$$
   		\item We only need to show $$\Hom(D_i,D_j) = 0,$$ which implies $H^0_{\varphi,\gamma_K}(D_i^{\vee}\otimes D_j) = 0$, and using the Tate duality from \cite[Theorem~4.4.5]{Kedlaya2012}, one can show $H^2_{\varphi,\gamma_K}(D_i^{\vee}\otimes D_j) = 0$ by similar computation.
   		
   	    Let $\delta_i$ (resp. $\delta_j$) denote a continuous character $K\rightarrow k(x)^\times$ such that $\delta|_{W_\tau}$ corresponds to $\Delta(x_i)$ (resp. $\Delta(x_j)$). Then $D_i\cong\mathbf{D}(M_{x_i})(\delta_i)$ and $D_j\cong\mathbf{D}(M_{x_j})(\delta_j)$. Assume $\Hom(D_i,D_j) \neq 0$, then we have a nonzero injective map
   	    $$\mathcal{R}(\delta_i\delta_j^{-1})\inj \mathbf{D}(M_{x_i}^{\vee}\otimes M_{x_j}). $$
   	    Hence $\delta:= \delta_i\delta_j^{-1}$ is deRham. Now we assume $\delta = \delta_{\underline{k}}\delta_{\mathrm{sm}}$ for some algebraic character of weight $\underline{k} =(k_\eta)_{\eta:K\inj C}$, and some smooth character $\delta_{\mathrm{sm}}$. Then $\delta$ corresponds to a rank one Filtered $(\varphi,N,G_{L/K})$-module $M(\delta)$, which has filtration weight $-\underline{k}$ and the $(\varphi,N,G_{L/K})$-structure corresponding to $\delta_\mathrm{sm}$ as a WD representation. As $\mathbf{D}$ is fully faithful, then it induces a nonzero map $M_{x_i}(\delta)\rightarrow M_{x_j}$, which implies there exists a nonzero $W_K$-equivariant map
   	    $$\mathrm{Ind}_{W_\tau}^{W_K}(\tilde{\tau}_i\otimes_C \delta_{\mathrm{sm}}|_{W_\tau}) \rightarrow \mathrm{Ind}_{W_\tau}^{W_K}(\tilde{\tau}_j).$$
        Hence by Frobenius reciprocity, it induces a nonzero $I_K$-equivariant ($W_\tau$-equivariant actually) map $\tilde{\tau}_i\otimes_C \delta_{\mathrm{sm}}|_{W_\tau}\rightarrow \tilde{\tau}_j^g$ for some $g\in W_K$. By assumption, one has $\tau_i\cong \tau_j$ and $\delta_{\mathrm{sm}}|_{W_{\tau}}$ is trivial.
        
        It follows that there exists a nonzero map $M_{x_i}(\delta)\rightarrow M_{x_j}$ (hence an isomorphism) of $(\varphi,N,G_{L/K})$-modules. Note that the $\eta$-filtration Weight of $M_{x_i}(\delta)$ is $\underline{k^{(i)}_\eta}-k_\eta$, and the $\eta$-filtration Weight of $M_{x_j}$ is $\underline{k^{(j)}_\eta}$. Hence we must have $k^{(i)}_{\eta,t}-k_\eta \leq k^{(j)}_{\eta,t}$ for each $1 \leq t\leq n_i$, which contradicts the condition of being in the regular locus. 
   		\item This follows from the Euler characteristic formula by \cite[Theorem~4.4.5]{Kedlaya2012}. 
   	\end{enumerate}
       
   \end{proof}

   \subsection{Refined Paraboline Varieties}

   \begin{Def}
   	\begin{enumerate}[leftmargin=1.5em] Let $X$ be a $C$-rigid space.
   		\item For $1\leq i\leq l$, let $f_i: X\rightarrow \mathcal{S}_i$ be a morphism of $C$-rigid spaces. A $(\varphi,\Gamma_K)$-module $D$ of rank $n$ over $X$ is \textit{paraboline with ordered parameter} $(f_1,\dots,f_l)$ if, after perhaps enlarging $C$, there exists an increasing filtration $(\mathcal{F}_iD)_{i=0,\dots,l}$ given by $(\varphi,\Gamma_K)$-submodules and line bundle $\mathcal{L}_1,\dots,\mathcal{L}_l$ on $X$ such that each $\mathrm{gr}_i(D)\cong \mathcal{R}_{K,X}(f_i)\otimes\mathcal{L}_i$. Such a filtration is called a \textit{parabolization} (\textit{with order parameters} $(f_1,\dots,f_l)$) of $D$.
   		\item In the case $X=\mathrm{Sp}(C)$, we say that a parabolized $(\varphi,\Gamma_K)$-module $(D,\mathcal{F}_\bullet D)$ is \textit{strictly paraboline with ordered parameters} $(f_1,\dots,f_l)$ over $X$ if $\mathcal{F}_{i-1}D$ is the unique $(\varphi,\Gamma_K)$-submodule of $\mathcal{F}_{i}D$, such that $\mathcal{F}_iD/\mathcal{F}_{i-1}D\cong \mathcal{R}(f_i)$ for all $1\leq i\leq l$. This is equivalent to say $H^0_{\varphi,\gamma_K}((\mathcal{F}_iD^{\vee})\otimes\mathcal{R}(f_{i}) )=C$ for all $1\leq i\leq l$. In particular, $\mathcal{F}_\bullet D$ is the unique filtration such that $\mathrm{gr}_i(D)\cong \mathcal{R}_{K,X}(f_i)$.
   		\item A $(\varphi,\Gamma_K)$-module $D$ over $X$ is called \textit{densely pointwise strictly paraboline} if there exist a $C$-morphisms $f_i:X\rightarrow \mathcal{S}_i$ for each $i=1,\dots,l$ and a Zariski dense subset $X_{\mathrm{alg}}\subseteq X$ such that $D_{z}$ is strictly paraboline with ordered parameters $(f_1(z),\dots,f_l(z))$ for each point $z\in X_{\mathrm{alg}}$.
   		 
   	\end{enumerate}
   
   \end{Def}

   \begin{Rmk}
   In \cite[Definition~6.3.1]{Kedlaya2012}, they claim that the condition of being strictly paraboline is equivalent to the condition that $$H^0_{\varphi,\gamma_K}(\mathcal{R}_{K,X}(f_{i+1})^{\vee}\otimes(D/\mathcal{F}_iD)) = C$$
   for all $1\leq i\leq l$. However it seems not true, even for the trianguline case.
   
   Let $\mathcal{R}:= \mathcal{R}_{K,C}$, let $\delta$ be a non-algebraic character (Hence $H_{\varphi,\gamma_K}^0(\mathcal{R}({\delta}^{\pm 1})) = 0$). Let $E$ be a nontrivial extension of $\mathcal{R}$ by $\mathcal{R}(\delta)$, i.e., we have the non-split short exact sequence:
   $$0\rightarrow \mathcal{R} \rightarrow E \rightarrow \mathcal{R}(\delta) \rightarrow 0.$$
   Then let $D: E\oplus \mathcal{R}$ with the filtration $0\subset \mathcal{R} \subset E \subset D$. It is obvious that $\Hom(\mathcal{R},D) = C\oplus C$, hence the "alternative" condition fails. On other hand, it is obvious that $\Hom(\mathcal{F}_iD,\mathrm{gr}_iD) = C$ for $i= 1,2$. For $i=3$, one has 
   $$\Hom(D,\mathcal{R}) = \Hom(E,\mathcal{R})\oplus \Hom(\mathcal{R},\mathcal{R}),$$
   Hence it is enough to show that $\Hom(E,\mathcal{R}) = 0$. Assume $f\in \Hom(E,\mathcal{R})$ is nonzero. Then $f$ restrict to $\mathcal{R}$ is zero, otherwise $f$ gives a splitting of the short exact sequence above, after scaling  by some constant in $C$, which is a contradiction. Hence $f$ factors though $E/\mathcal{R}\cong \mathcal{R}(\delta)$, which also contradicts to the condition that $\delta$ is non-algebraic.
   
   Actually, both conditions can ensure that the filtration is unique for the given parameter, but themselves are not equivalent. For the later proof (also for the proofs in \cite[Section 6.3]{Kedlaya2012}), we only need the condition that $$H^0_{\varphi,\gamma_K}((\mathcal{F}_iD^{\vee})\otimes\mathcal{R}_{K,X}(f_{i}) )$$
   Hence we put it as the definition of strictly paraboline $(\varphi,\Gamma_K)$-modules.
   \end{Rmk}

   \begin{Def}
   Let $A$ (resp. $B$) be an element in an Abelian category $\mathfrak{A}$ with a sub object $A'$ (resp. $B'$). The tuple $(A,A',B,B')$ is called \textit{satisfies the factor through condition} if there exists 
   \begin{enumerate}
   \item an increasing filtration $$0=\mathcal{F}_0A\overset{\alpha_1}{\inj} \mathcal{F}_1A \overset{\alpha_2}{\inj}  \cdots \overset{\alpha_s}{\inj} \mathcal{A}_s=A$$ with $\mathcal{F}_{s_0}A=A'$ for some $s_0$;
   \item and an increasing filtration $$0=\mathcal{F}_0B\overset{\beta_1}{\inj} \mathcal{F}_1B \overset{\beta_2}{\inj}  \cdots \overset{\beta_t}{\inj} \mathcal{B}_t=B$$ with $\mathcal{F}_{t_0}B=B'$ for some $t_0$,
   \end{enumerate}
   such that $\Hom_{\mathfrak{A}}(\mathrm{coker}\alpha_i,\mathrm{coker}\beta_j) = 0$, if $i\leq s_0$ or $j\geq t_0$.
   \end{Def}
  
   \begin{Prop}\label{FTC}
   Let $(A,A',B,B')$ be a tuple in some Abelian category such that satisfies the factor through condition, then the natural map $$\Hom(A/A',B') \rightarrow \Hom(A,B)$$ is a bijection.
   \end{Prop}
   
   \begin{proof}
   By the left exactness of the functor $\Hom(\mathcal{F}_1A,-)$, one has the following left exact short sequence
   $$0\rightarrow \Hom(\mathcal{F}_1A,\mathcal{F}_{j-1}B)\rightarrow \Hom(\mathcal{F}_1A,\mathcal{F}_{j}B) \rightarrow \Hom(\mathcal{F}_1A,\mathrm{coker}\beta_j) = 0$$
   for every $j=1,\dots,t$.
   It follows that $$\Hom(\mathcal{F}_1A,B) = \Hom(\mathcal{F}_1A,\mathcal{F}_{t-1}B) = \cdots \Hom(\mathcal{F}_1A,\mathcal{F}_0B) = 0.$$
   Hence one gets $\Hom(A,B) = \Hom(A/\mathcal{F}_1A,B)$ from the left exact sequence
   $$0\rightarrow \Hom(A/\mathcal{F}_1A,B)\rightarrow \Hom(A,B) \rightarrow \Hom(\mathcal{F}_1A,B) = 0.$$
   Iterating the argument above, it follows that $\Hom(A,B) = \Hom(A/A',B)$. And using the symmetric argument for $B$, it is easy to see that $$\Hom(A/A',B') = \Hom(A,B).$$
   \end{proof}
   
   \begin{Cor}\label{Cor:StrPar}
   Let $D,D_1,\dots,D_r$ be $(\varphi,\Gamma_K)$-modules over $C$, and let $(\mathcal{F}_iD)_{i=0,\dots,r}$ be an increasing filtration of $(\varphi,\Gamma_K)$-submodules such that $\mathrm{gr}_i(D)\cong D_i$. Suppose that

   \begin{equation*}
   \Hom(D_i,D_j)=\left\{
   \begin{array}{cc}
   C & i= j\\
   0 & i < j
   \end{array}
   \right.
   \end{equation*} 
    then one has
    $$H^0_{\varphi, \gamma_K}((\mathcal{F}_iD)^\vee\otimes D_{i}) = C$$ 
   \end{Cor}
   
   \begin{proof}
   Note that $H^0_{\varphi, \gamma_K}((\mathcal{F}_iD)^\vee\otimes D_{i}) = \Hom((\mathcal{F}_iD), D_{i})$, and $(\mathcal{F}_iD, \mathcal{F}_{i-1}D, D_{i},D_{i})$ satisfies the factor through condition. Therefor
   $$H^0_{\varphi, \gamma_K}((\mathcal{F}_iD)^\vee\otimes D_{i}) = \Hom(D_i, D_i) = C.$$
   \end{proof}

   For $\eta\in\Hom(K,C)$, let $\delta_{\eta}: a\mapsto \eta(a)^{-1}$ be the algebraic character. Since $\mathcal{R}_{K,C}$ is a product of B\'{e}zout domains, the ideal $\mathcal{R}(\delta_{\eta})\subset \mathcal{R}_{K,C}$ is generated by one element $t_{\eta}\in\mathcal{R}_{K,C}$. The element $t_{\eta}$ is not uniquely determined, but the ideal it generates is. Moreover, one has
   $$\mathcal{R}_{K,C}/(t)\cong\bigoplus_{\eta:K\inj C}\mathcal{R}_{K,C}/(t_\eta).$$
   For more details, see \cite[Notation~6.2.7]{Kedlaya2012}.

  \begin{Lemma} \label{Sat1}
   Let $D_1,D_2$ be $(\varphi,\Gamma_K)$-modules over $\mathcal{R}_{K,C}$. For any $f\in\Hom(D_1,D_2)$, the image of $f$ is saturated (recall the definition of being saturated in \cite[Notation~6.0.2]{Kedlaya2012}) in $D_2$ if and only if the image of induced map $$\tilde{f}:\mathcal{R}_{K,C}\rightarrow D_1^{\vee}\otimes D_2$$
   is saturated in $D_1^{\vee}\otimes D_2$.
   \end{Lemma}

   \begin{proof}
   	We may assume $f$ is nonzero, otherwise the assertion is trivial.
   	
   	Let $\mathcal{R}:=\mathcal{R}_{K,C}$. Assume $\mathrm{Im}(f)$ is saturated. Hence for any $\eta:K\inj C$, the base change map $f\otimes_{\mathcal{R}}\mathcal{R}/t_\eta$ is non-zero, which implies that the induced map $\tilde{f}\otimes_{\mathcal{R}}\mathcal{R}/(t_\eta)$ is non-zero. According to \cite[Lemma~2.6.10]{Kedlaya2012}, the image of $\tilde{f}$ is saturated. 
   	
   	Now assume the image of $\tilde{f}$ is saturated in $D_1^{\vee}\otimes D_2$. Then the image of 
   	$$\tilde{f}\otimes \id_{D_1}: D_1\rightarrow D_1\otimes D_1^{\vee}\otimes D_2$$
   	is saturated. Composing with the natural map $g:D_1\otimes D_1^{\vee}\otimes D_2\rightarrow D_2$, one obtains the map $f$ and concludes that the image is saturated (here we need the fact that $g$ splits as a morphism of free $\mathcal{R}$-modules).
   \end{proof}

   Recall that (see definition \ref{Def:Phigammaoftypetau}) a quasi-deRham $(\varphi,\Gamma_K)$-module $D$ over $C$ is called irreducible if $D$ is of the form $\mathbf{D}(M)(\delta)$ for some irreducible filtered $(\varphi,\Gamma_K)$-module $M$.

   \begin{Lemma} \label{Sat2}
   Let $D$ be an quasi-deRham irreducible $(\varphi,\Gamma_K)$-module over $\mathcal{R}_{K,C}$. Then
   the cokernel of the embedding of any nonzero $(\varphi,\Gamma_K)$-submodule $j:D'\inj D$ is killed by some power of $t$. And in particular, the only nonzero saturated submodule is $D$ itself. 
   \end{Lemma}

   \begin{proof}
   We can write $D\cong \mathbf{D}(M)(\delta)$ for some irreducible filtered $(\varphi,N,G_{L/K})$-module $M$ over $C$, and some character $\delta: K^\times \rightarrow C^\times$. Replacing $D$ by $D(\delta^{-1})$, we may assume $D$ is deRham. By the proof of \cite[Corollaire~III.2.5]{Berger2008}, every embedding $j:D'\inj D$ of $(\varphi,\Gamma_K)$-modules comes from some embedding $$\mathbf{D}^{-1}(j):M'\inj M$$ of filtered-$(\varphi,N,G_{L/K})$-modules, such that $D'\cong \mathbf{D}(M')$. As $M$ is irreducible, $\mathbf{D}^{-1}(j)$ is actually an isomorphism for the underlying $(\varphi,N,G_{L/K})$ structures. Hence there exists some $m$, such that one has the unique map $j'$ such that the composition
   $$M[m]\xrightarrow{j'} M' \xrightarrow{\mathbf{D}^{-1}(j)} M$$ 
   is the canonical shifting map. It follows that $(\mathbf{D}(M)/\mathbf{D}(M'))$ is killed by $t^m$ (note that $\mathbf{D}(M[m]) = t^{n}\mathbf{D}(M)$ by definition).
   \end{proof}
   
   The following theorem is a generalization of the results in \cite[Theorem~6.3.9]{Kedlaya2012}. Hence we will omit some detail of the proof which are exactly the same as the arguments in \textit{loc. cit.}.
   
   \begin{Theo}\label{KPX6.3.9}
   Let $X$ be a reduced rigid $C$-analytic space. Let $D$ (resp. $D'$) be a $(\varphi,\Gamma_K)$-module over $\mathcal{R}_{K,X}$ of rank $d$ (resp. $d'$) for $d'<d$. Suppose that $D'_z:=D'\otimes_X k_z$ is quasi-deRham irreducible for every closed point $z\in X$ ($k_z$ is the residue field of $z$). Suppose that there exists a Zariski dense subset $X_{\mathrm{alg}}$ of closed points of $X$ such that for every $z\in X_{\mathrm{alg}}$, the $k_z$-vector space $H^0_{\varphi,\gamma_K}(D_z^\vee \otimes D'_z)$ is one dimensional, and for any basis, the induced map $D_z\rightarrow D'_z$ is surjective. Then there exist canonical data of
   \begin{enumerate}[label=(\alph*)]
   	\item a proper birational morphism $f: X'\rightarrow X$ of reduced rigid $C$-analytic spaces,
   	\item a unique (up to $\mathcal{O}^\times_{X'}$) homomorphism $\lambda: f^*D\rightarrow f^*D'\otimes \mathcal{L}$ of $(\varphi,\Gamma_K)$-modules over $\mathcal{R}_{X',K}$, where $\mathcal{L}$ is a line bundle over $X'$ with trivial $(\varphi,\Gamma_K)$-action,
   \end{enumerate}
   such that the following conditions are satisfied.
   \begin{enumerate}[label=(\arabic*)]
   \item The set $Z$ of closed points $z\in X'$ failing to have the following property is Zariski closed and disjoint from $f^{-1}(X_{\mathrm{alg}})$ (hence its complement is Zariski open and dense): the induced homomorphism $\lambda_z:D_Z\rightarrow D'_z$ is surjective and the corresponding element spans $H^0_{\varphi,\Gamma_K}(D_z^\vee \otimes D'_z)$ (hence the latter is one-dimensional).
   \item Locally on $X'$, the cokernel of $\lambda$ is killed by some power of $t$, and is supported over $Z$ in the sense that for any analytic function $g$ vanishing along $Z$, some power of $g$ kills the cokernel of $\lambda$ too.
   \item The kernel of $\lambda$ is a $(\varphi,\Gamma_K)$-module over $\mathcal{R}_{X',K}$ of rank $d-d'$.
   \end{enumerate}
   \end{Theo}

   \begin{proof}
   As $X$ can be replaced by its normalization $b:\tilde{X}\rightarrow X$, we assume henceforth that $X$ is normal and connected. Moreover, for every affinoid subdomain $\mathrm{Sp}(A)$ in $X$, we see that $\mathrm{Spec}(A)$ is irreducible. It follows that any coherent sheaf on $X$, or pullback under any dominant morphism, has constant generic rank.
   
   By \cite[Corollary~6.3.3]{Kedlaya2012} and Corollary~6.3.6(2) in $\mathit{loc.\ cit.}$, there exists a proper birational morphism: $g: Y \rightarrow X$, such that the following condition holds for $D_0:= g^*(D^{\vee}\otimes D')$:
   \begin{enumerate}
   	\item $H^0_{\varphi,\gamma_K}(D_0)$ is flat and $H^i_{\varphi,\gamma_K}(D_0)$ has Tor-dimension less than or equal to one for $i=1,2$;
   	\item $H^0_{\varphi,\gamma_K}(D_0/t_\eta)$ is flat and $H^i_{\varphi,\gamma_K}(D_0/t_\eta)$ has Tor-dimension less than or equal to one for $i=1,2$, and for each $\eta:K\inj C$;
   \end{enumerate}
   (actually we construct such $g$ locally and canonically on $X$, which satisfies the conditions above. Then by Theorem 4.4.3(2) in $\mathit{loc.\ cit.}$, we are allowed to glue these morphisms together to get a global morphism satisfying the same conditions. Note that here $Y$ can also be replaced by the normalization of its nil-reduction, we may and will take $Y$ to be normal and reduced.)
   
   Condition (1) allows us to invoke Theorem 6.3.7 of \textit{loc. cit.}, so that there exists a Zariski open and dense subset $U_0$ of $Y$, such that $\mathrm{Tor}_1^{Y}(H^{1}_{\varphi,\gamma_K}(D_0),k_z) = 0$ if and only if $z\in U_0$. Then the base change spectral sequence
   $$E_2^{i,j}= \mathrm{Tor}_{-i}^{Y}(H^{j}_{\varphi,\gamma_K}(D_0),k_z)\Rightarrow H^{i+j}_{\varphi,\gamma_K}(D_{0,z})$$ gives the short exact sequence
   $$0\rightarrow H^{0}_{\varphi,\gamma_K}(D_0)\otimes k_z\rightarrow H^{0}_{\varphi,\gamma_K}(D_{0,z}) \rightarrow \mathrm{Tor}_1^{Y}(H^{1}_{\varphi,\gamma_K}(D_0),k_z) \rightarrow 0, $$
   which implies for any point $z$ in $U_0\cap g^{-1}(X_{\mathrm{alg}})$, the $k_z$-vector space $$H^{0}_{\varphi,\gamma_K}(D_0)\otimes k_z\cong H^{0}_{\varphi,\gamma_K}(D_{0,z})$$
   is one dimensional. By condition (1) again, $H^{0}_{\varphi,\gamma_K}(D_0)$ is a line bundle over $Y$. Let $\mathcal{L}_0$ denote the dual line bundle of $H^{0}_{\varphi,\gamma_K}(D_0)$. Dualizing the natural homomorphism $g^*D\otimes\mathcal{L}_0^{\vee}\rightarrow g^*D'$, one has the unique (up to $\mathcal{O}_Y'$) homomorphism $$\lambda_0: f^*D\rightarrow f^*D'\otimes \mathcal{L}_0.$$
   Now we check properties (1) and (2). And afterwards we will construct a morphism 
   $$f: X'\rightarrow Y \rightarrow X$$ by blowing up in $Y$ which will preserve the properties (1) and (2), and check (3).
   
   (1) Actually we have shown that $H^0_{\varphi,\gamma_K}(D_{0,z}) = k_z$ if and only if $z\in U_0$. For any non zero element $c\in H^0_{\varphi,\gamma_K}(D_{0,z})$, the corresponding map $D_z\rightarrow D'_z$ is surjective if and only if  the induced map $\mathcal{R}_{K,k_z}\rightarrow D_{0,z}$ is saturated by lemma \ref{Sat1} and \ref{Sat2}. By \cite[Lemma~2.6.10]{Kedlaya2012}, a map $D_z\rightarrow D'_{z}$ is surjective if and only if for each $\eta\in\Hom(K,C)$, the map
   $$H^{0}_{\varphi,\gamma_K}(D_z^{\vee}\otimes D'_z)\cong H^{0}_{\varphi,\gamma_K}(D_0)\otimes_Y k_z\rightarrow H^{0}_{\varphi,\gamma_K}(D_0/t_{\eta})\otimes_Y k_z\inj H^{0}_{\varphi,\gamma_K}((D_z^{\vee}\otimes D'_z)/t_{\eta})$$
   is nontrivial. Here the injectivity of the last homomorphism above follows from the base change spectral sequence $E^{i,j}_2 =\mathrm{Tor}^{Y}_{-i}(H^{j}_{\varphi,\gamma_K}(D_0/t_{\eta}),k_z)\Rightarrow  H^{i+j}_{\varphi,\gamma_K}(D_{0,z}/t_{\eta})$.
   This condition equals to $z$ is in the intersection of locus $Z_\eta$ that the natural map $H^0_{\varphi,\gamma_K}(D_0)\rightarrow H^0_{\varphi,\gamma_K}(D_0/t_\eta)$ vanishes for every $\eta:K\inj C$. Hence $U_0\setminus (\cap_\eta Z_\eta)$ is exactly the locus that the condition that $\lambda_{0,z}: D_z\rightarrow D'_z$ is surjective and the corresponding element spans $H^0_{\varphi,\gamma_K}(D_{0,z})$ holds, and contains $g^{-1}(X_{\mathrm{alg}})$ obviously.
   
   (2) The argument is exactly the same as the proof in \cite[Theorem~6.3.9]{Kedlaya2012}. And the only different is we need to invoke lemma $\ref{Sat2}$ to make sure that $\mathrm{coker}(\lambda_0)_z$ is killed by some power of $t$.
   
   (3) Note that $\lambda_0$ is the base change of a morphism
   $$\lambda_0^r: g^*(D^r)\rightarrow g^*(D'^r)$$
   Let $M^r$ denote the cokernel of $\lambda_0^r$. Then we can locally apply to \cite[Corollary~6.3.6(1)]{Kedlaya2012} to any finite presentation of $M^{[r/p,r]}$, and able to glue these local constructions globally to obtain the morphism $h: X' \rightarrow Y$ such that $h^*(M^r)$ has Tor-dimension at most one. Let $f:= g\circ h$ and let $\lambda := h^*(\lambda_0)$. Then the exactly same arguments used in the proof of Theorem 6.3.9 of \textit{loc. cit.} shows the kernel of $\lambda$ is of rank $d-d'$.
   \end{proof}
   
   \begin{Cor}\label{Cor:KPX6.3.10}
   Let $X$ be a reduced rigid analytic space over $C$. Let $D$ be a densely pointwise strictly paraboline $(\varphi,\Gamma_K)$-module over $\mathcal{R}_{X,K}$ of rank $n$, with respect to the ordered parameters $(f_1,\dots,f_r)$ and the Zariski dense subset $X_{\mathrm{alg}}$. Then there exist canonical data of
   \begin{enumerate}[label=(\alph*)]
   	\item a proper birational morphism $F:X'\rightarrow X$ of reduced analytic spaces,\\
   	\item a unique increasing filtration $(\mathcal{F}_i(F^*D))_{0\leq i\leq r}$ on the pullback $(\varphi,\Gamma_K)$-modules $f^*D$ over $\mathcal{R}_{X',K}$ via $(\varphi,\Gamma_K)$-stable coherent $\mathcal{R}_{X',K}$-submodules, 
   \end{enumerate} 
   such that the following conditions are satisfied.
   \begin{enumerate}[label=(\arabic*)]
   	\item The set $Z$ of closed points $z\in X'$ at which $(\mathcal{F}^\bullet(f^*D))_z$ fails to be a strictly parabolic filtration on $D_z$ with ordered parameters $(f_{1}(z),\dots,f_{r}(z))$ is Zariski closed in $X'$ and disjoint from $f^{-1}(X_{\mathrm{alg}})$ (hence the complement of $Z$ is Zariski open and dense).\\
   	\item Each $\mathrm{gr}_i(F^*D)$ embeds $(\varphi,\Gamma_K)$-equivariantly into $F^*(\mathcal{R}(f_i))\otimes_{X'}\mathcal{L}_i$ for some line bundle $\mathcal{L}_i$ over $X'$, and the cokernel of the embedding is, locally on $X'$, killed by some power of $t$ and supported on $Z$.
   \end{enumerate} 
   \end{Cor}
   \begin{proof}
   	The existence of data satisfying all properties follows from theorem \ref{KPX6.3.9} by induction.
   \end{proof}
   
   The following definition and propositions are parallel results of \cite[Section~2.2]{Hellmann2016}. Recall the definition of $\mathcal{S}_{\mathrm{reg}}$ from definition \ref{Def:RegLocus}.
   \begin{Def}
   	 Consider the functor $\mathcal{P}_{\mathcal{S}_{\mathrm{reg}}}$ that assigns to a rigid $C$-space $X$ the isomorphism classes of quadruples $(D,\mathcal{F}_\bullet(D),f,\nu)$, where $D$ is a $(\varphi,\Gamma_K)$-module over $X$ of rank $n$ and $\mathcal{F}^\bullet(D)$ is an increasing filtration of $D$ given by $(\varphi,\Gamma_K)$-submodules such that $\mathcal{F}_0D=0$ and $\mathcal{F}_lD=D$. Further $f=(f_1,\dots,f_l)\in \mathcal{S}_{\mathrm{reg}}(X)$ and $\nu=(\nu_1,\dots,\nu_l)$ is a collection of trivializations
   	$$\nu_i:\mathcal{F}_{i}(D)/\mathcal{F}_{i-1}(D) \xrightarrow{\sim} \mathcal{R}_{K,X}(f_i)$$
   \end{Def}
   \begin{Prop}
   Let $f = (f_1,\dots,f_l)\in\mathcal{S}_{\mathrm{reg}}(X)$ for some reduced rigid $C$-space $X$. Let $1\leq a_1<\dots<a_s\leq l$ be a sub-sequence of $1,2,\dots,l$, and let $D$ be a successive extension of $D_i:=\mathcal{R}(f_{a_i})$ for $1\leq i\leq s-1 $. Then $H^1_{\varphi,\gamma_K}(D\otimes D_s^{\vee})$ is a locally free $\mathcal{O}_X$-module of rank $[K:\mathbb{Q}_p]\mathrm{rk}(D\otimes D_s^{\vee})$.
   \end{Prop}

   \begin{proof}
   	It follows from \cite[Theorem~4.4.5]{Kedlaya2012} that the cohomology is a coherent sheaf and it is enough to compute the rank at all closed points. We proceed by induction on $s$. The $s= 2$ case has been proved by proposition \ref{Prop:dimforcoh}(3). For general $s$, consider the short exact sequence:
   	$$0\rightarrow D_1 \rightarrow D\rightarrow D'\rightarrow 0$$
   	tensored with $D_s^{\vee}$. By induction hypothesis, $H^1_{\varphi,\gamma_K}(D'\otimes D_s^{\vee})$ is locally free of rank $[K:\mathbb{Q}_p]\mathrm{rk}(D'\otimes D_s^{\vee}) $, and therefor the Euler characteristic formula \cite[Theorem~4.4.5(2)]{Kedlaya2012} implies $H^i_{\varphi,\gamma_K}(D'\otimes D_s^{\vee}) = 0$ for $i=0,2$.
   	
   	By proposition \ref{Prop:dimforcoh} again, $H^1_{\varphi,\gamma_K}(D_1\otimes D_s^{\vee})$ is free of rank $[K:\mathbb{Q}_p]\mathrm{rk}(D_1\otimes D_s^{\vee})$, and $H^2_{\varphi,\gamma_K}(D_1\otimes D_s^{\vee})=0$. Now the claim follows from the long exact sequence associated to the short exact sequence above (tensored with $D_s^{\vee}$).
   \end{proof}

   \begin{Theo} \label{Sd}
   	\text{ }
   	\begin{enumerate}
    \item The functor $\mathcal{P}_{\mathcal{S}_{\mathrm{reg}}}$ is represented by a rigid space.
    \item The natural map $\kappa': \mathcal{P}_{\mathcal{S}_{\mathrm{reg}}}\rightarrow \mathcal{S}$ is smooth of relative dimension $$[K:\mathbb{Q}_p]\sum_{1\leq i<j\leq r}n_in_j$$
   	\end{enumerate}
    
    \begin{proof}
    	The proof is quite similar as the proof of \cite[Theorem~2.4]{Hellmann2016}, hence we give a short sketch. Let $\mathcal{S}^{(i)}:= \mathcal{S}_1\times\cdots\times\mathcal{S}_i$, then one can define the functor $\mathcal{P}_{\mathcal{S}^{(i)}_{\mathrm{reg}}}$ in the similar way. Now we proceed the proof by induction on $i$.
    	
    	The case $i=1$ is settled by $\mathcal{P}_{\mathcal{S}^{(1)}_{\mathrm{reg}}}= \mathcal{S}_1$. Now assume $\mathcal{P}_{\mathcal{S}^{(i-1)}_{\mathrm{reg}}}$ is constructed with universal object $(\mathcal{D}_{i-1},\mathcal{F}_\bullet\mathcal{D}_{i-1},g_{i-1},\mu_{i-1})$. Let $U\subseteq \mathcal{P}_{\mathcal{S}^{(i-1)}_{\mathrm{reg}}}\times \mathcal{S}_i$ denote the preimage of $\mathcal{S}^{(i)}_{\mathrm{reg}}\subseteq \mathcal{S}^{(i-1)}_{\mathrm{reg}}\times \mathcal{S}_i$ under the canonical projection $g_{i-1}\times \id_{\mathcal{S}_i}$. The proposition (for the case $(a_1,\dots,a_s) = (1,2,\dots,i)$) above shows that 
    	$$\mathcal{M}_U:=\mathcal{E}xt^1_{\mathcal{R}_{K,U}}(\mathcal{D}_{i-1},\mathcal{R}_{K,U}(\mathrm{pr}_i)) = H^1_{\varphi,\gamma_K}(\mathcal{D}_{i-1}\otimes\mathcal{R}_{K,U}(\mathrm{pr}_i)^{\vee})$$
    	is a vector bundle of rank $[K:\mathbb{Q}_p]n_i(n_1+\cdots+n_{i-1})$, here $\mathrm{pr}_i$ is the natural projection $\mathrm{pr}_i:U\rightarrow \mathcal{S}_i$. Now $\mathcal{P}_{\mathcal{S}^{(i)}_{\mathrm{reg}}} = \underline{\mathrm{Spec}}_U(\mathrm{Sym}^\bullet\mathcal{M}_U^{\vee})$ is the geometric vector bundle over $U$ associate to $\mathcal{M}_U$.
    	
    	Then $\mathcal{D}_i$ is the universal extension
    	$$0\rightarrow  p_i^*\mathcal{D}_{i-1}\rightarrow \mathcal{D}\rightarrow \mathcal{R}(\mathrm{pr}'_i)\rightarrow 0$$ 
    	where $\mathrm{pr}_i':\mathcal{P}_{\mathcal{S}^{(i)}_{\mathrm{reg}}}\rightarrow \mathcal{S}_i$ and $p_i: \mathcal{P}_{\mathcal{S}^{(i)}_{\mathrm{reg}}}\rightarrow \mathcal{P}_{\mathcal{S}^{(i-1)}_{\mathrm{reg}}}$ are the natural projections.
    	
    	The filtration $\mathcal{F}_\bullet\mathcal{D}_i$ is $$0\subseteq p_i^*(\mathcal{F}_1\mathcal{D}_{i-1}) \subseteq \cdots\subseteq p_i^*(\mathcal{D}_{i-1}) \subseteq \mathcal{D}_i,$$ and $g_i$,$\mu_i$ are defined in the obvious way.
    	
    	Note that the morphism  $p_i\times \mathrm{pr}_i':\mathcal{P}_{\mathcal{S}^{(i)}_{\mathrm{reg}}} \rightarrow \mathcal{P}_{\mathcal{S}^{(i-1)}_{\mathrm{reg}}} \otimes S_i$ is smooth of relative dimension $[K:\mathbb{Q}_p]n_i(n_1+\cdots+n_{i-1})$. It follows that $\kappa'$ is smooth of relative dimension
    	\begin{align*}
    		\dim(\kappa') & = \sum_{2\leq i\leq l-1} \dim(p_i\times \mathrm{pr}_i')\\
    		               & = [K:\mathbb{Q}_p]\sum_{1\leq i<j\leq l}n_in_j
    	\end{align*}
    \end{proof}
    
   \end{Theo}

   \begin{Def}
   	Fix a continuous representation $\bar{r}:G_K\rightarrow \mathrm{GL}_n(k_C)$ and let $R_{\bar{r}}^\Box$ be the usual framed local deformation ring of $\bar{r}$, which pro-represents the functor of local artinian rings with residue field $k_C$:
   	$$A \mapsto \{r:G_K\rightarrow \mathrm{GL}_d(A)|r\otimes_A k_C = \bar{r}\}$$
   	It is a local complete noetherian $\mathcal{O}_C$-algebra of residue field $k_C$ and we denote by $\mathfrak{X}_{\bar{r}}:= (\mathrm{Spf}R_{\bar{r}}^\Box)^{\mathrm{rig}}$ the rigid analytic space over $C$ associated to the formal scheme $\mathrm{Spf}R_{\bar{r}}^\Box$. We define $X_{\mathrm{par}}(\bar{r})$ as the reduced rigid analytic space over $C$ which is the Zariski-closure in $\mathfrak{X}_{\bar{r}}\times \mathcal{S}$ of 
   	$$U_{\mathrm{par}}(\bar{r}) := \{ (r,f)\in\mathfrak{X}_{\bar{r}}\times\mathcal{S}_\mathrm{reg} |\ D^\dagger_{\mathrm{rig}}(r)\mathrm{\ is\ paraboline\ with\ parameters\ }f\}$$
   	We call $X_{\mathrm{par}}(\bar{r})$ the \textit{refined\ paraboline\ variety\ for} $\bar{r}$ (\textit{of\ shape\ }$\mathcal{S}$), and call $U_{\mathrm{par}}(\bar{r})$ the \textit{regular locus of} $X_{\mathrm{par}}(\bar{r})$ (we will prove that $U_{\mathrm{par}}(\bar{r})$ is Zariski dense in $X_{\mathrm{par}}(\bar{r})$ in the following theorem).
   \end{Def}

   We denote by $\omega$ the composition $X_{\mathrm{par}}(\bar{r}) \hookrightarrow \mathfrak{X}_{\bar{r}}\times \mathcal{S}\twoheadrightarrow \mathcal{S}$, and denote by  $\omega_i$ the composition $X_{\mathrm{par}}(\bar{r}) \xrightarrow{\omega} \mathcal{S} \surj \mathcal{S}_i$.\\
   
   \begin{Theo}{\ }\label{Thm:Geoofpar}
   	\begin{enumerate}
   		\item the rigid space $X_\mathrm{par}(\bar{r})$ is equidimensional of dimension 
   		$$[K:\mathbb{Q}_p](\frac{n(n-1)}{2}+l)+n^2;$$
   		\item the set $U_{\mathrm{par}}(\bar{r})$ is Zariski open in $X_{\mathrm{par}}(\bar{r})$, hence it is also Zariski dense in $X_{\mathrm{par}}(\bar{r})$;
   		\item the rigid space $U_{\mathrm{par}}(\bar{r})$ is smooth and the morphism $\omega$ restricted on $U_{\mathrm{par}}(\bar{r})$ is smooth.  
   	\end{enumerate}
      \end{Theo}
   \begin{proof}
   	Our strategy is as follows.  We will construct a smooth rigid $C$-space ${\mathcal{P}^\Box(\bar{r},\mathcal{S})}$ fitting into a commutative diagram as below
   	\begin{equation*}
   		\begin{tikzcd}
   		& {\mathcal{P}^\Box(\bar{r},\mathcal{S})} \arrow{rd}{\kappa} \arrow{ld}[swap]{\pi_{\bar{r}}} &                       \\
   		{X_{\mathrm{par}}(\bar{r})} \arrow{rr}{\omega} &                                                                                 & \mathcal{S}
   		\end{tikzcd}
   	\end{equation*}
   	and show that $\pi_{\bar{r}}$ is smooth of relative dimension $l$ with the image $U_{\mathrm{par}}(\bar{r})$ and the morphism $\kappa$ is smooth of relative dimension $n^2+\dim \kappa'$, here $\kappa'$ is the morphism in theorem \ref{Sd}.
   	
   	Consider the functor $\mathcal{P}^\Box(\bar{r},\mathcal{S})$ that assigns to a reduced rigid $C$-space $X$ the isomorphism classes of quadruples $(r,\mathcal{F}_\bullet(D),f,\nu)$, where $r:G_K\rightarrow \mathrm{GL}_d(\mathcal{O}^+_X) $ is a continuous representation such that for any closed point $x\in X$, the reduction of $r\otimes \mathcal{O}_{k_x}$ coincide with $\bar{r}$, and $\mathcal{F}_\bullet$ is an increasing filtration of $D_{\mathrm{rig}}^\dagger(r)$ given by $(\varphi,\Gamma_K)$-submodules, which are locally on $X$ direct summands as $\mathcal{R}_{K,X}$-modules, such that $\mathcal{F}_0 = 0$ and $\mathcal{F}_l = D_{\mathrm{rig}}^\dagger(r)$. Further $f=(f_1,\dots,f_l)\in \mathcal{S}_{\mathrm{reg}}(X)$ and $\nu=(\nu_1,\dots,\nu_l)$ is a collection of trivializations
    $$\nu_i:\mathrm{gr}_i (D_{\mathrm{rig}}^\dagger(r))\xrightarrow{\sim} \mathcal{R}(f_i)$$
    Now we are going to show that $\mathcal{P}^\Box(\bar{r},\mathcal{S})$ is represented by a rigid $C$-space. Actually the argument is almost the same in the proof of \cite[Theorem~2.6]{Breuil2017}. Using the similar notation as \textit{loc. cit.}, we can construct $\mathcal{P}^\Box(\bar{r},\mathcal{S})$ from the following series of morphisms:
    \begin{equation}\label{eq:kappa}
    	\kappa:\mathcal{P}^\Box(\bar{r},\mathcal{S})\hookrightarrow \mathcal{P}_{\mathcal{S}_{\mathrm{reg}}}^{\Box,\mathrm{adm}}\xrightarrow{\pi} \mathcal{P}_{\mathcal{S}_{\mathrm{reg}}}^\mathrm{adm}\hookrightarrow \mathcal{P}_{S_{\mathrm{reg}}}\xrightarrow{\kappa'}\mathcal{S}_{\mathrm{reg}}
    \end{equation}
    we can also compute the relative dimension of $\kappa$ and show that $\kappa$ is smooth hence $\mathcal{P}^\Box(\bar{r},\mathcal{S})$ is reduced. Now we explain the notations and morphisms in the composition (\ref{eq:kappa}).
    \begin{enumerate}
    	\item $\mathcal{P}_{\mathcal{S}_{\mathrm{reg}}}^\mathrm{adm}$ is the admissible open subspace of $\mathcal{P}_{\mathcal{S}_{\mathrm{reg}}}$ defined by \cite[Theorem~1.2]{hellmann2012families}, i.e. the maximal open subspace of $\mathcal{P}_{\mathcal{S}_{\mathrm{reg}}}$ such that there exists a rank $n$ vector bundle $\mathcal{V}$ over $\mathcal{P}_{\mathcal{S}_{\mathrm{reg}}}^\mathrm{adm}$ and a continuous morphism $G_K\rightarrow \mathrm{Aut}_{\mathcal{O}_{\mathcal{P}_{\mathcal{S}_{\mathrm{reg}}}^\mathrm{adm}}}(\mathcal{V})$ such that $D_{\mathrm{rig}}^\dagger(\mathcal{V})$ is isomorphic to the inverse image of the universal $(\varphi,\Gamma_K)$-module of $\mathcal{P}_{\mathcal{S}_{\mathrm{reg}}}$ over $\mathcal{P}_{\mathcal{S}_{\mathrm{reg}}}^\mathrm{adm}$.
    	\item $\pi:\mathcal{P}_{\mathcal{S}_{\mathrm{reg}}}^{\Box,\mathrm{adm}}\rightarrow \mathcal{P}_{\mathcal{S}_{\mathrm{reg}}}^{\mathrm{adm}}$ is the $\mathrm{GL}_n$ torsor of the trivialization of the vector bundle $\mathcal{V}$.
    	\item There exists a canonical isomorphism $\pi^*(\mathcal{V}) \cong \mathcal{O}^n_{ \mathcal{P}_{\mathcal{S}_{\mathrm{reg}}}^{\Box,\mathrm{adm}} }$, and therefor the action of $G_K$ on $\mathcal{V}$ induces a continuous representation
    	$$\tilde{r}: G_K\rightarrow \mathrm{GL}_n(\Gamma({ \mathcal{P}_{\mathcal{S}_{\mathrm{reg}}}^{\Box,\mathrm{adm}} },\mathcal{O}_{ \mathcal{P}_{\mathcal{S}_{\mathrm{reg}}}^{\Box,\mathrm{adm}} })).$$
    	As $G_K$ is topologically generated by finite many elements, then the set of points $x$ in $\mathcal{P}_{\mathcal{S}_{\mathrm{reg}}}^{\Box,\mathrm{adm}} $ such that $\tilde{r}$ factors through $$ \mathrm{GL}_n(\Gamma({ \mathcal{P}_{\mathcal{S}_{\mathrm{reg}}}^{\Box,\mathrm{adm}} },\mathcal{O}^{+}_{ \mathcal{P}_{\mathcal{S}_{\mathrm{reg}}}^{\Box,\mathrm{adm}} })) \subset \mathrm{GL}_n(\Gamma({ \mathcal{P}_{\mathcal{S}_{\mathrm{reg}}}^{\Box,\mathrm{adm}} },\mathcal{O}_{ \mathcal{P}_{\mathcal{S}_{\mathrm{reg}}}^{\Box,\mathrm{adm}} }))$$ 
    	is admissible open in $\mathcal{P}_{\mathcal{S}_{\mathrm{reg}}}^{\Box,\mathrm{adm}} $. Recall that we have fixed a residue  representation $\bar{r}:G_K\rightarrow \mathrm{GL}_n(k_C)$, then we define $\mathcal{P}^\Box(\bar{r},\mathcal{S})\subseteq \mathcal{P}_{\mathcal{S}_{\mathrm{reg}}}^{\Box,\mathrm{adm}} $ as the admissible open subspace of the points $x$ where $\tilde{r}_x$ modulo the maximal ideal of $\mathcal{O}_{k_x}$ is $\bar{r}$.
    \end{enumerate}
    From the construction, one can see that $\mathcal{P}^\Box(\bar{r},\mathcal{S})$ does represent the functor we described as above. We denote by $r_X:G_K\rightarrow \mathrm{GL}_n(\mathcal{O}^+_{ \mathcal{P}^\Box(\bar{r},\mathcal{S}) })$ the universal representation on $\mathcal{P}^\Box(\bar{r},\mathcal{S})$.
    
    By theorem \ref{Sd}, $\kappa$ is smooth of relative dimension $n^2+[K:\mathbb{Q}_p](\sum_{1\leq i<j\leq r}n_in_j)$. Recall $\mathcal{S}_{\mathrm{reg}}$ is smooth of dimension $r+[K:\mathbb{Q}_p](\sum_{1\leq i \leq r}\frac{n_i^2-n_i+2}{2})$ by remark \ref{Rmk:dimofS}. It follows that  $\mathcal{P}^\Box(\bar{r},\mathcal{S})$ is smooth and equidimensional of dimension
    $$l+n^2+[K:\mathbb{Q}_p](\sum_{1\leq i<j\leq l}n_in_j+\sum_{1\leq i \leq l}\frac{n_i^2-n_i+2}{2})=l+n^2+[K:\mathbb{Q}_p](\frac{n(n-1)}{2}+l).$$
    We have the natural morphism $$\pi_{\bar{r}}: \mathcal{P}^\Box(\bar{r},\mathcal{S}) \rightarrow \mathfrak{X}_{\bar{r}}\times \mathcal{S},\ (r,\mathcal{F}^\bullet(D),f,\nu)\mapsto (r,f) $$
    Moreover, by the description of $\mathcal{S}^\Box(\bar{r},\underline{d})$, we can factors $\pi_{\bar{r}}$ through $X_{\mathrm{par}}(\bar{r})$ with image $U_{\mathrm{par}}(\bar{r})$. Hence if we can show $\pi_{\bar{r}}$ is smooth of relative dimension $l$ then $U_{\mathrm{par}}(\bar{r})$ is Zariski open in $X_{\mathrm{par}}(\bar{r})$ and equidimensional of the dimension we claimed.
    
    By corollary \ref{Cor:StrPar}, the points in $U_{\mathrm{par}}(\bar{r})$ is strictly paraboline. Hence we can apply corollary \ref{Cor:KPX6.3.10} to get a proper birational, hence surjective morphism $$F:X \rightarrow X_{\mathrm{par}}(\bar{r})$$ and, if we denote $D:= D^\dagger_{\mathrm{rig}}(F^*r_X)$, an increasing filtration $\mathcal{F}_\bullet(D)$ of ($\varphi,\Gamma_K$)-submodules of $D$ such that there exists a short exact sequence of $(\varphi,\Gamma_K)$-modules over $X$ for $1\leq i\leq r$
    $$0\rightarrow \mathrm{gr}_iD\rightarrow F^*(\mathcal{R}(\omega_i))\otimes \mathcal{L}_i\rightarrow \mathcal{M}_i\rightarrow 0$$ 
    where $\mathcal{L}_i$ is a line bundle over $X$ with trivial $(\varphi,\Gamma_K)$-action and  locally on $X$, the cokernel $\mathcal{M}_i$ is killed by some power of $t$, and is supported over some Zariski closed set $Z_i$ which is disjoint from $F^{-1}(U_{\mathrm{par}}(\bar{r}))$.
    
    Let $U$ be the intersection of $X\setminus(\bigcup\limits_{1\leq i\leq l}Z_i)$ and the preimage of $\mathcal{S}_{\mathrm{reg}}\subseteq \mathcal{S}$ for the morphism $X\rightarrow X_{\mathrm{par}}(\bar{r})\rightarrow\mathcal{S}$. It is Zariski open and contains $F^{-1}(U_{\mathrm{par}}(\bar{r}))$, which implies it is Zariski dense in $X$. Let $t:U^\Box\rightarrow U$ denote the $\mathbb{G}_m^{l}$ torsor of the trivialization of the line bundles $\mathcal{L}_i|_{U}$. By construction, note that $U^\Box$ has the following universal property. There exists canonical isomorphisms $t_i:\mathcal{O}_{U^{\Box}}\cong t^*(\mathcal{L}_i)$ for $1\leq i\leq l$, and if $g: T\rightarrow U$ is a morphism, with trivialization $s_i:\mathcal{O}_Y\cong g^*(\mathcal{L}_i)$ for $1\leq i\leq l$, then there exists an unique lifting $h:T\rightarrow U^\Box$ factors through $t$ such that $h^*(t_i)= s_i$.
    
    We can construct a morphism $s: U^\Box\rightarrow \mathcal{P}^\Box(\bar{r},\mathcal{S})$, by the universal property of $\mathcal{P}^\Box(\bar{r},\mathcal{S})$, such that $\pi_{\bar{r}}\circ s$ is the composition
    $$U^\Box\xrightarrow{t} U\xrightarrow{F|_U} X_{\mathrm{par}}(\bar{r}).$$ As $F$ is a composition of blow-ups and normalizations, one can find a Zariski open and Zariski dense subspace $V\subseteq X_{\mathrm{par}}(\bar{r})$ such that $F^{-1}(V)\subseteq U$ and $F|_{F^{-1}(V)}$ is an isomorphism. Similarly, we denote $t':V^\Box\rightarrow V$ be the $\mathbb{G}_m^{l}$-torsor of $\mathcal{L}_i|_{V}$, and it is easy to see that we can identify $V^\Box$ with $t^{-1}(V)$ and $t'=t|_{V^\Box}$. Then the universal property of $V^\Box$ allows us to construct a morphism
    $$\pi^\Box:\pi_{\bar{r}}^{-1}(V)\rightarrow V^\Box$$ 
    such that $t\circ\pi^\Box = F^{-1}|_{V}\circ \pi_{\bar{r}}$. By the universal property of $V^\Box$ again, one has $\pi^\Box\circ s|_{V^\Box} = \id_{V_\Box}$. By the description of $\mathcal{P}^\Box(\bar{r},\mathcal{S})$, the restriction of $\pi^\Box$ on $\pi^{-1}_{\bar{r}}$ is injective, as $\mathcal{P}^\Box(\bar{r},\mathcal{S})$ is reduced, one has $s\circ \pi^\Box=\id_{\pi^{-1}_{\bar{r}}(V)}$, which implies $\pi^\Box$ is an isomorphism. As $t:V^\Box \rightarrow V$ is a $\mathbb{G}_m^{l}$-torsor, then it is of relative dimension $l$. Combined with our previous computation of the dimension of $\mathcal{P}^\Box(\bar{r},\mathcal{S})$, one has $V$ is equidimensional of dimension
    $$[K:\mathbb{Q}_p](\frac{n(n-1)}{2}+l)+n^2.$$
    As $V$ is Zariski open and dense in $X_{\mathrm{par}}(\bar{r})$, so is $X_{\mathrm{par}}(\bar{r})$. This proves the assertion (1).
    
    Now we are going to show that $\pi_{\bar{r}}$ is smooth of relative dimension $l$, which implies assertion (2) and (3). To do this, it is enough to show that if $x\in \mathcal{P}^\Box(\bar{r},\mathcal{S})$ and $y=(r_y,f_y) = \pi_{\bar{r}}(x)\in U_{\mathrm{par}}(\bar{r})$, there exists an isomorphism of complete local rings over $\hat{\mathcal{O}}_{X_{\mathrm{par}}(\bar{r}),y}$
    $$\hat{\mathcal{O}}_{\mathcal{P}^\Box(\bar{r},\mathcal{S}),x}\cong \hat{\mathcal{O}}_{X_{\mathrm{par}}(\bar{r}),y}\llbracket x_1,\dots,x_l\rrbracket.$$
    Denote $A:= \hat{\mathcal{O}}_{X_{\mathrm{par}}(\bar{r}),y}$ and $B:=\hat{\mathcal{O}}_{\mathcal{P}^\Box(\bar{r},\mathcal{S}),x}$, then $\pi_{\bar{r}}$ induces a local morphism of complete local rings $A\rightarrow B$.
    
    The natural projection $X_{\mathrm{par}}\rightarrow \mathfrak{X}_{\bar{r}}$ induces a local morphism of complete local rings
    $$\hat{\mathcal{O}}_{\mathfrak{X}_{\bar{r}},r_y} \rightarrow \hat{\mathcal{O}}_{X_{\mathrm{par}}(\bar{r}),y}.$$
    According to \cite[Lem.2.3.3 and Prop.2.3.5]{2009Kisin}, there also exists a topological isomorphism between $\hat{\mathcal{O}}_{\mathfrak{X}_{\bar{r}},r_y}$ and $R_{r_y}^\Box$, the framed universal deformation ring of $r_y$, where $r_y$ is the Galois representation corresponds to the point $y$. Let $\mathcal{F}_\bullet$ be the unique parabolization of the $(\varphi,\Gamma_K)$-module $D_{\mathrm{rig}}^\dagger(r_y)$ correspond to the point $y$, and let $R_{r_y,\mathcal{F}_\bullet}^\Box$ be the framed universal deformation ring of the pair $(r_y,\mathcal{F}_\bullet)$, in the obvious sense. As before, the ring $B = \hat{\mathcal{O}}_{\mathcal{P}^\Box(\bar{r},\mathcal{S}),x}$ is naturally isomorphic to a complete local  $R_{r_y,\mathcal{F}_\bullet}^\Box$-algebra, smooth of relative dimension $l$. 
    
    Let $(x_1,\dots,x_l)$ be a family of topological generators of the $R_{r_y,\mathcal{F}_\bullet}^\Box$-algebra $B$, which means the ring homomorphism
    \begin{align*}
    	R_{r_y,\mathcal{F}_\bullet}^\Box[\![X_1,\dots,X_l]\!]&\rightarrow B \\
    	X_i&\mapsto x_i
    \end{align*}
    is an isomorphism.
     Then one can define a $A$-linear map
    \begin{align*}
    	A\llbracket X_1,\dots,X_l \rrbracket & \rightarrow B \\
    	X_i&\mapsto x_i.
    \end{align*} 
   Composing with the map $R_{r_y}^\Box\rightarrow A$, one has
   the following commutative diagram
   
   \begin{equation*}
   	\begin{tikzcd}
   		R_{r_y}^\Box\llbracket X_1,\dots,X_l \rrbracket \arrow{r}\arrow{d}&A\llbracket X_1,\dots,X_l \rrbracket\arrow{d} \\
   		R_{r_y,\mathcal{F}_\bullet}^\Box\llbracket X_1,\dots,X_l \rrbracket \arrow{r}{\sim} & B
   	\end{tikzcd}.
   \end{equation*}
   According to \cite[Prop.2.3.6 and Prop.2.3.9]{BJG2009} (Actually I we need parallel properties for the paraboline case, see the lemma below), the natural map $R_{r_y}^\Box\rightarrow R_{r_y,\mathcal{F}_\bullet}^\Box$ is surjective. It follows that the morphism $A\llbracket X_1,\dots,X_l \rrbracket \rightarrow B$ is surjective as well. As $B$ and $A\llbracket X_1,\dots,X_l \rrbracket$ are noetherian, local and complete with the same dimensional, it is enough to show that $A$ is integral to prove $A\llbracket X_1,\dots,X_l \rrbracket \rightarrow B$ is an isomorphism.
   
   As $X_{\mathrm{par}}(\bar{r})$ is reduced, then $A$ is reduced by \cite[§7.2,Prop.8]{bosch1984non}. Hence it is enough to show that $A$ has an unique minimal ideal. If we $X^{\mathrm{norm}}$ denote the normalization of $X_{\mathrm{par}}(\bar{r})$, then it suffices to show the fiber of $y$ in $X^{\mathrm{norm}}$ is a closed point.
   
   As $\mathcal{P}^\Box(\bar{r},\mathcal{S})$ and $U^\Box$ are normal, their morphisms to $X_{\mathrm{par}}(\bar{r})$ factors through $X^{\mathrm{norm}}$, i.e., one has the following commutative diagram
   \begin{equation*}
   	\begin{tikzcd}
   		& {\mathcal{P}^\Box(\bar{r},\mathcal{S})} \arrow[d] \arrow[rd, "\pi_{\bar{r}}"] &                           \\
   		U^\Box \arrow[r] \arrow[ru, "s"] & X^{\mathrm{norm}} \arrow[r]                                                   & X_{\mathrm{par}}(\bar{r})
   	\end{tikzcd}
   \end{equation*}
   The construction of $s$ implies that all points in the fiber of $y$ in $X^{\mathrm{norm}}$ are in the image of ${\mathcal{P}^\Box(\bar{r},\mathcal{S})}$ in $X^{\mathrm{norm}}$. As $\pi_{\bar{r}}^{-1}(y)$ is connected (isomorphic to $\mathbb{G}_m^l$ actually.) So the image in $X^{\mathrm{norm}}$ is connected, which is the fiber of $y$. As the fiber of $y$ is finite, hence it is a closed point.
   
   Now we are left to prove assertions (2) and (3). By \cite[Prop.1.7.8]{huber2013}, $U_{\mathrm{par}}(\bar{r})$ is admissible open in $X_{\mathrm{par}}(\bar{r})$. The construction of $s$ also means $F(U) =  F \circ r (U^\Box) = \pi_{\bar{r}}\circ s(U^\Box) \subseteq U_{\mathrm{par}}(\bar{r})$. By definition, one has $F^{-1}(U_{\mathrm{par}}(\bar{r}))\subseteq U$. It follows that $F(U) = U_{\mathrm{par}}(\bar{r})$ as $F$ is projective. Hence $U_{\mathrm{par}}(\bar{r})$ is Zariski constructible, by \cite[Lem.2.14]{Hellmann2016} and admissible open in $X_{\mathrm{par}}(\bar{r})$, hence it is Zariski open in $X_{\mathrm{par}}(\bar{r})$ by Lemma 2.13 in \textit{loc. cit.}. This proves assertion (2).
   
   As $\pi_{\bar{r}}$ and $\kappa$ is free, hence $\omega$ restricted on $U_{\mathrm{par}}(\bar{r})$ is smooth by \cite[Lem.5.8]{Breuil2017}, and moreover $U_{\mathrm{par}}(\bar{r})$ is smooth as $\mathcal{S}$ is smooth. This proves assertion (3).
   \end{proof}

   In the proof above, we need the parallel results of \cite[Prop.2.3.6 and Prop.2.3.9]{BJG2009} for the paraboline case (i.e. replacing trianguline $(\varphi,\Gamma_K)$-modules by paraboline $(\varphi,\Gamma_K)$-modules of \textit{loc. cit.}). By the proof of \textit{loc. cit.}, the only additional statement we need to check is the following lemma. We only give the statement and omit the proof as it is exactly the same as in \textit{loc. cit.} as well.
   
   \begin{Lemma}
   Let $A$ be a local Artin $C$-algebra equipped with a map $A/\mathfrak{m}\xrightarrow{\sim} C$. Let $(D,\mathcal{F}_\bullet)$ be a parabolized $(\varphi,\Gamma_K)$-module over $C$ with parameter $f\in\mathcal{S}_{\mathrm{reg}}$. If $(D_A,\pi_{D_A})$ is a deformation of $D$ over $A$, i.e.
   \begin{enumerate}
   \item $D_A$ is a $(\varphi,\Gamma_K)$-module over $A$;
   \item $\pi_{D_A}:D_A\rightarrow D$ is an $\mathcal{R}_{K,A}$-linear $(\varphi,\Gamma_K)$-morphism such that the inducing map
   $$D_A\otimes_A C\xrightarrow{\sim} D$$
   is an isomorphism.
   \end{enumerate}
   then
   \begin{enumerate}
   \item there exists at most one parabolization $\mathcal{F}_\bullet D_A$ of $D_A$ deforming $\mathcal{F}_\bullet D$, i.e. $$\pi_{D_A}(\mathcal{F}_i D_A) = \mathcal{F}_iD.$$
   \item If $A\rightarrow A'$ is a local map of Artin $C$-algebra, whose residue fields are isomorphic to $C$. If $D_A$ has a (unique) parabolization deforming $\mathcal{F}_\bullet D$, then the same is true for  $D_A\otimes_A A'$.
   \item Let $A\rightarrow A'$ be as above. Suppose that $A\rightarrow A'$ is injective, then the converse holds.
   \item Let $A,A'$ be a local Artin $C$-algebra with residue fields equal to $C$. If $(D_A,\pi_{D_A})$ (resp. $(D_{A'},\pi_{D_{A'}})$) be a deformation of $D$ with unique parabolization deforming $\mathcal{F}_\bullet D$, then so is $(D',\pi_{D'})$, where $D' := D_A\times_D D_{A'}$, and $\pi_{D'} :=\pi_{D_A}\circ \mathrm{pr}_{D_A} = \pi_{D_{A'}}\circ \mathrm{pr}_{D_{A'}}$. 
   \end{enumerate}
   \end{Lemma}
   
   The following corollary and its proof is contained in the proof of theorem \ref{Thm:Geoofpar}. We explicitly write it down as we will need the statement in the next section.
   
   \begin{Cor}\label{Cor:Surjforregularlucos}
   Let $x=(r_x,f_x)$ be a closed point in $U_{\mathrm{par}}(\bar{r})$, the projection $X_{\mathrm{par}}(\bar{r}) \rightarrow \mathfrak{X}_{\bar{r}}$ induces the local map
   $$\hat{O}_{\mathfrak{X}_{\bar{r}},r_x}\rightarrow \hat{O}_{U_{\mathrm{Par}}(\bar{r}),x}$$
   of complete local rings is surjective.
   \end{Cor}

   \begin{proof}
   In the proof of theorem \ref{Thm:Geoofpar}, we have already show that
   \begin{enumerate}
   \item $\hat{O}_{\mathfrak{X}_{\bar{r}},r_x} \cong R^\Box_{r_x}$;
   \item  there exists an isomorphism $\hat{O}_{U_{\mathrm{Par}}(\bar{r}),x}\llbracket X_1,\dots,X_l\rrbracket \xrightarrow{\sim} B$ for some complete local ring $B$.
   \item the composition $R^\Box_{r_x}\llbracket X_1,\dots,X_l\rrbracket\rightarrow \hat{O}_{U_{\mathrm{Par}}(\bar{r}),x}\llbracket X_1,\dots,X_l\rrbracket\xrightarrow{\sim} B$
   is surjective.
   \end{enumerate}
   It follows that the local map $$\hat{O}_{\mathfrak{X}_{\bar{r}},r_x}\rightarrow \hat{O}_{U_{\mathrm{Par}}(\bar{r}),x}$$ is surjective.
   \end{proof}
   
   \subsection{Paraboline Varieties}
   In the last subsection, we gave the definition of paraboline $(\varphi,\Gamma_K)$-modules. Now let $D$ be such a $(\varphi,\Gamma_K)$-module, with a parabolization $(\mathcal{F}_i D)_{i=0,\dots,l}$ for some ordered parameter $f=(f_1,\dots,f_l)\in \mathcal{S}$. Suppose that $f$ is in $\mathcal{S}_\mathrm{reg}$, then $D$ is strictly paraboline according to corollary \ref{Cor:StrPar}. We let 
   \begin{align*}
   	\Delta: \mathcal{S}_i = (\mathcal{T}_{\tau_i} \times \prod_{\sigma: K\inj C} \mathrm{GL}_{n_i}/B_{n_i}) & \rightarrow \mathcal{T}_{\tau_i}\\
   	        f_i=(\delta_i,(\mathrm{Fil}^\bullet_{i,\sigma})_\sigma) & \mapsto \delta_i
   \end{align*}
   denote the natural projection. In particular, $\Delta$ is proper. We have:
   
   \begin{Prop}\label{Prop:betterparameter}
   Let $D,\mathcal{F}_\bullet,f=(f_1,\dots,f_l)$ be as above (i.e. $f\in\mathcal{S}_{\mathrm{reg}}$). Suppose that there exists another filtration $(\mathcal{F}'_iD)_{i=0\dots,l}$ which is a parabolization for some ordered parameter $f'=(f'_1,\dots,f'_l)$, such that $\Delta(f_i') = \Delta(f_i)$. Then $\mathcal{F}'_i D = \mathcal{F}_i D$ for all $i=0,\dots,l$, and in particular $f_i=f_i'$.
   \end{Prop}
   
   \begin{proof}
   Actually, the proof is almost the same as corollary \ref{Cor:StrPar}. Indeed, note that the regularity condition defined in definition \ref{Def:RegLocus} only depends on the data of characters (as the filtration weight has been fixed already), which means for each $i<j$, the regularity condition holds for $(f_i,f_j')$, and hence one can apply proposition \ref{Prop:dimforcoh} to see that 
   $$\Hom(\mathcal{R}(f_i),\mathcal{R}(f_j'))=0.$$
   Hence $(D,\mathcal{F}_{l-1}D,\mathcal{R}(f_j'),\mathcal{R}(f_j'))$ satisfies the factor through condition, then the natural projection $$D\rightarrow D/\mathcal{F}'_{l-1}D\cong\mathcal{R}(f_l')$$
    factors through $\mathcal{F}_{l-1}D$, which means $\mathcal{F}_{l-1}D\subseteq \mathcal{F}'_{l-1}D$. By symmetricity, one has $\mathcal{F}'_{l-1}D\subseteq \mathcal{F}_{l-1}D$, (note that $f'$ also in $\mathcal{S}_{\mathrm{reg}}$ by definition), and hence $$\mathcal{F}'_{l-1}D = \mathcal{F}_{l-1}D.$$
    It follows that $\mathcal{F}'_i D = \mathcal{F}_i D$ for all $i=0,\dots,l$ via induction.
   \end{proof}
   
   The proposition above indicates that, in this case, the paraboline filtration $\mathcal{F}_\bullet D$ is uniquely determined by $(\delta_1,\dots,\delta_l)$. Therefore it is reasonable to define paraboline varieties via revising our previous definition for refined paraboline varieties in the following way.
   
   \begin{Def}
   Let $V_{\mathrm{par}}(\bar{r})$ be the image of $U_{\mathrm{par}}(\bar{r})$ in $\mathfrak{X}_{\bar{r}}\times \mathcal{T}$ for the natural projection
   $$\id_{\mathfrak{X}_{\bar{r}}}\times \Delta: \mathfrak{X}_{\bar{r}}\times \mathcal{S} \rightarrow \mathfrak{X}_{\bar{r}} \times\mathcal{T}.$$
   We define the \textit{paraboline\ variety\ }$Y_{\mathrm{par}}(\bar{r})$ \textit{for\ }$\bar{r}$ (\textit{of\ shape\ }$\mathcal{S}$) as the reduced rigid analytic space over $C$ which is the Zariski-closure in $\mathfrak{X}_{\bar{r}} \times\mathcal{T}$ of $V_{\mathrm{par}}(\bar{r})$. And we call $V_{\mathrm{par}}(\bar{r})$ the \textit{regular\ locus\ of} $Y_{\mathrm{par}}(\bar{r})$ (we will prove that $V_{\mathrm{par}}(\bar{r})$ is Zariski open in $Y_{\mathrm{par}}(\bar{r})$ as well).
   \end{Def}

   \begin{Rmk}\text{ }
   \begin{enumerate}
   	\item By definition, $V_{\mathrm{par}}(\bar{r})$ can be described as the set of the closed points $(r,\delta)\in\mathfrak{X}_{\bar{r}}\times \mathcal{T}_{\mathrm{reg}}$ such that $D_{\mathrm{rig}}^{\dagger}(r)$ is paraboline with parameter $f$ for some closed point $f$ in $\Delta^{-1}(\delta)$.
   	\item Even though the ambient space $\mathfrak{X}_{\bar{r}}\times\mathcal{T}$ only depends on the Galois representation $\bar{r}$ and the ordered inertia types $\underline{\tau}=(\tau_1,\dots,\tau_l)$, but the paraboline variety $Y_{\mathrm{par}}(\bar{r})$ also depends on the filtration weight $\underline{k}$. As $\mathcal{S}$ is determined by a unique pair $(\underline{\tau},\underline{k})$ (as a functor), hence it really makes sense to say a paraboline variety $Y_{\mathrm{par}}(\bar{r})$ for $\bar{r}$ is of shape $\mathcal{S}$. We omit the information of the shape $\mathcal{S}$ in our notation of paraboline variety $Y_{\mathrm{par}}(\bar{r})$ as we fix $\mathcal{S}$ once and for all through out the whole article.
   	\item The preimage $\Delta^{-1}(Y_{\mathrm{par}}(\bar{r}))$ is closed in $\mathfrak{X}_{\bar{r}}\times \mathcal{S}$, hence contains $X_{\mathrm{par}}(\bar{r})$. As $X_{\mathrm{par}}(\bar{r})$ and $Y_{\mathrm{par}}(\bar{r})$ are reduced, then there exists a unique morphism $\Xi:X_{\mathrm{par}}(\bar{r})\rightarrow Y_{\mathrm{par}}(\bar{r})$ such that the following diagram
   	\begin{equation*}
   		\begin{tikzcd}
   		X_{\mathrm{par}}(\bar{r}) \arrow[hook]{r}\arrow{d}{\Xi}&\mathfrak{X}_{\bar{r}}\times\mathcal{S}\arrow{d}{\id\times\Delta} \\
   			Y_{\mathrm{par}}(\bar{r}) \arrow[hook]{r} & \mathfrak{X}_{\bar{r}}\times\mathcal{T}
   		\end{tikzcd}.
   	\end{equation*}
    commutes.
    \item $\Xi$ is surjective. Indeed, as $\id\times\Delta$ is proper, then the image of $X_{\mathrm{par}}(\bar{r})$ in $\mathfrak{X}_{\bar{r}}\times\mathcal{T}$ is closed. Then it contains $Y_{\mathrm{par}}(\bar{r})$.
   \end{enumerate}	
   \end{Rmk}

   Actually, we do have another stronger motivation to define the paraboline varieties. In the next section, we will define some eigenvarieties on the automorphic side and our goal is to understand the comparison between the paraboline varieties and the eigenvarieties. Indeed, those eigenvarieties also have a parameter spaces (like the parameter space $\mathcal{S}$ for the refined paraboline  variety $X_{\mathrm{par}}(\bar{r})$), but which turns out to be $\mathcal{T}$ rather than $\mathcal{S}$, which means there is not any filtration structure. This fact motivates us to revise our original definition of the refined paraboline varieties.

   \begin{Prop}
   One has $$\Xi^{-1}(V_{\mathrm{par}}(\bar{r})) = U_{\mathrm{par}}(\bar{r}).$$
   \end{Prop}
   \begin{proof}
   Let $(r,\delta)$ be a closed point in $V_{\mathrm{par}}(\bar{r})$. Let $(r,f)$ be a closed point in $X_{\mathrm{par}}(\bar{r})$ such that $\Delta(f)=\delta$. By proposition \ref{Prop:betterparameter}, there exists a unique point $(r,g)\in U_{\mathrm{par}}(\bar{r})$ such that $\Delta(g)=\delta$. It is enough to show $f=g$.
   
   Let $D_r:= D_{\mathrm{rig}}^{\dagger}(r)$, with the unique increasing filtration $$0=D_{r,0}\subset D_{r,1}\subset \cdots \subset D_{r,l}$$ of $(\varphi,\Gamma_K)$-submodule such that $D_{r,i}/D_{r,i-1}\cong \mathcal{R}(g_i)$. In particular, by proposition \ref{Prop:Multipropforsen}, one has
    $$\mathrm{Sen}_{D_{r,i}}(T) = \mathrm{Sen}_{D_{r,i-1}}(T)\cdot\mathrm{Sen}_{\mathcal{R}(g_i)}(T).$$
    One other hand, we can apply corollary \ref{Cor:KPX6.3.10} to get a proper birational, hence surjective morphism $F:X \rightarrow X_{\mathrm{par}}(\bar{r})$ and, if we denote $D:= D^\dagger_{\mathrm{rig}}(F^*r_X)$, an increasing filtration $\mathcal{F}_\bullet(D)$ of ($\varphi,\Gamma_K$)-submodules of $D$ such that there exists a short exact sequence of $(\varphi,\Gamma_K)$-modules over $X$ for $1\leq i\leq l$
   $$0\rightarrow \mathrm{gr}_iD\rightarrow F^*(\mathcal{R}(\omega_i))\otimes \mathcal{L}_i\rightarrow \mathcal{M}_i\rightarrow 0$$ 
   where $\omega_i:X_{\mathrm{par}}(\bar{r})\rightarrow \mathcal{S}_i$ is the nutural projection, $\mathcal{L}_i$ is a line bundle over $X$ with trivial $(\varphi,\Gamma_K)$-action and  locally on $X$, the cokernel $\mathcal{M}_i$ is killed by some power of $t$, and is supported over some Zariski closed set $Z_i$ which is disjoint from $F^{-1}(U_{\mathrm{par}}(\bar{r}))$.
   
   Let $x$ be a closed point in the preimage of $(r,f)$ in $X$. After perhaps enlarging $C$, we may assume their residue fields are the same and then we have:
   \begin{enumerate}
   	\item $D_x \cong D_r$ (we identify $D_x$ and $D_r$ from now on via choosing an isomorphism);
   	\item $F^*(\mathcal{R}(\omega_i))_x \cong \mathcal{R}(f_i)$;
   	\item the complex
   	$$0\rightarrow (\mathcal{F}_{i-1}D)_x \rightarrow  (\mathcal{F}_{i}D)_x \xrightarrow{\mu_i}  \mathcal{R}(f_i) \rightarrow \mathcal{M}_{i,x}\rightarrow 0$$
   	where $\mathcal{M}_{i,x}$ killed by some power of $t$ for $i=1,\dots,l$;
   	\item the equation (by proposition \ref{Prop:Multipropforsen})
   	          $$\mathrm{Sen}_{(\mathcal{F}_i D)_{y}}(T) = \mathrm{Sen}_{(\mathcal{F}_{i-1} D)_{y}}(T)\cdot \mathrm{Sen}_{(F^*(\mathcal{R}(\omega_i))_y}(T)$$
   	          holds for every closed point $y$ in $F^{-1}(U_{\mathrm{par}}(\bar{r}))$, hence holds for $y$ in the whole $X$ by continuity. In particular, the equation holds for $y=x$;
   \item $\mathrm{Sen}_{\mathcal{R}(f_i)}(T) = \mathrm{Sen}_{\mathcal{R}(g_i)}(T)$ by the formula in corollary \ref{Cor:SenpolyforquasideRham}.
   \end{enumerate}
   We are going to show that $(\mathcal{F}_{i}D)_x = D_{r,i}$ and $f_i=g_i$ inductively, via comparing Sen polynomials.
   
   Firstly, when $i=1$, the injection 
   $$0\rightarrow (\mathcal{F}_1 D)_x \rightarrow \mathcal{R}(f_1)$$ 
   is an isomorphism by corollary \ref{Cor:Senineqforinj} as their Sen polynomial are the same by (4).
   Using the regularity condition (note that $(f_1,g_2,\dots,g_l)\in \mathcal{S}_{\mathrm{reg}}$), one has $$\Hom(\mathcal{R}(f_1),\mathcal{R}(g_i)) = 0,$$ for $i\geq 2$. Hence $((\mathcal{F}_1 D)_x,0,D_r,D_{r_1})$ satisfies the factor through condition, and therefor $(\mathcal{F}_1 D)_x$ is a $(\varphi,\Gamma_K)$-submodule of $D_{r,1}$. Then $(\mathcal{F}_1 D)_x= D_{r,1}$ by corollary \ref{Cor:Senineqforinj} again, as their Sen polynomial are the same, and also, $f_1=g_1$.
   
   Now assume $(\mathcal{F}_{i-1}D)_x=D_{r,i-1}$, hence in particular, $(\mathcal{F}_{i-1}D)_x$ is saturated in $D_x$, and so is in $(\mathcal{F}_{i}D)_x$. Then the complex
   $$0\rightarrow (\mathcal{F}_{i-1}D)_x\rightarrow (\mathcal{F}_{i}D)_x \rightarrow \mathrm{Im}(\mu_i)\rightarrow 0$$
   is a exact as $$\mathrm{rank}((\mathcal{F}_{i-1}D)_x) + \mathrm{rank} (\mathrm{Im}(\mu_i))= \mathrm{rank} ((\mathcal{F}_{i}D)_x)$$ and $(\mathcal{F}_{i-1}D)_x/(\mathcal{F}_{i-1}D)_x$ is torsion free. This implies the Sen polynomial of ${\mathrm{Im}(\mu_i)}$ is equal to the Sen polynomial of $\mathcal{R}(f_i)$, and therefor $\mu_i$ is surjective by corollary \ref{Cor:Senineqforinj}. It follows that the complex
   $$0\rightarrow (\mathcal{F}_{i-1}D)_x\rightarrow (\mathcal{F}_{i}D)_x \xrightarrow{\mu_i} \mathcal{R}(f_i)\rightarrow 0$$ is a short exact sequence. Using the regularity argument again, one has 
   $$\Hom((\mathcal{F}_{i}D)_x,\mathcal{R}(g_j)) = 0$$
   for any $j > i$.
   
   Then $(\mathcal{F}_{i}D)_x,0,D_r,D_{r,i})$ satisfies the factor through condition, and then $(\mathcal{F}_{i}D)_x$ is a $(\varphi,\Gamma_K)$-submodule of $D_{r,i}$. Hence $(\mathcal{F}_{i}D)_x = D_{r,i}$, by corollary \ref{Cor:Senineqforinj} again, as their Sen polynomials are the same and also $f_i=g_i$. Then one finishes the proof via induction.
   \end{proof}

   \begin{Rmk}
   The Proposition above shows that the image of $X_{\mathrm{par}}(\bar{r}) \setminus U_{\mathrm{par}}(\bar{r})$ is $Y_{\mathrm{par}}(\bar{r})\setminus V_{\mathrm{par}}(\bar{r})$ in $Y_{\mathrm{par}}(\bar{r})$. As $\Xi$ is proper, in particular closed, then $Y_{\mathrm{par}}(\bar{r})\setminus V_{\mathrm{par}}(\bar{r})$ is Zariski closed  and $V_{\mathrm{par}}(\bar{r})$ is Zariski open in $Y_{\mathrm{par}}(\bar{r})$. It follows that $\Xi$ restricted to $U_{\mathrm{par}}(\bar{r})$ is closed and bijective onto $V_{\mathrm{par}}(\bar{r})$ (on the topological level). Hence $\Xi$ induces an isomorphism between the underlying topological spaces of $U_{\mathrm{par}}(\bar{r})$ and $V_{\mathrm{par}}(\bar{r})$. Hence we can see that $\Xi$ is actually induces an isomorphism of rigid spaces between $U_{\mathrm{par}}(\bar{r})$ and $V_{\mathrm{par}}(\bar{r})$ once we can show that $\Xi$ is smooth on $U_{\mathrm{par}}(\bar{r})$.
   \end{Rmk}

   \begin{Theo}
   $\Xi$ restricted to $U_{\mathrm{par}}(\bar{r})$ is an isomorphism onto $V_{\mathrm{par}}(\bar{r})$.
   \end{Theo}

   \begin{proof}
   By the remark about, it is enough to show that $\Xi$ is smooth in $U_{\mathrm{par}}(\bar{r})$. It is equivalent to show that for every closed point $x$ in $U_{\mathrm{par}}(\bar{r})$, the induced local map
   $$\Xi_x^*: \hat{B}_y\rightarrow \hat{A}_x$$
   is an isomorphism, where $y:= \Xi(x)$, and $A_x$ (resp. $B_y$) is the local ring of $x$ (resp. $y$), and $\hat{A}_x$ (resp. $\hat{B}_y$) is the completion of $A_x$ (resp. $B_y$).
   
   We first show that $\Xi_x^*$ is injective. Let $V=\mathrm{Sp}B$ be an affinoid subdomain of $V_{\mathrm{par}}(\bar{r})$ containing $y$, and let $U:= \Xi^{-1}(V)$. As $\Xi$ is proper with finite fiber, by \cite[Section~9.6, Corollary~6]{bosch1984non}, $\Xi$ is finite and in particular $U=\mathrm{Sp}A$ for some affinoid algebra (the notations are compatible with the notations above). 
   
   As $\Xi|_U$ is surjective onto $V$ and $A,B$ are reduced, then the kernel $\Xi^*:B\rightarrow A$ is $$(\Xi^*)^{-1}(\{0\}) = \bigcap_{\mathfrak{m}\in\mathrm{Sp}A}(\Xi^*)^{-1}(\{\mathfrak{m}\}) =  \bigcap_{\mathfrak{n}\in\mathrm{Sp}B}\{\mathfrak{n}\}=0,$$
   hence $\Xi^*$ is injective. Let $\mathfrak{n}_y$ denote the maximal ideal of $y$ in $B$, and Let $\mathfrak{m}_x$ denote the maximal ideal of $x$ in $A$. As $\mathfrak{m}_x$ is the unique maximal ideal above $\mathfrak{n}_y$, it follows that $A/\mathfrak{n}_yA$ is a local Artin ring with maximal ideal $\mathfrak{m}_x$. In particular, one has $\mathfrak{m}_x^k\subseteq \mathfrak{n}_y\subseteq \mathfrak{m}_x$, and therefor the $\mathfrak{m}_x$-adic topology is equal to the $\mathfrak{n}_y$-adic topology in $A$. It follows that the natural map $A\otimes_B \hat{B}_y\rightarrow \hat{A}_{x}$ is an isomorphism. As $\hat{B}_y$ is flat over $B$, then the morphism
   $$\Xi_x^*: \hat{B}_y = B\otimes_B \hat{B}_y\rightarrow A\otimes_B \hat{B}_y = \hat{A}_x$$
   is injective.
   
   For surjectivity, consider image $z$ of $x$ (resp. $y$) in $\mathfrak{X}_{\bar{r}}$ by the natural projection $(r,f)\mapsto r$ (resp. $(r,\delta)\mapsto r$). Let $\hat{C}_z$ denote the completion of the local ring of $z$. By corollary \ref{Cor:Surjforregularlucos}, the composition
   $$\hat{C}_z\rightarrow \hat{B}_y\rightarrow \hat{A}_x$$
   is surjective, in particular $\Xi_x^*$ is surjective.
   \end{proof}

   \begin{Rmk}\label{RefinedParandPar}
   Our comparison above means the refined paraboline variety $X_{\mathrm{par}}(\bar{r})$ is very close to the paraboline variety $Y_{\mathrm{par}}(\bar{r})$. Indeed, what we have proved are:
   \begin{enumerate}
   	\item $\Xi:X_{\mathrm{par}}(\bar{r}) \rightarrow Y_{\mathrm{par}}(\bar{r})$ is proper birational and surjective.
   	\item Their regular loci are isomorphic, i.e. $\Xi|_{U_{\mathrm{par}}(\bar{r})}$ is an isomorphism onto $V_{\mathrm{par}}(\bar{r})$.
   \end{enumerate}
   Hence we can derive similar geometric properties of $Y_{\mathrm{par}}(\bar{r})$ as those in theorem \ref{Thm:Geoofpar}:
   \begin{enumerate}
   	\item the rigid space $Y_\mathrm{par}(\bar{r})$ is equidimensional of dimension 
   	$$[K:\mathbb{Q}_p](\frac{n(n-1)}{2}+l)+n^2;$$
   	\item the set $V_{\mathrm{par}}(\bar{r})$ is Zariski open in $Y_{\mathrm{par}}(\bar{r})$, hence it is also Zariski dense in $Y_{\mathrm{par}}(\bar{r})$;
   	\item the rigid space $V_{\mathrm{par}}(\bar{r})$ is smooth and the morphism 
   	$$\omega':Y_{\mathrm{par}}(\bar{r})\inj \mathfrak{X}_{\bar{r}}\times\mathcal{T}\rightarrow\mathcal{T}$$ 
   	restricted to $V_{\mathrm{par}}(\bar{r})$ is smooth.  
   \end{enumerate} 
   \end{Rmk}

   \begin{Rmk}\label{Rmk:NotationCompforBD}
   Our construction of $V_{\mathrm{par}}(\bar{r})$ (resp. $Y_{\mathrm{par}}(\bar{r})$) actually coincides with Breuil and Ding's construction of $U_{\Omega,\mathbf{h}}(\bar{\rho})$ (resp. $X_{\Omega,\mathbf{h}}(\bar{\rho})$) in \cite[Section~4.2]{ChristopheBreuil2021BernsteinE}. As quit a lot notations, involved in our constructions, are different, here we list their correspondence (on the left are Breuil and Ding's notations and on the right are ours): 
   \begin{itemize}
   	\item $r:= l$.
   	\item $L:= K$.
   	\item $E:= C$.
   	\item $\bar{\rho} := \bar{r}$.
   	\item Let $K_1,K_2$ be $p$-adic local field. Even though both in \cite{ChristopheBreuil2021BernsteinE} and our article, the notation $\mathcal{R}_{K_1,K_2}$ denotes the relative Robba rings, but the roles of $K_1$ and $K_2$ are exactly reversed. Namely, $R_{K_1,K_2}$ in \cite{ChristopheBreuil2021BernsteinE} means $\mathcal{R}_{K_2,K_1}$ in our article. 
   	\item $\mathbf{h} = (h_{i,\eta})_{i=1,\dots,n \atop \eta: L\inj E} := (-k^{(s_i)}_{j_i,\eta})_{i=1,\dots,n \atop \eta: K\inj C}$, here $s_i$ and $j_i$ are unique integers such that $1\leq s_i \leq l-1$ and $1\leq j_i \leq n_{s_i+1}$ such that $n_1+\cdots + n_{s_i} + j_i = i$ (as our convention for the Hodge-Tate weight and Breuil and Ding's are differed by a sign, hence a minus sign occurs here).
   	
   	\item $\Omega= (\Omega_i)_{i=1,\dots,r} :=(\text{the cusipidal type for } \mathrm{GL}_{n_i}(K) \text{ associated to }\tau_i )_{i=1,\dots,r}$.
   	\item One has a canonical isomorphism of the parameter spaces:
   	$$\mathrm{Spec}(\mathcal{Z}_{{\Omega}_i})\times \widehat{\mathcal{O}_L^\times} \cong \mathcal{T}_i$$
   	which follows from \ref{BernsteinCenter} (i.e. the computation of the Bernstein center). This induces a canonical isomorphism of the whole parameter spaces:
   	$$\mathscr{Z}:= \prod_{i=1}^{r} (\mathrm{Spec}(\mathcal{Z}_{{\Omega}_i})\times \widehat{\mathcal{O}_L^\times}) \cong \mathcal{T} := \prod_{i=1}^{r}\mathcal{T}_i.$$
   	\item One can check directly that the definition of generic locus $\mathscr{Z}^{\mathrm{gen}}\subset \mathscr{Z} $ in \cite[Section~4.2]{ChristopheBreuil2021BernsteinE} coincides with ours of the regular locus $\mathcal{T}_{\mathrm{reg}}\subset \mathcal{T}$ in \ref{Def:RegLocus}.
   \item It follows from the proof of lemma \ref{Sat2} that the condition (4.17) in \cite[Section~4.2]{ChristopheBreuil2021BernsteinE} holds if and only if there exists some point $f\in\Delta^{-1}(\underline{x},\chi)$ such that $D_{\mathrm{rig}}^{\dagger}(\rho)$ is paraboline with parameter $f$ (recall that $\Delta: \mathcal{S} \twoheadrightarrow \mathcal{T}$ is the natural projection defined in the beginning of section \ref{Subsect:RLoP}). This means that the set $U_{\Omega,\mathbf{h}}(\bar{\rho})$ defined in \textit{loc. cit.} coincides with our definition of the set $V_{\mathrm{par}}(\bar{r})$. In particular, there Zariski closures in $\mathfrak{X}_{\bar{r}}\times\mathcal{T}$ coincide, i.e. $U_{\Omega,\mathbf{h}}(\bar{\rho}) = V_{\mathrm{par}}(\bar{r})$.
   \end{itemize}
   \end{Rmk}

   \section{Applications to Eigenvarieties}
   
   Let $F^+$ be a totally real number field and $F$ be an imaginary extension of $F^+$, such that every finite place $v$ in $F^+$ dividing $p$ splits in $F$. We fix a unitary group $\mathcal{G}$ in $n$ variables over $F^+$ which splits over $F$, and which is compact at all infinite places of $F^+$. Associated to such a group $\mathcal{G}$ (and the choice of a tame level, i.e. a compact open subgroup of $\mathcal{G}(\mathbb{A}_{F^+}^{p\infty})$), people have constructed a nice Hecke eigenvariety which is an equidimensional analytic rigid space of dimension $n[F^+:\mathbb{Q}]$, see e.g. \cite{GaetanChenevier2004FamillesPD} or \cite{BJG2009}. We say such Hecke eigenvariety is of trianguline type as to every \textit{p-adic overconvergent eigenform of finite slope}, attached to a point $x$ in such eigenvariety, one can associate a continuous semi-simple Galois representation $$\rho_x: \mathrm{Gal}(\bar{F}/F) \rightarrow \mathrm{GL}_n(\bar{\mathbb{Q}}_p)$$ which is trianguline in the sense of \cite{PierreColmez2008ReprsentationsTD} at all place of $F$ dividing $p$ (see \cite{Kedlaya2012}). In \cite{Breuil2017a}, Breuil, Hellmann and Schraen study the characteristic of the classical points in such eigenvariety (a point in such eigenvariety is called \textit{classical} if it is attached to a classical eigenform) and proof that they are Zariski dense in the eigenvariety. This is the critical property to construct a reasonable map from such eigenvarieties to the trianguline variety (\cite{Breuil2017a}) in a flavor of Langlands correspondence, which contributes to understand the Fontaine-Mazur conjectures.
   
   In this section, we generalize the results above to the paraboline case. More precisely, we will construct a Hecke eigenvariety of paraboline type, in sense that the Galois representation $\rho_x:\mathrm{Gal}(\bar{F}/F) \rightarrow \mathrm{GL}_n(\bar{\mathbb{Q}}_p)$, attached to a point $x$ of overconvergent form in such eigenvariety (with some extra technical conditions), is paraboline in the sense of section \ref{Global} at all places of $F$ dividing $p$, and similarly, construct a reasonable map from such eigenvariety to the paraboline variety that we defined in section \ref{Global}.
   
   As in this section, there are a lot of notations involved, and many of them are defined by the same rule, to avoid redundancy, here we give some notation conventions that we used through out this section.
   \begin{itemize}
   	\item If $\delta$ is a locally algebraic character, we always denote by $\delta_{\mathrm{sm}}$ (resp. $\delta_{\mathrm{alg}}$) the smooth part (resp. the algebraic part) of $\delta$. 
   	\item If we already defined a group $G$ and a subgroup $G_k$, then for any sub group $H$ (resp. quotient group $H$), we automatically denote $H_k:= H\cap G_k$ (resp. $H_k:=\mathrm{Im}(G_k\rightarrow H)$).
   	\item If for every place $v$ in some global field $F$ dividing $p$, the group $G_v$ is defined, then we denote $G_p:= \prod_{v|p}G_v$, and also if $\pi_v$ is a defined $G_v$-representation (resp. monoid, ring, algebra and etc.), we denote $\pi_p:= \otimes_{v|p} \pi_v$.
   	\item If we already defined a group scheme $\mathcal{G}$ over some local field $F$, for every place $v$ of $F$, we denote $G_v:= \mathcal{G}(F_v)$.
   	\item For an affinoid rigid space $X$, we denote $R_X:=\Gamma(X,\mathcal{O}_X)$.
   	\item Let $R$ be a commutative ring, and $M$ be a $R$-module. For any subset $\Delta\subseteq R$, we denote 
   	$$M[\Delta] := \{m\in M\ |\ am=0,\ \forall a\in \Delta\}.$$
   	For a point $x\in \mathrm{Spec}R$, we denote
   	$$M[{x}]:= M[{\mathfrak{p}_x}],$$
   	where $\mathfrak{p}_x$ is the prime ideal associated to $x$.
   \end{itemize}

      \subsection{Notations and Settings}  \label{Subsect:Notations}  
   
   In this subsection, we clarify all notations and settings we used for eigenvarieties. In particular, we will define the space of overconvergent eigenforms following the language of \cite{Loeffler2011}, and describe some basic properties of this space.
   
   Firstly, we recall the global setting, basically the same as \cite{Breuil2017a}. We fix a totally real field $F^+$, and denote by $S_p$ the set of places of $F^+$ dividing $p$. We fix a totally imaginary quadratic extension $F$ of $F^+$ that splits at all places in $S_p$. We fix a finite extension $C$ of $\mathbb{Q}_p$ which is assumed to be large enough such that $|\Hom(F^+_v:C)| = [F^+_v:\mathbb{Q}_p]$ for all $v$ in $S_p$.
   
    We fix a unitary group $\mathcal{G}$ in $n$ variables over $F^+$ such that  $\mathcal{G} (F^+\times_\mathbb{Q} \mathbb{R})$ is compact, and $\mathcal{G} \times_{F^+} F \cong \mathrm{GL}_{n,F}$. We fix a parabolic subgroup $\mathcal{P}$ of $\mathrm{GL}_{n,F}$, which corresponds to the upper block triangular matrices of size $\underline{n }=(n_1,\dots,n_l)$, with Levi subgroup $\mathcal{M} \cong \mathrm{GL}_{n_1,F} \times \dots \times \mathrm{GL}_{n_l,F} $ and with unipotent radical $\mathcal{N}$. Let $\bar{\mathcal{P}}$ denote the opposite parabolic with unipotent radical $\bar{\mathcal{N}}$, and let $\mathcal{H} \cong \mathbb{G}_m^l$ denote the maximal quotient torus of $\mathcal{M}$.
    
    We fix an isomorphism $i:\mathcal{G}\times_{F^+}F\xrightarrow{\sim}\mathrm{GL}_{n,F}$, and for every $v\in S_p$, we fix a place $\tilde{v}$ of $F$ dividing $v$. Then $i$ and the isomorphism $F_v^+\rightarrow F_{\tilde{v}}$ induce an isomorphism $i_{\tilde{v}}: \mathcal{G}(F_v^+) \rightarrow GL_n(F_{\tilde{v}})$.  In the notation conventions that we declared as above, we have $G_v:= \mathcal{G}(F_v^+)\cong \mathrm{GL}_n(F_{\tilde{v}})$. Then for any subgroup (for example $M_{\tilde{v}}$) in $\mathrm{GL}_n(F_{\tilde{v}})$, we always identify it with a subgroup of $\mathcal{G}(F_v^+)$ via $i_{\tilde{v}}$ (and denote it by $M_v$).
    
    We fix a tame level $U^p = \prod_v U_v\subseteq \mathcal{G}(\mathbb{A}_{F^+}^{p,\infty})$, where $U_v$ is a compact open subgroup of $\mathcal{G}(F^+_v)$. We fix $S$ a finite set of places of $F^+$ that split in $F$ containing all $v|p$ and the set of finite places $v\notin p$ (split in $F$) such that $U_v$ is not hyperspecial. We denote by $\mathbb{T}^S$ the commutative spherical Hecke algebra:
    $$\varinjlim_I\left(\bigotimes_{v\in I}\mathcal{O}_C[U_v\backslash G_v/U_v]\right),$$
    the inductive limit being taken over finite sets $I$ of finite places of $F^+$ that split in $F$ such that $I\cap S =\emptyset$. 
    
    Then we state our local setting and discuss the space of overconvergent eigenform, which is basically specialize Loeffler's setting to the case of unitary group $\mathcal{G}$.
    
    For every $v\in S_p$, fix a uniformizer $\varpi_v$ in $F_v^+$. Let $$G_{v,0} := \{g\in \mathrm{GL}_n(\mathcal{O}_{F_{\tilde{v}}})\ |\ g \in \mathcal{P}(k(F_{\tilde{v}})) \mod \varpi_v\},$$ and let $$G_{v,k} := \{g\in \mathrm{GL}_n(\mathcal{O}_{F_{\tilde{v}}})\ |\ g \in \mathcal{N}(k(F_{\tilde{v}})) \mod \varpi_v^{k}\}$$ for $k\geq 1$. Then $G_{v,k}$ form a basis of the topology of $G_{v,0}$. And for each $k\geq 0$, one has
    \begin{align*}
    	\overline{N}_{v,k}\times \overline{M}_{v,k}\times {N}_{v,k} & \xrightarrow{\sim} G_{v,k}\\
    	(\overline{n},m,n) & \mapsto \overline{n}mn. 
    \end{align*}
    
     Let $\Sigma_v:= \varpi_v^{\mathbb{Z}}I_{n_1}\times \dots \times\varpi_v^{\mathbb{Z}}I_{n_l}$. Then the subgroup $\Sigma_v$ of $M_v$ is contained in the center of $M_v$. We define the sub-monoid $$\Sigma^+_v:= \{\varpi_v^{k_1}I_{n_1}\times \dots \times\varpi_v^{k_r}I_{n_l}|k_1\leq\dots\leq k_l\},$$
    and the semi-group
    $$\Sigma^{++}_v:= \{\varpi_v^{k_1}I_{n_1}\times \dots \times\varpi_v^{k_l}I_{n_r}|k_1<\dots< k_l\}.$$ 
    Note that in this setting, one has $z(\overline{N}_{v,0})z^{-1} \subseteq \overline{N}_{v,0}$ for any $z\in \Sigma_v^+$. We denote $\Sigma := \Sigma_p$ (resp. $\Sigma^+:=\Sigma^+_p$ and $\Sigma^{++} := \Sigma^{++}_p$) for short.
    
    We denote by $\mathbb{I}_v\subset G_v$ the monoid generated by $G_{v,0}$ and $\Sigma_v^+$, and define the Hecke algebra 
    $$\mathcal{H}_v:= C[G_{v,0}\backslash\mathbb{I}_v/G_{v,0}].$$
    It follows from \cite[Lem~3.4.1]{Loeffler2011} that the map
    \begin{align*}
    	C[\Sigma_v^+] & \rightarrow \mathcal{H}_v \\
    	   z\in \Sigma_v^+& \mapsto \frac{1}{N_z}[G_{v,0}zG_{v,0}]
    \end{align*} 
    is an isomorphism, where $N_z := |G_{v,0}zG_{v,0}/G_{v,0}|$. In particular, we have an isomorphism
    $$\mathcal{H}_p := \bigotimes_{v\in S_p}\mathcal{H}_v \cong C[\Sigma^+]$$
    of $C$-algebras.
    
    For each $v\in S_p$, we fix an element $$a^{(v)}=(a^{(v)}_{\eta,1}\leq\dots\leq a^{(v)}_{\eta,l})_{\eta}\in (\mathbb{Z}^{n}_{+})^{\Hom(F_{\tilde{v}},C)}.$$ We denote by $W_v$ the irreducible algebraic representation of $M_v$ over $C$ with lowest weight $a^{(v)}$ w.r.t. the Borel subgroup of upper triangular matrices, i.e.
    $$W_v\cong\mathrm{Ind}_{M_v\cap\overline{B}_v}^{M_v}\delta_{a^{(v)}},$$
    here $\overline{B}_v$ is Borel subgroup of $G_v$ consisting of lower triangular matrices and $\delta_{a^{(v)}}$ is the algebraic character of $M_v$ of weight $a^{(v)}$.
    We write $W := W_p$ for short.
    
    For every $v\in S_p$, we fix a supercuspidal representation $\sigma_v$ of $M_{v}$ with coefficient in $C$. By \cite[Theorem~1.3]{Paskunas2005}, there exists a $M_{v,0}$-type $\sigma_{v,0}$ of $\sigma_v$. We write $\sigma:= \sigma_p$ and $\sigma_0:= \sigma_{p,0}$ for short. 
    
     Let $\Pi_v$ denote $\mathrm{End}_{M_v}(c\mathrm{-}\mathrm{Ind}^{M_v}_{M_{v,0}} \sigma_{v,0})$, where $c\mathrm{-}\mathrm{Ind}^{M_v}_{M_{v,0}}(-)$ denotes the compact induction, and $\mathfrak{Z}_{v}$ denote the Bernstein center of $\sigma_v$. According to \cite[Thm~4.1]{Dat1999CaractresV}, one has $\mathfrak{Z}_{v} \cong \Pi_v$. We denote $\mathfrak{Z}:= \mathfrak{Z}_p$ for short.
     
    \begin{Lemma}\label{Lem:ArtkintoBernstein}
    	For any $t$ element $M_{p}$, we denote by $r_t$ the right regular action on $c\mathrm{-}\mathrm{ind}^{M_p}_{M_{p,0}} \sigma_{p,0}$, i.e. $r_tf(x):= f(xt)$. Then the map $$i_{\Sigma}: t\in\Sigma \mapsto r_{t}$$ identifies $C[\Sigma]$ with a subring of $\Pi_p=\mathfrak{Z}$.
    \end{Lemma}
    
    \begin{proof}
    	For any $t\in\Sigma_p$ and $m\in M_p$, one has $r_m\circ r_{t} = r_{t} \circ r_m$ (note that $t$ is in the center of $M_p$). Hence $r_{t}$ is an element in $\Pi_p = \mathfrak{Z}$.
    	
    	Now we are going to show that the map is injective. For any $t\in \Sigma$ and any $x\in\sigma_{0}$, let $f_{t,x}$ denote the unique element in $c\mathrm{-}\mathrm{ind}^{M_p}_{M_{p,0}} \sigma_{p,0}$ such that $f_{t,x}(t) = x$ and $f_{t,x}(x') = 0$ for all $x'\in M_{p}\backslash tM_{p,0}$. Hence one can identify $\bigoplus\limits_{t\in\Sigma} \sigma_{0}$ as a sub vector space of $c\mathrm{-}\mathrm{ind}^{M_p}_{M_{p,0}} \sigma_{p,0}$, via the map $(x_t)_{t\in\Sigma}\mapsto \sum\limits_{t\in\Sigma} f_{x_t,t}$. Note that for any $t,t'\in\Sigma$ and $x\in c\mathrm{-}\mathrm{ind}^{M_p}_{M_{p,0}} \sigma_{p,0}$, one has $r_{t}f_{x,t'} = f_{x,t't^{-1}}$. This implies
    	$$\bigoplus\limits_{t\in\Sigma} \sigma_{0}\cong C[\Sigma]\otimes_{C} \sigma_0$$
    	as an $C[\Sigma]$-module. Hence
    	 $C[\Sigma]$ acts faithfully on $\bigoplus\limits_{t\in\Sigma} \sigma_{0}$, and the map $t\mapsto r_{t}$ is injective.
    \end{proof}

    We denote by $\hat{H}_{v}$ the group $C$-rigid space representing the functor 
    $$ R \rightarrow \{\delta:M_v \rightarrow  R^\times\ |\ \delta \mathit{\ is\ continuous}\}.$$
    We also denote by $\hat{H}_{v,0}$ the group $C$-rigid space representing the functor
    $$R \rightarrow \{\delta:M_{v,0} \rightarrow  R^\times\ |\ \delta \mathit{\ is\ continuous}\}$$
    (Note that any continuous character $\delta$ of $M_v$ (reps. $M_{v,0}$) always factors through $H_v$ (resp. $H_{v,0}$). Hence $\hat{H}_v$ (resp. $\hat{H}_{v,0}$) can also be described as the group (rigid space) of continuous characters of $\hat{H}_v$ (resp. $\hat{H}_{v,0}$)).
    
     We denote by $\hat{H}_{v}^{\sigma_v}$ the finite, closed and reduced sub space of $\hat{H}_v$, whose $C'$-points are
    $$\{\delta: M_p \rightarrow  (C')^\times\ |\  \sigma\otimes \delta \cong \sigma \}$$
    for any finite field extension $C'/C$. We write $\hat{H}_{\sigma_v} := \hat{H}_{v}/\hat{H}_{v}^{\sigma_v}$ and $\hat{H}_\sigma := \prod_{v\in S_p} \hat{H}_{\sigma_v}$. 
    \begin{Rmk}\label{Rmk:DecompforParSpofEV}
    	\textit{  }
    	\begin{enumerate}
    		\item 
    		For each $v\in S_p$, we have $\sigma_v \cong \sigma_{v,1}\otimes \cdots \otimes \sigma_{v,l}$, where each $\sigma_{v,i}$ is a supercuspidal representation of $\mathrm{GL}_{n_i}(F_{\tilde{v}})$.
    		Let $\tau_{v,i}$ denote the type of $\mathrm{rec}(\sigma_{v,i})$ (recall that the type of a WD representation is defined in subsection \ref{ParemeterSpaces}). By lemma \ref{Lem:InvUnramChar}, one has
    		$$\hat{H}^{\sigma_v}_v \cong \mu_{e_{\tau_{v,1}}}\times \cdots \times \mu_{e_{\tau_{v,l}}},$$
    		where $\mu_{e_{\tau_{v,i}}}$ is the finite group scheme of ${e_{\tau_{v,i}}}$-th roots of unity.
    		\item If we denote by $G_{\sigma_{v,i}}$ the unique normal subgroup of $\mathrm{GL}_{n_i}(F_{\tilde{v}})$ of finite index $e_{\tau_{v,i}}$, containing all compact subgroups of $\mathrm{GL}_{n,i}(F_{\tilde{v}})$, and write $G_{\sigma_{v}} := \prod\limits_{1\leq i\leq l} G_{\sigma_{v,i}}$ for the subgroup of $M_{v}$.
    		
    		Then $\hat{H}_{\sigma_v}$ (resp. $\hat{H}_{\sigma}$)  can be described as the group (rigid space) of continuous characters of $G_{\sigma_v}$ (resp. $\prod_{v\in S_p} G_{\sigma_v}$) via the map $\delta\mapsto \delta|_{G_{\sigma_v}}$. 
    		\item For any smooth character $\delta\in\hat{H}_{\sigma}$, according to remark \ref{Rmk:TwistedbycharofBcenter}.1, the notation $\sigma\otimes\delta$ makes sense. Actually, after perhaps enlarging $C$, one can extend $\delta$ to some smooth character $\tilde{\delta}$ of $M_p$, then $\sigma\otimes\delta := \sigma\otimes\tilde{\delta}$.
    		
    		\item 
    		Let $\mathcal{T}_{F_{\tilde{v}}}$ denote the group ($C$-rigid space) of continuous characters of $F_{\tilde{v}}^{\times}$, and Let $\mathcal{W}_{F_{\tilde{v}}}$ denote the group ($C$-rigid space) of continuous characters of $\mathcal{O}_{F_{\tilde{v}}}^{\times}$. Then we have the decomposition (depends on the choice of $\pi_v$):
    		$$\mathcal{T}_{F_{\tilde{v}}} \cong \mathcal{W}_{F_{\tilde{v}}}\times \mathbb{G}_m$$ 
    		as we do in subsection \ref{QdeRhamPhiGammaMod}.
    		
    		Note that $\hat{H}_{v} \cong \prod\limits_{1\leq i\leq l} \mathcal{T}_{F_{\tilde{v}}}$ and $\hat{H}_{v,0} \cong \prod\limits_{1\leq i\leq l} \mathcal{W}_{F_{\tilde{v}}}$. Then one has (also depending on the choice of $\varpi_v$)
    		\begin{align*}
    			\hat{H}_{v}\cong \hat{H}_{v,0}\times \mathbb{G}_m^l
    		\end{align*}
    	    Then it follows from theorem \ref{BernsteinCenter} that
    	    $$\hat{H}_{\sigma_v} \cong (\hat{H}_{v,0}\times \mathbb{G}_m^l)/\hat{H}_v^{\sigma_v}\cong \hat{H}_{v,0}\times \prod_{1\leq i \leq l}(\mathbb{G}_m/\mu_{e_{\tau_{v,i}}})\cong \hat{H}_{v,0}\times \mathrm{Spec}(\mathfrak{Z}_v)^{\mathrm{rig}}$$
    	    
    	\end{enumerate}
    \end{Rmk}
    
    \begin{Prop}\label{SpAtOnePointForBC}
    	Let $\delta\in \mathrm{Spec}(\mathfrak{Z})$ be a closed point. Then one has
    	$$(c\mathrm{-Ind}_{M_{p,0}}^{M_p}\sigma_0)\otimes_{\mathfrak{Z}}k_{\delta} \cong \sigma\otimes\delta$$ 
    \end{Prop}

    \begin{proof}
    	According to lemma \ref{Lem:Twistingbychar}.2, one has the following isomorphism
    	$$\mathfrak{Z}\cong C[T_{v,i}^{\pm e_{\tau_{v,i}}}]_{v\in S_p,1\leq i\leq l},$$
    	where $\tau_{v,i}$ is the type of $\mathrm{rec}(\sigma_{v,i})$, such that the action of $\mathfrak{Z}$ on $\sigma\otimes\delta$ is the character $\delta$. Then by \cite[Prop~3.10]{caraiani2018}, one has
    	$$(c\mathrm{-Ind}_{M_{p,0}}^{M_p}\sigma_0)\otimes_{\mathfrak{Z}}k_{\delta} \cong \sigma\otimes\delta$$
    \end{proof}

    Let $V_v := W_v \otimes \sigma_{v,0}^\vee$ be the locally algebraic representation of $M_{v,0}$ with algebraic part $W_v$ and smooth part $\sigma_{v,0}$, here $(-)^\vee$ denotes the dual representation, and let $V:= \prod_{v\in p} V_v$ be the corresponding locally algebraic representation of $M_{p,0}$.
   
    Let $k(V)$ be the smallest integer such that $V|_{M_{p,k}}$ is algebraic. Let $X_v$ be an affinoid subspace of $\hat{H}_{v,0}$ (Note that in general, $\hat{H}_{v,0}$ is only quasi-Stein instead of affinoid). Let $k(X_v)$ be the smallest integer such that the natural map
    $$ H_{p,k} \rightarrow \Gamma(X_v,\mathcal{O}_{X_v}^\times)$$
    is analytic. We follow the notations in \cite[Definition~2.3.1]{Loeffler2011} and let $\mathcal{C}(X_v,V_v,k)$ denote the set of locally $C$-analytic $(R_{X_v}\otimes_CV_v)$-valued functions of $N_{v,0}$ which are analytic in cosets of $N_{v,k}$. This set has a $G_{v,0}$-action defined as in \textit{loc. cit.}.
    
     Now for an affinoid subspace of $\hat{H}_{p,0}$ of the form 
    $$X = \prod_{v\in S_p} X_v,$$
    where each $X_v$ is an affinoid subspace of $\hat{H}_{v,0}$, we define $k(X) := \max_{v\in S_p}\{k(X_v)\}$.
     Then for any $k\geq\max\{k(V),k(X)\}$, let
    $\mathcal{C}(X,V,k)$ denote the $G_{p,0}$-representation
    $$\prod_{v\in S_p} \mathcal{C}(X_v,V_v,k).$$
     Indeed, $\mathcal{C}(X,V,k)$ is a Banach $R_X$-module with property (Pr) (see \cite[3.5]{Loeffler2011} for the definition of (Pr) and the proof). 
     
     Let $\delta\in\hat{H}_{v,0}$ be a locally algebraic point, i.e. $\delta = \delta_\mathrm{sm}\delta_{\mathrm{alg}}$ for some smooth character $\delta_{\mathrm{sm}}$ and some algebraic character $\delta_{\mathrm{alg}}$, then for any $k\geq \max\{k(V),k(\delta)\}$, we define the $G_{v,0}$-representation
     $$\mathcal{C}(\delta,V_v,k)^\mathrm{cl} := \left(\mathrm{Ind}_{\bar{P}_{v,0}}^{G_{v,0}} (W\otimes \delta_{\mathrm{alg}})\right)^{\mathrm{alg}}\otimes \mathrm{Ind}_{\bar{P}_{v,0}/\bar{P}_{v,k}}^{G_{v,0}/G_{v,k}}(\sigma_{v,0}^\vee\otimes \delta_{\mathrm{sm}}).$$
     Actually, $\mathcal{C}(\delta,V_v,k)^\mathrm{cl}$ is a natural sub-representation of $\mathcal{C}(\delta,V_v,k)$ according to \cite[Prop~2.3.2]{Loeffler2011}. For a locally algebraic point $\delta = (\delta_v)_{v\in S_p}\in \hat{H}_{p,0}$, we write
     $$\mathcal{C}(\delta,V,k)^\mathrm{cl} : = \prod_{v\in S_p} \mathcal{C}(\delta_v,V_v,k)^\mathrm{cl}.$$

     According to \cite[Thm~2.4.7]{Loeffler2011}, for each $v\in S_p$, we extend the $G_{v,0}$-action on $\mathcal{C}(X_v,V_v,k)$ to an $\mathbb{I}_v$-action (recall that $\mathbb{I}_v$ is the monoid generated by $G_{v,0}$ and $\Sigma_v^+$) by continuous $R_{X_v}$-linear operators (see the remark below).
     
     \begin{Rmk} \label{RmK:LocallyAlgRep}
     	\textit{ }
     	\begin{enumerate}
     		\item Write $\overline{N}_{\mathbb{I}_v}:= \mathbb{I}_v\cap \overline{N}_v$, and $M_{\mathbb{I}_v}:= \mathbb{I}_v\cap M_v$, and $N_{\mathbb{I}_v}:= \mathbb{I}_v\cap N_v$. Then $\mathbb{I}_v$ is Iwahori factorisable, i.e.
     		$$\mathbb{I}_v\cong  \overline{N}_{\mathbb{I}_v}\times M_{\mathbb{I}_v}\times N_{\mathbb{I}_v}.$$
     		Moreover, $M_{\mathbb{I}_v} = M_{v,0}\Sigma_v^{+}$ and $N_{\mathbb{I}_v} = N_{v,0}$. Hence if we regard $V_v$ (resp. $R_{X_v}$) as an $M_{\mathbb{I}_v}$-representation (resp. $M_{\mathbb{I}_v}$-character), on which $\Sigma_v^+$ acts trivially, then the $\mathbb{I}_v$-action on $\mathcal{C}(X_v,V_v,k)$ is induced by following bijective map
     		\begin{align*}
     			\left[\mathrm{Ind}^{\mathbb{I}_v}_{\overline{P}_{\mathbb{I}_v}}(V_v\otimes R_{X_v})\right]^{k\mathrm{-an}} &\xrightarrow{\sim} \mathcal{C}(X_v,V_v,k)\\
     			f &\mapsto f|_{N_{\mathbb{I}_v}}.
     		\end{align*}
     		Hence the $\mathbb{I}_v$-action on $\mathcal{C}(X_v,V_v,k)$ is non-canonical and depends on our choice of $\varpi_v$ (note that the action of $M_{\mathbb{I}_v}$ depends on our choice of $\Sigma_v^+$, which depends on the choice of $\varpi_v$).
     		\item Let $\delta = \delta_{\mathrm{sm}}\delta_{\mathrm{alg}}$ be a locally algebraic character of $M_{p,0}$. The algebraic $M_{p,0}$-representation $W$ (resp. an algebraic $M_{p,0}$-character $\delta_{\mathrm{alg}}$) can be naturally regarded as an algebraic representation of $M_{p}$ (resp. an algebraic character of $M_{p}$), on which the $\Sigma^+$ action is not trivially in general. Suppose that for each $v\in S_p$, the restriction of $\delta_{\mathrm{alg}}$ on $M_{v}$ is $$\delta_{\mathrm{alg},v} = \delta_{b^{(v)}}$$
     		for some $b^{(v)}=(b^{(v)}_{\eta,i})\in(\mathbb{Z}^{l})^{\Hom(F_{\tilde{v}},C)}$. We write
     		\begin{align*}
     			\Delta_{\delta_{\mathrm{alg},v}}: M_{v} &\rightarrow C^{\times}\\
     			(B_1,\dots,B_l)&\mapsto \prod_{1\leq i\leq l}\prod_{\eta\in\Hom(F_{\tilde{v}},C)} \eta(\varpi_v)^{b^{(v)}_{\eta,i}v_{F_{\tilde{v}}}(\mathrm{det}(B_i))},      			
     		\end{align*}
             which is an unramified character of $M_v$. Then $\Delta_{\delta_{\mathrm{alg}}}\delta_{\mathrm{alg}}^{-1}$ restricted on $\Sigma$ is trivial, where $\Delta_{\delta_{\mathrm{alg}}}:=\prod_v \Delta_{\delta_{\mathrm{alg},v}}$.
            
            Similarly, we can also choose and fix a unramified character $\Delta_W$ (it is not unique in general) such that $\Sigma$ acts trivially on $W\otimes \Delta_W^{-1}$.
            \item For a locally algebraic character $\delta$ of $M_{p,0}$. We can uniquely decompose the locally algebraic $\mathbb{I}_p$-representation $\mathcal{C}(\delta,V,k)^{\mathrm{cl}}$ into $U_{\mathrm{alg}}\otimes U_{\mathrm{sm}}$, where the factors are respectively algebraic and smooth representation. Then using the notations as above, one has
            $$U_{\mathrm{alg}} \cong \left. \left(\mathrm{Ind}^{G_p}_{\overline{P}_p}(W\otimes\delta_{\mathrm{alg}})\right) \right| _{\mathbb{I}_p}$$
            and 
            $$U_{\mathrm{sm}}\cong \left(\mathrm{Ind}_{\bar{P}_{v,0}/\bar{P}_{v,k}}^{G_{v,0}/G_{v,k}}(\sigma_{v,0}^\vee\otimes \delta_{\mathrm{sm}}) \right) \otimes \left. \left(\Delta_W^{-1}\Delta_{\delta_{\mathrm{alg}}}^{-1}\right) \right| _{\mathbb{I}_p}, $$
            where the first factor can be regarded as an $\mathbb{I}_p$-representation as we have explained before.
     	\end{enumerate}
     	
       \end{Rmk}

       For any  $k\geq\max\{k(V),k(X)\}$, we define
       $$M(X,V,k):= \{ \phi: \mathcal{G}(F^+)\backslash \mathcal{G}(\mathbb{A}^\infty_{F^+}) \rightarrow \mathcal{C}(X,V,k)\ |\ \phi(gu) = u_p^{-1}\phi(g) \mathit{\ for\ } u\in U^p\times G_{p,0}\}$$
       \textit{the space of overconvergent automorphic forms of tame level $U^{p}$ and ``weight $\mathcal{C}(X,V,k)$"},
       here $u_p$ is the natural projection of $u$ from $U^p\times G_{p,0}$ to $G_{p,0}$, which is an $R_X$-module, endowed with a smooth left action of the monoid $\mathcal{G}(\mathbb{A}_{F^+}^{p\infty})\times \mathbb{I}_p$ via the formula $$u(\phi)(g) = u_p\phi(gu).$$  In particular, $M(V,X,k)$ is a Banach $R_X$-module with property (Pr) by \cite[Proposition~ 3.5.2]{Loeffler2011}.
    
    For a locally algebraic character $\delta\in \hat{H}_{p,0}$, and $k\geq\max\{k(V),k(\delta)\}$, we define
    $$M(\delta,V,k)^{\mathrm{cl}}:= \{\phi\in M(\delta,V,k)\ |\ \mathrm{Im}(\phi) \subseteq \mathcal{C}(\delta,V,k)^\mathrm{cl} \}$$
    \textit{the space of classical automorphic forms of tame level $U^{p}$ and ``weight $\mathcal{C}(\delta,V,k)$"}, which is a closed subspace of $M(\delta,V,k)$, stable under the action of $\mathcal{G}(\mathbb{A}_{F^+}^{p\infty})\times \mathbb{I}_p$.
    
    \begin{Rmk}
    	\text{ }
    	\begin{enumerate}
    		\item Note that $G_{\infty}:= \mathcal{G}(F^+\times_\mathbb{Q} \mathbb{R})$ is connected, hence one has 
    		$$G(F^+)\backslash G(\mathbb{A}^\infty_{F^+})\cong G(F^+)\backslash G(\mathbb{A}_{F^+})/ G_{\infty}^0,$$
    		where $G_\infty^0$ is the connected component of $\id$ in $G_\infty$. Then our notation for $M(X,V,k)$ (resp. $M(\delta,V,k)^{\mathrm{cl}}$) coincides with the notation for $M(e_U,X,V,k)$ in \cite[Def~3.7.1]{Loeffler2011} (resp. for $M(e_U,1,V(\delta),k)$ \cite[Def~3.9.1]{Loeffler2011}), where $U:= U^p\times G_{p,0}$.
    		\item (See the paragraph below \cite[Def~3.3.2]{Loeffler2011}), $M(X,V,k)$ (resp. $M(\delta,V,K)^{\mathrm{cl}}$) is an $R_X$-module (resp. a $k_{\delta}$-vector space) with a linear $\mathbb{T}^S[1/p]\otimes_C\mathcal{H}_p$-action, induced by the $(\mathcal{G}(\mathbb{A}_{F^+}^{p\infty})\times \mathbb{I}_p)$-action.
    	\end{enumerate}
    	
    \end{Rmk}
    
    \subsection{Extending the Hecke Action to the Bernstein Center $\mathfrak{Z}$}
     In \cite{Loeffler2011}, Loeffler has already constructed the eigenvariety using the eigenvariety machine of \cite{Buzzard2010}. We have to refine his construction. Roughly speaking, this is because such eigenvariety only parametrizes the eigencharacters of the Atkin-Lehner algebra $\mathcal{H}_p$ at places of $F^+$ dividing $p$, but for our purpose, we want to parametrize to eigencharacter of the Bernstein center $\mathfrak{Z}$. Note that by lemma \ref{Lem:ArtkintoBernstein}, the Atkin-Lehner algebra can be identified with a sub $L$-algebra of the Bernstein center $\mathfrak{Z}$. So our strategy to solve this problem is to uniformly extend the $\mathcal{H}_p$-action on the eigenspaces of overconvergent forms to a $\mathfrak{Z}$-action, which is the aim of this subsection.
     
     To do this, we need to give a alternative description of the space of overconvergent automorphic forms.
    
     Let $\hat{S}(U^p,C)$ denote the space of $p$-adic automorphic forms on $\mathcal{G}(\mathbb{A}^{\infty}_{F^+})$ of tame level $U^p$ with coefficients in $C$, that is the $C$-vector space of continuous functions $f: \mathcal{G}(F^+)\backslash \mathcal{G}(\mathbb{A}^\infty_{F^+})/U^{p} \rightarrow C$, endowed the linear continuous action of $G_p$ by right translation on functions. We also denote by $\hat{S}(U^p,C)^{k\text{-}\mathrm{an}} \subset \hat{S}(U^p,C)$ the $C$-subvector space of $\mathbb{Q}_p$-analytic vectors for the action of $G_{p,k}$ (\cite[Sect~7]{PeterSchneider2003AlgebrasOP}). We regard $R_X\otimes_CV$ as a $\overline{P}_{p,0} = \overline{N}_{p,0}M_{p,0}$-representation, trivial on $\overline{N}_{p,0}$. 
     
     Then $\hat{S}(U^p,L) \hat{\otimes}_C R_X\otimes_C V$ can be identified with the $C$-vector space of the continuous functions
     $$f: \mathcal{G}(F^+)\backslash \mathcal{G}(\mathbb{A}^\infty_{F^+})/U^{p} \rightarrow R_X\otimes_C V,$$
     endowed the linear continuous $\overline{N}_{p,0}$-action by $u(f)(g):=u(f(g\cdot u))$. And $$\hat{S}(U^p,C)^{k\text{-}\mathrm{an}} \hat{\otimes}_C R_X\otimes_C V \subset \hat{S}(U^p,L) \hat{\otimes}_C R_X\otimes_C V$$
     is the $L$-subvector space of $\mathbb{Q}_p$-analytic vectors for the action of $\overline{N}_{p,0}$. 
     
     Moreover, We define the action of $\mathcal{H}_p=L[\Sigma^+]$ on $\left( \hat{S}(U^p,L)^{k\mathrm{-}\mathrm{an}} \right)^{\overline{N}_{p,0}} $ as
     $$(z\circ f)(g) := \frac{1}{[\overline{N}_{p,0}:z\overline{N}_{p,0}z^{-1}]} \sum_{n\in \overline{N}_{p,0}/z\overline{N}_{p,0}z^{-1}}f(gnz).$$
     This action is well defined and does not depend on the choice of representative $n$ in the cosets of $[\overline{N}_{p,0}:z\overline{N}_{p,0}z^{-1}]$ as for any $n_1,n_2\in \overline{N}_{p,0}$ such that $n_2^{-1}n_1\in z\overline{N}_{p,0}z^{-1}$, one has 
     $$f(gn_1z) = f(gn_2z\cdot z^{-1}n_2^{-1}n_1z) = f(gn_2z).$$
      We also regard $R_X$ as a representation of $M_p$ via the decomposition (this depends on our choice of $\varpi_v$ for $v\in S_p$) 
     $$\hat{H}_{p} \cong \hat{H}_{p,0}\times \mathcal{T}^{l}_{F_{\tilde{v}}},$$
     and in particular, $\Sigma$ acts trivially on $R_X$, and let $\Sigma^+$ acts trivially on $V$. Then we extend the action of $C[\Sigma^+]$ on the tensor product
     $$\left(\hat{S}(U^p,C)^{k\text{-}\mathrm{an}}\right)^{\overline{N}_{p,0}} \hat{\otimes}_C (R_X\otimes_C V)$$
     by letting $\Sigma^+$ acts trivially on the second factor. Note that the $C[\Sigma^+]$-action on $\left(\hat{S}(U^p,C)\right)^{\overline{N}_{p,0}} \hat{\otimes}_C (R_X\otimes_C V)$ is stable on the $M_{p,0}$-invariant subspace. Indeed, if $f$ is a $M_{p,0}$-invariant function, then so is $z\circ f$ as
     \begin{align*}
     	\sum_{n\in \overline{N}_{p,0}/z\overline{N}_{p,0}z^{-1}}mf(gmnz) & = \sum_{n\in \overline{N}_{p,0}/z\overline{N}_{p,0}z^{-1}}mf(g\cdot mnm^{-1} \cdot mz) \\
     	         & = \sum_{n\in \overline{N}_{p,0}/z\overline{N}_{p,0}z^{-1}}mf(g\cdot mnm^{-1} \cdot zm) \\
     	         & = \sum_{n\in \overline{N}_{p,0}/z\overline{N}_{p,0}z^{-1}}f(g\cdot mnm^{-1} \cdot z) \\
     	         & = \sum_{n\in \overline{N}_{p,0}/z\overline{N}_{p,0}z^{-1}} f(g n z)
     \end{align*}
     for $m\in M_{p,0}$ and  $n\in \overline{N}_{p,0}$ and $z\in\Sigma^+$ (note that $mz\overline{N}_{p,0}z^{-1}m^{-1} =z\overline{N}_{p,0}z^{-1}$, then the map $n\mapsto mmm^{-1}$ induces a bijection on $\overline{N}_{p,0}/z\overline{N}_{p,0}z^{-1}$). Hence we finally define an action of $\mathcal{H}_p = C[\Sigma^+]$ on
     $$(\hat{S}(U^p,C)^{k\mathrm{-}\mathrm{an}} \hat{\otimes}_C R_X\otimes_C V)^{\bar{P}_{p,0}} \cong ((\hat{S}(U^p,C)^{k\mathrm{-}\mathrm{an}})^{\overline{N}_{p,0}} \hat{\otimes}_C R_X\otimes_C V)^{M_{p,0}}$$ 
     
     The following proposition is the "family" version of \cite[Proposition~3.10.1]{Loeffler2011}. As the argument is almost the same as Loeffler did in \textit{loc. cit.}, we only give a sketch of the proof here.
    \begin{Theo}\label{ComparisonforAutomorphicForms}
    	There is an isomorphism 
    	$$M(X,V,k)\xrightarrow{\sim} (\hat{S}(U^p,C)^{k\text{-}\mathrm{an}} \hat{\otimes}_C R_X\otimes_C V)^{\bar{P}_{p,0}}$$
    	of $R_X$-modules, and commuting with the Hecke action of $\mathbb{T}^{S}$ and $\mathcal{H}_p$ on both sides. 
    \end{Theo}
    \begin{proof}
    	Note that $$\mathcal{C}(X,V,k) \cong \mathcal{C}^{k\text{-an}}(N_{p,0},R_X\otimes_C V),$$ with an evaluation map $\mu: \mathcal{C}(X,V,k) \rightarrow R_X\otimes_C V$ by evaluating the function on $\id\in N_{p,0}$. Given $f\in M(X,V,k)$, we can regard $f$ as a map $$\mathcal{G}(\mathbb{A}^\infty_{F^+}) \rightarrow \mathcal{C}(X,V,k).$$
    	Composing with $\mu$, we get a function $\tilde{f}: \mathcal{G}(\mathbb{A}^\infty_{F^+}) \rightarrow R_X\otimes_C V$. One can check $\tilde{f}$ is $G_{p,k}$-analytic, left $\mathcal{G}(F^+)$-invariant and right $U^p$-invariant and $\bar{P}_{p,0}$-invariant. Hence we can define $\tilde{f}$ as the image of $f$ in $(\hat{S}(U^p,C)^{k\mathrm{-an}} \hat{\otimes}_C R_X\otimes_C V)^{\bar{P}_{p,0}}$.
    	
    	Conversely, given $h \in (\hat{S}(U^p,C)^{\mathrm{an}} \hat{\otimes}_C  R_X\otimes_C V)^{\bar{P}_{p,0}}$, we regard it as a map $G(\mathbb{A}^\infty_{F^+}) \rightarrow  R_X\otimes_C V$. We define $f$ to be the function $G(\mathbb{A}^\infty_{F^+}) \times N_0 \rightarrow  R_X\otimes_C V$ given by $f(g)(n) = h (gn^{-1})$. One can also check this gives the inverse of the map $f \mapsto  \tilde{f}$.
    	
    	Moreover, this isomorphism is obviously $\mathbb{T}^{S}$-equivariant and it is also $\mathcal{H}_p$-equivalent by \cite[Proposition~3.10.2]{Loeffler2011}.
    \end{proof}
    
    \begin{Rmk}
    	For any $C$-point $x: R \rightarrow C$ in $X\subseteq\hat{H}_{p}$, we can specialize the isomorphism above at $x$, and the right hand side becomes 
    	$$(\hat{S}(U^p,C)^{k\text{-an}} \hat{\otimes}_C R_X\otimes V)^{\bar{P}_{p,0}}\otimes_{R,x} C \cong (\hat{S}(U^p,C)^{k\text{-an}} \hat{\otimes}_C V(\delta_x))^{\bar{P}_{p,0}},$$
    	which (passing to direct limit $\varinjlim_k$) gives the isomorphism in \cite[Proposition~3.10.1]{Loeffler2011}. 
    \end{Rmk}

    Let $\mathfrak{m}^S$ a maximal ideal of $\mathbb{T}^S$ with residue field $k_C$ (enlarging $C$ if necessary) such that $\hat{S}(U^p,C)_{\mathfrak{m}^S} \neq 0$. Let $\bar{\rho}= \bar{\rho}_{\mathfrak{m}^S}: \mathrm{Gal}(\bar{F}/F)\rightarrow \mathrm{GL}_n(k_C)$ denote the unique absolutely semi-simple Galois representation associated to $\mathfrak{m}^S$ (see \cite[Prop~3.4.2]{CHT08}). Additionally, we assume that $\mathfrak{m}^S$ is \textit{non-Eisenstein}, i.e. $\bar{\rho}$ is absolutely irreducible. Then it follows from \cite[Prop~3.4.4]{CHT08} that the space $\hat{S}(U^p,C)_{\mathfrak{m}^S}$ becomes a module over $R_{\bar{\rho},S}$, the complete local noetherian $\mathcal{O}_C$-algebra of residue field $k_C$ pro-representing the functor of deformations $\rho$ of $\bar{\rho}$ that are unramified outside $S$ and such that $\rho'\circ c\cong \rho\otimes \chi^{n-1}$ ($c\in\mathrm{Gal}(F/F^+)$ is the complex conjugation, where $\rho'$ is the dual of $\rho$ and $\chi: \mathrm{Gal}(\bar{F}/F)\rightarrow \mathbb{Z}_p^\times$ is the $p$-adic cyclotomic character).
    
    \begin{Rmk}
    	Note that the $\mathbb{T}^{S}$-action commutes with the $G_p$-action on $\hat{S}(U^p,C)$, and in particular preserves $G_{p,k}$-analyticity. Hence the isomorphism in theorem \ref{ComparisonforAutomorphicForms} implies
    	$$M(X,V,k)_{\mathfrak{m}^S} \cong (\hat{S}(U^p,C)^{k-\mathrm{an}}_{\mathfrak{m}^S} \hat{\otimes}_C R_X\otimes_C V)^{\bar{P}_{p,0}},$$ 
    	which makes $M(X,V,k)_{\mathfrak{m}^S}$ into an $R_{\bar{p},S}$-module.
    \end{Rmk}

    We denote by $J_{\overline{P}_p}(-)$ the Emerton's Jacquet functor, defined as in \cite[Def~3.4.5]{Emerton2006}.

    \begin{Prop}\label{Prop:Finiteslopepart}
    	For an element $\varsigma\in \Sigma^{++}$, and $\lambda\in C^\times$. 
    	The natural quotient map
    	$$(\hat{S}(U^p,C)^{k-\mathrm{an}})^{\overline{N}_{p,0}}\rightarrow J_{\bar{P}_p}(\hat{S}(U^p,C)^{k-\mathrm{an}})$$
    	induces an isomorphism of $\lambda$-eigenspace (resp. generalized $\lambda$-eigenspace) between $M(X,V,k)$ and $[J_{\bar{P}_p}(\hat{S}(U^p,C)^{k-\mathrm{an}})\hat\otimes_CR_X\otimes_CV ]^{M_{p,0}}$ for $\varsigma$.
    \end{Prop}

    \begin{proof}
    	As $\overline{N}_{p,0}$ acts trivially on $R_X$ and $V$, one has
    	$$[(\hat{S}(U^p,C)^{k-\mathrm{an}})\hat\otimes_CR_X\otimes_CV ]^{\overline{P}_{p,0}}\cong [(\hat{S}(U^p,C)^{k-\mathrm{an}})^{\overline{N}_{p,0}}\hat\otimes_CR_X\otimes_CV ]^{M_{p,0}}.$$
    	Note that $C[\Sigma] \cong C[\Sigma^+]_{(\varsigma)}$ (the localization by $\varsigma$), hence the $C[\Sigma^+]$-action on 
    	$$M(X,V,k)^{\varsigma=\lambda}$$
    	uniquely extends to an $C[\Sigma]$-action. We denote by $(-)_{\mathrm{fs}}$ the Emerton's finite slope part functor (see the definition in \cite[Def~3.2.1]{Emerton2006}). By the universal property of finite slope part (see Prop~3.2.4 of \textit{loc. cit.}). One has
    	$$M(X,V,k)^{\varsigma=\lambda} \cong M(X,V,k)^{\varsigma=\lambda}_{\mathrm{fs}}.$$
    	Then one can invoke Prop~3.2.9 of \textit{loc. cit.} to get the conclusion ( note that $J_{\overline{P}_p}(-)$ is defined as $(-)^{\overline{N}_{p,0}}_{\mathrm{fs}}$).
    \end{proof}

    \begin{Prop}\label{Prop:ext to Bcenter}
    	Let $X\subset \hat{H}_{p,0}$ be an affinoid subdomain. We can extend the $\mathcal{H}_p$-action on
    	\begin{equation}\label{Jacmod}
    		[J_{\bar{P}_p}(\hat{S}(U^p,C)^{k-\mathrm{an}})\hat\otimes_CR_X\otimes_C V]^{M_{p,0} }
    	\end{equation}
    to an action of the Bernstein center $\mathfrak{Z}$, via the embedding
    \begin{align*}
    	i_{\Sigma}: \mathcal{H}_p & \inj \mathfrak{Z} \\
    	           z\in \Sigma^+ & \mapsto r_z,
    \end{align*}
    defined in lemma \ref{Lem:ArtkintoBernstein}.
    \end{Prop}
    
    \begin{proof}
    	Before we start with the proof, note that we can naturally regard $W$ as an algebraic representation of $M_p$, on which the center $Z(M_p)$ acts by an algebraic character $\delta_{W}$. We also extend the $M_{p,0}$-action on $R_X$ to an action of $M_{p}$ via the decomposition 
    	$$\hat{H}_{v} \cong \hat{H}_{v,0} \times \mathbb{G}_m^{l}$$
    	as before (depending to the choices of $\pi_v$). In particular, $\Sigma^+$ acts trivially on $R_X$.
    	    	 
    	Let $J$ denote $J_{\bar{P}_p}(\hat{S}(U^p,C)^{G_k-\mathrm{an}})$ for short temporarily. 
    	One has
    	\begin{align}\label{eq:Jac}
    		[J\hat\otimes_CR_X\otimes_C V]^{M_{p,0} } & \cong (\sigma_{p,0}^\vee\otimes_C J\otimes R_X\otimes_C W)^{M_{p,0}}\nonumber\\
    		& \cong\Hom_{M_{p,0}}(\sigma_{p,0}, J\otimes R_X\otimes_C W)\nonumber \\
    		& \cong\Hom_{M_{p}}(c\mathrm{-Ind}^{M_p}_{M_{p,0}}\sigma_{p,0}, J\otimes R_X\otimes_C W(\delta_W^{-1})).
    	\end{align}
        For any $z\in\Sigma^+$ and $\alpha\in \Hom_{M_{p}}(c\mathrm{-Ind}^{M_p}_{M_{p,0}}\sigma_{p,0}, J\otimes R_X\otimes_C W(\delta_W^{-1}))$, by the definition of $\mathcal{H}_p$-action, one has
        $$z (\alpha): f\mapsto z(\alpha(f)),$$
        where $z$ acts trivially on $R_X$ and $W(\delta_W^{-1})$. Note that the space (\ref{eq:Jac}) also has a natural left $\mathfrak{Z}$-module structure as follows:
        $$F(\alpha):f\mapsto \alpha \circ F(f)$$
        Then one has
        \begin{align*}
        	(i_{\Sigma}(z)(\alpha))(f)  & =  (\alpha\circ r_z) (f)\\
        	&= z \circ \alpha(f)\\
        	&= (z(\alpha))(f),
        \end{align*}
    for any $z\in\Sigma$, and any $\alpha$ in space (\ref{eq:Jac}) and any $f$ in $c$-$\mathrm{ind}^{M_p}_{M_{p,0}}\sigma_{p,0}$. It follows that if we regard $\mathcal{H}_p$ as a subring of $\mathfrak{Z}$ via $i_{\Sigma}$, then their action on space (\ref{eq:Jac}) coincide. Hence we can extend the $\mathcal{H}_p$-action on space (\ref{eq:Jac}) to an action of $\mathfrak{Z}$ via the embedding $i_{\Sigma}$.
    \end{proof}

   \begin{Rmk}
   	From now on, we keep such extension of $\mathfrak{Z}$-action on space (\ref{Jacmod}), and keep in mind that such extension depends on our choice of $\varpi_v$.
   \end{Rmk}
    
    \subsection{Construction of Eigenvarieties}
    
    Now we can give the construction of the eigenvariety. Even though we can not directly apply the eigenvariety machine from \cite[Construction~5.7]{Buzzard2010}, but the procedure is quite similar.
    
    For an affinoid domain $X\subset \hat{H}_{p,0}$, let $M_X$ denote $M(X,V,k)_{\mathfrak{m}^S}$ for short. Note that $M_X$ is a direct summand of $M(X,V,k)$ as a closed $R_X$-submodule. Hence $M_X$ is also a Banach $R_X$-module with property (Pr). We also fix an element $\varsigma\in \Sigma^{++}$, then $$\mathbb{T}_X:=R_X\otimes_C(\mathcal{H}_p\otimes_C R_{\bar{\rho},S}[\frac{1}{p}])$$ is a commutative unital $R_X$-algebra equipped with an $R_X$-algebra homomorphism $\mathbb{T}_X\rightarrow \mathrm{End}_{R_X}(M_X)$, such that the endomorphism of $M_X$ induced by $\varsigma\in \mathbb{T}$ is compact (recall that $\mathcal{H}_p\cong C[\Sigma_p^+]$, and the compactness of the action of $\varsigma$ follows from \cite[Thm~3.7.2(4)]{Loeffler2011}). Let $Z_\varsigma$ be the closed subspace of $X\times \mathbb{A}^{1,\mathrm{rig}}_C$ defined by the zero locus of the characteristic power series $F\in R_X\{\{T\}\}$ of $\varsigma$, with an affinoid admissible cover $\mathcal{C}$ of $Z_\varsigma(X,V)$, constructed as in \cite[Sect~4]{Buzzard2010}. 
    
     Let $Y$ be an element in $\mathcal{C}$, with image $X_Y\subseteq X$. By the construction of $\mathcal{C}$, the set $X_Y$ is an affinoid subdomain of $X$, with reduced ring of global sections $R_{X_Y}$. Let $M_{X_Y}$ denote $M_X\hat{\otimes}_{R_X} R_{X_Y}$, and for $t\in\mathbb{T}_X$, let $t_{X_Y}:M_{X_Y}\rightarrow M_{X_Y}$ denote the $R_{X_Y}$-linear continuous endomorphism induced by $t: M_X\rightarrow M_X$. By \cite[Lemma~2.13]{Buzzard2010}, $\varsigma_{X_Y}$ is still compact with characteristic power series $F_{X_Y}$ on $M_{X_Y}$, where $F_{X_Y}$ is just the image of $F$ in $R_{X_Y}\{\{T\}\}$.
    
    If $Y$ is connected, finite over $X_Y$ of degree $d$, one can associate $Y$ with a degree $d$ monic factor $Q_Y(T)$ of $F_{X_Y}(T)$ such that $R_Y$ is canonically isomorphic to $R_{X_Y}[T]/(Q_Y(T))$. Hence we can invoke \cite[3.3]{Buzzard2010} to decompose $M_{X_Y}$  as direct sum of $\varsigma$-invariant closed submodules $N_Y \oplus N'_Y$, such that $Q^*(\varsigma)$ is zero on $N_Y$ and invertible on $N'_Y$ (where $Q^*(T):=T^dQ(T^{-1})$), where $N_Y$ is a projective $R_{X_Y}$-module of rank $d$.
    
    Hence $N_Y$ is the kernel of map $Q^*(\varsigma): M_{X_Y}\rightarrow M_{X_Y}$(i.e. $N_Y = M_{X_Y}[{\{Q^*(\varsigma)\}}]$). Then by proposition \ref{Prop:Finiteslopepart}, $N_Y$ is isomorphic to $$[J_{\bar{P}_p}(\hat{S}(U^p,C)^{k-\mathrm{an}}_{\mathfrak{m}^S})\hat\otimes_CR_{X_Y}\otimes_C V]^{M_{p,0}}[\{Q^*(\varsigma)\}].$$ And by proposition \ref{Prop:ext to Bcenter}, we can extend the $\mathcal
    H_p$-action on $N_Y$ to a $\mathfrak{Z}$-action. Then we define $\mathbb{T}_Y$ to be the $R_{X_Y}$-algebra generated by the image of $\mathbb{T}\otimes_{\mathcal{H}_p}\mathfrak{Z}$ in $\mathrm{End}_{R_X}(N_Y)$, and let $D_Y:= (\mathrm{Spec}(\mathbb{T}_Y))^{\mathrm{rig}}$ denote the associated affinoid variety. Note that as $Q^*(\varsigma)$ is zero on $N$, the ring $\mathbb{T}_Y$ is a finite $R_Y$-algebra (via send $T$ to $\varsigma^{-1}$), and hence there is a natural finite map $D_Y\rightarrow Y$.
    
    For general $Y$, the image $X$ of $Y$ in $\mathrm{Max}(R)$ may not be connected. Suppose $X$ can be written as a finite disjoint union $X:= \coprod X_i$ of connected component and we denote $Y_i$ the preimage of $X_i$ in $Y$. We define $D_Y$ as the disjoint union of $D_{Y_i}$, together with a finite map $D_Y\rightarrow Y$. By the following lemmas, we can glue together the $D_Y$, as $Y$ ranges through all elements elements of $\mathcal{C}$, and the resulting rigid space $D_\varsigma(X,V)$, which is a finite cover of $Z_\varsigma(X,V)$, is called the eigenvariety of $M(X,V,k)$.
    
    \begin{Lemma}\label{BC of EigenVar}
    	If $Y\in\mathcal{C}$ with image $X_Y\subseteq X$, and $X'$ is an affinoid subdomain of $X_Y$, we denote by $Y'$ the pre-image of $X'$ under the map $Y\rightarrow X_Y$. Then $Y'$ is in $\mathcal{C}$ and is an affinoid subdomain of $Y$. Furthermore, $D_{Y'}$ is canonically isomorphic to the pre-image of $Y'$ under the map $D_Y\rightarrow Y$.
    \end{Lemma}

    \begin{proof}
    	The first part is exactly from \cite[Lemma~5.1]{Buzzard2010}. For the second part, note that $N_{Y'}$ is the base change from $R_X$ to $R_{X'}$ for the kernel $\mathcal{Q}^*(\varsigma): M_{X}\rightarrow M_{X}$, and $R_{X'}$ is flat over $R_X$. Hence $N_{Y'}$ is canonically isomorphic ot $N_Y \otimes_{R_X} R_{X'}$, and $\mathbb{T}_{Y'}$ is canonically isomorphic to $\mathbb{T}_Y \otimes_{R_X} R_{X'}$ (note that the extension to the action of $\mathfrak{Z}$, constructed in the proof of proposition \ref{Prop:ext to Bcenter}, are obviously compatible with base change), which implies $$D_{Y'} \cong D_Y \times_{X_Y} X' \cong D_Y\times_Y Y'.$$
    \end{proof}
    
    \begin{Lemma}
    	If $Y_1,Y_2\in\mathcal{C}$, then $Y:=Y_1\cap Y_2\in\mathcal{C}$. Furthermore $Y$ is an affinoid subdomain of $Y_i$ for $i=1,2$, and $D_Y$ is canonically isomorphic to the pre-image of $Y$ under the map $D_{Y_i}\rightarrow Y_i$.
    \end{Lemma}
    
    \begin{proof}
    	This lemma is almost the same as \cite[Lemma~5.2]{Buzzard2010} except for the revised construction of $D_Y$, and just take an extra note that our construction of the extension to an action of $\mathfrak{Z}$ is compatible with base change. Hence in particular the assertions about $Y$ still hold in our case. Namely, if we denote $X':= X_1\cap X_2$ where $X_i$ is the image of $Y_i$ in $X$, and $Y_i':= Y_i\times_{X_i} X'$ for $i=1,2$. Then we have $Y = Y_1'\cap Y_2'$ and $Y$ is a union of the component of $Y_i'$. Hence by the definition of $D_Y$ and lemma \ref{BC of EigenVar}, one has the canonical isomorphism:
    	$$D_Y \cong D_{Y_i}\times_{Y_i} Y$$
    \end{proof}
    
    \begin{Cons}\label{ConsforEigenVar}
    	By \cite[Proposition~9.3.2/1]{bosch1984non}, we can glue the system $\{D_Y|Y\in\mathcal{C}\}$ to a rigid space $\tilde{D}_\varsigma(X,V,k)$, and denote $D_\varsigma(X,V,k)$ the reduction of $\tilde{D}_\varsigma(X,V,k)$. 
    	
    	Note that by \cite[Cor~3.7.3]{Loeffler2011}, for any $k_1\geq k_2\geq \max\{k(X),k(V)\}$, the natural embedding $M(X,V,k_2)_{\mathfrak{m}^S}^{Q^*(\varsigma)}\inj M(X,V,k_1)_{\mathfrak{m}^S}^{Q^*(\varsigma)}$ is an isomorphism. Then we have a canonical isomorphism 
    	$$D_\varsigma(X,V,k_2)\cong D_\varsigma(X,V,k_1),$$
    	and denote by $D_\varsigma(X,V)$ for short.
    	
        When $X$ running through the affinoid subdomain of $\hat{H}_{p,0}$, we can glue the system $\{D_\varsigma(V,X)\ |\ X\subseteq \hat{H}_{p,0}\}$ further to a rigid space $D_\varsigma(V)$. We call $D_\varsigma(V)$ \textit{the eigenvariety of tame level $U^p$ and weight $V$}. 
    \end{Cons}

   \begin{Rmk} \label{EigenVarAndSpecVar}
   	\text{ }
   	\begin{enumerate}
   		
   		\item From the construction of the eigenvariety, one can see that $D_{\varsigma}(V,X)$ is Zariski closed in $X\times (\mathrm{Spec}(\mathfrak{Z}))^{\mathrm{rig}} \times \mathfrak{X}_{\bar{\rho},S}$. Hence $D_{\varsigma}(V)$ is Zariski closed in $\hat{H}_{\sigma,0}\times(\mathrm{Spec}(\mathfrak{Z}))^{\mathrm{rig}}\times \mathfrak{X}_{\bar{\rho},S} \cong \hat{H}_{\sigma}\times \mathfrak{X}_{\bar{\rho},S}$.
   		
   		\item
   		When $X$ running through the affinoid subdomain of $\hat{H}_{p,0}$, we can also glue the system of spectral varieties $\{Z_{\varsigma}(V,X)\ |\ X\subset \hat{H}_{p,0} \}$ further to a rigid space $Z_{\varsigma}(V)$, which is Zariski closed in $\hat{H}_{p,0}\times \mathbb{G}_m$.
   		
   	    Then natural projection $Z_{\varsigma}(V)\rightarrow \hat{H}_{p,0}$ is flat (and quasi-finite) by \cite[Lem~4.1]{Buzzard2010}. And even though the finite morphism 
   		$$ D_{\varsigma}(V) \rightarrow Z_{\varsigma}(V) $$
   		is not flat in general, but each irreducible component of $D_{\varsigma}(V)$ maps surjectively to an irreducible component of $Z_\varsigma(V)$ (see \cite[Prop~6.4.2]{GaetanChenevier2004FamillesPD} for the proof). Hence in particular, the image in $\hat{H}_{p,0}$ of each irreducible component of $D_{\varsigma}(V)$ is Zariski open and dense in a component of $\hat{H}_{p,0}$. 
   		\item
   		For now on, for the sake of constructing the map to paraboline varieties, we regard $D_{\varsigma}(V)$ as a closed subspace of $\hat{H}_{\sigma}\times \mathfrak{X}_{\bar{\rho},S}\cong \hat{H}_{p,0} \times \left(\mathrm{Spec}(\mathfrak{Z})\right)^{\mathrm{rig}} \times \mathfrak{X}_{\bar{\rho},S}$ via composing the close embedding above with the twisting
   		\begin{align*}
   			\hat{H}_{p,0} \times \left(\mathrm{Spec}(\mathfrak{Z})\right)^{\mathrm{rig}} \times \mathfrak{X}_{\bar{\rho},S} & \xrightarrow{\sim} \hat{H}_{p,0} \times \left(\mathrm{Spec}(\mathfrak{Z})\right)^{\mathrm{rig}} \times \mathfrak{X}_{\bar{\rho},S} \\
   			(\delta_0,\delta_{\mathfrak{Z}},\rho) &\mapsto (\delta_0^{-1},\delta_{\mathfrak{Z}},\rho)
   		\end{align*}
   	\end{enumerate}
   
   \end{Rmk}

   
   \begin{Rmk}
   	Let $x = (\delta,\rho)\in \hat{H}_{\sigma}\times \mathfrak{X}_{\bar{\rho},S}$ be a point in $D_{\varsigma}(V)$. Via remark \ref{Rmk:DecompforParSpofEV}(3), one can decompose $\delta$ into $ (\delta_0,\delta_{\mathfrak{Z}})$, where $\delta_0:= \delta|_{M_{p,0}}$ is the restriction of $\delta$ to $M_{p,0}$ and $\delta_{\mathfrak{Z}}: \mathfrak{Z} \rightarrow k_{\delta}$ is a closed point in $(\mathrm{Spec}(\mathfrak{Z}))^{\mathrm{rig}}$. Via the embedding $i_{\Sigma}$, one can further restrict $\delta_{\mathfrak{Z}}$ on $\Sigma$, and write $\delta_{\Sigma}:= \delta_{\mathfrak{Z}}|_{\Sigma}: \Sigma \rightarrow k_{\delta}$.
   	
   	\subsection{Classical points}
   	In this part, our aim is to describe the classical points in the eigenvarieties.
   	
   	Let $X$ be an affinoid subdomain of $\hat{H}_{\sigma}$ containing $x$, and let $k\geq\max\{k(X),k(V)\}$. For a closed point $x=(\delta,\rho)\in\hat{H}_{\sigma}\times \mathfrak{X}_{\bar{\rho},S}$, we define the notation $M(X,V,k)_{\mathfrak{m}^S}[x]$ in the following way (perhaps enlarging $C$): firstly,
   	$$M(\delta_0^{-1},V,k)_{\mathfrak{m}^S}[\rho] \subseteq M(\delta_0^{-1},V,k)_{\mathfrak{m}^S}$$
   	is the closed Banach subspace, consisting of the eigenvectors of $\rho: R_{\rho,S}\rightarrow k_{\delta}$. Then we consider 
   	$$M(\delta_0^{-1},V,k)_{\mathfrak{m}^S}[\rho,\delta_\Sigma]:=\left(M(\delta_0^{-1},V,k)_{\mathfrak{m}^S}[\rho]\right)[\delta_{\Sigma}],$$  
   	then eigenspace of the character $\delta_{\Sigma}$ of $\Sigma$. According to proposition \ref{Prop:ext to Bcenter}, we have extended to a $\mathfrak{Z}$-action on this space. Hence we denote by
   	$$M(X,V,k)[x] := \left( M(\delta_0^{-1},V,k)_{\mathfrak{m}^S}[\rho,\delta_\Sigma] \right)[\delta_{\mathfrak{Z}}]$$
   	the eigenspace of $\delta_{\mathfrak{z}}$ on $M(\delta_0^{-1},V,k)_{\mathfrak{m}^S}[{\rho,\delta_\Sigma}]$.
   	
   	 Then by almost the same argument in \cite[Lem~5.9]{Buzzard2010}, one can show that $x=(\delta,\rho)\in\hat{H}_{\sigma}\times \mathfrak{X}_{\bar{\rho},S}$ is a point in $D_{\varsigma}(V)$ if and only if $M(X,V,k)[{x}]\neq\emptyset$. 
   \end{Rmk}

    The following definition is similar as Loeffler's for classical points in \cite[Sect~3.13]{Loeffler2011}.
    
    \begin{Def}\label{Def:ClPt}
    	A closed point $x=(\delta,\rho)\in D_\varsigma(V)$ is called classical if
    	\begin{enumerate}
    		\item $\delta_0$ is locally algebraic.
    		\item $M(X,V,k)^{x}_{\mathfrak{m}^S}\cap \left[\varinjlim\limits_{k'\geq k} M(\delta_0^{-1},V,k')^{\mathrm{cl}}\right] \neq \emptyset$ for some (and hence any) affinoid subdomain $X$ in $\hat{H}_{\sigma}$, and $k\geq\max\{k(X),k(V)\}$. 
    	\end{enumerate}
    \end{Def}

    We will give a different description of classical points that more explicitly relate to classical automorphic forms.

    \begin{Lemma}\label{InvandCoinv}
    	Let $E$ be a smooth representation of $G_v$ over $C$. Then for any $z\in\Sigma_v^+$, the map $$x\mapsto \frac{1}{[\overline{N}_{v,0}:z\overline{N}_{v,0}z^{-1}]}\sum\limits_{g\in \overline{N}_{v,0}/z\overline{N}_{v,0}z^{-1}}gzx$$ (resp. $[x] \mapsto [zx]$) makes the $\overline{N}_{v,0}$-invariant space  $E^{\overline{N}_{v,0}}$ (resp. the $\overline{N}_{v,0}$-coinvariant space  $E_{\overline{N}_{v,0}}$) into an $C[\Sigma_v^+]$-module, such that the map
    	\begin{align*}
    		E^{\overline{N}_{v,0}} & \rightarrow E_{\overline{N}_{v,0}} \\
    		x &\mapsto [x]
    	\end{align*}
        is an isomorphism of $C[\Sigma_v^+]$-modules. 
        
        Moreover, the natural projection $E_{\overline{N}_{v,0}}\rightarrow E_{\overline{N}_{v}}$ induces an isomorphism
        \begin{equation}\label{eq:Emerton=classical}
        	E_{\overline{N}_{v,0}}\otimes_{C[\Sigma_v^+]} C[\Sigma_v] \cong E_{\overline{N}_{v}}
        \end{equation}
        of $C[\Sigma_v]$-modules.
    \end{Lemma}
    
    \begin{proof}
    	One can easily check the first assertion directly. We only proof the moreover part.
    	
    	Let $z$ be an element in $\Sigma_v^{++}$, note that 
    	$$\bigcup_{m\in\mathbb{N}}z^{-m}{\overline{N}_{v,0}}z^{m} = {\overline{N}_{v}}.$$
    	Then the kernel of the natural projection $E_{\overline{N}_{v,0}}\rightarrow E_{\overline{N}_{v}}$ is $\bigcup_{m} (E_{\overline{N}_{v,0}}[z^{m}])$. As $C[\Sigma_v] \cong C[\Sigma_v^+]_{z}$, and $z$ is invertible in $E_{\overline{N}_{v}}$, it follows that
        $$E_{\overline{N}_{v,0}}\otimes_{C[\Sigma_v^+]} C[\Sigma_v] \cong E_{\overline{N}_{v}}.$$
    \end{proof}

    \begin{Rmk}\label{Rmk:Emerton=classical}
    	The left hand side of isomorphism (\ref{eq:Emerton=classical}) is the Emerton's Jacquet functor applied to the smooth (hence locally analytic) representation $E$. The right hand side is the classical Jacquet functor for smooth representation theory. Hence we can use the notation $J_{\overline{P}_v}(-)$ (resp. $J_{\overline{P}_p}(-)$) for smooth representations without any ambiguity as these two functors coincide.
    \end{Rmk}

    We fix an embedding $\iota: \overline{\mathbb{Q}}_p\rightarrow \mathbb{C}$. Then the composition $$\iota_{v,\eta}: F^+ \overset{v}{\inj} F^+_{v}\cong F_{\tilde{v}} \overset{\eta}{\inj} C \inj \overline{\mathbb{Q}}_p \overset{\iota}{\inj} \mathbb{C}$$
    gives a bijection of sets:
    \begin{align*}
    	\{(v,\eta)\ |\ v\in S_p\mathit{\ and\ }\eta\in\Hom(F_{\tilde{v}},C) \} & \rightarrow S_{\infty}:= \{\mathit{infinite\ places\ of\ }F^+\}\\
    	(v,\eta) &\mapsto \iota_{v,\eta}
    \end{align*}

    \begin{Def}
    	Let $\delta_{\mathrm{alg}}$ be an algebraic character of $M_p$ of the form
    	$$\delta_{\mathrm{alg}}:(g_{v,1},\dots,g_{v,l})_{v\in S_p}\mapsto \prod_{v\in S_p}\left(\prod_{\eta\in \Hom(F_{\tilde{v}},C)}\eta\left(\det (g_{v,i})^{b^{(v)}_{\eta,i}}\right)\right)$$
    	We say $\delta_{\mathrm{alg}}^{-1}$ \textit{is anti-dominant w.r.t.} $W$ if, for any $v\in S_p$ and $\eta\in \Hom(F_{\tilde{v}},C)$, one has
    	$$a^{(v)}_{\eta,1}-b^{(v)}_{\eta,1}\leq \cdots\leq a^{(v)}_{\eta,j}-b^{(v)}_{\eta,i_j}\leq \cdots\leq a^{(v)}_{\eta,n}-b^{(v)}_{\eta,l},$$
        here $i_j$ is the smallest integer such that $n_1+\cdots+n_{i_j}\geq i$.
        
        A point $x=(\delta,\rho)\in D_{\varsigma}(V)$, such that $\delta$ is locally algebraic, is called \textit{dominant} if the algebraic part $\delta_{\mathrm{alg}}^{-1}$ of $\delta^{-1}$ is anti-dominant w.r.t. $W$.	
    \end{Def}
    
    \begin{Rmk}
    	In this case, $\mathrm{Ind}^{G_p}_{\overline{P}_p}(W(\delta^{-1}_{\mathrm{alg}}))$ is an irreducible algebraic representation of $G_p$, with lowest weight
    	$$c^{(v)}_{\eta}:=\left(a^{(v)}_{\eta,1}-b^{(v)}_{\eta,1}\leq \cdots\leq a^{(v)}_{\eta,j}-b^{(v)}_{\eta,i_j}\leq \cdots\leq a^{(v)}_{\eta,n}-b^{(v)}_{\eta,l}\right).$$
    	We write $\iota(\mathrm{Ind}^{G_p}_{\overline{P}_p}(W(\delta^{-1}_{\mathrm{alg}})))$ for the irreducible $\mathcal{G}(F^+\otimes_{\mathbb{Q}}\mathbb{R})$-representation over $\mathbb{C}$, of lowest $\iota_{v,\eta}$-weight $c^{(v)}_{\eta}$.
    \end{Rmk}

    \begin{Prop}\label{CharOfDomClPt}
    	A dominant point $x=(\delta,\rho)\in \hat{H}_{\sigma}\times \mathfrak{X}_{\overline{\rho},S}$ in $D_{\varsigma}(V)$ is classical if and only if there exists an automorphic representation $\pi= \pi_\infty \otimes \pi_f^p\otimes \pi_p$ of $G(\mathbb{A}_{F^+})$ over $\mathbb{C}$ such that the following conditions hold:
    	\begin{enumerate}
    		\item the $\mathcal{G}(F^+\otimes_{{\mathbb{Q}}}\mathbb{R})$-representation $\pi_\infty$ is isomorphic to $\iota(\mathrm{Ind}^{G_p}_{\overline{P}_p}(W(\delta^{-1}_{\mathrm{alg}})))^{\vee}$ (in this case, we say $\pi$ is of weight $W(\delta^{-1}_{\mathrm{alg}})$);  
    		\item the $\mathrm{Gal}(\overline{F}/F)$-representation $\rho$ is the Galois representation  associate to $\pi$; 
    		\item the invariant subspace $(\pi_f^p)^{U^p}$ is nonzero;
    		\item the $M_p$-representation $\sigma(\delta_\mathrm{sm}\Delta_W)\otimes_{k(\delta),\iota} \mathbb{C}$ is a subrepresentation of $J_{\overline{P}_p}(\pi_p)$ (recall remark \ref{RmK:LocallyAlgRep}(2) for the definition of the character $\Delta_W$).
    	\end{enumerate}
    \end{Prop}

    \begin{proof}
    	For simplicity, through out this proof, we always regard the $v$-component $\pi_v$ of $\pi$ as a representation over $\overline{\mathbb{Q}}_p$ via $\iota^{-1}$ for any finite place $v$ of $F^{+}$.
    	
    	Consider the decomposition
    	$$\hat{H}_\sigma \cong \hat{H}_{p,0}\times \mathrm{Spec}(\mathfrak{Z})^{\mathrm{rig}}$$
    	in remark \ref{Rmk:DecompforParSpofEV}(3). We write $\delta = (\delta_{0},\delta_{\mathfrak{Z}})$, where $\delta_0$ is a point in $\hat{H}_{p,0}$ and $\delta_{\mathfrak{Z}}$ is a point in $\mathrm{Spec}(\mathfrak{Z})^{\mathrm{rig}}$. Consider the decomposition $\delta = \delta_{\mathrm{sm}}\cdot\delta_{\mathrm{alg}}$ for some smooth character $\delta_{\mathrm{sm}}$ and some algebraic character $\delta_{\mathrm{alg}}$. Similarly, we write $\delta_{\mathrm{sm}} = (\delta_{\mathrm{sm},0},\delta_{\mathrm{alg},\mathfrak{Z}})$ and $\delta_{\mathrm{alg}} = (\delta_{\mathrm{alg},0},\delta_{\mathrm{alg},\mathfrak{Z}})$. One has $\delta_{\mathfrak{Z}} = \delta_{\mathrm{sm},\mathfrak{Z}}\delta_{\mathrm{alg},\mathfrak{Z}}$.
    	
    	Let $\Delta_{\delta_{\mathrm{alg}}}$ be the unramified character of $M_{p}$ associated to $\delta_{\mathrm{alg}}$ defined by remark \ref{RmK:LocallyAlgRep}(2). Note that $\Delta_{\delta_{\mathrm{alg}}}$ can be regarded as a point in $\mathrm{Spec}(\mathfrak{Z})^{\mathrm{rig}}$ via the quotient map
    	$$\hat{H}_p\surj \hat{H}_\sigma.$$
    	Then one can easily see that $\Delta_{\delta_{\mathrm{alg}}} = \delta_{\mathrm{alg,\mathfrak{Z}}}$. 
    	
    	Let $X$ be an affinoid subdomain in $\hat{H}_{p,0}$ such that $\delta_0\in X$. It suffices to show that $(\delta,\rho)$ is a classical point in $D_{\varsigma}(X,V)$. Let $k\geq \max\{k(X),k(V)\}$. Note that, by \cite[Prop~2.3.2~and~Cor~3.3.5]{Loeffler2011}, one has 
    	$$M(X,V,k)[{\delta_0}]\cong M(X,V,k)\otimes_{R_X,\delta_0} k(\delta_0)\cong M(1,V(\delta_0^{-1}),k).$$
    	 
        By \cite[Cor~3.7.3]{Loeffler2011}, one has $\varinjlim\limits_{k'\geq k} M(X,V,k')_{\mathfrak{m}^S}[{x}] = M(X,V,k)_{\mathfrak{m}^S}[{x}]$. Hence the space
        $$M(1,V(\delta_0^{-1}),k)_{\mathfrak{m}^S}[{x}]\cap \varinjlim_{k} M(1,V(\delta_0^{-1}),k)^{\mathrm{cl}}$$
        is non-empty if and only if the space
        $$\varinjlim_{k'\geq k} M(1,V(\delta_0^{-1}),k)_{\mathfrak{m}^S}^{\mathrm{cl}}[x]$$ 
        is non-empty. 
        
        Let $\Delta_W$ (resp. $\Delta_{\mathrm{alg}}$) be the unramified character of $M_{p}$ associated to $W$ (resp. $\delta_{\mathrm{alg}}$).
        It follows from \cite[Theo~3.9.2]{Loeffler2011} that 
        $\varinjlim_{k'\geq k}M(1,V(\delta_0),k)^{\mathrm{cl}}$ is isomorphic, as an $\left(\mathcal{H}_p\otimes_{L}\mathbb{T}^S[\frac{1}{p}]\right)$-module, to
        \begin{equation}\label{eq:DecompForClAutoForms}
        	\bigoplus_{\pi} m(\pi)\left[ (\pi_{f}^{p})^{U^p}\otimes \left(\pi_p\otimes \left(\mathrm{Ind}_{\bar{P}_{p,0}}^{G_{p,0}}(\sigma_0^\vee\otimes\delta_{\mathrm{sm},0}^{-1})\right)_{\mathrm{sm}}\otimes\Delta_W^{-1}\Delta_{\delta_{\mathrm{alg}}}\right) ^{G_{p,0}}\right]
        \end{equation}
        where the direct sum is over all automorphic representations $\pi$ of weight $W(\delta_{\mathrm{alg}})$ such that $(\pi_f^p)^{U^p}\neq 0$. Note that $G_{p,0}$ is decomposable in the sense of \cite[Def~2.2.1]{Loeffler2011}, i.e. $G_{p,0}\cong \overline{N}_{p,0}\times M_{p,0}\times N_{p,0}$. Then one has:
        \begin{align*}
        	&\left(\pi_p\otimes \left(\mathrm{Ind}_{\bar{P}_{p,0}}^{G_{p,0}}(\sigma_0^\vee\otimes\delta_{\mathrm{sm},0}^{-1})\right)_{\mathrm{sm}}\otimes\Delta_W^{-1}\Delta_{\delta_{\mathrm{alg}}}\right) ^{G_{p,0}}\\
        	\cong &\left(\left(\left(\pi_p\otimes \left(\mathrm{Ind}_{\bar{P}_{p,0}}^{G_{p,0}}(\sigma_0^\vee\otimes\delta_{\mathrm{sm},0}^{-1})\right)_{\mathrm{sm}}\right) ^{N_{p,0}}\right)^{M_{p,0}}\right)^{\overline{N}_{p,0}}\otimes\Delta_W^{-1}\Delta_{\delta_{\mathrm{alg}}}\\
        		\cong &\left(\left(\left( \mathcal{C}^{\mathrm{sm}}(N_{p,0},\pi_p\otimes\sigma_0^\vee\otimes \delta_{\mathrm{sm},0}^{-1})\right) ^{N_{p,0}}\right)^{M_{p,0}}\right)^{\overline{N}_{p,0}}\otimes\Delta_W^{-1}\Delta_{\delta_{\mathrm{alg}}}\\
        		\cong& \left(\left(\pi_p\otimes\sigma_0^\vee\otimes \delta_{\mathrm{sm},0}^{-1}\right)^{M_{p,0}}\right)^{\overline{N}_{p,0}}\otimes\Delta_W^{-1}\Delta_{\delta_{\mathrm{alg}}}\\
        		\cong& \left[\Hom_{M_{p,0}}(\sigma_0,\pi_p\otimes\delta_{\mathrm{sm},0}^{-1})\right]^{\overline{N}_{p,0}} \otimes\Delta_W^{-1}\Delta_{\delta_{\mathrm{alg}}}\\
        		\cong& \Hom_{M_{p,0}}(\sigma_0,\pi_p^{\overline{N}_{p,0}}\otimes\delta_{\mathrm{sm},0}^{-1})\otimes\Delta_W^{-1}\Delta_{\delta_{\mathrm{alg}}} \\
        		\cong& \Hom_{M_{p,0}}(\sigma_0,(\pi_p)_{\overline{N}_{p,0}}\otimes\delta_{\mathrm{sm},0}^{-1})\otimes\Delta_W^{-1}\Delta_{\delta_{\mathrm{alg}}}
        \end{align*}       
        (The second to last equality comes from the fact that $\overline{N}_{p,0}$ acts trivially on $\sigma_0$ and $\xi_x$, and the last one follows from lemma \ref{InvandCoinv}).
         
         After localizing at $\mathfrak{m}^S$, the space $M(1,V(\delta_0^{-1}),k)^{\mathrm{cl}}$ becomes an $R_{\overline{\rho},S}$-module. Then after taking the $\mathfrak{p}_{\rho}$-torsion ($\mathfrak{p}_{\rho}$ is the prime ideal in $R_{\overline{\rho},S}$ associated to $\rho$), it becomes the closed subspace which expressed as a direct sum of the same terms as in (\ref{eq:DecompForClAutoForms}), for $\pi$ runs through all automorphic representation $\pi$ of weight $W(\delta_\mathrm{alg}^{-1})$ such that $(\pi_p^f)^{U^p}\neq 0$, and the $\mathrm{Gal}(\overline{F}/F)$-representation $\rho$ is the Galois representation associated to $\pi$. It follows that the space 
        $$[\varinjlim_{k'\geq k}M(1,V(\delta_0^{-1}),k)^{\mathrm{cl}}]_{\mathfrak{m}^S}[{x}]$$
        is non-empty if and only if there exists a automorphic representation $\pi$ satisfies conditions (1), (2) and (3) such that the space
        \begin{align}\label{Isom:Clpts}
        	&\left(\Hom_{M_{p,0}}(\sigma_0,(\pi_{p})_{\overline{N}_{p,0}}\otimes\delta_{\mathrm{sm},0}^{-1}\otimes\Delta_W^{-1}\Delta_{\delta_{\mathrm{alg}}})\right)[\delta_{\mathfrak{Z}}]\nonumber\\
        	\cong & \left(\Hom_{M_{p,0}}(\sigma_0,J_{\overline{N}_{p}}(\pi_p)\otimes\delta_{\mathrm{sm},0}^{-1}\otimes\Delta_W^{-1}\Delta_{\delta_{\mathrm{alg}}})\right)[\delta_{\mathfrak{Z}}]\\
        	\label{Isom:Clpts2}\cong & \left(\Hom_{M_{p}}(c\mathrm{-Ind}_{M_{p,0}}^{M_p}\sigma_0,J_{\overline{N}_{p}}(\pi_p)\otimes\delta_{\mathrm{sm},0}^{-1}\otimes\Delta_W^{-1}\Delta_{\delta_{\mathrm{alg}}})\right)[\delta_{\mathfrak{Z}}] 
        \end{align}
        is non-empty, here $\delta_{\mathrm{sm},0}$ is regarded as a character of $H_{p}$ such that $$\delta_{\mathrm{sm},0}((\varpi^{k_{v,1}}_{v,1},\dots,\varpi^{k_{v,l}}_{v,l})_{v\in S_p})=1.$$
        Indeed, the isomorphism (\ref{Isom:Clpts}) follows from the same argument in proposition \ref{Prop:Finiteslopepart}, as the eigenvalue of $\varsigma$ is invertible under the character $\delta_{\mathfrak{Z}}$ (and note that, by remark \ref{Rmk:Emerton=classical}, the Jacquet functor $J_{\overline{N}_{p}}(-)$ here is the classical one). And note that the $\mathfrak{Z}$-action on the space (\ref{Isom:Clpts2}), induced by the construction in proposition \ref{Prop:ext to Bcenter}, is exactly the natural action on the first factor $c\mathrm{-Ind}_{M_{p,0}}^{M_p}\sigma_0$.  
         
        Note that $(c\mathrm{-Ind}_{M_{p,0}}^{M_p}\sigma_0)\otimes_{\mathfrak{Z}}\delta_{\mathfrak{Z}} \cong \sigma(\delta_{\mathfrak{Z}})$. It follows that the space above is non-empty if and only if there exist a $M_{p}$-equivariant map 
        $$\alpha: c\mathrm{-Ind}_{M_{p,0}}^{M_p}\sigma_0 \rightarrow J_{\overline{N}_{p}}(\pi_p)\otimes\delta_{\mathrm{sm},0}^{-1}\otimes\Delta_W^{-1}\Delta_{\delta_{\mathrm{alg}}}$$ 
        such that 
        \begin{align*}
        	\alpha\circ F(f) = \delta_{\mathfrak{Z}}(F)\cdot \alpha(f)
        \end{align*}
    for any $F\in\mathfrak{Z}$ (note that by our construction in proposition \ref{Prop:ext to Bcenter}, the Bernstein action is $\alpha\mapsto \alpha\circ F$).
    Such $\alpha$ exists if and only if it factors through the projection
    $$c\mathrm{-Ind}_{M_{p,0}}^{M_p}\sigma_0\surj (c\mathrm{-Ind}_{M_{p,0}}^{M_p}\sigma_0)\otimes_{\mathfrak{Z},\delta_{\mathfrak{Z}}}k(\delta_{\mathfrak{Z}})\cong \sigma(\delta_{\mathfrak{Z}}),$$
    i.e. if and only if there exists an $M_p$-equivariant map
    $$\bar{\alpha}: \sigma(\delta_{\mathfrak{Z}}) \rightarrow J_{\overline{N}_{p}}(\pi_p)\otimes\delta_{\mathrm{sm},0}^{-1}\otimes\Delta_W^{-1}\Delta_{\delta_{\mathrm{alg}}}$$
    
    Because $\Delta_{\delta_{\mathrm{alg}}} = \delta_{\mathrm{alg},\mathfrak{Z}}$ as a point in $\mathrm{Spec}(\mathfrak{Z})^{{\mathrm{rig}}}$, then $\delta_{\mathfrak{Z}}\cdot\Delta_{\delta_{\mathrm{alg}}}^{-1} = \delta_{\mathrm{sm},\mathfrak{Z}}$. Hence if we twisting $\bar{\alpha}$ with $\Delta_{\delta_{\mathrm{alg}}}^{-1}\delta_{\mathrm{sm,0}}\Delta_{W}$, the left hand side is 
    \begin{align*}
    	\sigma(\delta_{\mathfrak{Z}}\Delta_{\delta_{\mathrm{alg}}}^{-1}\delta_{\mathrm{sm,0}}\Delta_{W}) &\cong \sigma(\delta_{\mathrm{sm},\mathfrak{Z}}\delta_{\mathrm{sm,0}}\Delta_{W})\\
    	&\cong \sigma(\delta_{\mathrm{sm}}\Delta_{W})
    \end{align*}
    It follows that such $\alpha$ exists if and only if there exists an $M_p$-equivariant map
    $$\sigma(\delta_{\mathrm{sm}}\Delta_W) \rightarrow J_{\overline{P}_p}(\pi_p),$$
    which is injective as $\sigma$ is irreducible. This is exactly the condition (4), and hence we finish the proof.
    \end{proof}

    \subsection{Density of Classical Points and Applications}
    \label{Subsect:Density}
    In this part, we will prove the Zariski density of the classical points (actually, of the dominant, very regular, classical, non-critical points. See the definition below). Then we use this property to compare the eigenvarieties and paraboline varieties via the local-global compatibility of the Langlands correspondence, and show that the eigenvariety $D_{\varsigma}(V)$ (up to isomorphism) does not depend on the choice of $\varsigma\in\Sigma^{++}$.
    
    For this purpose, we need to clarify our setting in section \ref{Global} for the parameter spaces of paraboline varieties to make them compatible with our setting for eigenvarieties. 
    
    Through out this section, we always denote $s_i:= n_1+\cdots+n_l$. And fix an isomorphism: $\iota: \overline{\mathbb{Q}}_p \xrightarrow{\sim} \mathbb{C}$. We denote by $\mathrm{val}$ the normalized valuation on any $p$-adic local field, such that $\mathrm{val}(p)=1$. 
    
     For each $v\in S_p$, we fix a supercuspidal representation $\sigma_v$ of $M_v$, which is isomorphic to $\mathrm{GL}_{n_1}(F_{\tilde{v}})\times\cdots\times \mathrm{GL}_{n_l}(F_{\tilde{v}})$ as in the previous section. Hence $\sigma_v$ is of the form:
    $$\sigma_{v,1}\otimes\cdots \otimes \sigma_{v,l},$$
    where each $\sigma_{v,i}$ is a supercuspidal representation of $\mathrm{GL}_{n_i}(F_{\tilde{v}})$.
    
    We denote by $\varrho_{v,i} = \mathrm{rec}(\sigma_{v,i})$, the Weil-Deligne representation of $F_{\tilde{v}}$ attached to $\sigma_{v,i}$ via the local Langlands correspondence in \cite{harris2001geometry}. And denote
    by $\tau_{v,i}$ the type of $\varrho_{v,i}$.
    
    \begin{Rmk}
    	Recall that, by proposition \ref{Prop:WD-decomp}, for each $v\in S_p$, and $1\leq i\leq l$, there exists an unique subgroup $W_{\tau_{v,i}}$ of $ W_{F_{\tilde{v}}}$ of finite index with the same inertial group $I_{F_{\tilde{v}}}$, and some irreducible smooth representation $\tilde{\tau}_{v,i}$ of $W_{\tau_{v,i}}$ such that $\tilde{\tau}_{v,i}|_{I_{F_{\tilde{v}}}}$ is irreducible and  
    	$$\varrho_{v,i}\cong \mathrm{Ind}_{W_{\tau_{v,i}}}^{W_{F_{\tilde{v}}}}(\tilde{\tau}_{v,i})$$
    	
    	Then we say $\tau_{v,i} := \tilde{\tau}_{v,i}|_{I_{F_{\tilde{v}}}}$ is the type of $\varrho_{v,i}$.
    \end{Rmk}
     Following the notations in section \ref{Subsect:DecompofWDRep}, we write $K_{\tau_{v,i}}$ for the unique unramified extension of $F_{\tilde{v}}$ of degree $e_{\tau_{v,i}} =[W_{F_{\tilde{v}}}:W_{\tau_{v,i}}]$. We write $F_{\tau_{v,i}}$ for the image of $K_{\tau_{v,i}}$ in $F_{\tilde{v}}$ by the norm map, and write $\mathcal{T}_{\tau_{v,i}}$ for the moduli ($C$-rigid) space of continuous character of $F_{\tau_{v,i}}$, and write
    $$\mathcal{T}_{v} := \prod_{i=1}^{l}\mathcal{T}_{\tau_{v,i}}.$$
    
    For simplicity, if $\sigma_{v,i}$ and $\sigma_{v,j}$ are in the same Bernstein Block, we may always assume that $\sigma_{v,i}\cong \sigma_{v,j}$ and $\tilde{\tau}_{v,i}\cong \tilde{\tau}_{v,j}$.
    
    Recall that, for each $v\in S_p$, we fix an irreducible algebraic representation $M_v$ of lowest $\eta$-weight
    $$a^{(v)}_{\eta,1}\leq \cdots \leq a^{(v)}_{\eta,n}.$$
    
    Let $\mathcal{S}_v$ be the parameter space, defined in section \ref{Global}, associated to the Weil-Deligne data $(\tilde{\tau}_{v,i},\dots,\tilde{\tau}_{v,i})$, and filtration weights
     $\underline{k}^{(v)}:=(\underline{k}^{(v)}_1,\dots,\underline{k}^{(v)}_l)$ for
     $$\underline{k}^{(v)}_i:=(a^{(v)}_{\eta,s_{i-1}+1}+s_{i-1}< a^{(v)}_{\eta,s_{i_{i-1}+2}}+s_{i-1}+1< \cdots < a^{(v)}_{\eta,s_i} +s_i-1)_{\eta\in\Hom(F_{\tilde{v}},C)}.$$
     Then $\mathcal{S}_{v}\cong \mathcal{T}_v\times \mathrm{Flag}_v$ by the construction of $\mathcal{S}_v$, where $\mathrm{Flag}_v$ is the flag variety for the weight data $\underline{k}^{(v)}$.
     
    \begin{Rmk}
    	In section \ref{Global}, we require that the largest weight of the parameter space $\mathcal{S}_v$ should be 0. Here, for the convention to compare with eigenvarieties, we shift it by $a_{\eta,s_i}^{(v)}+s_i-1$ which actually still represents the same moduli space, i.e. the moduli space of $l$ quasi-deRham $(\varphi,\Gamma_{F_{\tilde{v}}})$-modules $(D_1,\dots,D_l)$, such that each $D_i$ is of the form $\mathrm{D}(M_i)(\delta_i)$, for continuous character $\delta_i$, and filtered $(\varphi,N,G_{F_{\tilde{v}}})$-module $M_i$ of type $\tau_{v,i}$, and weight $\underline{k}_i^{(v)}$.  
    	
    	Indeed, after replacing $\underline{k}_i^{(v)}$ by $\underline{k}_i^{(v)}-a_{i}^{(v)}$, we still get the same moduli space via the the map
    	$$(M_i,\delta_i) \mapsto (M_i(\delta_{a_{i}^{(v)}}),\delta_i \cdot \delta^{-1}_{a_{ i}^{(v)}}),$$
    	where $a_{i}^{(v)} := (a_{\eta,s_i}^{(v)})_{\eta} + s_i-1$.  
    \end{Rmk}

     Let $\overline{\rho}: \mathrm{Gal}(\overline{F}/F)\rightarrow \mathrm{GL}_n(k_C)$ be the absolutely irreducible  Galois representation of the previous section. For each $v\in S_p$, we denote by $\overline{\rho}_v$ the Galois representation of $\mathrm{Gal}(\overline{F}_{\tilde{v}}/F_{\tilde{v}})$ obtained from $\overline{\rho}$ by the restriction $\mathrm{Gal}(\overline{F}_{\tilde{v}}/F_{\tilde{v}}) \inj \mathrm{Gal}(\overline{F}/F)$. We denote by $j_{\overline{\rho}}$ the natural map:
    \begin{align*}
    	\mathfrak{X}_{\overline{\rho},S} &\rightarrow \prod_{v\in S_p}\mathfrak{X}_{\overline{\rho}_v}\\
    	\rho & \mapsto (\rho_v)_{v\in S_p},
    \end{align*}
    where $\rho_v$ is the restriction of $\rho$ in $\mathrm{Gal}(\bar{F}_{\tilde{v}}/F_{\tilde{v}})$.

    \begin{Lemma}
    	Let $\delta$ be an unramified character of $\mathrm{GL}_{n_i}(F_{\tilde{v}})$ over $C$. Then $$\sigma_{v,i} \cong \sigma_{v,i}(\delta)$$
    	if and only if $\delta$ restricted on $\{g\ |\ \mathrm{det}(g)\in {F_{\tau_{v,i}}^{\times}}\}$ is trivial. 
    \end{Lemma}

     \begin{proof}
     	It follows from the local Langlands correspondence that $\sigma_{v,i} \cong \sigma_{v,i}(\delta)$ if and only if $\varrho_{v,i} \cong \varrho_{v,i}\otimes\mathrm{rec}({\delta})$ . Then the assertion follows from lemma \ref{Lem:Twistingbychar}.
     \end{proof}
 
     \begin{Cor}
     	The local Artin map:
     	$$\mathrm{rec:} \mathrm{GL}_{1}(F_{\tilde{v}}) \rightarrow F_{\tilde{v}}^{\times}$$
     	induces an isomorphism
     	\begin{align*}
     		j_{v,i}: \hat{H}_{\sigma_{v,i}} & \xrightarrow{\sim} \mathcal{T}_{\tau_{v,i}}\\
     		\delta &\mapsto\delta\circ \mathrm{rec}^{-1}|_{F_{\tau_{v,i}}^{\times}}
     	\end{align*}
     \end{Cor}
 
     \begin{proof}
     	It follows from the lemma above that $\hat{H}_{\sigma_{v,i}}$ is the moduli space of continuous character of  $\{g\ |\ \mathrm{det}(g)\in {F_{\tau_{v,i}}^{\times}}\}$. Then one gets the conclusion.
     \end{proof}
 
    
    We denote by $\delta_{\mathcal{P}}$ the modulus character:
    $$(g_1,\dots,g_l) \mapsto |\mathrm{det}(g_1)|_{F_{\tilde{v}}}^{n-n_1}\otimes\cdots\otimes |\mathrm{det}(g_i)|_{F_{\tilde{v}}}^{n+n_i-2s_i} \otimes\cdots\otimes |\mathrm{det}(g_l)|_{F_{\tilde{v}}}^{n_l-n}$$ 
    for $M_v$.
    
     We denote by $j_{v,i}'$ the map
    \begin{align*}
    	j_{v}': \hat{H}_{\sigma_{v}} & \xrightarrow{\sim} \mathcal{T}_{\tau_{v}}\\
    	\delta  & \mapsto j_v(\Delta_{W}\cdot \delta_{\mathcal{P}}^{-\frac{1}{2}}\cdot |\mathrm{det}|^{\frac{1-n}{2}}\cdot \delta)
    \end{align*}
    where $j_v:= (j_{v,1},\dots,j_{v,l})$ 
    
    \begin{Def}
    	A locally algebraic character $\delta=(\delta_v)_{v\in S_p}\in \hat{H}_{\sigma}$ is called \textit{very regular} if for each $v\in S_p$, the character $j_v'(\delta_v)=(\delta_{v,1}'\dots,\delta_{v,l}')$ fails to have the following property:
    	
    	 there exist $1\leq i,j\leq l$ such that $\tau_{v,i}\cong \tau_{v,j}$, and for such $i,j$, the smooth part of $\delta_{v,i}'/\delta_{v,j}'$ is an unramified character such that $(\delta_{v,i}'/\delta_{v,j}')_{\mathrm{sm }}(\varpi_{v}^{e_{\tau_{v,i}}}) \in\{ 1, q_{v,i} \}$, where $q_{v,i}$ is the cardinality of the residue field of $K_{\tau_{v,i}}$. 
    \end{Def}
    \begin{Prop} \label{VCptProp}
    	Let $(\delta,\rho)$ be a dominant classical point in $D_{\varsigma}(V)$. Suppose that $\delta$ is very regular. And suppose that 
    	$$\delta_{\mathrm{alg}}\left((g_{v,1},\dots,g_{v,l})_{v\in S_p}\right) = \prod_{v\in S_p\atop  \eta:F_{\tilde{v}}\inj C} \left( \prod_{i=1}^{l}\eta(\mathrm{det}(g_{v,i}))^{b^{(v)}_{\eta,i}} \right),$$
    	for $g_{v,i} \in \mathrm{GL}_{n_i}(F_{\tilde{v}})$.
    	
    	 Then for each $v\in S_p$, the $\mathrm{Gal}(\overline{F}_{\tilde{v}}/F_{\tilde{v}})$-representation $\rho_v$ is deRham with the following properties:
    	\begin{enumerate}
    		\item For each $\eta\in\Hom(F_{\tilde{v}},C)$, the $\eta$-Hodge-Tate weight (see the definition in \cite[Def~2.3.4]{Brinon09cmisummer}) of $\rho_v$ is 
    		$$a^{(v)}_{\eta,1}-b^{(v)}_{\eta,1}<\dots< a^{(v)}_{\eta,j}-b^{(v)}_{\eta,i_j}+ (j-1) <\dots < a^{(v)}_{\eta,n}-b^{(v)}_{\eta,l}+(n-1),$$
    		where $i_j$ is the smallest integer such that $j\leq s_{i_j}$;
    		\item The Weil-Deligne representation $\mathrm{WD}(\rho_v)$ attached to $\rho_v $ is isomorphic to
    		$$\bigoplus_{1\leq i\leq l}\varrho_{v,i}\otimes (\delta_{v,i}')_{\mathrm{sm}},$$
    		where $(\delta_{v,1}'\dots,\delta_{v,l}')=j_v'(\delta_v)$.
    	\end{enumerate}
    	 
    \end{Prop}

    \begin{proof}
    	The assertion (1) and that $\rho_v$ is deRham  follows from \cite[Thm~1.2]{ThomasBarnetLamb2011AFO}. (note that, by our definition, the Hodge-Tate weights of a deRham representation is the filtration weights of the attached filtered $(\varphi,N,G_{L/K})$-module).
    	
    	For the assertion (2), let $\pi = \pi_\infty\otimes \pi_f^{v}\otimes \pi_v$ be the automorphic representation associated to $\rho$. We write $$\delta =(\delta_v)_v\in \prod_{v\in S_p}\hat{H}_{\sigma_v}$$ and $$\delta_v =(\delta_{v,1},\dots,\delta_{v,l})\in\prod_{i=1}^{l}\hat{H}_{\sigma_{v,i}}.$$ By the local Langlands correspondence (see \cite[4.2.2]{TorstenWedhorn2000TheLL}), one has $$\iota^{-1}(\mathrm{rec}(\pi_v)) \cong \bigoplus_{1\leq i\leq l} \varrho_{v,i} \otimes (\delta_{\mathcal{P}}^{-\frac{1}{2}}\cdot\Delta_{W}\cdot\delta_{v,i})_{\mathrm{sm}},$$
    	(it follows from the condition (4) of proposition \ref{CharOfDomClPt} that one has a surjective morphism
    	$$ \mathrm{Ind}_{P_v}^{G_v}(\sigma_v \otimes \Delta_W\otimes\delta_{v,\mathrm{sm}}) \surj \iota^{-1}(\pi_v)$$
    	of $G_v$ representation, note that the very regular condition means the components are not linked to each other in the sense of (2.2.8) of \textit{loc. cit.}, which implies $$\iota^{-1}(\pi_v) \cong \mathrm{Ind}_{P_v}^{G_v}(\sigma_v \otimes \Delta_W\otimes\delta_{v,\mathrm{sm}}).$$ Then one applies (4.2.2) of \textit{loc. cit.}).
    	
    	Then by \cite[Thm~A]{ThomasBarnetLamb2011LocalglobalCF}, the semi-simplification of the WD representation $\mathrm{WD}(\rho_v)$ associated to $\rho_v$ is 
    	$$\bigoplus_{1\leq i\leq l}\varrho_{v,i}\otimes (\delta_{v,i}')_{\mathrm{sm}}.$$
    	Actually, $\mathrm{WD}(\rho_v)$ must be Frobenius semi-simple with $N=0$. Indeed, one has $$\Hom(\varrho_{v,i}\otimes (\delta'_{v,i})_{\mathrm{sm}},\varrho_{v,j}\otimes (\delta'_{v,j})_{\mathrm{sm}}) = 0$$
    	as $(\delta'_{v,i})_{\mathrm{sm}}\neq (\delta'_{v,j})_{\mathrm{sm}}$ for any $\varrho_{v,i}\cong \varrho_{v,j}$. It follows that $\mathrm{WD}(\rho_v)$ is Frobenius semi-simple. And the condition that $((\delta'_{v,i})/(\delta'_{v,j}))_{\mathrm{sm}}$ fails to be unramified with $$((\delta'_{v,i})/(\delta'_{v,j}))_{\mathrm{sm}}(\varpi_v^{e_{\tau_{v,i}}}) = q_{v,i}$$
    	for any $\varrho_{v,i}\cong \varrho_{v,j}$ implies $N=0$.
    \end{proof}

    \begin{Rmk}
    	Note that if $N=0$ and $\varrho_{v,i}\otimes (\delta'_{v,i})_{\mathrm{sm}}\ncong\varrho_{v,j}\otimes (\delta'_{v,j})_{\mathrm{sm}}$for any $i\neq j$, then there exists a unique increasing filtration of the filtered $(\varphi,N,G_{F_{\tilde{v}}})$-module $D_v = \mathbf{D}_{\mathrm{pst}}(\rho_{v})$:
    	$$0=D_{v,0}\subset D_{v,1} \subset \cdots \subset D_{v,l}=D_v,$$
    	given by filtered $(\varphi,N,G_{F_{\tilde{v}}})$-submodules, such that the Weil-Deligne representation $\mathrm{WD}(D_{v,i}/D_{v,i-1})$ attached to $D_{v,i}/D_{v,i-1}$ is isomorphic to $\varrho_{v,i}\otimes (\delta'_{v,i})_{\mathrm{sm}}$ for each $i=1,\dots,l$. 
    \end{Rmk}

    \begin{Def}
    	\text{ }
    	\begin{enumerate}
    		\item Let $K$ be a $p$-adic local field with some finite Galois extension $L$, and let $C$ be a $p$-adic local field such that $[K:\mathbb{Q}_p]=\#\Hom(K,C)$. Let $M$ be a filtered $(\varphi,N,G_{L/K})$-module over $C$, with an increasing filtration
    		\begin{equation}\label{eq:refinement}
    			0= M_0 \subset M_1 \subset \cdots \subset M_r = M
    		\end{equation}
    		given by filtered $(\varphi,N,G_K)$-submodules. Suppose that the rank of $M$ (resp. $M_i$) is $d$ (resp. $d_i$) and $\eta$-filtration weights of $M$ are
    		$$k_{\eta,1}\leq \cdots \leq k_{\eta,d}$$
    		for each $\eta\in\Hom(K,C)$. Note that the filtration $\mathcal{F}^{\bullet}$ structure of $M_i\otimes_{L_0}L$ is define by $\mathcal{F}^{j}(M_i\otimes_{L_0}L):= (M_i\otimes_{L_0}L)\cap\mathcal{F}^{j}(M\otimes_{L_0}L)$. Hence the filtration weights of $M_i$ is a sub-sequence of the filtration weights of $M$ with $d_i$ elements.
    		
    		We say this increasing filtration (\ref{eq:refinement}) is \textit{non-critical} if the $\eta$-filtration weights of $M_i$ are the smallest $d_i$ integers of the $\eta$-filtration weights of $M$, i.e.
    		$$k_{\eta,1}\leq \cdots \leq k_{\eta,d_i}$$
    		for each $1\leq i\leq r$ and $\eta\in\Hom(K,C)$.
    		\item Let $x=(\delta,\rho)$ be a very regular dominant classical point in $D_{\varsigma}(V)$. Let $(D_{v,1}\subset \cdots\subset D_{v,l})$ be the filtration of $\mathbf{D}_{\mathrm{pst}}(\rho_v)$ as in the remark above. We say $x$ is \textit{non-critical} if the filtration $(D_{v,1}\subset \cdots\subset D_{v,l})$ is non-critical.
    	\end{enumerate}
    	   	
    \end{Def}

    \begin{Lemma}\label{Noncriticalcriterion}
    Let $K$ be a $p$-adic local field with finite Galois extension $L$. Let $M$ be a filtered $(\varphi,N,G_{L/K})$-module over $ C$ of rank $n$, here $C$ is an another $p$-adic local field such that $|\Hom(K,C)|=[K:\mathbb{Q}_p]$. Suppose that
    	\begin{enumerate}
    		\item $M$ is weakly admissible in the sense of \cite[Def~8.2.1]{Brinon09cmisummer}.
    		\item the associated WD representation $(\varrho,N)$ of $W_K$ is isomorphic to ($N=0$) $$\varrho_1\oplus\dots\oplus\varrho_l,$$ where each $\varrho_i$ is an irreducible WD representation of rank $n_i$.
    		\item Suppose that for each $\eta\in\Hom(K,C)$, the $\eta$-filtration weights of $M$ is 
    		$$k_{\eta,1}\leq \cdots\leq k_{\eta,n},$$
    		 and for any $1\leq i\leq l-1$ and $\eta\in\Hom(K,C)$, one has
    		$$(k_{\eta,s_{i+1}}-k_{\eta,s_i})+\left(\sum_{\eta: K\inj C}\sum_{j=1}^{s_i} k_{\eta,j}\right) > [K:K_0]\left(\sum_{j=1}^{i}\mathrm{val}(\mathrm{det}(\varrho_j(\mathrm{Fr_K^{-1}}))) \right),$$
    		here $\mathrm{Fr}_K^{-1}$ is some geometric Frobenius in $W_K$.
    	\end{enumerate}
    If we write $M_i$ for the unique subobject of $M$ as filtered $(\varphi,N,G_{L/K})$-module such that the associated WD representation is isomorphic to
    $$\varrho_1\oplus\dots\oplus\varrho_i$$ for each $0\leq i\leq l$,
    then the filtration
    $$0=M_0\subset M_1 \subset\cdots\subset M_l=M,$$
    is non-critical.  
    \end{Lemma}

    \begin{proof}
    	As $M$ is weakly admissible, then for each $1\leq i\leq l$, one has
    	$$t_H(M_i)\leq t_{N}(M_i),$$
    	here we denote by $t_H$ (resp. $t_N$) the Hodge number (resp. the Newton number).
    	
    	By definition, one can compute that the Newton number $t_N(M_i)$ of $M_i$ is  $$\frac{1}{[K_0:\mathbb{Q}_p]}\sum_{j=1}^{j}\mathrm{val}(\mathrm{det}(\varrho_j(\mathrm{Fr_K^{-1}})))$$
    	Suppose that the $\eta$-filtration weights of $M_i$ are
    	$$k_{\eta,\omega_i(\eta,1)} \leq \cdots \leq k_{\eta,\omega_i(\eta,l)},$$
    	where $\omega_i(\eta,-)$ is some order-preserving map $\{1,\dots,s_i\}\rightarrow\{1,\dots,n\}$. Then the Hodge number $t_{H}(M_i)$ of $M_i$ is
    	$$\frac{1}{[K:\mathbb{Q}_p]}\sum_{\eta: K\inj C}\sum_{j=1}^{s_i} k_{\eta,\omega_i(\eta,j)}$$ 
    	 For $m\leq n$, define the partial order on the set
    	$$\Omega(m,n):=\{\omega: \{1,\dots,m\}\rightarrow\{1,\dots,n\}\ |\ \omega(i)<\omega(i+1),\forall 1\leq i\leq n-1\}$$
    	by $\omega_1\leq \omega_2$ if and only if $\omega_1(i)\leq \omega_2(i)$ for all $1\leq i\leq l$. Then one can easily see the following fact
    	\begin{enumerate}
    		\item the identity map $\omega_{\id}$ is the unique minimal element in $\Omega(m,n)$;
    		\item if we write $\omega_0$ for the map, that $\omega_0(i) = i$ for $i<m$ and $\omega_0(m)=m+1$. Then $\omega_0$ is the unique minimal element in $\Omega(m,n)\setminus \{\omega_{\id}\}$;
    		\item $\omega_1\leq\omega_2$ implies that 
    		$$\sum_{j=1}^{m}k_{\eta,\omega_1(j)} \leq \sum_{j=1}^{m}k_{\eta,\omega_2(j)}.$$
    	\end{enumerate}
    Now we claim that $\omega_i(\eta,-) = \omega_{\id}$ in $\Omega(s_i,n)$ for each $1\leq i\leq l$ and $\eta\in\Hom(K,C)$. If it is true, then one can see that it is exactly the condition that the refinement
    $$M_1\subset \cdots\subset M_l=M$$
    is non-critical. 
    
    Otherwise, we may assume that $\omega_i(\eta',-)\neq \omega_{\id}$ for some $i$ and $\eta'$. By the observation above, one has
    \begin{align*}
    	[K:\mathbb{Q}_p]t_{H}(M_i) & = \sum_{\eta: K\inj C}\sum_{j=1}^{s_i} k_{\eta,\omega_i(\eta,j)}\\
    	& \geq \left(\sum_{j=1}^{s_i} k_{\eta',\omega_0(j)}\right) + \left( \sum_{\eta\neq \eta'}\sum_{j=1}^{s_i} k_{\eta,j} \right)\\
    	& \geq (k_{\eta',s_{i+1}}-k_{\eta',s_{i}}) + \left(\sum_{j=1}^{s_i} k_{\eta',j}\right) + \left( \sum_{\eta\neq \eta'}\sum_{j=1}^{s_i} k_{\eta,j} \right) \\
    	&= (k_{\eta',s_{i+1}}-k_{\eta',s_{i}}) + \left( \sum_{\eta:K\inj C}\sum_{j=1}^{s_i} k_{\eta,j} \right)\\
    	& > [K:K_0]\left(\sum_{j=1}^{i}\mathrm{val}(\mathrm{det}(\varrho_j(\mathrm{Fr_K^{-1}}))) \right)\\
    	& = [K:\mathbb{Q}_p] t_N(M_i),
    \end{align*}
    which contradicts to the condition that $M$ is weakly admissible.
    \end{proof}

    \begin{Prop} \label{VCNcriterion}
    	Let $x=(\delta,\rho)$ be a point in $D_{\varsigma}(V)$ such that $\delta=\delta_{\mathrm{sm}}\delta_{\mathrm{alg}}$ is locally algebraic. We write $\delta=(\delta_0,\delta_{\mathfrak{Z}})$ for the decomposition (depending on the choices of $\varpi_v$)
    	$$\hat{H}_{\sigma} \cong \hat{H}_{p,0}\times \mathrm{Spec}(\mathfrak{Z})^{\mathrm{rig}}.$$
    	We denote $B_{v,i}:=\mathrm{diag}(I_{n_1},\dots,I_{n_{i-1}},\varpi_vI_{n_{i}},I_{n_{i+1}},\dots,I_{n_{l}})$, and denote $$c_{v,i}:=\mathrm{val}((\Delta_{W}\cdot\delta_{\mathfrak{Z}})(B_{v,i}))$$ (note that $B_{v,i}\in \Sigma$ and recall that we have fixed an embedding $i_{\Sigma}:C[\Sigma]\inj \mathfrak{Z}$). And suppose that
    	$$\delta_{\mathrm{alg}}\left((g_{v,1},\dots,g_{v,l})_{v\in S_p}\right) = \prod_{v\in S_p\atop  \eta:F_{\tilde{v}}\inj C} \left( \prod_{i=1}^{l}\eta(\mathrm{det}(g_{v,i}))^{b^{(v)}_{\eta,i}} \right),$$
    	for $g_{v,i} \in \mathrm{GL}_{n_i}(F_{\tilde{v}})$. 
    	
    	We write $\delta=(\delta_v)_v\in\prod_{v\in S_p}\hat{H}_{\sigma_{v,i}}$, and write $ (\delta_{v,1}',\dots,\delta_{v,l}') := j'_v(\delta_v)$.
    	Then
    	\begin{enumerate}
    		\item if $$b_{\eta,i}^{(v)}-b_{\eta,j}^{(v)} > [F_{\tilde{v}}:\mathbb{Q}_p](1 + n+ |c_{v,i}-c_{v,j}|)$$ for each $v\in S_p$ and $i<j$, then $x$ is very regular;
    		\item if $$b_{\eta,i}^{(v)}-b_{\eta,i+1}^{(v)} > [F_{\tilde{v},0}:\mathbb{Q}_p]\mathrm{val}(\delta_{\mathfrak{Z}}(\varsigma))$$ for each $v\in S_p$ and $1\leq i\leq l$, then $x$ is classical;
    		\item if
    		$$b^{(v)}_{\eta,i}-b^{(v)}_{\eta,i+1} > \frac{[F_{\tilde{v}}:F_{\tilde{v},0}]}{e_{\tau_{v,i}}}\mathrm{val}\left((\delta'_{v,1}\dots \delta'_{v,i})(\varpi_v^{-e_{\tau_{v,i}}})\right) + C_i+C_i',$$
    		where $(\delta'_{v,1},\dots,\delta'_{v,l}) := j_v'(\delta)$, and 
    		$$C_i := [K:K_0]\sum_{1\leq j\leq i}\mathrm{val}(\mathrm{det}(\varrho_{v,j}(\mathrm{Fr}^{-1}_{F_{\tilde{v}}})))$$
    		and
    		$$C_i':= -\sum_{\eta\in \Hom(F_{\tilde{v}},C) \atop 1\leq j\leq s_i} (a^{(v)}_{\eta,j}+j-1),$$
    		then $x$ is non-critical .
    	\end{enumerate}
    \end{Prop}

    \begin{proof}
    	\textit{ }
    	\begin{enumerate}
    		\item We write $\delta_{v} =(\delta_{v,1},\dots,\delta_{v,l})$.
    		
    		Then one can compute that
    		$$(\delta_{v,i}')_{\mathrm{sm}}(\varpi_{v}^{n_i}) = \Delta_W(B_{v,i})|\varpi_v|_{F_{\tilde{v}}}^{\frac{n_i(2s_i-2n-n_i+1)}{2}}(\delta_{v,i})_{\mathrm{sm}}(B_{v,i}).$$
    		Note that
    		\begin{align*}
    			(\delta_{v,i})_{\mathrm{sm}}(B_{v,i}) = \delta_{\mathfrak{Z}}(B_{v,i})\left(\prod_{\eta\in{\Hom(F_{\tilde{v}},C)}}\eta(\varpi_v)^{n_ib^{(v)}_{\eta,i}}\right)^{-1},
    		\end{align*}
    	Then for any $1\leq i<j\leq l$ such that $\tau_i\cong \tau_j$ (which implies $n_i=n_j$), one has
    	\begin{align*}
    		&\frac{1}{e_{\tau_{v,i}}}\mathrm{val}\left( (\delta'_{v,j}/ \delta'_{v,i})_{\mathrm{sm}}(\varpi_{v,i}^{e_{\tau_{v}}})\right)\\ 
    		=& \frac{1}{n_i}\mathrm{val}(\Delta_W(B_{j,v}/B_{i,v}))+(s_i-s_j)[F_{\tilde{v},0}:\mathbb{Q}_p]\\
    		+& \frac{1}{n_i}\left[\mathrm{val}((\delta_{v,j})_{\mathrm{sm}}(B_{j,v}))-\mathrm{val}((\delta_{v,i})_{\mathrm{sm}}(B_{i,v}))\right]\\
    		=& \frac{1}{n_i}\mathrm{val}((\Delta_W\cdot\delta_\mathfrak{Z})(B_{j,v}/B_{i,v}))+(s_i-s_j)[F_{\tilde{v},0}:\mathbb{Q}_p]\\
    		+&\frac{1}{[F_{\tilde{v}}:F_{\tilde{v},0}]}\sum_{\eta\in\Hom(F_{\tilde{v}},C)}(b_{\eta,i}^{(v)}-b_{\eta,j}^{(v)})\\
    		>& \frac{1}{n_i}\mathrm{val}((\Delta_W\cdot\delta_\mathfrak{Z})(B_{j,v}/B_{i,v})) + \frac{1}{[F_{\tilde{v}}:F_{\tilde{v},0}]}\sum_{\eta}(b_{\eta,i}^{(v)}-b_{\eta,j}^{(v)})-n[F_{\tilde{v},0}:\mathbb{Q}_p]\\
    		>& [F_{\tilde{v},0}:\mathbb{Q}_p]
    	\end{align*}
        Hence $(\delta,\rho)$ is very regular.
        \item 
        Suppose that $(\varsigma_v)_{v\in S_p} = \varsigma\in \Sigma^+$ for $$\varsigma_v = \mathrm{diag}(\varpi_{v,1}^{t_{v,1}}I_{n_1},\dots,\varpi_{v,1}^{t_{v,l}}I_{n_l})$$
        with $t_{v,1}< \cdots < t_{v,l}$.
        
        Applying [Prop~2.6.3'] in the corrigendum of \cite{Loeffler2011} in the special case of $\mathrm{GL}_n$, one can compute the inequality ($\dagger$) directly in \textit{loc. cit. }and see that $\mathrm{val}(\delta_{\mathfrak{Z}}(\varsigma))$ is a small slope for $\varsigma$ at point $(\delta,\rho)$ if 
        $$ \frac{t_{v,i+1}-t_{v,i}}{[F_{\tilde{v},0}:\mathbb{Q}_p]}\left((b^{(v)}_{\eta,i}-b^{(v)}_{\eta,i+1})-(a^{(v)}_{\eta,s_i}-a^{(v)}_{\eta,s_i+1}-1)\right) > \mathrm{val}(\delta_{\mathfrak{Z}}(\varsigma)),$$
        which follows from our condition (2). Then $(\delta,\rho)$ is classical by \cite[Thm~3.9.6]{Loeffler2011} of \textit{loc. cit.}.
        \item It follows from lemma \ref{Noncriticalcriterion}, using the equation:
        \begin{align*}
        	& [F_{\tilde{v}}:F_{\tilde{v},0}]\mathrm{val}\left(\mathrm{det}\left(((\delta'_{v,i})_{\mathrm{sm}}\otimes \varrho_{v,j})(\mathrm{Fr}_{F^{-1}_{\tilde{v}}})\right)\right) - \sum_{\eta\in \Hom(F_{\tilde{v}},C) \atop s_{i-1}+1\leq j\leq s_i} (a^{(v)}_{\eta,j}- b^{(v)}_{\eta,j} +j-1) \\
        	&= (C_i-C_{i-1}) +{n_i[F_{\tilde{v}}:F_{\tilde{v},0}]}\left[\mathrm{val}((\delta'_{v,i})_{\mathrm{sm}}(\varpi_{v})) + \mathrm{val}((\delta'_{v,i})_{\mathrm{alg}}(\varpi_{v})) \right]\\
        	&- \sum_{\eta\in \Hom(F_{\tilde{v}},C) \atop s_{i-1}+1\leq j\leq s_i} (a^{(v)}_{\eta,j}+j-1) \\
        	& = (C_i+C_i'-C_{i-1}-C'_{i-1}) + {n_i[F_{\tilde{v}}:F_{\tilde{v},0}]}\left[\mathrm{val}((\delta'_{v,i})(\varpi_{v}))\right]
        \end{align*}
    	\end{enumerate}
    \end{proof}
  
    \begin{Theo}\label{VCNareZarDense}
    	  The set of dominant, very regular, classical, non-critical points in $D_{\varsigma}(V)$  is Zariski dense.   	
    \end{Theo}

    \begin{proof} A point $x=(\rho,\delta)\in D_{\varsigma}(V)$ is called \textit{locally algebraic} if $\delta$ is a locally algebraic (i.e. deRham) character.
    	Consider the following diagram:
    	\begin{equation*}
    		\begin{tikzcd}
    			D_{\varsigma}(V) \arrow[r, "\mu"] \arrow[rd, "f_D"'] & Z_{\varsigma}(V) \arrow[d, "f_Z"] \\
    			& {\hat{H}_{p,0}}             
    		\end{tikzcd}
    	\end{equation*}
        where $\mu$ is the natural finite morphism from eigenvariety to spectral variety, and $f_D,f_Z$ are the projections to the weight space. We write $D_{\mathrm{VCN}}$ (resp. $D_{\mathrm{la}}$) for the set of very regular, classical, non-critical points (resp. locally algebraic points) in $D_{\varsigma}(V)$.
        
        Firstly, we claim that $D_{\mathrm{la}}$ is Zariski dense in $D_{\varsigma}(V)$. Indeed, it is obvious that the set of locally algebraic characters are Zariski dense in $\hat{H}_{p,0}$. Then by the proof of \cite[Cor~6.4.4]{GaetanChenevier2004FamillesPD}, the image of each irreducible component of $Z_\varsigma(V)$ is Zariski open in $\hat{H}_{p,0}$. Hence $\mu(D_{\mathrm{la}})$ is Zariski dense in $Z_\varsigma(V)$. 
        It follows from [Lem~6.2.10] of \textit{loc. cit.} that the finite map $\mu:D_{\varsigma}(V)\rightarrow Z_{\varsigma}(V)$ has the property that each irreducible component of $D_{\varsigma}(V)$ maps surjectively to an irreducible component of $Z_{\varsigma}(V)$. Hence by [Lem~6.2.8] of \textit{loc. cit.}, one has $D_{\mathrm{la}}$ is Zariski dense in $D_{\varsigma}(V)$.
        
        Then, by the construction, we can choose an affinoid admissible cover $\mathcal{C}$ of $Z_{\varsigma}(V)$, such that $\{f_Z(X) | X\in \mathcal{C}\}$ (resp. $\{\mu^{-1}(X)| X\in \mathcal{C}\}$) is an affinoid admissible cover of $\hat{H}_{p,0}$ (resp. ($D_{\varsigma}(V)$)). Hence we can restrict the diagram above on each $X\in \mathcal{C}$ and it is enough to show that $V_{\mathrm{VCN}}\cap \mu^{-1}(X)$ is Zariski dense in each $\mu^{-1}(X)$ if $D_{\mathrm{la}}\cap \mu^{-1}(X)\neq \emptyset$.
        
        As $\mu^{-1}(X)$ is quasi-compact, the function
        \begin{align*}
        	f_{\varsigma}: D_{\varsigma}(V) &\rightarrow \mathbb{Q}\\
        	x=(\delta_0,\delta_{\mathfrak{Z}},\rho)&\mapsto \mathrm{val}(\delta_{\mathfrak{Z}}(\varsigma)),
        \end{align*}
        
         and, for $v\in \Hom(F_{\tilde{v}},C)$ and $1\leq i\leq l$, the function
        \begin{align*}
        	f_{v,i}: D_{\varsigma}(V) &\rightarrow \mathbb{Q}\\
        	x=(\delta_0,\delta_{\mathfrak{Z}},\rho)&\mapsto \mathrm{val}(\delta_{\mathfrak{Z}}(B_{v,i})),
        \end{align*}
   where $B_i$ is the block diagonal matrix of size $(n_1,\dots,n_l)$ such that the $i$-th component is $\varpi_vI_{n,i}$ and other components are identity, are bounded above. Hence it follows from proposition \ref{VCNcriterion} that there exists a constant $c>0$, such that for any locally algebraic point in $\mu^{-1}(X)$ with $(v,\eta)$-weight $(b^{(v)}_{\eta,1},\dots,b^{(v)}_{\eta,l})$, for $v\in S_p$ and $\eta\in\Hom(F_{\tilde{v}},C)$, is dominant, very regular, classical and non-critical if $b^{(v)}_{\eta,i}-b^{(v)}_{\eta,i+1} >c$ for each $i,v$ and $\varsigma$.
   
   It is obviously that the subset $\hat{H}_{>c}\subset \hat{H}_{p,0}$ of locally algebraic characters satisfying the property above is Zariski dense in any affinoid subdomain $X$ of $\hat{H}_{p,0}$ if $X$ contains at least one locally algebraic character. Let $X$ be an element in $\mathcal{C}$ such that $\mu^{-1}(X)\cap D_{\mathrm{la}} \neq \emptyset$. Then $f_{Z}(X)$ contains at least one locally algebraic character, hence $\hat{H}_{>c}\cap f_Z(X)$ is Zariski dense in $f_Z(X)$. As the map $X\rightarrow f_{Z}(X)$ is flat, one has $f_z^{-1}(\hat{H}_{>c})\cap X$ is Zariski dense in $X$. By \cite[Lem~6.2.8]{GaetanChenevier2004FamillesPD} again, one has $f_D^{-1}(\hat{H}_{>c})\cap \mu^{-1}(X)$ is Zariski dense in $\mu^{-1}(X)$. It follows that
    $D_{\mathrm{VCN}}$ is Zariski dense in $D_{\varsigma}(V)$ as $\mu^{-1}(X)\cap f_{D}^{-1}(\hat{H}_{>c})\subseteq D_{\mathrm{VCN}}$.
    \end{proof}
 
    \begin{Cor}
    	The eigenvariety $D_\varsigma(V)$ can be describe as the reduced closed rigid subspace of $\hat{H}_{\sigma}\times \mathfrak{X}_{\overline{\rho},S}$, whose underlying topological space is the Zariski closure of set of classical points. In particular, the eigenvariety $D_\varsigma(V)$ does not depend on the choice of $\varsigma$. Hence we can denote by $D(V)$ for short.
    \end{Cor}
    
    \begin{proof}
    	It follows from the property above immediately.
    \end{proof}

    \begin{Rmk}
    	The eigenvariety $D(V)$ is isomorphic to the Bernstein eigenvariety $\mathcal{E}_{\Omega,\lambda}(U^{p})$ defined in \cite{ChristopheBreuil2021BernsteinE} (recall remark \ref{Rmk:ForBergersFunc1} for the definition of $\Omega$, and the dominant weight $\lambda$ corresponds to the algebraic representation $W$), even they are constructed in different methods. In fact, by the same argument as in the proof of \cite[Prop. 3.4]{Breuil2017a}, the strictly dominant classical points of $\mathcal{E}_{\Omega,\lambda}(U^{p})$ has the same characterization in proposition \ref{CharOfDomClPt} of those in $D(V)$. By theorem \ref{VCNareZarDense} and \cite[Thm. 3.2.11]{ChristopheBreuil2021BernsteinE}, the classical points are dense in both $D(V)$ and $\mathcal{E}_{\Omega,\lambda}(U^{p})$. Hence they are isomorphic.
    \end{Rmk}
    
    \begin{Theo} \label{ComparisonTheorem}
    	For each dominant, very regular, classical, non-critical point $(\rho,\delta)\in D(V)$, the $(\varphi,\Gamma_{F_{\tilde{v}}})$-module $D_{\mathrm{rig}}^\dagger(\rho_v)$ is strictly paraboline with parameters $j_{v}' (\delta_v)$ for each $v\in S_p$.
    	
    	 In particular, if we write $j':=\prod_{v\in S_p} j'_v$, then the composition 
    	$$D(V)\hookrightarrow \hat{H}_{\sigma}\times \mathfrak{X}_{\bar{\rho},S}  \xrightarrow{j'\times j_{\bar{\rho}}} \prod_{v\in S_p} ( \mathcal{T}_v\times \mathfrak{X}_{\bar{\rho}_{v}}) $$
    	factors through the product of the paraboline varieties $$\prod_{v\in S_p}Y_{\mathrm{par}}(\bar{\rho}_{v}) \hookrightarrow \prod_{v\in S_p} (\mathcal{T}_v\times \mathfrak{X}_{\bar{\rho}_{v}}).$$
    \end{Theo} 

    \begin{proof}
    	Let $x=(\delta,\rho)\in D(V)$ be a dominant, very regular, classical, non-critical point. And suppose that 
    	$$\delta_{\mathrm{alg}}\left((g_{v,1},\dots,g_{v,l})_{v\in S_p}\right) = \prod_{v\in S_p\atop  \eta:F_{\tilde{v}}\inj C} \left( \prod_{i=1}^{l}\eta(\mathrm{det}(g_{v,i}))^{b^{(v)}_{\eta,i}} \right),$$
    	for $g_{v,i} \in \mathrm{GL}_{n_i}(F_{\tilde{v}})$. We write $ (\delta_{v,1}',\dots,\delta_{v,l}') := j'_v(\delta_v)$.
    	
    	We write $M_v$ for the filtered $(\varphi,N,G_{F_{\tilde{v}}})$-module associated to $\rho_v$. It follows from proposition \ref{VCptProp} that the WD representation $\mathrm{WD}(M_v)$ attached to $M_v$ is isomorphic to
    	$$ \bigoplus_{1\leq i\leq l} \varrho_{v,i}\otimes (\delta'_{v,i})_{\mathrm{sm}},$$
    	 and the $\eta$-Hodge-Tate weights of $M_v$ are
    		$$a^{(v)}_{\eta,1}-b^{(v)}_{\eta,1}<\dots< a^{(v)}_{\eta,j}-b^{(v)}_{\eta,i_j} + (j-1) <\dots < a^{(v)}_{\eta,n}-b^{(v)}_{\eta,l}+(n-1),$$
    	where $i_j$ is the smallest integer such that $j\leq s_{i_j}$. 
    	
    	For each $1\leq i\leq l$, we write $M_{v,i}$ for the unique sub object of $M_v$ (as a filtered $(\varphi,N,G_{F_{\tilde{v}}})$-module) such that the WD representation $\mathrm{WD}(M_{v,i}/M_{v,i-1})$ attached to $M_{v,i}/M_{v,i-1}$ is isomorphic to 
    	$ \varrho_{v,j}\otimes (\delta'_{v,j})_{\mathrm{sm}}$ (set $M_{v,0}=0$).
    	
    	As $x$ is non-critical, the $\eta$-Hodge-Tate weights of $M_{v,i}/M_{v,i-1}$ are
    	$$a^{(v)}_{\eta,s_{i-1}+1}-b^{(v)}_{\eta,i}+(s_{i-1})<\dots< a^{(v)}_{\eta,j}-b^{(v)}_{\eta,i} + (j-1) <\dots < a^{(v)}_{\eta,s_i}-b^{(v)}_{\eta,i}+(i-1).$$
    	It follows that (note that $ (\delta'_{v,i})_{\mathrm{alg}} = (\delta_{v,i})_{\mathrm{alg}}$)
    	$$M_{v,i}/M_{v,i-1}\cong M_{v,i}'(\delta'_{v,i})$$
    	for some filtered $(\varphi,N,G_{F_{\tilde{v}}})$-module $M_{v,i}'$ such $\mathrm{WD}(M_{v,i}')\cong \varrho_{v,i}$ and $\eta$-Hodge-Tate weights of $M_{v,i}'$ are
    	$$a^{(v)}_{\eta,s_{i-1}+1}+(s_{i-1})<\dots< a^{(v)}_{\eta,j} + (j-1) <\dots < a^{(v)}_{\eta,s_i}+(s_i-1),$$
    	which exactly means that $(M_{v,i}/M_{v,i-1})_{1\leq i\leq l}$ is a point in $\mathcal{S}_v$ and its image in $\mathcal{T}_v$ under the natural projection $\mathcal{S}_v\rightarrow \mathcal{T}_v$ is $j'_v(\delta_v)$ by construction. 
    	
    	Note that $j'_v(\delta_v)$ is in $\mathcal{T}_{v,\mathrm{reg}}$ as being very regular implies being regular. Hence $D_{\mathrm{rig}}^{\dagger}(\rho_v)$ is strictly paraboline with parameter $j'_v(\delta_v)$. In particular, $(j'\times j_{\bar{\rho}})(\delta,\rho)$ is a point in $\prod_{v\in S_p}Y_{\mathrm{par}}(\bar{\rho}_v)$.
    	
    	It follows that the map $j'\times j_{\bar{\rho}}$ restricted on $D(V)$ factors through $\prod_v Y_{\mathrm{par}}(\bar{\rho}_v)$ as the set of dominant, very regular, classical, non-critical points is Zariski dense in $D(V)$ by theorem \ref{VCNareZarDense}.
    \end{proof}
    
    \appendix
    \section{Some Computation for Sen Polynomials}
    For convenience, let $A$ denote a connected affinoid $C$-algebra through out this section. One can easily generalize our assertions in this section to the case of general rigid $C$-spaces.
    
    Let $\pi= \mathrm{exp}(t)-1$ in $\mathcal{R}_K$ (resp. $\mathcal{R}^r_K$). Write $q_n:=\frac{\varphi^n(\pi)}{\varphi^{n-1}(\pi)}$, and write $n(r)$ the minimal integer such that $p^{n(r)-1}(p-1)\geq r$. Then $(q_n)$ is a maximal ideal of $\mathcal{R}^r_K$ with residue field $K_n = K(\xi_{p^n})$ for $n\geq n(r)$.
    
    For a rank $d$ $(\varphi,\Gamma_K)$-module $D^r$ over $\mathcal{R}^r_{K,A}$, by \cite[Lemma~3.2.3]{Kedlaya2012}, one has
    $$D^r/tD^r\cong\prod_{n\geq n(r)} D^r/q_nD^r,$$
    where each $D^r/q_nD^r$ is a locally free $(K_n\otimes_{\mathbb{Q}_p} A)$-module of rank $d$ and $\varphi^{n'-n}$ induces an isomorphism $$D^r/q_nD^r\otimes_{K_n} K_{n'}\cong D^r/q_{n'}D^r$$ of $A[\Gamma_K]$-modules for $n'\geq n\geq n(r)$. 
    
    Hence if $D$ is a $(\varphi,\Gamma_K)$-module over $\mathcal{R}_{K,A}$ such that $D \cong D^r\otimes_{\mathcal{R}_{K,A}^r} \mathcal{R}_{K,A}$ for some $(\varphi,\Gamma_K)$-module $D^r$ over $\mathcal{R}_{K,A}^r$, then 
    $$(D/tD)^{\varphi=1}\cong D^r/q_nD^r\otimes_{K_n} K_{\infty}$$
    for $n\geq n(r)$, which is a locally free $(K_\infty\otimes_{\mathbb{Q}_p} A)$-modules of dimension $d$ with continuous semi-linear $\Gamma_K$-action.
    \begin{Def} \cite[Definition~6.2.11]{Kedlaya2012}
    	Keep our the notations be as above. Consider the \textit{Sen operator} $\Theta_{\mathrm{Sen}} = \frac{\log(\gamma)}{\log(\chi(\gamma))}$ for $\gamma\in\Gamma_K$ close to 1 enough. We may assume $\gamma\in \mathrm{Gal}(K_\infty/K_n)$ hence $\Theta_{\mathrm{Sen}}$ induces a $K_n\otimes_{\mathbb{Q}_p} A$-linear map on $D^r/q_nD^r$. 
    	
    	We define the \textit{Sen polynomial} $\mathrm{Sen}_{D}$ of $D$ as the characteristic polynomial of this linear map, which is monic with coefficient in $K_n\otimes_{\mathbb{Q}_p}A$.
    \end{Def}
    
    \begin{Rmk}\label{rmkforsen} \textit{ }
    	\begin{enumerate}
    		\item As the Sen operator $\Theta_{\mathrm{Sen}}$ commutes with $\Gamma_K$-action, hence the Sen polynomial $\mathrm{Sen}_D$ is $\Gamma_K$-invariant, then with coefficients in $K\otimes_{\mathbb{Q}_p} A$.
    		\item One can show that $\mathrm{Sen}_D$ does not depends of the choice $r$ and $n$.
    		\item Under the isomorphism $$K\otimes_{\mathbb{Q}_p}A[T] \xrightarrow{\sim}\prod_{\sigma:K\inj C}A[T]  $$
    		we can decompose $\mathrm{Sen}_{D}(T) = (\mathrm{Sen}_{D,\sigma}(T))_\sigma$, where $\mathrm{Sen}_{D,\sigma}$ is the characteristic polynomial of $\Theta_{\mathrm{Sen}}$ for $(D/t_\sigma D)^{\varphi=1}$. Usually, we often write a Sen polynomial of the form $(\mathrm{Sen}_{D,\sigma}(T))_\sigma$.
    		\item When $D$ is a $(\varphi,\Gamma_K)$-module of rank 1, then $D$ corresponds to a continuous character $\delta: K^\times \rightarrow A^\times$ by \cite[Theorem~6.2.14]{Kedlaya2012}. Lemma 6.2.12 of \textit{loc. cit.} shows that  the Sen operator acts by multiplication by an element $\mathrm{wt}(\delta)$ in $K\otimes_{\mathbb{Q}_p} A$, and the Sen polynomial is $T-\mathrm{wt}(\delta)$.
    	\end{enumerate}
    
    \end{Rmk}
   
    \begin{Prop}\label{Prop:Multipropforsen}
    Let 
    $$0\rightarrow D_1 \rightarrow D_2 \rightarrow D_3 \rightarrow 0$$
    be a short exact sequence of $(\varphi,\Gamma_K)$-module over $A$, then
    $$\mathrm{Sen}_{D_2} = \mathrm{Sen}_{D_1} \cdot\mathrm{Sen}_{D_3}$$
    \end{Prop}
    
    \begin{proof}
    	There exists some $r>0$ such that $D_i\cong D^r_i\otimes_{\mathcal{R}^r_{K,A}} {\mathcal{R}_{K,A}}$ for some $(\varphi,\Gamma_K)$-module over ${\mathcal{R}^r_{K,A}}$ for $i=1,2,3$, and the short exact sequence is induced by some short exact sequence
    	$$0\rightarrow D_1^r \rightarrow D_2^r \rightarrow D_3^r \rightarrow 0$$
    	As the short exact sequence above is locally split as ${\mathcal{R}^r_{K,A}}$-module, hence the complex
    	$$0\rightarrow D_1^r/q_nD_1^r \rightarrow D_2^r/q_nD_3^r \rightarrow D_3^r/q_nD_3^r \rightarrow 0$$
    	is exact. It follows that
    	  $$\mathrm{Sen}_{D_2} = \mathrm{Sen}_{D_1} \cdot\mathrm{Sen}_{D_3}.$$
    \end{proof}

    \begin{Lemma} \label{GaldescentforSenpoly}
    	Let $A$ be a affinoid $C$-algebra. Let $D$ be a $(\varphi,\Gamma_K)$-module over $A$, and denote $D_L:=D\otimes_{\mathcal{R}_{K,A}} \mathcal{R}_{L,A}$, then 
    	$$\mathrm{Sen}_{D_L}(T)=\mathrm{Sen}_{D}(T)$$
    \end{Lemma}

    \begin{proof}
    	By direct computation, one can show that 
    	$$D_L^r/q_n D_L^r \cong (D^r/q_nD^r)\otimes_{K_n}L_n,$$
    	for some $r>0$.
    	Note that $D^r/q_nD^r$ is stable under the action of $\Theta_{\mathrm{Sen}}$. Hence the characteristic polynomial of $\Theta_{\mathrm{Sen}}$ on $D_L^r/q_n D_L^r$ is equal to the the characteristic polynomial of $\Theta_{\mathrm{Sen}}$ on $D^r/q_n D^r$, i.e.
    	$$\mathrm{Sen}_{D_L}(T)=\mathrm{Sen}_{D}(T).$$
    \end{proof}
    
    \begin{Prop}
    	Let $M$ be a filtered $(\varphi,N,G_{L/K})$-module over $L\otimes_{\mathbb{Q}_p}A$ of rank $d$ of weight $(\underline{k_\sigma})_{\sigma:K\inj C}$, where $\underline{k_\sigma}:=(k_{\sigma,1}\leq \dots\leq k_{\sigma,d})$. Then
    	$$\mathrm{Sen}_{\mathbf{D}(M)}(T)=(\prod_{1\leq i\leq d}(T+k_{\sigma,i}))_\sigma.    $$
    \end{Prop}

    \begin{proof}
    	For $n\geq n(r)$, by the construction of $\iota_n: \mathcal{R}_{L}^r\hookrightarrow L_n\llbracket t_n\rrbracket$ ($\varphi^{-n}$ in \textit{loc. cit.} \cite{Cherbonnier1999} for example), one can check $\iota_n(q_n)$ is a uniformizer in $L_n\llbracket t_n\rrbracket$, hence the reduction map $$\bar{\iota}_n:\mathcal{R}_{L}^r/q_n\mathcal{R}_{L}^r\rightarrow L_n\llbracket t_n\rrbracket \surj L_n$$ is an isomorphism.
    	 
    	 Fix an embedding $K\inj L$, we have the surjection
    	 \begin{align*}
    	 	\Hom(L,C)&\rightarrow \Hom(K,C)\\
    	 	\tilde{\sigma} &\mapsto (\sigma:K\inj L\xrightarrow{\tilde{\sigma}} C)
    	 \end{align*} 
    	 Recall that one has the canonical isomorphism
    	$$\mathbf{D}^r_L(M) \otimes_{\mathcal{R}^r_{L,A},\iota_n} L_n\llbracket t_n\rrbracket \xrightarrow{\sim} \mathbf{D}^n(M),$$
    	where $\mathbf{D}^n(M)$ is of the form  $\mathrm{Fil}^0(L_n\llbracket t_n\rrbracket\otimes_{(L\otimes_{\mathbb{Q}_p} A)} M')$, for some  locally free filtered $(L\otimes_{\mathbb{Q}_p} A)$-module $M'$, with $\tilde{\sigma}$-weight $\underline{k_\sigma}$. Locally on $\mathrm{Spec}A$, let $(e_{\tilde{\sigma},1},\dots,e_{\tilde{\sigma},d})$ be a basis of $M'_{\tilde{\sigma}}:= M'\otimes_{(L\otimes_{\mathbb{Q}_p} A), \tilde{\sigma}\otimes\id} A$, such that $(e_{\tilde{\sigma},1},\dots,e_{\tilde{\sigma},i_j})$ is a basis of $\mathrm{Fil}^jM'$, where $i_j$ is the rank of $\mathrm{Fil}^jM'$. Hence one can see that 
    	$$(t_n^{-k_{\sigma,1}}e_{\tilde{\sigma},1},\dots,t_n^{-k_{\sigma,d}}e_{\tilde{\sigma},d})$$
    	is a basis of $\mathbf{D}^n(M)$, and $\Theta_{\mathrm{Sen}}(t_n^{-k_{\sigma,i}}e_{\tilde{\sigma},i}) = -k_{\sigma,i} t_n^{-k_{\sigma,i}}e_{\tilde{\sigma},i}$. It follows that characteristic polynomial of $\Theta_{\mathrm{Sen}}$ on $\mathbf{D}^r_L(M)/q_n \mathbf{D}^r_L(M) \cong \mathbf{D}^n(M)\otimes_{L_n\llbracket t_n\rrbracket}L_n$ is
    	$$(\prod_{1\leq i\leq d}(T+k_{\sigma,i}))_{\tilde{\sigma}:L\inj C} = (\prod_{1\leq i\leq d}(T+k_{\sigma,i}))_{\sigma:K\inj C}.$$
    	(Note that if an element $a \in\prod_{\tilde{\sigma}:L\inj C}A$ is of the form $(a_\sigma)_{\tilde{\sigma}}$, then $a\in K\otimes_C A$ and $a = (a_\sigma)_\sigma \in\prod_{\sigma:K\inj C}A$).
    	
    	As $\mathbf{D}(M)\otimes_{\mathcal{R}_{K,A}} \mathcal{R}_{L,A} \cong \mathbf{D}_L(M)$, by lemma \ref{GaldescentforSenpoly}, one has
    	$$\mathrm{Sen}_{\mathbf{D}(M)}(T) = (\prod_{1\leq i\leq d}(T+k_{\sigma,i}))_{\sigma:K\inj C}.$$
    \end{proof}

    \begin{Prop}\label{senpolytwistedbychar}
    	Let $D$ be a $(\varphi,\Gamma_K)$-module over $A$, and let $\delta: K^\times \rightarrow A^\times$ be a continuous character, of weight $\mathrm{wt}$. Then
    	$$\mathrm{Sen}_{D(\delta)}(T) = \mathrm{Sen}_{D(\delta)}(T-\mathrm{wt}(\delta)).$$	
    \end{Prop}
    
    \begin{proof}
        The assertion follows from remark \ref{rmkforsen}(4) and the following general fact: 
    	
    	For two $(\varphi,\Gamma_K)$-module $D_1,D_2$, the operator $\Theta_{\mathrm{Sen}}$ on $D_1\otimes D_2$ satisfies the Leibniz's rule
    	$$\Theta_{\mathrm{Sen}}(v\otimes w) = v\otimes \Theta_{\mathrm{Sen}}(w) + \Theta_{\mathrm{Sen}}(v)\otimes w.$$
    	Indeed, we only need to show the Leibniz's rule for $\log(\gamma)$ for an element $\gamma$ in $\Gamma_K$ close to $1$ enough. Let $\alpha:= \gamma-1$. Assume 
    	$$(X+Y+XY)^{n} = \sum_{i,j\in\mathbb{N}^{2}} b(n,i,j)X^iY^j,$$
    	for some integer $b(n,i,j)$. By induction, one can show that
    	$$\alpha^n(v\otimes w) = \sum_{i,j\in\mathbb{N}^{2}} b(n,i,j)\alpha^i(v)\otimes\alpha^j(w).$$
    	Then if we write
    	$$a(i,j) = \sum_{n\in\mathbb{N}\setminus\{0\}} (-1)^{n-1}\frac{b(n,i,j)}{n},$$
    	one has 
    	$$\log(\gamma)(v\otimes w) = \sum_{i,j\in\mathbb{N}^2} a(i,j)\alpha^i(v)\otimes\alpha^j(w).$$
    	On other hand, one has
    	\begin{equation*}
    		a(i,j)=\left\{
    		\begin{array}{cc}
    			1/i, & j= 0\\
    			1/j, & i= 0\\
    			0,          & \mathrm{otherwise}
    		\end{array}
    		\right.
    	\end{equation*} 
    via consider the equation of formal power series of the equation
    $$\log(1+Y+XY) = \log X+\log Y.$$
    It follows that $\log(\gamma)(v\otimes w) = \log(\gamma)(v)\otimes w + v\otimes \log(\gamma)(w)$.
    \end{proof}
   
    \begin{Cor}\label{Cor:SenpolyforquasideRham}
    Let $D$ be a rank $d$ quasi-deRham $(\varphi,\Gamma_K)$-module over $A$ of form $\mathbf{D}(M)(\delta)$. Suppose that $\mathrm{wt}(\delta) = (\mathrm{wt}_\sigma(\delta))_{\sigma}$ in $\prod\limits_{\sigma: K\hookrightarrow C} A$, and $M$ is of filtration weight $(\underline{k}_{\sigma})_{\sigma:K\inj C}$, where $\underline{k}_{\sigma} =(\underline{k}_{\sigma,1}\leq\dots \leq \underline{k}_{\sigma,d}))$, then
    $$\mathrm{Sen}_D(T) = (\prod_{1\leq i\leq d}(T-\mathrm{wt}_\sigma(\delta)+k_{\sigma,i}))_{\sigma:K\inj C}.$$
    \end{Cor}
   
   \begin{Def}
   	Let $S=(S_\sigma(T))_\sigma$ be a polynomial in $K\otimes_{\mathbb{Q}_p} A[T]$ of degree $d$, of the form
   	$$S_\sigma(T) = \prod_{1\leq i\leq d}(T - a_{\sigma,i}).$$
   	Let $S'(T)$ be a polynomial of degree $d$ in $K\otimes_{\mathbb{Q}_p} A[T]$. We say $S'(T)\leq S(T)$ if there exist sets of non negative integers $(k_{\sigma_1},\dots,k_{\sigma_d})$ for $\sigma:K\inj C$, such that 
   	$$S'(T) =(S'_\sigma(T) = \prod_{1\leq i\leq d}(T-a_{\sigma,i}-k_{\sigma,i}))_{\sigma:K\inj C}.$$
   	One can easily see that this is indeed a partial order of polynomials in this form.
   \end{Def}

   \begin{Cor}\label{Cor:Senineqforinj}
   Let $A$ be a finite field over $C$. Let $D$ be a quasi-deRham $(\varphi,\Gamma_K)$-module over $A$. Then for any $(\varphi,\Gamma_K)$-submodule  $D'$ of $D$ with the same rank, one has $\mathrm{Sen}_{D'}(T)\leq \mathrm{Sen}_{D}(T)$, and the equality holds if and only if $D'=D$.
   
   Moreover, if we replace the quasi-deRham condition by paraboline, the same property holds. 

   \end{Cor}

   \begin{proof}
   	First, we assume $D$ is quasi-deRham of the form $\mathbf{D}(M)(\delta)$. Then by proposition \ref{senpolytwistedbychar},  we can twisted by $\delta^{-1}$ for $D$ and $D'$, and may assume $D = \mathbf{D}(M)$ for some filtered $(\varphi,N,G_{L/K})$-module. According to \cite[Lem~III.1.3~and~Prop~III.2.4]{Berger2008}, there exist a injective morphism $i:M'\inj M$, such that
   	$$D'\cong \mathbf{D}(M')$$ and the natural inclusion $D'\subseteq D$ comes from $i:M'\inj M$. As $D'$ has the same rank of $D$, the injection $i$ must be an isomorphism on the underlying $L_0\otimes_{\mathbb{Q}_p}A$-module, and $\mathrm{Fil}^jM'\subseteq \mathrm{Fil}^jM$. By the formulation in corollary \ref{Cor:SenpolyforquasideRham}, one has
   	$$\mathrm{Sen}_{D'}(T)\leq \mathrm{Sen}_{D}(T),$$
   	and the equality holds if and only if $\mathrm{Fil}^jM'= \mathrm{Fil}^jM$, which means $i$ is an isomorphism and $D'=D$.
   	
   	Then for the case that $D$ is paraboline, note that 
   	$$\mathrm{Sen}_D(T) = \prod_{i\in\mathbb{Z}} \mathrm{Sen}_{\mathrm{gr}_iD}(T)$$
   	and 
   	$$\mathrm{Sen}_{D'}(T) = \prod_{i\in\mathbb{Z}} \mathrm{Sen}_{\mathrm{gr}_iD'}(T),$$
   	where $\mathcal{F}_iD'= D'\cap \mathcal{F}_iD$. Hence $\mathrm{Sen}_{\mathrm{gr}_iD'}(T)\leq \mathrm{Sen}_{\mathrm{gr}_iD}(T)$ implies
   	$$\mathrm{Sen}_{D'}(T)\leq \mathrm{Sen}_{D}(T),$$
   	and the equality holds if and only if  $\mathrm{gr}D'=\mathrm{gr}D$ for every $i$, which is equivalent to $D'=D$.
   \end{proof}

   \section{The Bernstein Center}
   
   Let $L$ be a field of characteristic $0$ and let $K$ be a $p$-adic local field. Let $d$ be a positive integer, and let $\sigma$ be a supercuspidal representation of $\mathrm{GL}_d(K)$ over $L$, and let $\varrho:=\mathrm{rec}(\sigma)$ be the corresponding WD representation of $K$ via the local Langlands correspondence, which is of type $\tau$.
   
   We assume that $L$ contains $\mu_{e_\tau}(\bar{L}):=\{e_\tau\mathrm{th-roots\ of \ unity\ in\ } \bar{L}\}$. Through out this section, we always denote $\mathrm{GL}_d(K)$ by $\mathrm{GL}_d$ for short.
   
   \begin{Lemma}\label{Lem:InvUnramChar}
   	Let $\psi: \mathrm{GL}_d\rightarrow L^\times$ be an unramified character of $\mathrm{GL}_d$, such that $\psi\otimes\sigma\cong\sigma$. Then there exists an (unique) element $\zeta\in\mu_{e_\tau}(L)$, such that
   	\begin{align*}
   		\psi: \mathrm{GL}_d & \rightarrow L^\times\\
   		B & \mapsto \zeta^{v_K(\det(B))}.
   	\end{align*}
   \end{Lemma}
   
   \begin{proof}
   	Note that $\psi$ must of the form $B\mapsto \zeta^{v_K(\det(B))}$ as it is unramified. Hence it is enough to show that $\psi(B)$ is a $e_\tau$th-root of unity for one (and hence any) matrix $B\in \mathrm{GL}_d$ such that $v_K(\mathrm{det}(B))=1$. 
   	
   	It follows from the local Langlands correspondence that $\psi\otimes\sigma\cong\sigma$ if and only if $\mathrm{rec}(\psi)\otimes\varrho\cong \varrho$. Then $\mathrm{rec}(\psi)$ restricted on $W_\tau$ is trivial by lemma \ref{Lem:Twistingbychar}, which means if $e_\tau|v_K(\mathrm{det}(B))$, then $\psi(B)=1$. Hence $\psi(B)$ is an $e_{\tau}$th-root of unity for any matrix $B\in \mathrm{GL}_d$ such that $v_K(\mathrm{det}(B))$=1 . 
   \end{proof}
   
   Choose and fix a generator $\zeta$ of $\mu_{e_\tau}(L)$, and let $\psi_\zeta$ denote the unramified character
   $$B\mapsto \zeta^{v_K(\det(B))}.$$
   Then the group $\mathcal{G}_\pi:=\{\psi{\ }|\mathit{\ unramified\ character\ such\ that\ }\psi\otimes\pi\cong\pi \}$ is generated by $\psi_\zeta$. Choose an isomorphism $\nu_\zeta:\psi_\zeta\otimes\sigma\xrightarrow{\sim} \sigma$. Then $\nu_\zeta$ induces an isomorphism (also denote by $\nu_\zeta$ in the understandable way)
   $$\nu_\zeta: \psi_\zeta^i\otimes\sigma\xrightarrow{\sim} \psi_\zeta^{i-1}\otimes\sigma.$$ 
   Then $\nu_\zeta^{e_\tau}: \sigma= \psi_\zeta^{e_\tau}\otimes\sigma\xrightarrow{\sim} \sigma$ is an automorphism of $\pi$. It follows that $\nu_\zeta^{e_\tau}$ is a scalar $c_{\nu}\in L$.
   
   Let $\Pi_\sigma$ denote the family of $\mathrm{GL}_d$-representation $L[T,T^{-1}] \otimes_L \sigma$, where $L[T,T^{-1}]$ is the universal family of characters:
   \begin{align*}
   	\mathrm{GL}_d(K) & \rightarrow L[T,T^{-1}]\\
   	B & \mapsto T^{v_K(\mathrm{det}(B))}.
   \end{align*}
   Then by \cite[Lem.22 and Prop.27]{Bernstein1992}, the Bernstein block $\mathrm{Rep}_{[\sigma]}(\mathrm{GL}_d)$ is isomorphic to the category of right $\mathrm{End}_{\mathrm{GL}_d}(\Pi_\sigma)$-modules.
   
   \begin{Theo}\label{BernsteinCenter}
   	Let $\mathfrak{Z}_\sigma:= L[T^{e_\tau},T^{-e_\tau}]$. We have a natural inclusion $$\mathfrak{Z}_{\sigma}\subseteq L[T^{},T^{-1}] \subseteq \mathrm{End}_{\mathrm{GL}_d}(\Pi_\sigma),$$
   	such that the center of $\mathrm{End}_{\mathrm{GL}_d}(\Pi_\sigma)$ is $\mathfrak{Z}_\sigma$, i.e., the Bernstein center of $\mathrm{Rep}_{[\sigma]}(\mathrm{GL}_d)$ is $\mathfrak{Z}_\sigma$.
   \end{Theo}
   
   \begin{proof}
   	Note that $\Pi_\sigma$ is a free $L[T^{},T^{-1}]$-module and the $\mathrm{GL}_d$-action is $L[T^{},T^{-1}]$-linear. Hence we have a natural inclusion
   	$$L[T^{},T^{-1}] \subseteq \mathrm{End}_{\mathrm{GL}_d}(\Pi_\sigma).$$
   	
   	Let $\nu:= \nu_{\zeta}$, and let $F:= L[T^{},T^{-1}]$. The following $\mathrm{GL}_d$-equivariant $L$-linear bijective map
   	\begin{align*}
   		\sigma \otimes_L F &\rightarrow (\sigma\otimes \psi_{\zeta}) \otimes_L F\\
   		x\otimes T^{m} &\mapsto \zeta^{m}x\otimes T^{m}
   	\end{align*}
   induces an $F$-module structure on $(\sigma\otimes \psi_{\zeta}) \otimes_L F$ defined by:
   $$f(T)(x\otimes g(T)) := x\otimes f(\zeta T)g(T).$$
   Then $\nu$ induces an automorphism (also denoted by $\nu$)
   \begin{align*}
   	\nu:\sigma \otimes_L F & \rightarrow (\sigma\otimes \psi_{\zeta}) \otimes_L F \cong \sigma \otimes_L F\\
   x\otimes g(T)	&\mapsto \nu(x) \otimes g(T)
   \end{align*}   	
   of $\mathrm{GL}_d$-representation, such that for any $x\in \sigma$ and $f(T),g(T)\in F$, one has
   \begin{align*}
   	\nu \circ f(T) (x\otimes g(T)) &= \nu (x\otimes f(T)g(T))\\
   	& = \nu(x)\otimes f(T)g(T)\\
   	& = f(\zeta^{-1}T)(\nu(x)\otimes g(T))\\
   	& = f(\zeta^{-1})\circ\nu(x\otimes g(T))
   \end{align*}
    
   	 It follows that $\nu\circ f(T) = f(\zeta^{-1}T)\circ \nu$.
   	
   	By the computation in \cite[Prop.28]{Bernstein1992}, one has
   	$$\mathrm{End}_{\mathrm{GL}_d}(\Pi_\sigma)\cong \bigoplus_{0\leq i\leq e_\tau-1} F\nu^{i}$$
   	such that $\nu\circ f(T) = f(\zeta^{-1}T)\circ\nu$ for any $f(T)\in F$, and $\nu^{e_\tau} = c_{\nu}$.
   	
   	As $T^{e_{\tau}}$ commutes with $\nu$, it follows that the center of $\mathrm{End}_{\mathrm{GL}_d}(\Pi_\sigma)$ contains $\mathfrak{Z}_\pi$. On other hand, suppose that 
   	$$x= (f_0(T),f_1(T)\nu,\dots,f_{e_\tau-1}(T)\nu^{e_\tau-1})$$
   	is an element in the center of $\mathrm{End}_{\mathrm{GL}_d}(\Pi_\sigma)$. Then $x\circ T = T\circ x$ implies
   	$$f_i(T)T = f_i(T)\zeta^{-i}T,$$
   	which means $f_i(T) = 0$ for $i\neq 0$. Then after considering the equation $$x\circ \nu = \nu \circ x,$$
   	one has $f_0(T) \in\mathfrak{Z}_\sigma$. Our conclusion follows.
   \end{proof}

   Let $G^{0}$ denote the subgroup of $\mathrm{GL}_d$ generated by all compact subgroups (i.e. $G^0=\{B\in \mathrm{GL}_d|\ \ \mathrm{det}(B)\in\mathcal{O}_K^\times\}$), and let $Z$ denote the center of $\mathrm{GL}_d$. According to \cite[Prop~25]{Bernstein1992}, $\sigma|_{G^{0}}$ is semi-simple of finite length and each irreducible component is stable under the action of $Z$. Then let $\sigma^{0}\subset \sigma$ be an irreducible component as a $G^{0}$ representation and let $G_{\sigma}\subseteq \mathrm{GL}_d$ be the maximal normal subgroup such that $\sigma^0$ is stable under $G_{\sigma}$. As $G_{\sigma}$ contains $ZG^{0}$, then $G_{\sigma}$ has finite index in $\mathrm{GL}_d$.
   
   \begin{Prop}\label{G^0Induction}
   	The natural map
   	$$\mathrm{Ind}_{G_{\sigma}}^{\mathrm{GL}_d} \sigma^{0} \rightarrow \sigma$$
   	induced by $\sigma^{0}\inj \sigma$ is an isomorphism of $\mathrm{GL}_d$ representation. Moreover, one has $e_{\tau} = [\mathrm{GL}_d:G_{\sigma}]$. In particular, $e_\tau | d$.
   \end{Prop}  

   \begin{proof}
   For each $B\in\mathrm{GL}_d$, let $(\sigma^{0})^{B}$ denote the $B$-conjugation of $\sigma^{0}$ as $G_{\sigma}$ representation, i.e., $C^{B}(x) := B^{-1}CB(x)$ for each $x\in \sigma^{0}$, and $C\in G_{\sigma}$, and $C^{B}$ denotes the $B$-conjugate action of $C$ in $(\sigma^{0})^{B}$.
   
   Let $\{B_1,\dots,B_e\}$ be a set of representatives of $\mathrm{GL}_d/ G_{\sigma}$, where $e := [\mathrm{GL}_d: G_{\sigma}]$. Then one has $(\sigma^{0})^{B_i}\cong B_i(\sigma^{0})$ via the map $x\mapsto B(x)$, and
   $$B_i(\sigma^{0})\cap B_j(\sigma^{0}) = \emptyset,$$
   for $1\leq i\neq j\leq e$, by the definition of $G_{\sigma}$ and note that $\sigma^{0}$ is irreducible. It follows that 
   $$\sigma = \bigoplus_{1\leq i\leq e} B_i(\sigma^{0})\cong \mathrm{Ind}_{G_{\sigma}}^{\mathrm{GL}_d} \sigma^{0}.$$
   
   We write $G_{\sigma}':= \{B\in\mathrm{GL}_d\ |\ (\sigma^{0})^{B} \sim \sigma^{0}\}$, and claim that $G_{\sigma}'\cong G_{\sigma}$. Indeed, it is obvious that $G_\sigma\subseteq G_{\sigma}'$. If $G_{\sigma} \neq G_{\sigma}'$, one can extend the $G_{\sigma}$-action on $\sigma^{0}$ to a $G_{\sigma}'$-action as $G_{\sigma}'/G_{\sigma}$ is cyclic (this extension is similar as we do for WD representations). We denote by $\tilde{\sigma}^{0}$ the extended $G_{\sigma}'$-action. Then the natural map
   $$\mathrm{Ind}_{G_\sigma}^{G_\sigma'}\sigma^{0} \rightarrow \tilde{\sigma}^0$$
   is surjective but not injective. This implies that the natural map
   $$\sigma \cong \mathrm{Ind}_{G_\sigma}^{\mathrm{GL}_d}\sigma^{0} = \mathrm{Ind}_{G_\sigma'}^{\mathrm{GL}_d}(\mathrm{Ind}_{G_\sigma}^{G_{\sigma}'}\sigma^{0}) \rightarrow \mathrm{Ind}_{G_\sigma'}^{\mathrm{GL}_d}\tilde{\sigma}^{0}$$
   is surjective but not injective, which contradicts to the irreducibility of $\sigma$.
   
   Now assume $\delta$ is an unramified character of $\mathrm{GL}_d$ over $L$. Then $\sigma\otimes\delta\cong\sigma$ if and only if 
   there exists a $G_\sigma$-equivariant map (by Frobenius reciprocity)
   $$\sigma^{0} \inj \mathrm{Ind}_{G_{\sigma}}^{\mathrm{GL}_d} (\sigma^{0}\otimes\delta|_{G_\sigma}).$$
   As $\mathrm{Ind}_{G_{\sigma}}^{\mathrm{GL}_d} (\sigma^{0}\otimes\delta|_{G_\sigma}) \cong \bigoplus_{1\leq i\leq e} ((\sigma^{0})^{B_i}\otimes\delta|_{G_\sigma})$, then the image of $\sigma^{0}$ must be in $(\sigma^{0})\otimes\delta|_{G_\sigma}$. It follows that    $$\sigma\otimes\delta\cong\sigma$$
   if and only if $\delta|_{G_\sigma}$ is trivial, which implies $e=e_\tau$. It follows that $$e_\tau| d$$ as $d=[\mathrm{GL}_d:ZG^0]$.
   
   \end{proof}

   \begin{Rmk} \label{Rmk:TwistedbycharofBcenter}
   	\textit{  }
   	\begin{enumerate}
   		\item For any closed point $\delta\in\mathrm{Spec}(\mathfrak{Z}_\sigma)$ (i.e. $\delta: \mathfrak{Z}_{\sigma} \rightarrow k_{\delta}^{\times}$), after perhaps enlarging $k_{\delta}$, one can extend $\delta$ to an unramified character $\tilde{\delta}$ of $\mathrm{GL}_d$. By definition, for any two different extensions $\tilde{\delta}$ and $\tilde{\delta}'$, one has $\sigma\otimes\tilde{\delta}\cong \sigma\otimes\tilde{\delta}'$. Hence the notation $\sigma\otimes \delta$ makes sense. One can also describe $\sigma\otimes \delta$ as following:
   		
   		Actually, $\mathrm{Spec}(\mathfrak{Z}_{\sigma})$ can be regarded as the space of unramified character of $G_\sigma$, which means $\delta$ can be regarded as an unramified character of $G_\sigma$. Then
   		$$\sigma\otimes\delta \cong \mathrm{Ind}_{G_\sigma}^{\mathrm{GL}_d}(\sigma^0\otimes\delta).$$
   		
   		One can see that this description is similar as we do for WD representations in lemma \ref{Lem:Twistingbychar}.
   		\item Regarding $\mathrm{Spec}(\mathfrak{Z}_{\sigma})$ as the space of unramified character of $G_\sigma$ gives a group scheme structure of $\mathrm{Spec}(\mathfrak{Z}_{\sigma})$, which is not canonical and depends on our choice of $\sigma$ in the Bernstein block $[\sigma]$. Actually, this is the unique group scheme structure of $\mathrm{Spec}(\mathfrak{Z}_{\sigma})$, which is isomorphic to $\mathbb{G}_{m,L}$ such that for any point $\delta: G_\sigma\rightarrow k_{\delta}^{\times}$ in $\mathrm{Spec}(\mathfrak{Z}_{\sigma})$, the natural action of $\mathfrak{Z}_{\sigma}$ on $\sigma\otimes\delta$ is $\delta$ ($\delta$ can be also regarded as character of $\mathfrak{Z}_\sigma$). 
   	\end{enumerate}
   
    \begin{proof}
    	Part (1) follows from lemma \ref{Lem:InvUnramChar}. 
    	
    	For part(2), it is suffice of prove the claim for the case that $\delta$ is the identity in the group scheme $\mathrm{Spec}(\mathfrak{Z}_{\sigma})$. Note that $\mathrm{Z}_\sigma = L[T^{\pm e_\tau}]$. Hence it suffices to show that $T^{e_\tau}$ acts trivially on $\sigma$.
    	
    	As the functor $\Hom_{\mathrm{GL}_d}(\Pi_\sigma,-)$ induces an equivalence of categories between 
    	\begin{align*}
    		\mathrm{Rep}_{[\sigma]}(\mathrm{GL}_d) & \rightarrow \{\mathrm{right\ } \mathrm{End}_{\mathrm{GL}_d}(\Pi_\sigma)\mathrm{-modules}\}\\
    		\sigma' &\mapsto \Hom_{\mathrm{GL}_d}(\Pi_\sigma,\sigma' )
    	\end{align*}
    Now we are going to compute $\Hom_{\mathrm{GL}_d}(\Pi_\sigma,\sigma)$. Let $F:= L[T^{\pm 1}]$. As
    $$\Pi_{\sigma}\otimes_F F/(f_{\zeta}(T))\cong \bigoplus_{0\leq i\leq e_\tau-1}\sigma\otimes\psi_{\zeta}^{i},$$
    where $f_{\zeta}(T) = (T-1)(T-\zeta)\cdots(T-\zeta^{e_\tau-1})$. Then
    \begin{align*}
    	\Hom_{\mathrm{GL}_d}(\Pi_\sigma,\sigma) & \inj \Hom_{\mathrm{GL}_d}(\Pi_\sigma\otimes_F F/(f_\zeta(T)),\sigma) \\
    	& = \bigoplus_{0\leq i\leq e_\tau-1}\Hom_{\mathrm{GL}_d}(\sigma\otimes \psi^{i}_{\zeta},\sigma) \\
    	& = \bigoplus_{0\leq i\leq e_\tau-1} L \nu_{\zeta}^i,
    \end{align*}
    In particular, $\dim_L\Hom_{\mathrm{GL}_d}(\Pi_\sigma,\sigma) \geq e_\tau$. On the other hand, it is easy to see that 
    $$\Pi_\sigma = (\mathrm{Ind}_{G^{0}}^{\mathrm{GL}_d} L) \otimes_L \sigma \cong  c\mathrm{-Ind}_{G^{0}}^{\mathrm{GL}_d} (\sigma|_{G^0}),$$
    where $L$ is regarded as the trivial representation of $G^{0}$.
    Hence by the proof of \ref{G^0Induction}, one has
    \begin{align*}
    	\Hom_{\mathrm{GL}_d}(\Pi_\sigma,\sigma) & = \Hom_{G^{0}}(\sigma|_{G^{0}},\sigma|_{G^{0}}) \\
    	& = \bigoplus_{1\leq i\leq e_\tau}\Hom_{G^{0}}((\sigma^0)^{B_i},(\sigma^0)^{B_i})\\
    	& = L^{\oplus e_\tau}.
    \end{align*}
    It follows that $\dim_L\left(\Hom_{\mathrm{GL}_d}(\Pi_\sigma,\sigma)\right) = e_\tau$, and
    $$\Hom_{\mathrm{GL}_d}(\Pi_\sigma,\sigma) = \Hom_{\mathrm{GL}_d}(\Pi_\sigma\otimes_F F/(f_\zeta(T)),\sigma).$$
    Note that $f_{\zeta}(T) = T^{e_\tau}-1$, then the right action of $T^{e_\tau}$ on  $\Hom_{\mathrm{GL}_d}(\Pi_\sigma,\sigma)$ is trivial. Namely, $T^{e_\tau}$ acts trivially on $\sigma$.
    \end{proof}
   	
   \end{Rmk}

  	\bibliography{para}
\end{document}